\documentclass[12pt]{article}
\usepackage[utf8]{inputenc}
\usepackage{cite}
\usepackage{amsfonts}
\usepackage{amsmath}
\usepackage{amsthm}
\usepackage{enumitem}
\usepackage{amssymb}
\usepackage{caption}
\usepackage{graphicx}
\usepackage{mathrsfs}
\usepackage{multirow}
\usepackage{xcolor}
\usepackage{setspace}
\usepackage{booktabs}
\usepackage{latexsym,mathdots,array,extarrows}
\usepackage{algpseudocode,float}
\usepackage{lipsum}
\usepackage{subfig}
\usepackage{pifont}
\usepackage{diagbox}

\usepackage{amsmath,amsfonts,amsthm,amssymb}
\usepackage{algorithm2e}
\usepackage{bm,dsfont}
\usepackage{booktabs}
\usepackage{graphicx}
\usepackage{array}
\usepackage{xcolor}
\usepackage{sidecap}
\usepackage{listings}
\usepackage{times}
\usepackage{url}
\usepackage{hyperref}

\usepackage{geometry}
\newtheorem{definition}{Definition}
\newtheorem{theorem}{Theorem}
\newtheorem{prop}{Proposition}
\newtheorem{lemma}{Lemma}
\newtheorem{assumption}{Assumption}

\newtheorem*{remark}{Remark}

\allowdisplaybreaks[4]

\title{Wasserstein Generative Adversarial Uncertainty Quantification
in Physics-Informed Neural Networks}

\author{Yihang Gao\thanks{Department of Mathematics, The University of Hong Kong,
Pokfulam, Hong Kong ({gaoyh@connect.hku.hk}).} \and 
Michael K. Ng\thanks{Department of Mathematics, The University of Hong Kong,
Pokfulam, Hong Kong ({mng@maths.hku.hk}). 
Research is supported by HKRGC GRF 12300218, 12300519, 17201020 and 17300021.}}

\date{}

\begin{document}

\maketitle
\begin{abstract}
In this paper, we study a physics-informed algorithm for Wasserstein Generative Adversarial Networks (WGANs) for uncertainty quantification in solutions of partial differential equations. By using groupsort activation functions in adversarial network discriminators, network generators are utilized to learn the uncertainty in solutions of partial differential equations observed from the initial/boundary data. Under mild assumptions, we show that the generalization error of the computed generator converges to the approximation error of the network with high probability, when the number of samples are sufficiently taken. According to our established error bound, we also find that our physics-informed WGANs have higher requirement for the capacity of discriminators than that of generators. Numerical results on synthetic examples of partial differential equations are reported to validate our theoretical results and demonstrate how uncertainty quantification can be obtained for solutions of partial differential equations and the distributions of  initial/boundary data. However, the quality or the accuracy of the uncertainty quantification theory in all the points in the interior is still the theoretical vacancy, and required for further research.
\end{abstract}

\section{Introduction}
Deep Learning becomes really popular in recent several years due to its extraordinary applications in computer vision \cite{krizhevsky2012imagenet, arjovsky2017wasserstein,goodfellow2014generative}, natural language processing \cite{vinyals2015grammar,bowman2015generating,bahdanau2014neural}, and healthcare \cite{miotto2018deep,liang2014deep,wang2014similarity} etc. Its excellent performance brings much confidence to researchers about the potential applications in scientific computing. In particular, it would be useful to deal with curse of dimensionality 
in solutions of partial differential equations.

Solving partial differential equations by deep learning is first studied by Psichogios \cite{psichogios1992hybrid} and Lagaris et al. \cite{lagaris1998artificial}. 
Recently, 
Raissi et al. \cite{raissi2019physics} proposed the physics-informed neural networks (PINNs) which regard residuals of differential equations as regularization terms. PINNs behave well in solving both the forward and inverse problems of various kinds of partial differential equations. 
Besides introducing physical information by residuals, classical methods in solving PDEs coexist well with deep neural networks, providing another interesting ways to apply deep learning in scientific computing. 
Sirignano et al. \cite{sirignano2018dgm} designed a Deep Galerkin Method algorithm to solve partial differential equations. 
Neural networks replace linear combinations of basis functions in classic Galerkin method due to advantages of data-driven methods and universal approximation capabilities of deep neural networks.
Deep learning has been widely applied in solving more complicated and specific problems, e.g., fractional PDEs \cite{pang2019fpinns}, stochastic PDEs \cite{zhang2019quantifying, yang2019adversarial,chen2021learning} as well as high dimensional problems \cite{han2020algorithms} etc.

Like most of deep learning models, PINNs suffer from training failure \cite{wang2020and,mishra2020estimates}, slow convergence\cite{jagtap2020adaptive,luo2020two,jagtap2020locally}, and curse of dimensionality \cite{darbon2020overcoming,jentzen2018proof,hutzenthaler2020overcoming,bolcskei2019optimal}. Jagtap et al. \cite{jagtap2020adaptive} introduced learnable parameters in activation functions to adaptively control the learning rate for PINNs. Wang et al. \cite{wang2020and} proposed a kind of adaptive hyper-parameters selection algorithm that can not only accelerate the convergence of PINNs but also improve its prediction accuracy.

\subsection{Uncertainty Quantification}

Uncertainties (e.g., experimental observations) are usually inevitable that the system may not strictly follow mathematical formulations due to the external interference and lack of suitable experimental conditions. Researchers may be most interested in how the system is affected without the need to repeat experiments and observe results numerously. Traditional methods like Gaussian process \cite{graepel2003solving, raissi2018numerical, bilionis2016probabilistic}, 
Monte Carlo sampling\cite{barth2011multi}, as well as statistical Bayesian inference \cite{yang2021b,stuart2010inverse,zhu2018bayesian} require strict assumptions and priors for modeling that they are only applicable in specific and limited problems. 
Yang et al. \cite{yang2021b} used Bayesian techniques to estimate solutions of PDEs with random boundary data, and prior information (e.g., Gaussian distribution) for the distribution of boundary data is required.  However, the explicit distribution of our collected data are usually unknown and incorrect prior may lead to worse results or even training failure. Zhang et al. \cite{zhang2019quantifying} adopted deep neural networks as a surrogate to 
solutions in polynomial chaos expansions in solving stochstic PDEs. 
In their setting, priors for noises (schochasticity and randomness) are also prerequisite. 

The recent explosive growth of deep generative models (e.g., VAEs \cite{kingma2013auto} and GANs \cite{goodfellow2014generative,arjovsky2017wasserstein}) greatly reduce prior information and 
formulations, and achieve empirical successes especially in image \cite{goodfellow2014generative,arjovsky2017wasserstein,odena2017conditional} and sentence generation\cite{bowman2015generating}. Yang et al. \cite{yang2019adversarial} applied GANs with PINNs to 
solve PDEs with uncertainty. Their method performs well empirically. 
Yang et al. \cite{yang2021capacity} recently proved the capacity of deep ReLU neural networks in approximating distributions in Wasserstein distance and maximum mean discrepancy. 
Wasserstein generative adversarial networks (WGANs) achieve much better results both empirically and theoretically \cite{arjovsky2017wasserstein} than traditional GANs especially for the data lying on low dimensional manifolds \cite{gulrajani2017improved}. 

\subsection{The Contribution}


The main aim of this paper is to study 
WGANs with PINNs for 
for uncertainty quantification in solutions of partial differential equations. 
We develop a probabilistic model that can learn the uncertainty (noise) of the boundary/initial data by WGANs, and propagates it to the interior domain with physical constraints by PINNs. 
By using groupsort activation functions in adversarial network
discriminators, network generators 
are utilized to learn the uncertainty in solutions of partial differential equations 
observed from the initial/boundary data. We analyze  
the generalization error of the computed generator, and show that the exact loss 
converges to the approximation error 
(the minimal error among all generators in the pre-defined class)
with high probability, 
when the number of samples are sufficiently taken. 
Moreover, we find in our established error bound that our physics-informed WGANs have higher requirement for the capacity of discriminators than that of generators. 

The outline of this paper is given as follows. In Section \ref{not_prel}, the reviews of WGANs and PINNs are introduced and the proposed generative model is presented. In Section \ref{conv}, we show the convergence of our proposed WGANs with PINNs. In Section \ref{num_results}, we present experimental results for uncertainty quantification of PDEs solutions to demonstrate our theoretical results. Finally, some concluding remarks are given in Section 5.

\subsection{Some Limitations}

It is necessary to mention some potential limitations in our study.


Firstly, we miss the theoretical guarantee for the quality of uncertainty propagation. In other words, we do not theoretically analyze the role of physics-informed (PINNs) term in uncertainty propagation. The accuracy of interior data matching is theoretically guaranteed for some classes of deterministic PDEs if with boundary/initial samples and governing equations \cite{shin2020convergence}. However, for the proposed probabilistic models, we encountered theoretical issues and technical difficulties to derive similar results. This also means that the propagated uncertainty can be entirely wrong and may not even converge to a correct one in the interior. Experimental results of the better distribution matching on the boundary than in the interior in section \ref{num_results} further imply the concern. We left it as a future research work for probabilistic models, which is interesting and unsolved, to the best of our knowledge.

Secondly, the performance of the proposed model on more complicated problems. We test the model on some well-known PDEs and show its capability in uncertainty quantification. However, strictly speaking, the estimation for uncertainty quantification in the interior domain is somewhat less than satisfactory even in some simple problems with non-Gaussian and nonlinear uncertainty distributions. Therefore, it is still an open problem and is left as a future research study to provide satisfactory uncertainty quantification in non-Gaussian multi-modal distributions.

\subsection{Notations}

Throughout the paper, we combine the spatial coordinates and the temporal coordinate together, simply denoted by $\mathbf{x}$. For two positive real numbers or functions $A$ and $B$, the term $A \vee B $ is equivalent to $\max \{A,B\}$ while $A \wedge B$ is equivalent to $\min\{A,B\}$. For $\mathbf{V}=(\mathbf{V}_{i,j})$,
we let $||\mathbf{V}||_{\infty}=\sup_{|| {\bf y}||_{\infty}=1}||\mathbf{V} {\bf y}||_{\infty}$. We 
also use the $(2,\infty)$ norm of $\mathbf{V}$, i.e., $||\mathbf{V}||_{2, \infty}=\sup_{||{\bf y}||_{2}=1}
||\mathbf{V} {\bf y}||_{\infty}$.

\section{The Generative Adversarial Model for Uncertainty Quantification}
\label{not_prel}

\subsection{Wassertein Generative Adversarial Networks} 

Different from the original model in \cite{goodfellow2014generative}, Generative Adversarial Networks (GANs) can be formulated in a more general form:
\begin{equation}
\label{general_GAN}
    \min_{g_{\theta} \in \mathcal{G}} \max_{f_{\alpha} \in \mathcal{F}} \mathbb{E}_{z \sim \pi} f_{\alpha}(g_{\theta}(z)) - \mathbb{E}_{x \in \nu} f_{\alpha}(x)
\end{equation}
where $\mathcal{G}$ and $\mathcal{F}$ are the generator class and the discriminator class respectively,
$\pi$ is the source distribution and $\nu$ is the target distribution to be approximated. 

Wasserstein GANs first proposed by Arjovsky et al. \cite{arjovsky2017wasserstein} adopts 1-Lipschitz functions as discriminators, i.e., 
$$
\mathcal{F} = \{f: \|f\|_{Lip} \leq 1\},
$$ 
where $\|f\|_{Lip}$ is the Lipschitz constant of $f$
and both the empirical and theoretical results show better performance of WGANs than GANs \cite{arjovsky2017wasserstein, arora2018gans,arora2017generalization, gulrajani2017improved}. 
In many WGANs applications, 
ReLU feedforward neural networks are used to approximate 1-Lipschitz functions.
However, they are not guaranteed to be 1-Lipschitz and the ReLU activation function is norm vanishing in half planes \cite{yarotsky2017error, tanielian2021approximating}. 
To overcome such shortcoming, Anil et al. \cite{anil2019sorting} used groupsort as the activation function which is norm preserving, and Tanielian et al. \cite{tanielian2021approximating} recently proved the approximation ability of groupsort 
neural networks to 1-Lipschitz functions. A groupsort neural network is defined as
\begin{equation}
\label{groupsortnn}
 \text{GS}(\mathbf{x}|W_f,D_f,\alpha) = \mathbf{V}_{D_f} \cdot  \sigma_{group_{size}}(\mathbf{V}_{D_f-1} \cdot \sigma_{group_{size}}(\cdots \sigma_{group_{size}}(\mathbf{V}_0 \cdot \mathbf{x} + \mathbf{v}_0)+ \cdots )+\mathbf{v}_{D_f-1})+\mathbf{v}_{D_f}    
\end{equation}
with constraints
\begin{equation}
\label{constraints}
\begin{aligned}
\left\|\mathbf{V}_{1}\right\|_{2, \infty} \leqslant 1 \hspace{1em} \text{ and  }& \hspace{1em} \max \left(\left\|\mathbf{V}_{2}\right\|_{\infty}, \ldots,\left\|\mathbf{V}_{D_f}\right\|_{\infty}\right) \leqslant 1 
\end{aligned}    
\end{equation}
where $\alpha=\{\mathbf{V}_i,\mathbf{v}_i\}_{i=0}^{D_f}$ denotes parameters of a groupsort neural network, $W_f$ is the width (number of neurons) of each layer, $D_f$ is the depth (number of layers) and $\sigma_{group_{size}}$ is the groupsort activation function with grouping size ($group_{size}$). A 
groupsort neural networks defined in (\ref{groupsortnn}) with constraints (\ref{constraints}) is proven to be 1-Lipschitz \cite{tanielian2021approximating}.  In the paper, we choose $group_{size}=2$ and the discriminator class can be 
defined as 
\begin{equation}
\label{discriminator_class}
    \mathcal{F}_{GS} = \{f_{\alpha}(\mathbf{x})=\text{GS}(\mathbf{x}|W_f,D_f,\alpha): \text{GS}(\mathbf{x}|W_f,D_f,\alpha) \text{ is a neural network of form (\ref{groupsortnn})} \}
\end{equation}

\subsection{Physics-Informed Neural Networks}

Physics-informed neural networks (PINNs) introduced by Raissi et al. \cite{raissi2019physics} is a data driven machine learning method to solve partial differential equations. 
For a partial differential equation
\begin{align}
    & \mathcal{L}\mathbf{u}(\mathbf{x}) = \mathbf{b}(\mathbf{x}) \hspace{1em} x \in \Omega\\
    & \mathcal{B}\mathbf{u}(\mathbf{x}) = \mathbf{c}(\mathbf{x})\hspace{1em} x \in \Gamma = \partial \Omega
\end{align}
where $\mathcal{L}$ is the differential operator in the interior domain and $\mathcal{B}$ is the operator on the boundary (includes initial conditions). Given the observed 
data $\{(\mathbf{x}_i, \mathbf{b}_i)\}_{i=1}^{k}$
and $\{(\mathbf{x}_i, \mathbf{c}_i)\}_{i=1}^{n}$  for both the differential equation and the boundary respectively, the empirical loss function for PINNs can be given by 
$$
\lambda \cdot \frac{1}{k} \sum_{i=1}^{k} \| \mathcal{L}g_{\theta}(\mathbf{x}_i) - \mathbf{b}_i \|^2 + \frac{1}{n} \sum_{i=1}^{n} 
\| \mathcal{B}g_{\theta}(\mathbf{x}_i) - \mathbf{c}_i \|^2
$$
where $g_{\theta}$ is the predefined deep neural networks parametrized by parameters $\theta$, $\mathcal{L}g_{\theta}$ is computed by automatic differentiation \cite{baydin2018automatic} and 
$\lambda$ is a positive number to balance the term for interior data points and 
the term for boundary data points.
Also the first term for differential equations acts as a regularization term for physical constraints. For Dirichlet boundary conditions, the operator $\mathcal{B}$ is the identity operator, meaning that we have observed solutions on the boundary. The optimization problem can be effectively and efficiently solved by stochastic gradient descent, see for instance \cite{kingma2014adam,raissi2019physics,liu1989limited}. 
The core idea of PINNs is to introduce the physical laws/information (differential equations) into the loss function as a regularization term to constrain the training process of neural networks. 

In the following discussion, we mainly focus on PDEs with Dirichlet boundary conditions, i.e.
\begin{align}
    & \mathcal{L}\mathbf{u}(\mathbf{x}) = \mathbf{b}(\mathbf{x}) \hspace{1em} \mathbf{x} \in \Omega \\
    & \mathbf{u}(\mathbf{x}) = \mathbf{c}(\mathbf{x}) \hspace{1em} \mathbf{x} \in \Gamma = \partial \Omega
\end{align}
where $\Omega \subset \mathbb{R}^{d}$ is a bounded domain and $\Gamma = \partial \Omega$ is its boundary. Note that we do not distinguish the initial and boundary conditions, which implies that the spatial variable $\mathbf{x}$ here contains both the spatial and the temporal variables of PDEs.
The domain is bounded by $M_x$, i.e., $\|\mathbf{x} \|_{\text{max}} \leq M_x$ for all $\mathbf{x} \in \Omega$. The right hand side $\mathbf{b}(\mathbf{x})$ is bounded in $\Bar{\Omega}$, i.e., $\| \mathbf{b}(\mathbf{x}) \|_2 \leq M_b$ for all $\mathbf{x} \in \Bar{\Omega}$ and $\mathbf{u}(\mathbf{x}) \in \mathbb{R}^{r}$ is the corresponding solution. Note that the proposed model can be studied similarly for the other boundary/initial conditions in PDEs.

Suppose that the following class $\mathcal{G}$ of neural networks are utilized to approximate the solution of PDEs,
\begin{equation}
    \label{generator_class}
    \mathcal{G}=\{g_{\theta}(\mathbf{x})= \text{NN}(\mathbf{x}|W_g,D_g,\theta): \lVert \mathbf{W}_{i} \rVert_{\text{max}} \leq M, 0 \leq i \leq D_g, ||\mathbf{w}_{D_g}||_{\text{max}} \leq M, \theta=\{\mathbf{W}_i,\mathbf{w}_i\}_{i=0}^{D_g} \} 
\end{equation}
where $g_{\theta}=\text{NN}(W_g,D_g,\theta): \mathbb{R}^{d} \to \mathbb{R}^r$ denotes neural networks with $D_g$ hidden layers (depth), $W_g$ neurons in each layer (width), and $\theta$ is the set of parameters. We use $\mathbf{W}_i$ to represent the matrix of linear transform between $(i-1)$-th and $i$-th layer and $\mathbf{w}_i$ is the bias vector.  The architecture of the neural networks is formulated as 
$$
\text{NN}(\mathbf{x}|W_g,D_g,\theta)=\mathbf{W}_{D_g} \cdot \sigma(\mathbf{W}_{D_g-1}\cdot \sigma( \cdots \sigma(\mathbf{W}_0 \cdot \mathbf{x}+\mathbf{w}_0)+ \cdots) +\mathbf{w}_{D_g-1})+\mathbf{w}_{D_g}
$$
where $\sigma$ is the activation function and $\theta=\{\mathbf{W}_i,\mathbf{w}_i\}_{i=0}^{D_g}$ is the set of parameters for the neural network. Due to the smoothness of PDEs solutions, $g_{\theta}$ is usually 
required to approximate solutions in $H^{q}(\Bar{\Omega})$ where $q$ depends on the differentiability of PDEs 
solutions. Several strategies are applied for the smoothness of $g_{\theta}$, usually to activation functions, e.g., 
$\sigma=tanh$ and $\sigma = ReLU^{q}$. Pinkus \cite{pinkus1999approximation} proved that one-hidden layer neural networks with a non-polynomial activation function which is $q$-order differentiable are dense in $C^q(\mathbb{R}^{d})$. In our paper, we adopt $tanh$ as the activation function. 

\subsection{The WGAN-PINNs Model}
\label{proposed_wgan_pinns_model}

Suppose that there exists uncertainty in the observation of boundary data in PDEs, 
i.e., $\mathbf{u}$ on $\mathbf{x} \in \partial \Omega$ can be interpreted as a random variable rather than a deterministic value. Our aim is to understand uncertainity quantification of PDEs solutions.
PINNs are not able to do uncertainty quantification, and they even perform worse especially for noisy data with 
non-Gaussian and spatially dependent noises, see for instance \cite{raissi2019physics, yang2021b}. 

For the generator $g_{\theta}$, besides the spatial (and temporal) coordinates $\mathbf{x}$, 
a random latent variable $\mathbf{z}$ is also input into the model to summarize the uncertainty and 
stochasticity of PDEs solutions.  We aim to obtain a generator that can provide a 
distribution of PDEs solutions well and also follows physical laws governed by given PDEs. 
Here, we study a kind of probabilistic model with physical constraints (i.e., the combination of WGANs and PINNs
into the model): 
\begin{equation}
\label{model_original}
    \begin{split}
        & \min_{g_{\theta} \in \mathcal{G}} \ 
        \mathbf{D}((\mathbf{x},g_{\theta}(\mathbf{x},\mathbf{z})), (\mathbf{x},\mathbf{u})), \hspace{1em} \mathbf{x} \sim p_{\Gamma}(\mathbf{x}), \ \mathbf{z} \sim p(\mathbf{z}) \\
        & s.t. \hspace{1em} \mathcal{L}g_{\theta}(\mathbf{x},\mathbf{z})=\mathbf{b}(\mathbf{x}), \hspace{0.5em} \forall \mathbf{x} \in \Omega  \text{ and } \mathbf{z} \sim p(\mathbf{z})
    \end{split}
\end{equation}
where $p_{\Gamma}(\mathbf{x})$ is the distribution of $\mathbf{x}$ on the boundary $\Gamma$ and $p(\mathbf{z})$ is the prior distribution of the latent variable $\mathbf{z}$. Also $\mathbf{D}$ is used to measure 
the distance between the joint distribution of generated data and PDE solutions. 
With a random latent variable $\mathbf{z}$, generators are able to generate data with uncertainty compatible with our observations and propagate the uncertainty into the interior domain by physics-informed neural networks with differential equations constraints. Compared with most classical methods of uncertainty quantification for PDEs,
our model works without some strong and special prior information. 
Note that when $\mathbf{D}$ is KL divergence or JS divergence, it is a model  proposed 
in \cite{yang2019adversarial}. However, KL divergence has some disadvantages. 
Firstly, it lacks generalization property with respect to finite data \cite{arjovsky2017wasserstein}. 
Moreover, it requires that the supports of two distributions must coincide well, otherwise, their KL divergence is infinitely 
large, see \cite{arjovsky2017towards}. 

In this paper, we propose to employ Wasserstein-1 distance in the model, and the resulting optimization is given by 
\begin{equation} \label{ww}
    \begin{split}
        & \min_{g_{\theta} \in \mathcal{G}} \ \mathbf{Wass}_1((\mathbf{x},g_{\theta}(\mathbf{x},\mathbf{z})), (\mathbf{x},\mathbf{u})), \hspace{1em} \mathbf{x} \sim p_{\Gamma}(\mathbf{x}), \mathbf{z} \sim p(\mathbf{z})  \\
        & s.t. \hspace{1em} \mathcal{L}g_{\theta}(\mathbf{x},\mathbf{z})=\mathbf{b}(\mathbf{x}), \hspace{0.5em} \forall \mathbf{x} \in \Omega  \text{ and } \mathbf{z} \sim p(\mathbf{z})
    \end{split}
\end{equation}

We can further re-formulate (\ref{ww}) into unconstrained optimization problem as follows: 
\begin{eqnarray}
\label{loss_func}
   &  & \min_{g_{\theta} \in \mathcal{G}} \
   \mathbf{Wass}_1((\mathbf{x},g_{\theta}(\mathbf{x},\mathbf{z})), (\mathbf{x},\mathbf{u})) + \lambda \cdot \mathbb{E}_{\mathbf{x}, \mathbf{z} \sim p_{\Omega}(\mathbf{x}),  p(\mathbf{z})} \| \mathcal{L}g_{\theta}(\mathbf{x},\mathbf{z})-\mathbf{b}(\mathbf{x}) \|_2^2 \nonumber \\
    & = & \min_{g_{\theta} \in \mathcal{G}} \max_{\|f \|_{\text{Lip}} \leq 1} \mathbb{E}_{\mathbf{x},\mathbf{z} \sim p_{\Gamma}(\mathbf{x}), p(\mathbf{z})} f(\mathbf{x}, g_{\theta}(\mathbf{x},\mathbf{z})) - \mathbb{E}_{(\mathbf{x},\mathbf{u}) \sim p_{\Gamma}(\mathbf{x},\mathbf{u})} f(\mathbf{x},\mathbf{u})\\
    & &  \hspace{5em} +  \lambda \cdot \mathbb{E}_{\mathbf{x}, \mathbf{z} \sim p_{\Omega}(\mathbf{x}),  p(\mathbf{z})} \| \mathcal{L}g_{\theta}(\mathbf{x},\mathbf{z})-\mathbf{b}(\mathbf{x}) \|_2^2
\end{eqnarray}
where $p_{\Gamma}(\mathbf{x},\mathbf{u})$ is the exact joint distribution of data $(\mathbf{x},\mathbf{u})$ whose marginal for $\mathbf{x}$ is $p_{\Gamma}(\mathbf{x})$ on the boundary, $p_{\Omega}(\mathbf{x})$ is the distribution of variable $\mathbf{x}$ in the interior domain $\Omega$ and the hyper-parameter $\lambda$ balances 
the term $\mathbf{Wass}_1$ distance and the regularization term governed by PDEs residual norm. 
By using the constraints posed in (\ref{generator_class}), for each generator $g \in \mathcal{G}$, we have 
$$
\|g(\mathbf{x},\mathbf{z})\|_2 \leq \sqrt{r} \cdot (M + W_g \cdot M),
$$
$$
\left \|\frac{\partial g_i(\mathbf{x},\mathbf{z})}{\partial \mathbf{x}} \right \|_2 \leq (W_g \cdot M)^{D_g},
$$
$$
\|\mathbf{b}(\mathbf{x})\|_2 \leq M_b,
$$
and all its derivatives are bounded depending on $M$, $W_g$ and $D_g$. Therefore, the bound for $\| \mathcal{L}g(\mathbf{x},\mathbf{z})-\mathbf{b}(\mathbf{x})\|_2^2$ depends on the operator $\mathcal{L}$, $M$, $W_g$, $D_g$ and $M_b$.

In practice, we work on the model with a finite number of data samples: 
$\{(\mathbf{x}_i, \mathbf{z}_i) \}_{i=1}^{m}$, $\{(\mathbf{x}_i, \mathbf{u}_i) \}_{i=1}^{n}$ on the boundary 
and $\{(\mathbf{x}_i, \mathbf{z}_i, \mathbf{b}_i)\}_{i=1}^{k}$ in the interior domain. Therefore, we deal with 
the following empirical loss given:
\begin{eqnarray}
\label{emp_loss_func}
\widehat{\text{Loss}} = 
\min_{g_{\theta} \in \mathcal{G}} \max_{f_{\alpha} \in \mathcal{F}_{GS}} & \hat{\mathbb{E}}^{m}_{\mathbf{x},\mathbf{z} \sim p_{\Gamma}(\mathbf{x}), p(\mathbf{z})}  f_{\alpha}(\mathbf{x}, g_{\theta}(\mathbf{x},\mathbf{z})) - \hat{\mathbb{E}}^{n}_{(\mathbf{x},\mathbf{u}) \sim p_{\Gamma}(\mathbf{x},\mathbf{u})} f_{\alpha}(\mathbf{x},\mathbf{u}) \\
& + \lambda \cdot \hat{\mathbb{E}}^{k}_{\mathbf{x}, \mathbf{z} \sim p_{\Omega}(\mathbf{x}),  p(\mathbf{z})} \| \mathcal{L}g_{\theta}(\mathbf{x},\mathbf{z})-\mathbf{b}(\mathbf{x}) \|_2^2,
\end{eqnarray}
where $\hat{\mathbb{E}}^{m}$ denotes the expectation on the empirical distribution of $m$ i.i.d. samples on random input to generator, 
$\hat{\mathbb{E}}^{n}$ denotes the expectation on the empirical distribution of 
$n$ i.i.d. samples on boundary data, and 
$\hat{\mathbb{E}}^{k}$ denotes the expectation on the empirical distribution of 
$k$ i.i.d. samples on interior data. 
The empirical loss of our model can be solved by learning $g_{\theta}$ and $f_{\alpha}$ 
based on 
stochastic gradient descent algorithm, see \cite{kingma2014adam,liu1989limited}. 

\section{Generalization and Convergence}
\label{conv}

In this section, we theoretically investigates how well the trained generator (obtained from (\ref{emp_loss_func})) is, 
with respect to the exact loss in (\ref{loss_func}). 
In the literature, 
Bai et al. \cite{bai2018approximability} designed some special discriminators with restricted approximability to let the trained generator approximate the target distribution in Wasserstein distance. However, their designs of discriminators are not applicable enough because it can only be applied to some specific statistical distributions (e.g., Gaussian distributions and exponential families) and generators that are invertible or injective neural networks. Liang \cite{liang2018well} proposed a kind of oracle inequality to develop the generalization error of GANs' learning in TV distance. However, several assumptions are required and KL divergence must be adopted to measure the distance between distributions. It may be inappropriate in generalization (convergence) analysis of GANs because many distributions in natural cases lie on low dimensional manifolds. The support of the generated distribution must coincide well with that of the target distribution, otherwise, the KL divergence between these two are infinite, which may be not helpful in generalization analysis.

Here our purpose is to show that the WGAN-PINNs model 
has good generalization and can converge to the best one among the whole generator class given large enough number of training data (i.e., $m$, $n$ and $k$) and strong discriminators (i.e., large enough $W_f$ and $D_f$) with a high probability.
Suppose that $(\Tilde{g}, \Tilde{f})$ is the obtained optimizer of (\ref{emp_loss_func}). 
For any 1-Lipschitz function $f$, we derive the following inequality:
\begin{align*}
    & \mathbb{E}_{\mathbf{x},\mathbf{z} \sim p_{\Gamma}(\mathbf{x}),p(\mathbf{z})}f(\mathbf{x},\Tilde{g}(\mathbf{x},\mathbf{z})) - \mathbb{E}_{(\mathbf{x},\mathbf{u})\sim p_{\Gamma}(\mathbf{x},\mathbf{u})} f(\mathbf{x},\mathbf{u})+ \lambda \cdot  \mathbb{E}_{\mathbf{x},\mathbf{z} \sim p_{\Omega}(\mathbf{x}),p(\mathbf{z})} \| \mathcal{L}\Tilde{g}(\mathbf{x},\mathbf{z})-\mathbf{b}(\mathbf{x})\|_2^2 \\
    = & \mathbb{E}_{\mathbf{x},\mathbf{z} \sim p_{\Gamma}(\mathbf{x}),p(\mathbf{z})}f(\mathbf{x},\Tilde{g}(\mathbf{x},\mathbf{z})) - \mathbb{E}_{\mathbf{x},\mathbf{z} \sim p_{\Gamma}(\mathbf{x}),p(\mathbf{z})}\Tilde{f}(\mathbf{x},\Tilde{g}(\mathbf{x},\mathbf{z}))\\
     & + \mathbb{E}_{\mathbf{x},\mathbf{z} \sim p_{\Gamma}(\mathbf{x}),p(\mathbf{z})}\Tilde{f}(\mathbf{x},\Tilde{g}(\mathbf{x},\mathbf{z})) - \hat{\mathbb{E}}^{m}_{\mathbf{x},\mathbf{z} \sim p_{\Gamma}(\mathbf{x}),p(\mathbf{z})}\Tilde{f}(\mathbf{x},\Tilde{g}(\mathbf{x},\mathbf{z}))\\
     & + \hat{\mathbb{E}}^{m}_{\mathbf{x},\mathbf{z} \sim p_{\Gamma}(\mathbf{x}),p(\mathbf{z})}\Tilde{f}(\mathbf{x},\Tilde{g}(\mathbf{x},\mathbf{z})) - \hat{\mathbb{E}}^{n}_{(\mathbf{x},u)\sim p_{\Gamma}(\mathbf{x},\mathbf{u})} \Tilde{f}(\mathbf{x},\mathbf{u})\\
     & + \hat{\mathbb{E}}^{n}_{(\mathbf{x},\mathbf{u})\sim p_{\Gamma}(\mathbf{x},\mathbf{u})} \Tilde{f}(\mathbf{x},\mathbf{u}) - \mathbb{E}_{(\mathbf{x},\mathbf{u})\sim p_{\Gamma}(\mathbf{x},\mathbf{u})} \Tilde{f}(\mathbf{x},\mathbf{u}) \\
     & + \mathbb{E}_{(\mathbf{x},\mathbf{u})\sim p_{\Gamma}(\mathbf{x},u)} \Tilde{f}(\mathbf{x},\mathbf{u}) - \mathbb{E}_{(\mathbf{x},\mathbf{u})\sim p_{\Gamma}(\mathbf{x},\mathbf{u})} f(\mathbf{x},\mathbf{u})\\
     & + \lambda \cdot  \mathbb{E}_{\mathbf{x},\mathbf{z} \sim p_{\Omega}(\mathbf{x}),p(\mathbf{z})} \| \mathcal{L}\Tilde{g}(\mathbf{x},\mathbf{z})-\mathbf{b}(\mathbf{x})\|_2^2 - \lambda \cdot \hat{\mathbb{E}}^{k}_{\mathbf{x},\mathbf{z} \sim p_{\Omega}(\mathbf{x}),p(\mathbf{z})} \| \mathcal{L}\Tilde{g}(\mathbf{x},\mathbf{z})-\mathbf{b}(\mathbf{x})\|_2^2 \\
     & + \lambda \cdot  \hat{\mathbb{E}}^{k}_{\mathbf{x},\mathbf{z} \sim p_{\Omega}(\mathbf{x}),p(\mathbf{z})} \| \mathcal{L}\Tilde{g}(\mathbf{x},\mathbf{z})-\mathbf{b}(\mathbf{x})\|_2^2\\
    \leq & \min_{\Tilde{f} \in \mathcal{F}_{GS}} \{\mathbb{E}_{\mathbf{x},\mathbf{z} \sim p_{\Gamma}(\mathbf{x}),p(\mathbf{z})}f(\mathbf{x},\Tilde{g}(\mathbf{x},\mathbf{z})) - \mathbb{E}_{\mathbf{x},\mathbf{z} \sim p_{\Gamma}(\mathbf{x}),p(\mathbf{z})}\Tilde{f}(\mathbf{x},\Tilde{g}(\mathbf{x},\mathbf{z}))  + \mathbb{E}_{(\mathbf{x},\mathbf{u})\sim p_{\Gamma}(\mathbf{x},\mathbf{u})} \Tilde{f}(\mathbf{x},\mathbf{u})\\
    & \hspace{3em} - \mathbb{E}_{(\mathbf{x},\mathbf{u})\sim p_{\Gamma}(\mathbf{x},\mathbf{u})} f(\mathbf{x},\mathbf{u}) \}\\
     & + \max_{{\Tilde{f} \in \mathcal{F}_{GS}}} \{ \mathbb{E}_{\mathbf{x},\mathbf{z} \sim p_{\Gamma}(\mathbf{x}),p(\mathbf{z})}\Tilde{f}(\mathbf{x},\Tilde{g}(\mathbf{x},\mathbf{z})) - \hat{\mathbb{E}}^{m}_{\mathbf{x},\mathbf{z} \sim p_{\Gamma}(\mathbf{x}),p(\mathbf{z})}\Tilde{f}(\mathbf{x},\Tilde{g}(\mathbf{x},\mathbf{z})) \} \\
     & + \max_{{\Tilde{f} \in \mathcal{F}_{GS}}} \{ \hat{\mathbb{E}}^{n}_{(\mathbf{x},\mathbf{u})\sim p_{\Gamma}(\mathbf{x},\mathbf{u})} \Tilde{f}(\mathbf{x},\mathbf{u}) - \mathbb{E}_{(\mathbf{x},\mathbf{u})\sim p_{\Gamma}(\mathbf{x},\mathbf{u})} \Tilde{f}(\mathbf{x},\mathbf{u}) \}\\
     & + \lambda \cdot  \left | \mathbb{E}_{\mathbf{x},\mathbf{z} \sim p_{\Omega}(\mathbf{x}),p(\mathbf{z})} \| \mathcal{L}\Tilde{g}(\mathbf{x},\mathbf{z})-\mathbf{b}(\mathbf{x})\|^2 - \lambda \cdot \hat{\mathbb{E}}^{k}_{\mathbf{x},\mathbf{z} \sim p_{\Omega}(\mathbf{x}),p(\mathbf{z})} \| \mathcal{L}\Tilde{g}(\mathbf{x},\mathbf{z})-\mathbf{b}(\mathbf{x})\|_2^2 \right | \\
     & + \max_{{\Tilde{f} \in \mathcal{F}_{GS}}} \{ \hat{\mathbb{E}}^{m}_{\mathbf{x},\mathbf{z} \sim p_{\Gamma}(\mathbf{x}),p(\mathbf{z})}\Tilde{f}(\mathbf{x},\Tilde{g}(\mathbf{x},\mathbf{z})) - \hat{\mathbb{E}}^{n}_{(\mathbf{x},\mathbf{u})\sim p_{\Gamma}(\mathbf{x},\mathbf{u})} \Tilde{f}(\mathbf{x},\mathbf{u}) + \lambda \cdot  \hat{\mathbb{E}}^{k}_{\mathbf{x},\mathbf{z} \sim p_{\Omega}(\mathbf{x}),p(\mathbf{z})} \| \mathcal{L}\Tilde{g}(\mathbf{x},\mathbf{z})-\mathbf{b}(\mathbf{x})\|_2^2 \}\\
     =: & I_1 + I_2 + I_3 + I_4 + I_5,
\end{align*}
i.e.,
\begin{equation} \label{five}
\mathbf{Wass}_1((\mathbf{x},\Tilde{g}(\mathbf{x},\mathbf{z}),(\mathbf{x},\mathbf{u})))+\lambda \cdot \mathbb{E}_{\mathbf{x}, \mathbf{z} \sim p_{\Omega}(\mathbf{x}),  p(\mathbf{z})} \| \mathcal{L}\Tilde{g}(\mathbf{x},\mathbf{z})-\mathbf{b}(\mathbf{x}) \|_2^2  \leq I_1 + I_2 + I_3 + I_4 + I_5
\end{equation}
Note that
$\widehat{\text{Loss}}$ in  
(\ref{emp_loss_func}) is equal to $I_5$.

Next we analyze the above five terms: the approximation of discriminators to 1-Lipschitz functions, the generalization of generated distributions and the target distribution (solution), the convergence of PINNs regularization term as well as the empirical loss respectively. The following five mild assumptions are 
imposed to the convergence guarantee of our algorithm. Note that Assumption \ref{decay}, \ref{large_capacity} and \ref{cong_sol_dis} are essential for our generalization bound while Assumption \ref{large_generator} and \ref{assump_pinns} are for better understanding of our theorem. We will illustrate the necessities and properties for each assumptions below.

\begin{assumption}[Decaying condition]
\label{decay}
There exist a constant $s \geq 2$ and a large enough $M_u>0$ such that for each $\mathbf{x} \in \Gamma$,
$$
\mathbb{P}(\|\mathbf{u}(\mathbf{x})\|_2 \geq M_u) \leq C \cdot M_u^{-s}.
$$
For simplicity, we set $C=1$ and further assume that $M_u \geq \sqrt{r} \cdot (M+W_g \cdot M)$ and $M_u \geq \sqrt{d} \cdot M_x$. 
\end{assumption}

\begin{remark}
Assumption \ref{decay} implies that the distribution of the solution at each location $\mathbf{x}$ has decaying tails (polynomial decay) after large enough value $M_u$. The condition is common for noisy data (from target distribution) in real applications. For example, if $\mathbf{u}(\mathbf{x})$ are from Gaussian distributions (or sub-Gaussian) with uniformly bounded means and deviations for all $\mathbf{x}$, 
the decaying condition is satisfied. The assumption is mild because it merely requires polynomial decay. Moreover, 
$$
M_u \geq \sqrt{r} \cdot (M+W_g \cdot M) \geq \|g(\mathbf{x},\mathbf{z}) \|_2
$$ 
indicates that the generated data is bounded by $M_u$, obviously satisfying the decaying condition. 
Note that the coincidence of supports of generators with the solution data is not required.
But we know there is a very low probability that the generator cannot cover the solution data.
\end{remark}

\begin{assumption}
\label{large_capacity}
Suppose $D_f$ and $W_f$ are large enough such that 
$$
M_u = O (2^{D_f/(2(d+r)^2)} \wedge  W_f^{1/(2(d+r)^2)}),
$$
where $O(\cdot)$ means that the two numbers are of the same order of magnitude, 
$M_u$ is defined in Assumption \ref{decay}.
\end{assumption}

\begin{remark}
The discriminators are strong (deep and wide) enough that they can discriminate the solution data. 
If a solution decays fast enough (e.g., Gaussian distributions with small deviations), then discriminators with small width and depth are sufficient. Furthermore, 
$$
\sqrt{r} \cdot (M+W_g \cdot M) \leq M_u = O(2^{D_f/(2C(d+r)^2)} \wedge  W_f^{1/(2(d+r)^2)})
$$ 
implies that the capacity of discriminators should be stronger than that of generators.
\end{remark}

\begin{assumption}
\label{large_generator}
We assume 
$ \sqrt{r} \cdot (W_g+1)\cdot M \geq \sqrt{d} \cdot M_x$ and 
the number of training data $m$ is large enough such that 
$$
(W_g+1) \cdot M \leq m^{1/2(d+r)}.
$$
\end{assumption}

\begin{remark}
This assumption requires the number of training data to be large enough such that the 
empirical generated distribution is close to the exact generated distribution. Otherwise, the limited and insufficient data will lead to the learned model $\Tilde{g}$ with large bias and variance. For more details and reasons of the assumption, please see the remark of Lemma \ref{bound_I2}.
\end{remark}

\begin{assumption}
\label{cong_sol_dis}
The exact solution data has finite 3-moment, i.e.,
$$
\mathbb{E}_{(\mathbf{x},\mathbf{u}) \sim p_{\Gamma}(\mathbf{x},\mathbf{u})} \|(\mathbf{x},\mathbf{u}) \|_2^3 < +\infty
$$
\end{assumption}

\begin{remark}
The boundedness of data $(\mathbf{x},\mathbf{u})$ in 3-moment is easy to be achieved. The spatial (with temporal) variable $\mathbf{x}$ is bounded by $M_x$, so we only focus on $\mathbf{u}(\mathbf{x})$. Similarly, if $\mathbf{u}(\mathbf{x})$ are from Gaussian (or sub-Gaussian) distributions with uniformly bounded mean and standard deviations for all $\mathbf{x}$, then the boundedness of 3-moment of data holds. Moreover, if Assumption \ref{decay} holds for a $M_u$ (fixed and independent with model architectures) with $s>3$, then the 3-moment of the data is bounded.
More precisely,
$$
\mathbb{E}_{(\mathbf{x},\mathbf{u}) \sim p_{\Gamma}(\mathbf{x},\mathbf{u})} \|(\mathbf{x},\mathbf{u})\|_2^3 \leq \mathbb{E}_{(\mathbf{x},\mathbf{u}) \sim p_{\Gamma}(\mathbf{x},\mathbf{u})} (\| \mathbf{x}\|_2 + \|\mathbf{u}(\mathbf{x})\|_2)^3
$$
and
$$
\mathbb{E}_{(\mathbf{x},\mathbf{u}) \sim p_{\Gamma}(\mathbf{x},\mathbf{u})} \|\mathbf{u}(\mathbf{x})\|_2^3 \leq M_u^3 + \mathbb{E}_{(\mathbf{x},\mathbf{u}) \sim p_{\Gamma}(\mathbf{x},\mathbf{u})} \{\|\mathbf{u}(\mathbf{x})\|_2^3 \cdot \mathbb{I}(\|\mathbf{u}(\mathbf{x})\|_2 \geq M_u)\} \leq M_u^3 + C \cdot M_u^{-s+3}
$$
imply the boundedness of the data.
\end{remark}

\begin{assumption}
\label{assump_pinns}
Assume that the number of training data $k$ is large enough such that for all $(\mathbf{x},\mathbf{z})$, 
the following inequality holds
$$
\| \mathcal{L}g(\mathbf{x},\mathbf{z})-\mathbf{b}(\mathbf{x})\|_2^2 \leq k^{1/4}
$$
for all $g \in \mathcal{G}$.
\end{assumption}

\begin{remark}
For a fixed predefined generator class $\mathcal{G}$ in (\ref{generator_class}), all the derivatives of $g \in \mathcal{G}$ are bounded by a constant depending on $M$, $W_g$ and $D_g$.
The reason is due to the boundedness of $tanh$ (and its derivatives) in $\mathcal{G}$. 
Therefore, 
\begin{equation}
\label{bound_pinns}
    \| \mathcal{L}\Tilde{g}(\mathbf{x},\mathbf{z})-\mathbf{b}(\mathbf{x})\|_2^2 \leq \bar{C}(\mathcal{L}, M, W_g, D_g, M_b),
\end{equation}
where $\bar{C}$ is a constant depending on model architectures and given PDEs. For other activation functions 
like $ReLU^q$, similar assumptions also hold for the boundedness of 
$\mathbb{E}_{\mathbf{z} \sim p(\mathbf{z})}\| \mathcal{L}\Tilde{g}(\mathbf{x},\mathbf{z})-\mathbf{b}(\mathbf{x})\|_2^2$ for all $\mathbf{x} \in \Omega$ with the decaying condition of $\mathbf{z}$ (e.g. sub-Gaussian distributions). For more details and reasons of the assumption, please see the remark of Lemma \ref{bound_I4}.
\end{remark}

\subsection{Convergence Analysis of Empirical Loss}


In this subsection, we establish the convergence of empirical loss, and 
show that with sufficiently large amount of training data, the 
empirical loss will converge to the approximation error with a high probability.

We first recall the following results for our analysis.


\begin{prop}[Theorem 2, \cite{tanielian2021approximating}]
\label{approx_dis}
Let $d \geq 2$, the grouping size $group_{size}=2$. For any $ f \in \operatorname{Lip}_{1}\left([-M_u,M_u]^{d}\right)$, where $\operatorname{Lip}_1$ denotes the class of 1-Lipschitz functions defined on $[-M_u,M_u]^{d}$, satisfying that 
$$
|f(\mathbf{x})-f(\mathbf{y})| \leq \| \mathbf{x}-\mathbf{y} \|_2, \hspace{1em} \mathbf{x}, \mathbf{y} \in [-M_u,M_u]^{d},
$$
there exists a neural network $\Tilde{f}$ of the form (\ref{groupsortnn}) with depth $D_f$ and width $W_f$ such that
\begin{equation*}
    \|f-\Tilde{f}\|_{L^{\infty}([-M_u,M_u]^{d})} \leq 2M_u \cdot C \cdot ( 2^{-D_f/d^2} \vee W_f^{-1/d^2}),
\end{equation*}
where the constant $C \approx 2 \sqrt{d}$.
\end{prop}

\begin{prop}[Convergence in $\mathbf{Wass}_1$ distance \cite{NEURIPS2020_2000f632}]
\label{convg_W1}
Assume that the distribution $\nu$ (lies on $\mathbb{R}^{d}$) satisfies that $M_{3}=\mathbb{E}_{\mathbf{x} \sim \nu}\|\mathbf{x}\|_2^{3}<\infty .$ Then there exists a constant $C^{'}$ depending on $M_{3}$ such that
$$
\mathbb{E} \mathbf{Wass}_{1}\left(\nu, \hat{\nu}^m\right) \leq C^{'} \cdot\left\{\begin{array}{ll}
m^{-1 / 2}, & d=1 \\
m^{-1 / 2} \log m, & d=2 \\
m^{-1 / d}, & d \geq 3
\end{array}\right.
$$
where $\hat{\nu}^m$ denotes the empirical distribution of $m$ i.i.d. samples from the distribution $\nu$; the constant $C^{\prime}$ is proportional to $\sqrt[3]{M_3}$ but is independent of $d$ and $m$.
\end{prop}

Let us study the five terms in (\ref{five}) in the following analysis. 

\begin{lemma}
\label{bound_I1}
Suppose that Assumption \ref{decay} and \ref{large_capacity} hold. 
Then 
\begin{equation}
\label{up_bound_1}
    I_1 \leq C_1 \cdot ( 2^{-D_f/(2 (d+r)^2)} \vee W_f^{-1/2(d+r)^2}) 
\end{equation}
where the constant $C_1$ is independent of the model architecture. 
\end{lemma}

\begin{proof}
By the definition of $I_1$, we obtain
\begin{align*}
    I_1 = & \min_{\Tilde{f} \in \mathcal{F}_{GS}} \{\mathbb{E}_{\mathbf{x},\mathbf{z} \sim p_{\Gamma}(\mathbf{x}),p(\mathbf{z})}f(\mathbf{x},\Tilde{g}(\mathbf{x},\mathbf{z})) - \mathbb{E}_{\mathbf{x},\mathbf{z} \sim p_{\Gamma}(\mathbf{x}),p(\mathbf{z})}\Tilde{f}(\mathbf{x},\Tilde{g}(\mathbf{x},\mathbf{z})) \\
    & + \mathbb{E}_{(\mathbf{x},\mathbf{u})\sim p_{\Gamma}(\mathbf{x},\mathbf{u})} \Tilde{f}(\mathbf{x},\mathbf{u}) - \mathbb{E}_{(\mathbf{x},\mathbf{u})\sim p_{\Gamma}(\mathbf{x},\mathbf{u})} f(\mathbf{x},\mathbf{u}) \}\\
    \leq & \min_{\Tilde{f} \in \mathcal{F}_{GS}} \{ \| \Tilde{f}-f \|_{L^{\infty}([-M_u,M_u]^{d+r})} \cdot \left (\mathbb{P}(\|\mathbf{u}(\mathbf{x})\|_2 \leq M_u) +1 \right ) \\
    & +  \mathbb{E}_{(\mathbf{x},\mathbf{u})\sim p_{\Gamma}(\mathbf{x},\mathbf{u})} [\Tilde{f}(\mathbf{x},\mathbf{u}) - f(\mathbf{x},\mathbf{u})] \cdot \mathbb{I}(\|\mathbf{u}(\mathbf{x})\|_2 \geq M_u) \}
\end{align*}
By using Assumption \ref{decay} and Proposition \ref{approx_dis}, we can conclude that
$$
I_1 \leq C_0 \cdot M_u \cdot  ( 2^{-D_f/(d+r)^2} \vee W_f^{-1/(d+r)^2}) + C_1 \cdot M_u^{-s+1}
$$
where $s \geq 2$, $C_0$ and $C_1$ are independent of the model architecture. 
Then substituting $M_u$ by $D_f$ and $W_f$ in Assumption \ref{large_capacity} with $s=2$, the upper bound (\ref{up_bound_1}) for $I_1$ can be obtained.
\end{proof}


The 
upper bound for $I_2$ is used to measure the convergence of the generated distribution by a large amount of but finite samples. The property is essential for theoretical analysis of algorithms and models because practically we can only access limited training data due to the storage and computation.

\begin{lemma}
\label{bound_I2}
For simplicity, we assume that $ \sqrt{r} \cdot (W_g+1)\cdot M \geq \sqrt{d} \cdot M_x$, then with probability of at least $1-m^{-1/4(d+r)}$ (over the choice of $m$ i.i.d. training samples), $I_2$ defined in (\ref{five}) satisfies the following inequality:
\begin{equation}
\label{error_bound_I2_1}
    I_2 \leq \mathbf{Wass}_1(p_{\Tilde{g}}(\mathbf{x},\mathbf{u}), \hat{p}^m_{\Tilde{g}}(\mathbf{x},\mathbf{u})) \leq C_2 \cdot ((W_g+1)\cdot M) \cdot \left\{\begin{array}{ll}
m^{-3 / 8} \log m, & d+r=2 \\
m^{-3 / 4(d+r)}, & d+r \geq 3
\end{array}\right.
\end{equation}
If we further assume that Assumption \ref{large_generator} holds for generators and training data, then with probability of at least $1-m^{-1/4(d+r)}$ (over the choice of $m$ i.i.d. training samples), $I_2$ defined in (\ref{five}) satisfies the following inequality:
\begin{equation}
\label{error_bound_I2_2}
    I_2 \leq \mathbf{Wass}_1(p_{\Tilde{g}}(\mathbf{x},\mathbf{u}), \hat{p}^m_{\Tilde{g}}(\mathbf{x},\mathbf{u})) \leq C_2 \cdot \left\{\begin{array}{ll}
m^{-1 / 8} \log m, & d+r=2 \\
m^{-1 / 4(d+r)}, & d+r \geq 3
\end{array}\right.
\end{equation}
where the constant $C_2$ is a universal constant, independent of the model architecture.
\end{lemma}

\begin{proof}
For each $g \in \mathcal{G}$ and for all $\mathbf{x} \in \Omega \cup \Gamma$, $\|g(\mathbf{x},\mathbf{z})\|_2 \leq \sqrt{r} \cdot (W_g+1) \cdot M$, therefore, $\| (\mathbf{x}, g(\mathbf{x},\mathbf{z})) \|_2 \leq \sqrt{d}\cdot M_x + \sqrt{r} \cdot (W_g+1)\cdot M$. Therefore, by properties of Wasserstein-1 distance, we have 
\begin{align*}
\mathbb{E} \mathbf{Wass}_1(p_{\Tilde{g}}(\mathbf{x},\mathbf{u}), \hat{p}^m_{\Tilde{g}}(\mathbf{x},\mathbf{u})) & \leq C^{\prime} \cdot (\sqrt{d}\cdot M_x +  \sqrt{r} \cdot (W_g+1)\cdot M) \cdot \left\{\begin{array}{ll}
m^{-1 / 2} \log m, & d+r=2 \\
m^{-1 / (d+r)}, & d+r \geq 3
\end{array}\right.\\
& \leq C_2 \cdot ((W_g+1)\cdot M) \cdot \left\{\begin{array}{ll}
m^{-1 / 2} \log m, & d+r=2 \\
m^{-1 / (d+r)}, & d+r \geq 3
\end{array}\right.
\end{align*}
where $C_2$ is independent of the model architecture and the numbers of training data. By Markov-Inequality, with probability of at least $1-m^{-1/4(d+r)}$,
\begin{equation}
    \mathbf{Wass}_1(p_{\Tilde{g}}(\mathbf{x},\mathbf{u}), \hat{p}^m_{\Tilde{g}}(\mathbf{x},\mathbf{u})) \leq C_2 \cdot ((W_g+1)\cdot M) \cdot \left\{\begin{array}{ll}
m^{-3 / 8} \log m, & d+r=2 \\
m^{-3 / 4(d+r)}, & d+r \geq 3
\end{array}\right.
\end{equation}

By Assumption \ref{large_generator}, we have
\begin{equation*}
\mathbb{E} \mathbf{Wass}_1(p_{\Tilde{g}}(\mathbf{x},\mathbf{u}), \hat{p}^m_{\Tilde{g}}(\mathbf{x},\mathbf{u})) \leq C_2 \cdot \left\{\begin{array}{ll}
m^{-1 / 4} \log m, & d+r=2 \\
m^{-1 / 2(d+r)}, & d+r \geq 3
\end{array}\right.
\end{equation*}
Similarly, by Markov-Inequality, with probability of at least $1-m^{-1/4(d+r)}$,
\begin{equation}
    \mathbf{Wass}_1(p_{\Tilde{g}}(\mathbf{x},\mathbf{u}), \hat{p}^m_{\Tilde{g}}(\mathbf{x},\mathbf{u})) \leq C_2 \cdot \left\{\begin{array}{ll}
m^{-1 / 8} \log m, & d+r=2 \\
m^{-1 / 4(d+r)}, & d+r \geq 3
\end{array}\right.
\end{equation}
\end{proof}

\begin{remark}
Error bound for $I_2$ in (\ref{error_bound_I2_1}) involves the generator capacity $(W_g+1)\cdot M$. If the number of training data $m$ is small w.r.t. the generator capacity (range of generated data), then the error will be large where $(W_g+1)\cdot M$ dominates the error. Conversely, if the range of generated data is small compared with the number of training data (Assumption \ref{large_generator}), then $(W_g+1) \cdot M$ can be ignored as is shown in (\ref{error_bound_I2_2}).
\end{remark}

Similar to the analysis of $I_2$, we analyze the convergence of the distribution of the solution by a large number of but 
a finite number of samples. 



\begin{lemma}
\label{bound_I3}
Suppose that Assumption \ref{cong_sol_dis} holds, then with probability of at least $1-n^{-1/2(d+r)}$ (over the choice of $n$ i.i.d. training samples), $I_3$ defined in (\ref{five}) satisfies the following inequality:
\begin{equation}
    I_3 \leq \mathbf{Wass}_1(p_{\Gamma}(\mathbf{x},\mathbf{u}),\hat{p}^n_{\Gamma}(\mathbf{x},\mathbf{u})) \leq C_3 \cdot \left\{\begin{array}{ll}
n^{-1 / 4} \log n, & d+r=2 \\
n^{-1 / 2(d+r)}, & d+r \geq 3,
\end{array}\right.
\end{equation}
where $\hat{p}^n_{\Gamma}(\mathbf{x},\mathbf{u}))$ is the empirical distribution of $p_{\Gamma}(\mathbf{x},\mathbf{u})$ with $n$ i.i.d. samples and the constant $C_3$ depends only on the 3-moment of the data on the boundary.
\end{lemma}

Now we mainly focus on the convergence of the PINNs regularization term. We hope that with sufficient samples, the empirical expectation will converge to the exact expectation. 

\begin{lemma}
\label{bound_I4}
For generators defined in (\ref{generator_class}), we have that with probability of at least $1-2 \cdot k^{-1}$ (over the choice of $k$ i.i.d. samples in the interior domain),
\begin{eqnarray}
\label{error_bound_I4_1}
    I_4 & = \lambda \cdot \left | \mathbb{E}_{\mathbf{x},\mathbf{z} \sim p_{\Omega}(\mathbf{x}),p(\mathbf{z})} \| \mathcal{L}\Tilde{g}(\mathbf{x},\mathbf{z})-\mathbf{b}(\mathbf{x})\|_2^2 - \hat{\mathbb{E}}^{k}_{\mathbf{x},\mathbf{z} \sim p_{\Omega}(\mathbf{x}),p(\mathbf{z})} \| \mathcal{L}\Tilde{g}(\mathbf{x},\mathbf{z})-\mathbf{b}(\mathbf{x})\|_2^2 \right | 
    \\
    & \leq \lambda \cdot C_4 \cdot (W_g \cdot M)^{\Tilde{C}D_g} \cdot  \sqrt{\frac{\log k}{2k}}.
\end{eqnarray}
where constant $C_4$ only depends on the differential operator $\mathcal{L}$ and $M_b$; constant $\Tilde{C}$ depends on the differential operator $\mathcal{L}$. Note that $\bar{C}:=C_4 \cdot (W_g \cdot M)^{\Tilde{C}D_g}$ defined in (\ref{bound_pinns}) is the upper bound for PINNs term with given generators class.
If we further assume that Assumption \ref{assump_pinns} holds for a large number of training data $k$, then with probability of at least $1-\frac{2}{k}$ (over the choice of $k$ i.i.d. samples in the interior domain),
\begin{equation}
\label{error_bound_I4_2}
    I_4 = \lambda \cdot \left | \mathbb{E}_{\mathbf{x},\mathbf{z} \sim p_{\Omega}(\mathbf{x}),p(\mathbf{z})} \| \mathcal{L}\Tilde{g}(\mathbf{x},\mathbf{z})-\mathbf{b}(\mathbf{x})\|_2^2 -  \hat{\mathbb{E}}^{k}_{\mathbf{x},\mathbf{z} \sim p_{\Omega}(\mathbf{x}),p(\mathbf{z})} \| \mathcal{L}\Tilde{g}(\mathbf{x},\mathbf{z})-\mathbf{b}(\mathbf{x})\|_2^2 \right | \leq \lambda \cdot \sqrt{\frac{\log k}{2}} \cdot  k^{-1/4}.
\end{equation}
\end{lemma}
\begin{proof}
It is a straightforward result by using Hoeffding's Inequality with Assumption \ref{assump_pinns}.
\end{proof}

\begin{remark}
Error bound for $I_4$ in (\ref{error_bound_I4_1}) involves the constant $\bar{C}:=C_4 \cdot (W_g \cdot M)^{\Tilde{C}D_g}$ which is related to the generator capacity ($W_g$, $D_g$ and $M$) as well as the given PDEs problem ($\mathcal{L}$ and $M_b$). If the number of training data $k$ is small w.r.t. $\bar{C}$, then the error will be large because $\bar{C}$ dominates the error. Conversely, if the training data is sufficient compared with the $\bar{C}$ (Assumption \ref{assump_pinns}), then $\bar{C}$ can be ignored as is shown in (\ref{error_bound_I4_2}). Here, we provide an example to depict how $\bar{C}$ is related to the PDEs problem. For the pedagogical example (\ref{exp1_ode}) in the first part of our numerical experiments. The given PDE is 
\begin{equation*}
        u_{xx} - u^2 u_x = b(x), \hspace{1em} x \in [-1,1]
\end{equation*}
with the pre-defined generator class (\ref{generator_class}). The right hand side function $b(x) = - \pi^2 \sin (\pi x) - \pi \cos (\pi x) \sin^2 (\pi x)$ is bounded by $M_b = \pi^2 + \pi$ in the domain $x \in [-1,1]$. By the definition of the generator class, $|u_{xx}| \leq 2(W_g \cdot M)^{2D_g}$, $|u_x| \leq (W_g \cdot M)^{D_g}$ and $|u| \leq (M + W_g \cdot M)$. Therefore, the constant $\bar{C}=2(W_g \cdot M)^{2D_g} + (W_g \cdot M)^{D_g} \cdot (M + W_g \cdot M)^2 + \pi^2 + \pi:=C_4(\mathcal{L},M_b)\cdot (W_g \cdot M)^{\Tilde{C}D_g}$, where constant $\Tilde{C}$ depends on the differential operator $\mathcal{L}$ and is equal to $2$ here. Note that $C_4 \cdot (W_g \cdot M)^{\Tilde{C}D_g}$ is a uniform upper bound and may be not tight in many examples. But we can imagine a worst case that the solution of a given PDE is the neural networks with $M$ as the value of elements of all transformation matrix. In most of the real applications, $\bar{C}$ is of the same order of magnitude with $M_b$ or smaller. Practically, parameters of neural networks are initialized by random samples from a Gaussian distribution (i.e. the values of $\mathcal{L}g(\mathbf{x},\mathbf{z})$ with initial parameters are small) and the stochastic gradient descent method gradually force $\mathcal{L}g(\mathbf{x},\mathbf{z})$ to be close to $\mathbf{b}(\mathbf{x})$. Therefore, during the training, $\|\mathcal{L}g(\mathbf{x},\mathbf{x}) -\mathbf{b}(\mathbf{x})\|_2$ will not be strikingly greater than $\| \mathbf{b}(\mathbf{x})\|_2$. In this sense, $\bar{C}$ is strict in applications while it is uniform for the whole class of generators theoretically.

\end{remark}

One important thing for machine learning is whether the empirical loss will converge to a (small) value with sufficient training data. Otherwise, the learning becomes meaningless without convergence guarantees. In this part, we study the convergence of the empirical loss. And we prove that with large amount of training data, our empirical loss will converge to the approximation error with a high probability.

\begin{theorem}[Convergence of the empirical loss]
Suppose that Assumption \ref{cong_sol_dis} holds for our problem and $(\Tilde{g},\Tilde{f})$ is the optimal solution by solving (\ref{emp_loss_func}), then with probability of at least $1-m^{-1/4(d+r)}-n^{-1/2(d+r)}-2\cdot k^{-1}$, the empirical loss satisfies the following inequality,
\begin{equation}
\label{bound_emp_loss_1}
\begin{split}
    \widehat{\text{Loss}} = I_5 \leq &  C_2 \cdot ((W_g+1)\cdot M) \cdot \left\{\begin{array}{ll}
m^{-3 / 8} \log m, & d+r=2 \\
m^{-3 / 4(d+r)}, & d+r \geq 3
\end{array}\right. + C_3 \cdot \left\{\begin{array}{ll}
n^{-1 / 4} \log n, & d+r=2 \\
n^{-1 / 2(d+r)}, & d+r \geq 3
\end{array}\right. \\
& + \lambda \cdot C_4 \cdot (W_g \cdot M)^{\Tilde{C}D_g} \cdot  \sqrt{\frac{\log k}{2k}}\\
& + \min_{g \in \mathcal{G}} \{ \mathbf{Wass}_1(p_g(\mathbf{x},\mathbf{u}),p_{\Gamma}(\mathbf{x},\mathbf{u})) + \lambda \cdot \mathbb{E}_{\mathbf{x},\mathbf{z} \sim p_{\Omega}(\mathbf{x}),p(\mathbf{z})} \| \mathcal{L}g(\mathbf{x},\mathbf{z})-\mathbf{b}(\mathbf{x})\|_2^2 \}.
\end{split}
\end{equation}
If we further assume that Assumption \ref{large_generator} and \ref{assump_pinns} hold for our problem, then with probability of at least $1-m^{-1/4(d+r)}-n^{-1/2(d+r)}-2\cdot k^{-1}$, the empirical loss satisfies the following inequality,
\begin{equation}
\label{bound_emp_loss_2}
\begin{split}
    \widehat{\text{Loss}} = I_5 \leq &  C_2 \cdot \left\{\begin{array}{ll}
m^{-1 / 8} \log m, & d+r=2 \\
m^{-1 / 4(d+r)}, & d+r \geq 3
\end{array}\right. + C_3 \cdot \left\{\begin{array}{ll}
n^{-1 / 4} \log n, & d+r=2 \\
n^{-1 / 2(d+r)}, & d+r \geq 3
\end{array}\right. + \lambda \cdot \sqrt{\frac{\log k}{2}} \cdot  k^{-1/4} \\
 & + \min_{g \in \mathcal{G}} \{ \mathbf{Wass}_1(p_g(\mathbf{x},\mathbf{u}),p_{\Gamma}(\mathbf{x},\mathbf{u})) + \lambda \cdot \mathbb{E}_{\mathbf{x},\mathbf{z} \sim p_{\Omega}(\mathbf{x}),p(\mathbf{z})} \| \mathcal{L}g(\mathbf{x},\mathbf{z})-\mathbf{b}(\mathbf{x})\|_2^2 \},
\end{split}
\end{equation}
where constants $C_2$ and $C_3$ are independent of the model architecture and $m$, $n$ as well as $k$; constant $C_4$ depends on the differential operator $\mathcal{L}$ and $M_b$ while constant $\Tilde{C}$ depends on the differential operator $\mathcal{L}$.
\end{theorem}

\begin{proof}
\begin{align*}
I_5 = & \max_{{\Tilde{f} \in \mathcal{F}_{GS}}} \{ \hat{\mathbb{E}}^{m}_{\mathbf{x},\mathbf{z} \sim p_{\Gamma}(\mathbf{x}),p(\mathbf{z})}\Tilde{f}(\mathbf{x},\Tilde{g}(\mathbf{x},\mathbf{z})) - \hat{\mathbb{E}}^{n}_{(\mathbf{x},\mathbf{u})\sim p_{\Gamma}(\mathbf{x},\mathbf{u})} \Tilde{f}(\mathbf{x},\mathbf{u}) + \lambda \cdot \hat{\mathbb{E}}^{k}_{\mathbf{x},\mathbf{z} \sim p_{\Omega}(\mathbf{x}),p(\mathbf{z})} \| \mathcal{L}\Tilde{g}(\mathbf{x},\mathbf{z})-\mathbf{b}(\mathbf{x})\|_2^2 \}  \\
\leq & \max_{\Tilde{f} \in \mathcal{F}_{GS}} \{ \hat{\mathbb{E}}^{m}_{\mathbf{x},\mathbf{z} \sim p_{\Gamma}(\mathbf{x}),p(\mathbf{z})}\Tilde{f}(\mathbf{x},g(\mathbf{x},\mathbf{z})) - \hat{\mathbb{E}}^{n}_{(\mathbf{x},\mathbf{u})\sim p_{\Gamma}(\mathbf{x},\mathbf{u})} \Tilde{f}(\mathbf{x},\mathbf{u}) + \lambda \cdot \hat{\mathbb{E}}^{k}_{\mathbf{x},\mathbf{z} \sim p_{\Omega}(\mathbf{x}),p(\mathbf{z})} \| \mathcal{L}g(\mathbf{x},\mathbf{z})-\mathbf{b}(\mathbf{x})\|_2^2 \}\\
& \hspace{1em} \text{ (the above inequality holds for all $g \in \mathcal{G}$)}\\
\leq & \max_{f \text{ is 1-Lipschitz}} \{ \hat{\mathbb{E}}^{m}_{\mathbf{x},\mathbf{z} \sim p_{\Gamma}(\mathbf{x}),p(\mathbf{z})}f(\mathbf{x},g(\mathbf{x},\mathbf{z})) - \hat{\mathbb{E}}^{n}_{(\mathbf{x},\mathbf{u})\sim p_{\Gamma}(\mathbf{x},\mathbf{u})} f(\mathbf{x},\mathbf{u}) + \lambda \cdot \hat{\mathbb{E}}^{k}_{\mathbf{x},\mathbf{z} \sim p_{\Omega}(\mathbf{x}),p(\mathbf{z})} \| \mathcal{L}g(\mathbf{x},\mathbf{z})-\mathbf{b}(\mathbf{x})\|_2^2 \}\\
= & \max_{f \text{ is 1-Lipschitz}} \{ \hat{\mathbb{E}}^{m}_{\mathbf{x},\mathbf{z} \sim p_{\Gamma}(\mathbf{x}),p(\mathbf{z})}f(\mathbf{x},g(\mathbf{x},\mathbf{z})) - \mathbb{E}_{\mathbf{x},\mathbf{z} \sim p_{\Gamma}(\mathbf{x}),p(\mathbf{z})}f(\mathbf{x},g(\mathbf{x},\mathbf{z})) + \mathbb{E}_{\mathbf{x},\mathbf{z} \sim p_{\Gamma}(\mathbf{x}),p(\mathbf{z})}f(\mathbf{x},g(\mathbf{x},\mathbf{z}))\\
& -\hat{\mathbb{E}}^{n}_{(\mathbf{x},\mathbf{u})\sim p_{\Gamma}(\mathbf{x},\mathbf{u})} f(\mathbf{x},\mathbf{u}) + \mathbb{E}_{(\mathbf{x},\mathbf{u})\sim p_{\Gamma}(\mathbf{x},\mathbf{u})} f(\mathbf{x},\mathbf{u}) - \mathbb{E}_{(\mathbf{x},\mathbf{u})\sim p_{\Gamma}(\mathbf{x},\mathbf{u})} f(\mathbf{x},\mathbf{u})\}\\
& + \lambda \cdot \hat{\mathbb{E}}^{k}_{\mathbf{x},\mathbf{z} \sim p_{\Omega}(\mathbf{x}),p(\mathbf{z})} \| \mathcal{L}g(\mathbf{x},\mathbf{z})-\mathbf{b}(\mathbf{x})\|_2^2 - \lambda \cdot \mathbb{E}_{\mathbf{x},\mathbf{z} \sim p_{\Omega}(\mathbf{x}),p(\mathbf{z})} \| \mathcal{L}g(\mathbf{x},\mathbf{z})-\mathbf{b}(\mathbf{x})\|_2^2 \\
& + \lambda \cdot \mathbb{E}_{\mathbf{x},\mathbf{z} \sim p_{\Omega}(\mathbf{x}),p(\mathbf{z})} \| \mathcal{L}g(\mathbf{x},\mathbf{z})-\mathbf{b}(\mathbf{x})\|_2^2\\
\leq & \max_{f \text{ is 1-Lipschitz}} \{ \hat{\mathbb{E}}^{m}_{\mathbf{x},\mathbf{z} \sim p_{\Gamma}(\mathbf{x}),p(\mathbf{z})}f(\mathbf{x},g(\mathbf{x},\mathbf{z})) - \mathbb{E}_{\mathbf{x},\mathbf{z} \sim p_{\Gamma}(\mathbf{x}),p(\mathbf{z})}f(\mathbf{x},g(\mathbf{x},\mathbf{z})) \} \\
& + \max_{f \text{ is 1-Lipschitz}} \{ \hat{\mathbb{E}}^{n}_{(\mathbf{x},u)\sim p_{\Gamma}(\mathbf{x},\mathbf{u})} f(\mathbf{x},\mathbf{u}) - \mathbb{E}_{(\mathbf{x},\mathbf{u})\sim p_{\Gamma}(\mathbf{x},\mathbf{u})} f(\mathbf{x},\mathbf{u}) \}\\
& + \lambda \cdot \hat{\mathbb{E}}^{k}_{\mathbf{x},\mathbf{z} \sim p_{\Omega}(\mathbf{x}),p(\mathbf{z})} \| \mathcal{L}g(\mathbf{x},\mathbf{z})-\mathbf{b}(\mathbf{x})\|_2^2 - \lambda \cdot \mathbb{E}_{\mathbf{x},\mathbf{z} \sim p_{\Omega}(\mathbf{x}),p(\mathbf{z})} \| \mathcal{L}g(\mathbf{x},\mathbf{z})-\mathbf{b}(\mathbf{x})\|_2^2\\
& + \max_{f \text{ is 1-Lipschitz}} \{ \mathbb{E}_{\mathbf{x},\mathbf{z} \sim p_{\Gamma}(\mathbf{x}),p(\mathbf{z})}f(\mathbf{x},g(\mathbf{x},\mathbf{z})) -  \mathbb{E}_{(\mathbf{x},\mathbf{u})\sim p_{\Gamma}(\mathbf{x},\mathbf{u})} f(\mathbf{x},\mathbf{u}) \} \\
& + \lambda \cdot \mathbb{E}_{\mathbf{x},\mathbf{z} \sim p_{\Omega}(\mathbf{x}),p(\mathbf{z})} \| \mathcal{L}g(\mathbf{x},\mathbf{z})-\mathbf{b}(\mathbf{x})\|_2^2,
\end{align*}
for all $g \in \mathcal{G}$.
Therefore,
\begin{align*}
    I_5 \leq & I_2 + I_3 + I_4 + \min_{g \in \mathcal{G}} \{ \mathbf{Wass}_1(p_g(\mathbf{x},\mathbf{u}),p_{\Gamma}(\mathbf{x},\mathbf{u})) + \lambda \cdot \mathbb{E}_{\mathbf{x},\mathbf{z} \sim p_{\Omega}(\mathbf{x}),p(\mathbf{z})} \| \mathcal{L}g(\mathbf{x},\mathbf{z})-\mathbf{b}(\mathbf{x})\|_2^2 \}.
\end{align*}
By Lemma \ref{bound_I2}, \ref{bound_I3} and \ref{bound_I4}, substituting $I_2$, $I_3$ and $I_4$ by the corresponding inequalities concludes (\ref{bound_emp_loss_1}) and  (\ref{bound_emp_loss_2}).
\end{proof}

\begin{remark}
The last term of the right hand side of (\ref{bound_emp_loss_1}) and (\ref{bound_emp_loss_2}):
$$
\min_{g \in \mathcal{G}} \{ \mathbf{Wass}_1(p_g(\mathbf{x},\mathbf{u}),p_{\Gamma}(\mathbf{x},\mathbf{u})) + \lambda \cdot \mathbb{E}_{\mathbf{x},\mathbf{z} \sim p_{\Omega}(\mathbf{x}),p(\mathbf{z})} \| \mathcal{L}g(\mathbf{x},\mathbf{z})-\mathbf{b}(\mathbf{x})\|_2^2 \}
$$
is called the approximation error in machine learning, representing the minimal error among all generators in the predefined class. 
There is still open question whether the approximation error of large scale (deep and wide) generators 
is small.
Recently, Ryck et al. \cite{de2021approximation} 
proved that shallow (only 1-hidden layer) but wide tanh neural networks can approximate Sobolev regular 
and analytic functions. His results reveal that wider neural networks have larger capacity. As we can see that in our previous analysis, the depth of the generator does not have effects on the range of generated data. More strikingly, shallow feed-forward neural networks with $tanh$ or $ReLU^q$ activation functions are not subsets of deeper neural networks, which implies that we cannot intuitively conclude larger capacity of deeper $tanh$ (and $ReLU^q$) neural networks. 
Therefore, we have no reasons to assume that the approximation error decays with the use of deeper neural networks.
\end{remark}

\subsection{Generalization Error}

By combining the above results, we establish 
the generalization and convergence property of our WGAN-PINNs.

\begin{theorem}
[Generalization of WGAN-PINNs] 
\label{main_theorem}
Suppose that Assumption \ref{decay}, \ref{large_capacity} and \ref{cong_sol_dis} hold for our problem and $\Tilde{g}$ is the optimal solution of empirical loss (\ref{emp_loss_func}),
then with probability of at least $1-m^{-1/4(d+r)}-n^{-1/2(d+r)}-2\cdot k^{-1}$ (over the choice of training data), the exact loss converges to the approximation error, i.e., 
\begin{equation}
\label{error_bound_1}
    \begin{split}
            & \mathbf{Wass}_1(p_{\Tilde{g}}(\mathbf{x},\mathbf{u}),p_{\Gamma}(\mathbf{x},\mathbf{u})) + \lambda \cdot \mathbb{E}_{\mathbf{x},\mathbf{z} \sim p_{\Omega}(\mathbf{x}),p(\mathbf{z})} \| \mathcal{L}\Tilde{g}(\mathbf{x},\mathbf{z})-\mathbf{b}(\mathbf{x})\|_2^2 \\
    \leq & \hspace{0.5em} C_1 \cdot ( 2^{-D_f/(2(d+r)^2)} \vee W_f^{-1/2(d+r)^2}) +  C_2 \cdot ((W_g+1)\cdot M) \cdot \left\{\begin{array}{ll}
m^{-3 / 8} \log m, & d+r=2 \\
m^{-3 / 4(d+r)}, & d+r \geq 3
\end{array}\right. \\
& + C_3 \cdot \left\{\begin{array}{ll}
n^{-1 / 4} \log n, & d+r=2 \\
n^{-1 / 2(d+r)}, & d+r \geq 3
\end{array}\right.  + \lambda \cdot C_4 \cdot (W_g \cdot M)^{\Tilde{C}D_g} \cdot  \sqrt{\frac{2 \log k}{k}} \\& \hspace{0.5em}  + \min_{g \in \mathcal{G}} \{ \mathbf{Wass}_1(p_g(\mathbf{x},\mathbf{u}),p_{\Gamma}(\mathbf{x},\mathbf{u})) + \lambda \cdot \mathbb{E}_{\mathbf{x},\mathbf{z} \sim p_{\Omega}(\mathbf{x}),p(\mathbf{z})} \| \mathcal{L}g(\mathbf{x},\mathbf{z})-\mathbf{b}(\mathbf{x})\|_2^2 \}.
\end{split}
\end{equation}
If we further assume that Assumption \ref{large_generator} and \ref{assump_pinns} hold for our problem,
then with probability of at least $1-m^{-1/4(d+r)}-n^{-1/2(d+r)}-2\cdot k^{-1}$ (over the choice of training data), the exact loss converges to the approximation error, i.e., 
\begin{equation}
\label{error_bound_2}
    \begin{split}
            & \mathbf{Wass}_1(p_{\Tilde{g}}(\mathbf{x},\mathbf{u}),p_{\Gamma}(\mathbf{x},\mathbf{u})) + \lambda \cdot \mathbb{E}_{\mathbf{x},\mathbf{z} \sim p_{\Omega}(\mathbf{x}),p(\mathbf{z})} \| \mathcal{L}\Tilde{g}(\mathbf{x},\mathbf{z})-\mathbf{b}(\mathbf{x})\|_2^2 \\
    \leq & \hspace{0.5em} C_1 \cdot ( 2^{-D_f/(2(d+r)^2)} \vee W_f^{-1/2(d+r)^2}) +  C_2 \cdot \left\{\begin{array}{ll}
m^{-1 / 8} \log m, & d+r=2 \\
m^{-1 / 4(d+r)}, & d+r \geq 3
\end{array}\right.\\
& + C_3 \cdot \left\{\begin{array}{ll}
n^{-1 / 4} \log n, & d+r=2 \\
n^{-1 / 2(d+r)}, & d+r \geq 3
\end{array}\right.  +  \lambda \cdot \sqrt{2 \log k} \cdot  k^{-1/4} \\& \hspace{0.5em} + \min_{g \in \mathcal{G}} \{ \mathbf{Wass}_1(p_g(\mathbf{x},\mathbf{u}),p_{\Gamma}(\mathbf{x},\mathbf{u})) + \lambda \cdot \mathbb{E}_{\mathbf{x},\mathbf{z} \sim p_{\Omega}(\mathbf{x}),p(\mathbf{z})} \| \mathcal{L}g(\mathbf{x},\mathbf{z})-\mathbf{b}(\mathbf{x})\|_2^2 \}.
\end{split}
\end{equation}
where constants $C_1$, $C_2$ and $C_3$ are independent of the model architecture and $m$, $n$ as well as $k$; constant $C_4$ depends on the differential operator $\mathcal{L}$ and $M_b$ while constant $\Tilde{C}$ depends on the differential operator $\mathcal{L}$.
\end{theorem}

\begin{remark}
The above theorem demonstrates that the exact loss (generalization error) of our obtained model from (\ref{emp_loss_func}) converges to the approximation error with a high probability when 
sufficient training data and strong (wide and deep) discriminators are employed, i.e.,
\begin{equation*}
\begin{split}
        & \lim_{(m,n,k) \to \infty} \lim_{(D_f,W_f) \to \infty} \mathbf{Wass}_1(p_{\Tilde{g}}(\mathbf{x},\mathbf{u}), p_{\Gamma}(\mathbf{x},\mathbf{u}))+\lambda \cdot \mathbb{E}_{\mathbf{x}, \mathbf{z} \sim p_{\Omega}(\mathbf{x}),  p(\mathbf{z})} \| \mathcal{L}\Tilde{g}(\mathbf{x},\mathbf{z})-\mathbf{b}(\mathbf{x}) \|_2^2 \\
    = & \min_{g \in \mathcal{G}} \{ \mathbf{Wass}_1(p_g(\mathbf{x},\mathbf{u}),p_{\Gamma}(\mathbf{x},\mathbf{u})) + \lambda \cdot \mathbb{E}_{\mathbf{x},\mathbf{z} \sim p_{\Omega}(\mathbf{x}),p(\mathbf{z})} \| \mathcal{L}g(\mathbf{x},\mathbf{z})-\mathbf{b}(\mathbf{x})\|_2^2 \}.
\end{split}
\end{equation*} 
According to Assumption \ref{decay} and \ref{large_capacity}, 
for stronger discriminators, 
we observe that the value of $I_1$ is small, and generators with larger capacities (i.e., larger $W_g$ and $M$) can be adopted to have a lower approximation error. 

\end{remark}

\subsection{Discussion}
\subsubsection{Remarks for generalization error bound}
From the deep learning theorey perspective, decaying rates of a preferred generalization error bound with respect to the number of training data $m$ and $n$ are independent of the data dimension. Observe that the error bound (\ref{bound_emp_loss_1}), (\ref{error_bound_1}) and (\ref{error_bound_2}) suffers from curse of dimensionality that the decaying rates will exponentially degrade if we increase the dimension $d+r$. Here, we provide another generalization error bound whose decaying rates are irrelevant to the dimension $d+r$. 
\begin{definition}
\label{def_expec_max}
For $n$ independent and identically distributed variables $X_1$, $\cdots$, $X_n$ from $\nu$, the function $F(\nu,n)$ is defined as the expectation of maximal values over all these $n$ variables. Mathematically, we define
\begin{equation}
    F(\nu, n) = \mathbb{E}_{X_1, \cdots, X_n \sim \nu} \max\{\|X_1\|_2, \cdots, \|X_n\|_2\}.
\end{equation}
For a bounded distribution $\nu$, $F(\nu, n)$ is bounded and irrelevant with $n$. For the distribution $\nu$ whose tail decays exponentially, $F(\nu, n)$ is proportional to $\log(n)$.  For example, if $\nu=\text{Unif}([0,1]^{r})$, then $F(\nu, n) \leq \sqrt{r}$; for standard gaussian distribution $\nu=\mathcal{N}(\mathbf{0},\mathbf{I}_r)$, we have $F(\nu, n) \leq \sqrt{2r \cdot \log(2n)}$.
\end{definition}
\begin{theorem}
\label{improved_generalization_bound}
If all conditions hold as in Theorem \ref{main_theorem}, then with probability of at least $1-m^{-1/4}-n^{-1/4}-2\cdot k^{-1}$ (over the choice of training data), the exact loss for the obtained generator $\Tilde{g}$ converges to the approximation error, i.e., 
\begin{equation}
\label{error_bound_3}
    \begin{split}
            & \mathbf{Wass}_1(p_{\Tilde{g}}(\mathbf{x},\mathbf{u}),p_{\Gamma}(\mathbf{x},\mathbf{u})) + \lambda \cdot \mathbb{E}_{\mathbf{x},\mathbf{z} \sim p_{\Omega}(\mathbf{x}),p(\mathbf{z})} \| \mathcal{L}\Tilde{g}(\mathbf{x},\mathbf{z})-\mathbf{b}(\mathbf{x})\|_2^2 \\
    \leq & \hspace{0.5em} C_1 \cdot ( 2^{-D_f/(2(d+r)^2)} \vee W_f^{-1/2(d+r)^2}) +  C_5 \cdot \left((W_g+1)\cdot M \right) \cdot \sqrt{Pdim(\mathcal{F}_{\text{GS}})\cdot \log m} \cdot m^{-1/4}\\& \hspace{0.5em} + C_6 \cdot F\left(p_{\Gamma}(\mathbf{x},\mathbf{u}),n \right ) \cdot \sqrt{Pdim(\mathcal{F}_{\text{GS}})\cdot \log n} \cdot n^{-1/4}  + \lambda \cdot C_4 \cdot (W_g \cdot M)^{\Tilde{C}D_g} \cdot  \sqrt{\frac{2 \log k}{k}} \\& \hspace{0.5em}  + \min_{g \in \mathcal{G}} \{ \mathbf{Wass}_1(p_g(\mathbf{x},\mathbf{u}),p_{\Gamma}(\mathbf{x},\mathbf{u})) + \lambda \cdot \mathbb{E}_{\mathbf{x},\mathbf{z} \sim p_{\Omega}(\mathbf{x}),p(\mathbf{z})} \| \mathcal{L}g(\mathbf{x},\mathbf{z})-\mathbf{b}(\mathbf{x})\|_2^2 \},
\end{split}
\end{equation}
where $\text{Pdim}(\mathcal{F}_{\text{GS}})$ represents the pseudo-dimension \cite{anthony1999neural} of the discriminator class $\mathcal{F}_{\text{GS}}$; constants $C_1$, $C_5$ and $C_6$ are independent of the model architecture and $m$, $n$ as well as $k$; constant $C_4$ depends on the differential operator $\mathcal{L}$ and $M_b$ while constant $\Tilde{C}$ depends on the differential operator $\mathcal{L}$.
\end{theorem}
Proof of Theorem \ref{improved_generalization_bound} is available in \ref{appendix_proof}. The main improvement for the derived error bound lies in for $I_2$ and $I_3$. In Theorem \ref{main_theorem}, upper bounds for $I_2$ and $I_3$ (in Lemma \ref{bound_I2} and Lemma \ref{bound_I3}) are derived by regarding the discriminator class as a subset of the class of 1-Lipschitz functions (i.e., $\mathcal{F}_{\text{GS}} \subset \mathcal{F}_{\text{Lip}}$) and using the convergence properties of $\mathbf{Wass}_1$ in optimal transport \cite{NEURIPS2020_2000f632, lei2020convergence}. In deep learning theory literature, the convergence is studied by dividing the function space into finitely several covering balls. However, the derived convergence error bound (\ref{error_bound_3}) 
is proportional to the complexity of function spaces (the complexity is related to the numbers of balls required to cover the space). It is reasonable because the upper bound consider the worst case that we may rarely achieve in practise. Moreover, the error bound (\ref{error_bound_1}) and (\ref{error_bound_2}) are monotonically decreasing when we enlarge the capacity of discriminators (i.e., larger $W_f$ and $D_f$), compatible with numerical results for GANs in \cite{arora2018gans} and for the proposed WGAN-PINNs (as are shown in Figure \ref{exp1_dis}) that stronger discriminators contribute to higher generation qualities. Interestingly,  although the error bound (\ref{error_bound_3}) escapes the curse of dimensionality for $m$ and $n$, it seems to be a contradiction to numerical observations that $Pdim(\mathcal{F}_{\text{GS}})$ are proportional to depth and width of $\mathcal{F}_{\text{GS}}$. When $D_f$ and $W_f$ are large that $\mathcal{F}_{\text{GS}}$ is close to $\mathcal{F}_{\text{Lip}}$, error bounds (\ref{error_bound_1}) and (\ref{error_bound_2}) are stricter than (\ref{error_bound_3}), where we derive the former by regarding $\mathcal{F}_{\text{GS}} \subset \mathcal{F}_{\text{Lip}}$ and the latter merely by the finite capacity of the neural networks class $\mathcal{F}_{\text{GS}}$. After all, the derived error bounds do not represent all behaviors of the exact generalization error. 

\subsubsection{Stabilizing the training procedure}
Note that the constant $(W_g \cdot M)^{\Tilde{C}D_g}$ in the derived error bound (\ref{bound_emp_loss_1}) and (\ref{error_bound_1}) are exponentially increasing w.r.t. the depth and width of generators, which is undesired, especially for generators with large capacities when $W_g \cdot M \gg 1$. Recall that to stabilize training, we usually adopt $L_2$ regularization to softly constrain the Frobenius norm of parameters, i.e., there exists $L>0$ such that $\sum_{i=1}^{D_g}\|\mathbf{W}_i\|_F^2 \leq L$. Therefore, we have $\sum_{i=1}^{D_g}\|\mathbf{W}_i\|_2^2 \leq L$ and  $\prod_{i=1}^{D_g}\|\mathbf{W}_i\|_2 \leq \sqrt{\left (\frac{L}{D_g}\right )^{D_g}}$ by Cauchy–Schwarz inequality. Here, the upper bound for the Lipschitzness of generators $\prod_{i=1}^{D_g}\|\mathbf{W}_i\|_2 \leq \sqrt{\left (\frac{L}{D_g}\right )^{D_g}} \to 0$ when $D_g \to \infty$, which implies that we allow larger $L$ when neural networks are deeper. The selection of hyperparameters for $L_2$ regularization term is tricky but essential, and depends on both the given PDEs problems and the predefined generator class. In other words, we have to not only control the Lipschitzness of neural networks for stable training but also maintain their expressive power. Similarly, to stabilize GANs training, Brock et al. \cite{brock2016neural} suggested to use orthonormal regularization $\sum_{i=1}^{D_g} (\|\mathbf{W}_i^T\mathbf{W}_i - \mathbf{I}\|_F^2$ for generators to softly constrain the Lipschitzness of neural networks. 

Other methods like Bjorck orthonormalization and spectral normalization, which are much stronger than $L_2$ and orthonormal regularization, are not recommended for generators in approximating solutions of PDEs without strong prior that the solution is 1-Lipschitz w.r.t. the variable $\mathbf{x}$ and the uncertainty $\mathbf{z}$. Otherwise, although we stabilize training, the expressive power of neural networks is weakened.

\subsubsection{Some limitations}
\label{discussion_limitations}
In this part, we mainly discuss some limitations of using the proposed method to do uncertainty quantifications for solutions of PDEs. Actually, some of them also exist for solving deterministic PDEs problems by deep learning (e.g., PINNs \cite{raissi2019physics}).

The reasonableness of the proposed model with boundary/initial samples and governing equations. As is mentioned in section \ref{proposed_wgan_pinns_model}, the core idea of the proposed model comes from traditional PDEs solvers that solve PDEs with given boundary/initial data and governing equations, where we do not require interior solution data. Empirically, it is not that easy to collect a large amount of solution data in the interior domain randomly. That is the main reason we propose the model (\ref{model_original}) which minimizes the distance of the generated data and the solution data on the boundary and try to propagate the uncertainty from the boundary to the interior. If we are capable to sample sufficient solution data in the interior domain $\Omega$, then we can directly transform the model (\ref{model_original}) to the following (\ref{model_original_solution}) which minimizes the distance of the generated data and the solution data in the interior domain with the governing PDEs constraints:
\begin{equation}
\label{model_original_solution}
    \begin{split}
        & \min_{g_{\theta} \in \mathcal{G}} \ 
        \mathbf{D}((\mathbf{x},g_{\theta}(\mathbf{x},\mathbf{z})), (\mathbf{x},\mathbf{u})), \hspace{1em} \mathbf{x} \sim p_{\Theta}(\mathbf{x}), \ \mathbf{z} \sim p(\mathbf{z}) \text{ and } \Theta = \Gamma \cup \Omega \\
        & s.t. \hspace{1em} \mathcal{L}g_{\theta}(\mathbf{x},\mathbf{z})=\mathbf{b}(\mathbf{x}), \hspace{0.5em} \forall \mathbf{x} \in \Omega  \text{ and } \mathbf{z} \sim p(\mathbf{z}).
    \end{split}
\end{equation}
where $\Theta = \Gamma \cup \Omega$ represents the union of boundary and interior domains.
The derived generalization error bound makes sense because it guarantees the quality of the generation in the interior domain.  As is discussed in \cite{raissi2019physics, yang2019adversarial}, the interior data definitely contributes to the model performance. We leave them as future works for further improvements and applications of the proposed model in solving forward (e.g., solving stochastic PDEs) and inverse problems (e.g., estimating coefficients in PDEs).

Approximation capabilities of neural networks to solutions. Similar worries also arise in solving deterministic PDEs by deep neural networks (PINNs). As a matter of fact, the approximation error of PDEs can be transformed into a kind of function approximation error. Specifically, if the solution $\mathbf{u}$ is governed by the spatial (and temporal) variable $\mathbf{x}$ and the uncertainty $\mathbf{z}$, i.e., $\mathbf{u} = \mathbf{u}(\mathbf{x},\mathbf{z})$, then \begin{align*}
    & \min_{g \in \mathcal{G}} \mathbf{Wass}_1(p_g(\mathbf{x},\mathbf{u}),p_{\Gamma}(\mathbf{x},\mathbf{u})) + \lambda \cdot \mathbb{E}_{\mathbf{x},\mathbf{z} \sim p_{\Omega}(\mathbf{x}),p(\mathbf{z})} \| \mathcal{L} g(\mathbf{x},\mathbf{z})-\mathbf{b}(\mathbf{x})\|_2^2\\
  = & \min_{g \in \mathcal{G}} \max_{f \text{ is 1-Lipschitz}} \{ \mathbb{E}_{\mathbf{x},\mathbf{z} \sim p_{\Gamma}(\mathbf{x}),p(\mathbf{z})}f(\mathbf{x},g(\mathbf{x},\mathbf{z})) -  \mathbb{E}_{\mathbf{x},\mathbf{z} \sim p_{\Gamma}(\mathbf{x}),p(\mathbf{z})} f(\mathbf{x},\mathbf{u}(\mathbf{x},\mathbf{z})) \}\\
  & + \lambda \cdot \mathbb{E}_{\mathbf{x},\mathbf{z} \sim p_{\Omega}(\mathbf{x}),p(\mathbf{z})} \| 
  \mathcal{L}g(\mathbf{x},\mathbf{z})-
  \mathcal{L}\mathbf{u}(\mathbf{x},\mathbf{z})\|_2^2\\
  \leq & \min_{g \in \mathcal{G}} \mathbb{E}_{\mathbf{x},\mathbf{z} \sim p_{\Gamma} (\mathbf{x}),p(\mathbf{z})} \left \|g(\mathbf{x},\mathbf{z})-\mathbf{u}(\mathbf{x},\mathbf{z})\right \|_2 + \lambda \cdot \mathbb{E}_{\mathbf{x},\mathbf{z} \sim p_{\Omega}(\mathbf{x}),p(\mathbf{z})} \left \|\mathcal{L} g(\mathbf{x},\mathbf{z})- \mathcal{L}\mathbf{u}(\mathbf{x},\mathbf{z})\right\|_2^2.
\end{align*}
If the deep neural networks have large capacities in approximating functions (e.g., there exists $g \in \mathcal{G}$, such that $g(\mathbf{x},\mathbf{z}) \approx \mathbf{u}(\mathbf{x},\mathbf{z})$ and $\mathcal{L} g(\mathbf{x},\mathbf{z}) \approx \mathcal{L} \mathbf{u}(\mathbf{x},\mathbf{z})$ ), then the approximation error is small. The approximation capability of deep ($tanh$ or $ReLU^{q}$) neural networks is an interesting and essential work for future research. Up to now, most of related works studies the approximation power of shallow neural networks \cite{luo2020two, de2021approximation, siegel2021high, siegel2020approximation, wojtowytsch2020representation}. 

The generalization analysis does not guarantee the quality of uncertainty quantification of the solution in the interior domain, where we match the data distribution on the boundary and the interior residual for differential equations. In other words, the generalization analysis does not reflect the role of physics-informed (PINNs) regularization term in uncertainty propagation. The derived generalization error bound is directly based on the generalization of GroupSort WGAN and PINNs. This is a considerable limitation. Back to PDEs problems, the distance $\mathbf{Wass}_1(p_{\Tilde{g}}(\mathbf{x},\mathbf{u}),p_{\Gamma}(\mathbf{x},\mathbf{u}))$ measures the difference (in terms of the distribution) between the generated data and the solution data on the  boundary and we call it the boundary residual. And similarly,  $\mathbb{E}_{\mathbf{x},\mathbf{z} \sim p_{\Omega}(\mathbf{x}),p(\mathbf{z})} \| \mathcal{L}\Tilde{g}(\mathbf{x},\mathbf{z})-\mathbf{b}(\mathbf{x})\|_2^2$ indicates the interior residual. For  deterministic problems, the quality of solutions is guaranteed for some typical PDEs which are continuously data dependent, where the error of solutions are controlled by the interior residual as well as the boundary residual \cite{mishra2020estimates,shin2020convergence}. However, for probabilistic models, the rigorous derivation of data dependency (in terms of distributions) for PDEs becomes really difficult. In traditional numerical methods for stochastic PDEs, prior knowledge for the exact distributions of random coefficients and boundary conditions are essential. However, in our setting, we have no priors but merely some random samples. WGAN aims to learn the random boundary conditions (uncertainty) from observed samples and PINNs propagates the uncertainty to the interior domain by PDEs constraints. In this case, the approximation of the random boundary conditions raises another problem of data dependency for stochstic PDEs in terms of distribution distances (e.g., Wasserstein distance etc.). That means even if we theoretically derive the generalization error bound (\ref{error_bound_2}) for the proposed model, the quality of the generated data in the interior domain which is our target, is not guaranteed. This is actually an unsolved problem for uncertainty quantification of PDEs, not for machine learning or deep learning. Moreover, the dependency of the solution on the boundary (in Wasserstein distance) and in the interior is unknown for probabilistic models, therefore, the selection of hyperparameters $\lambda$ has no theoretical supports, mainly by empirical experience (e.g., validation sets).

In our setting, we do not distinguish 
initial/boundary data. However, for PDEs, initial and boundary data have different contributions and solution error should have unequal weights for boundary and initial residuals. In deterministic problems, separating initial and boundary data in the model is applicable but requires additional hyperparameters. Moreover, two discriminators are needed in the proposed probabilistic model if we distinguish boundary and initial conditions, leading to training instability and even failure. Therefore, we do not recommend to distinguish the boundary and initial conditions without further information or corresponding algorithmic designs.

\section{Experimental Results}
\label{num_results}

Our theoretical analysis in the last section
conveys that with sufficient training data and discriminators of large capacity, the exact loss will converge to the approximation error (the minimal error). In this section, numerical experiments are conducted on different PDEs examples to verify our theories and show the effectiveness of our WGAN-PINNs model in solving PDEs with uncertain boundary conditions. Related codes and data of four examples below are available at the website\footnote{\url{https://github.com/yihang-gao/WGAN_PINNs}.}.

Firstly, we discuss the setting of our experimental results.
We set the number $m$ of samples on random input to generator is equal
to the number $n$ of samples on boundary data.
Here our observations are $\{(\mathbf{\bar{x}}_i,\mathbf{z}_i)\}_{i=1}^m$, $\{(\mathbf{\bar{x}}_i,\mathbf{u}_i)\}_{i=1}^m$ on the boundary and $\{(\mathbf{x}_j,\mathbf{b}_j)\}_{j=1}^k$ in the interior domain. Therefore, the empirical loss becomes
\begin{equation}
    \widehat{\text{Loss}} = \min_{g_{\theta} \in \mathcal{G}} \max_{f_{\alpha} \in \mathcal{F}} \frac{1}{m} \sum_{i=1}^m f_{\alpha}(\mathbf{\bar{x}}_i,g_{\theta}(\mathbf{\bar{x}}_i,\mathbf{z}_i)) - \frac{1}{m}\sum_{i=1}^m f_{\alpha}(\mathbf{\bar{x}}_i,\mathbf{u}_i) + \lambda \cdot \frac{1}{k} \sum_{j=1}^k \| \mathcal{L}g_{\theta}(\mathbf{x}_j,\mathbf{z}_j) - \mathbf{b}_j \|_2^2
\end{equation}
The following numerical method is used to solve 
min-max problem. 
We optimize two neural networks (generators and discriminators) in an alternating manner, i.e.,
\begin{equation*}
    \begin{split}
        & \max_{f_{\alpha} \in \mathcal{F}} \frac{1}{m} \sum_{i=1}^m f_{\alpha}(\mathbf{\bar{x}}_i,g_{\theta}(\mathbf{\bar{x}}_i,\mathbf{z}_i)) - \frac{1}{m}\sum_{i=1}^m f_{\alpha}(\mathbf{\bar{x}}_i,\mathbf{u}_i) \\
        & \min_{g_{\theta} \in \mathcal{G}} \frac{1}{m} \sum_{i=1}^m f_{\alpha}(\mathbf{\bar{x}}_i,g_{\theta}(\mathbf{\bar{x}}_i,\mathbf{z}_i)) +  \lambda \cdot \frac{1}{k} \sum_{j=1}^k \| \mathcal{L}g_{\theta}(\mathbf{x}_j,\mathbf{z}_j)-\mathbf{b}_j \|_2^2
    \end{split}
\end{equation*}
In our numerical method, we merely update the parameters several times alternately in each iteration. The overall computational cost of training can be reduced. 
Such strategy is adopted in 
\cite{goodfellow2014generative,arjovsky2017wasserstein} for adversarial learning.

The estimated Wasserstein distances between the target distributions and the generated distributions are computed by POT package \cite{flamary2021pot} 
on 10,000 samples while the residuals are estimated on 10,000 samples. For each experiment, we repeat ten times and calculate the averaged value to represent the corresponding expectation. To show the performance of our WGAN-PINNs model, we calculate the empirical mean value of $\Tilde{g}(\mathbf{x},\mathbf{z})$ (i.e. $\mu(\mathbf{x})=\frac{1}{N} \sum_{i=1}^N \Tilde{g}(\mathbf{x},\mathbf{z}_i)$) to estimate the solution $\mathbf{u}(\mathbf{x})$ without uncertainly.
Relative $L_2$ error $\mathcal{E}$ in (\ref{relative_error}) is utilized to measure the error between the empirical mean value $\mathbf{\mu}(\mathbf{x})$ and the solution $\mathbf{u}(\mathbf{x})$ without uncertainly:
\begin{equation}
    \label{relative_error}
    \mathcal{E} = \frac{\sqrt{\sum_{i=1}^N \|\mathbf{\mu}(\mathbf{x}_i)-\mathbf{u}(\mathbf{x}_i)\|_2^2}}{\sqrt{\sum_{i=1}^N \|\mathbf{u}(\mathbf{x}_i)\|_2^2}}
\end{equation}

Features of our interest are the number of training data (denoted as $m$, $n$ and $k$ consistent with those in theoretical analysis), as well as width ($W_g$ and $W_f$, respectively) and depth ($D_g$ and $D_f$, respectively) of both the generator and the discriminator. A relaxed (proximal) constraints on parameters of discriminators are adopted \cite{anil2019sorting,tanielian2021approximating}. Similarly, we use the Bjorck orthonormalization method to constrain parameters to be unit in $||\cdot||_2$ norm. For more details of Bjorck orthonormalization applied in Wasserstein GANs, please refer to \cite{bjorck1971iterative, anil2019sorting}. To improve the convergence and stability of both the generator and the discriminator, they are trained by Adam \cite{kingma2014adam} with default hyperparameters 
(listed in Table \ref{tab_hyperprm}) for each experiments.

\begin{table}
\footnotesize
\centering
\begin{tabular}{| c|c |} 
\hline
Operation & Features  \\
\hline
 Generator & \multirow{2}{*}{the architecture is changing in each experiments} \\ 
 Discriminator &  \\ 
\hline
 Bjorck iteration steps & 5\\
 Bjorck order & 2 \\
 \hline
 Optimizer & Adam: $\beta_1 = 0.9$, $\beta_2 = 0.99$ \\ 
 Learning rate (for both the generator and &  \multirow{2}{*}{$1 \times e^{-4}$}\\
 the discriminator) & \\
 \hline
\end{tabular}
\caption{Hyperparameters.}
\label{tab_hyperprm}
\end{table}

\subsection{A Pedagogical Example}
\label{pedagogical_example}
In the first example, 
we test our model on a non-linear 1-D ordinary differential equation to both verify our theory and show its effectiveness in solving uncertain differential equation problems. The ODE is
\begin{equation}
\label{exp1_ode}
        u_{xx} - u^2 u_x = f(x), \hspace{1em} x \in [-1,1]
\end{equation}
with random boundary conditions $u(-1) \sim \mathcal{N}(0, \sigma_1^2)$ and $u(1) \sim \mathcal{N}(0, \sigma_2^2)$.
Here the right hand side function is equal to $f(x) = - \pi^2 \sin (\pi x) - \pi \cos (\pi x) \sin^2 (\pi x)$. For deterministic problem, 
there is no uncertainly (i.e., $\sigma_1=\sigma_2=0$) and 
then the solution is $u(x)=\sin(\pi x)$. 


\begin{table}[!h]
\scriptsize
\centering
\begin{tabular}{c|c c|c c|c c|c c|c c}
  \multirow{2}{*}{\diagbox{$\sigma$}{$\lambda$}} & \multicolumn{2}{c|}{10} & \multicolumn{2}{c|}{50} & \multicolumn{2}{c|}{100} & \multicolumn{2}{c|}{500} & \multicolumn{2}{c}{1000} \\
  & $\mathcal{E}$ & Loss & $\mathcal{E}$ & Loss & $\mathcal{E}$ & Loss & $\mathcal{E}$ & Loss & $\mathcal{E}$ & Loss\\
  \hline
  0 & 2.97e-03 & 2.64e-03 & 1.76e-03 & 4.28e-03 & \textbf{1.02e-03} & 6.50e-03 & 2.13e-03 & 2.18e-02 & 1.88e-03 & 3.32e-2 \\
  0.05 & 4.79e-02 & 4.88e-02 & 2.87e-02 & 2.81e-02 & \textbf{1.89e-02} & 2.01e-02 & 2.09e-02 & 4.09e-02 & 2.19e-02 & 5.52e-02 \\ 
  0.1 & 9.47e-02 & 8.15e-02 & 3.80e-02 & 3.52e-02 & 3.13e-02 & 3.14e-02 & \textbf{2.87e-02} & 6.32e-02 & 3.33e-02 & 7.63e-02\\
  0.2 & 2.14e-01 & 2.03e-01 & 7.10e-02 & 6.12e-02 & 6.48e-02 & 6.15e-02 & \textbf{6.28e-02} & 1.42e-01 & 6.98e-02  & 1.68e-01 \\
\end{tabular}
\caption{\label{tab1}
The effect of parameter $\lambda$ and noise level $\sigma$ with 
$(D_g,W_g)=(D_f,W_f)=(3,50)$ and $m=n=40$, $k=100$.}
\end{table}

\begin{table}[!h]
\scriptsize
\centering
\begin{tabular}{c|c|c|c|c|c}
  \multirow{2}{*}{\diagbox{$\sigma$}{$\lambda$}} & 1 & 10 & 50 & 100 & 500 \\
  & $\mathcal{E}$  & $\mathcal{E}$ & $\mathcal{E}$  & $\mathcal{E}$  & $\mathcal{E}$\\
  \hline
  0 & 6.17e-03 & 2.63e-03 & \textbf{1.24e-03} & 1.70e-03 & 1.52e-03\\
  0.05 & 5.01e-02 & 5.67e-02 & \textbf{3.76e-02} & 6.45e-02 & 5.56e-02\\
  0.1 & 7.31e-02 & 1.21e-01 & 8.13e-02 & \textbf{7.35e-02} & 9.67e-02\\
  0.2 & 2.15e-01 & 1.89e-01 & 2.77e-01 & \textbf{1.31e-01} & 1.58e-01
\end{tabular}
\caption{\label{tab2}
The effect of parameter $\lambda$ and noise level $\sigma$ with $(D_g,W_g)=(D_f,W_f)=(3,50)$ and $m=n=40$, $k=100$.}.
\end{table}

\begin{figure*}[h!]
\makebox[\linewidth][c]{
  \centering
  \subfloat
  [$\sigma_1 = \sigma_2 = 0$] 
  {
     \label{exp1_pre_0}     
    \includegraphics[width=0.5\textwidth]{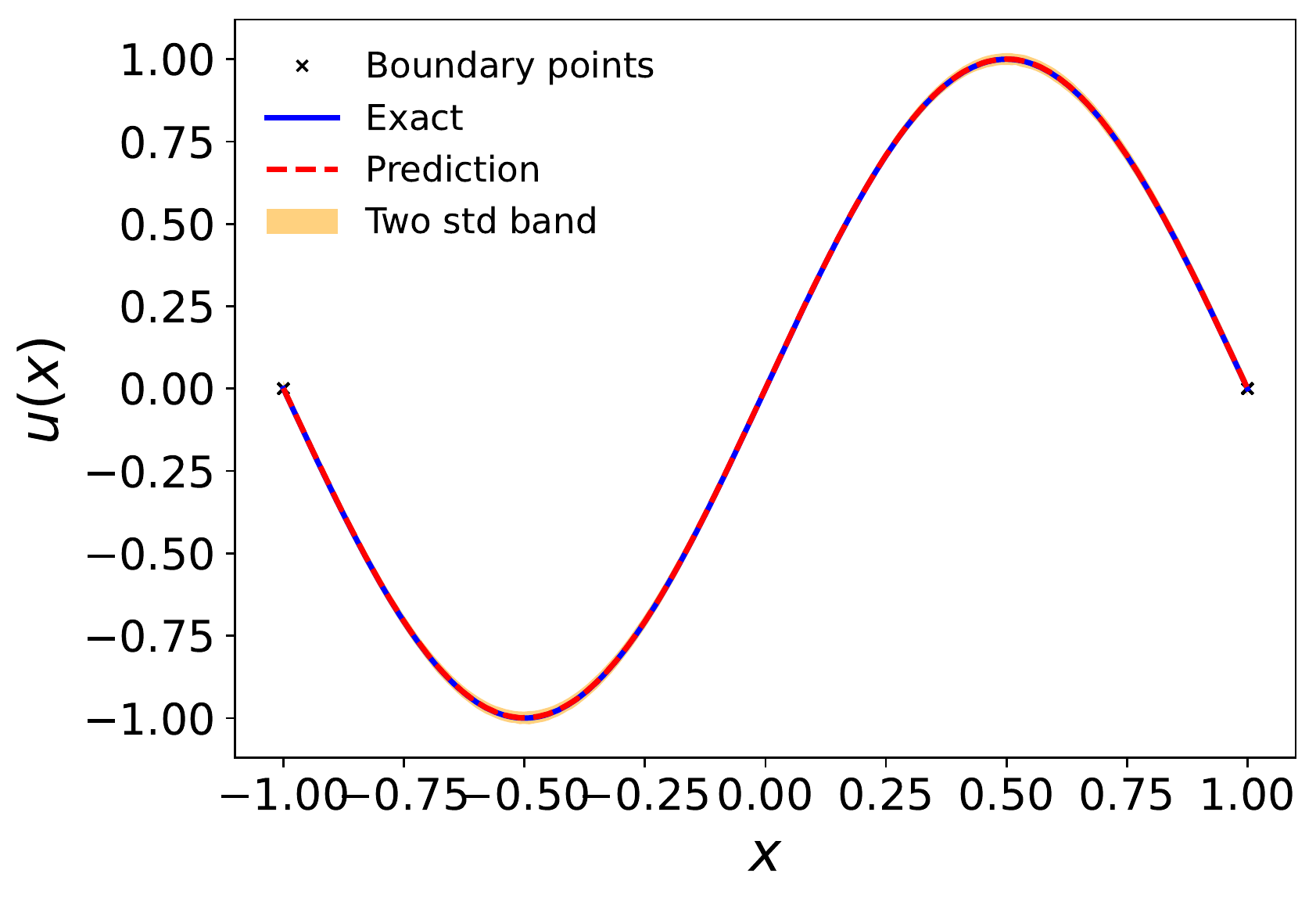}
    }\  
    \subfloat
    [$\sigma_1 = \sigma_2 = 0.05$] 
    {
    \label{exp1_pre_0.05}     
    \includegraphics[width=0.5\textwidth]{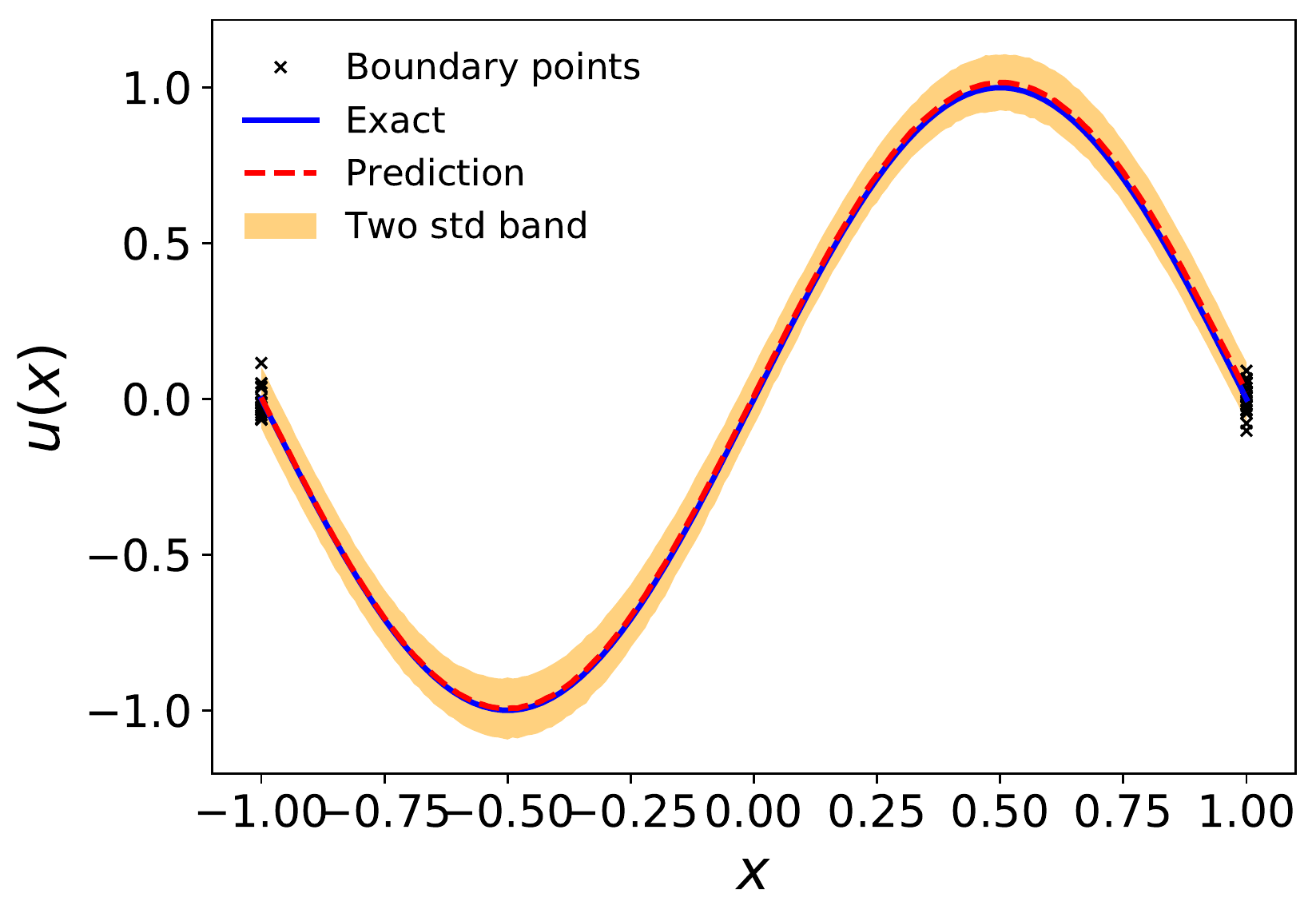}
    }\ 
    %
		}
    \makebox[\linewidth][c]{%
  \centering
  \subfloat
  [$\sigma_1 = \sigma_2 = 0.1$] 
  {
     \label{exp1_pre_01}     
    \includegraphics[width=0.5\textwidth]{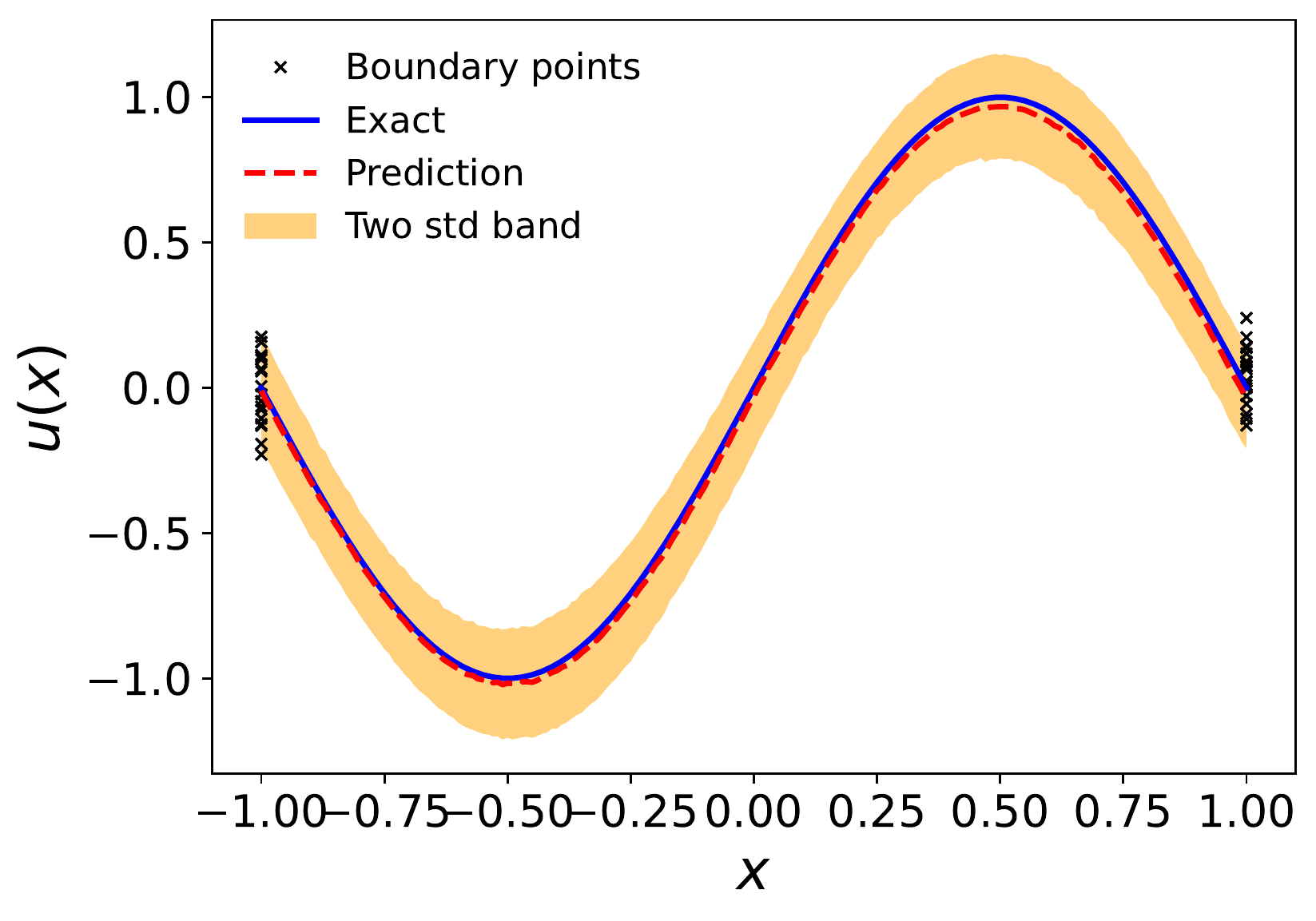}
    }\  
    \subfloat
    [$\sigma_1 =0$, \hspace{0.3em} $\sigma_2 = 0.2$] 
    {
    \label{exp1_pre_0.2}     
    \includegraphics[width=0.5\textwidth]{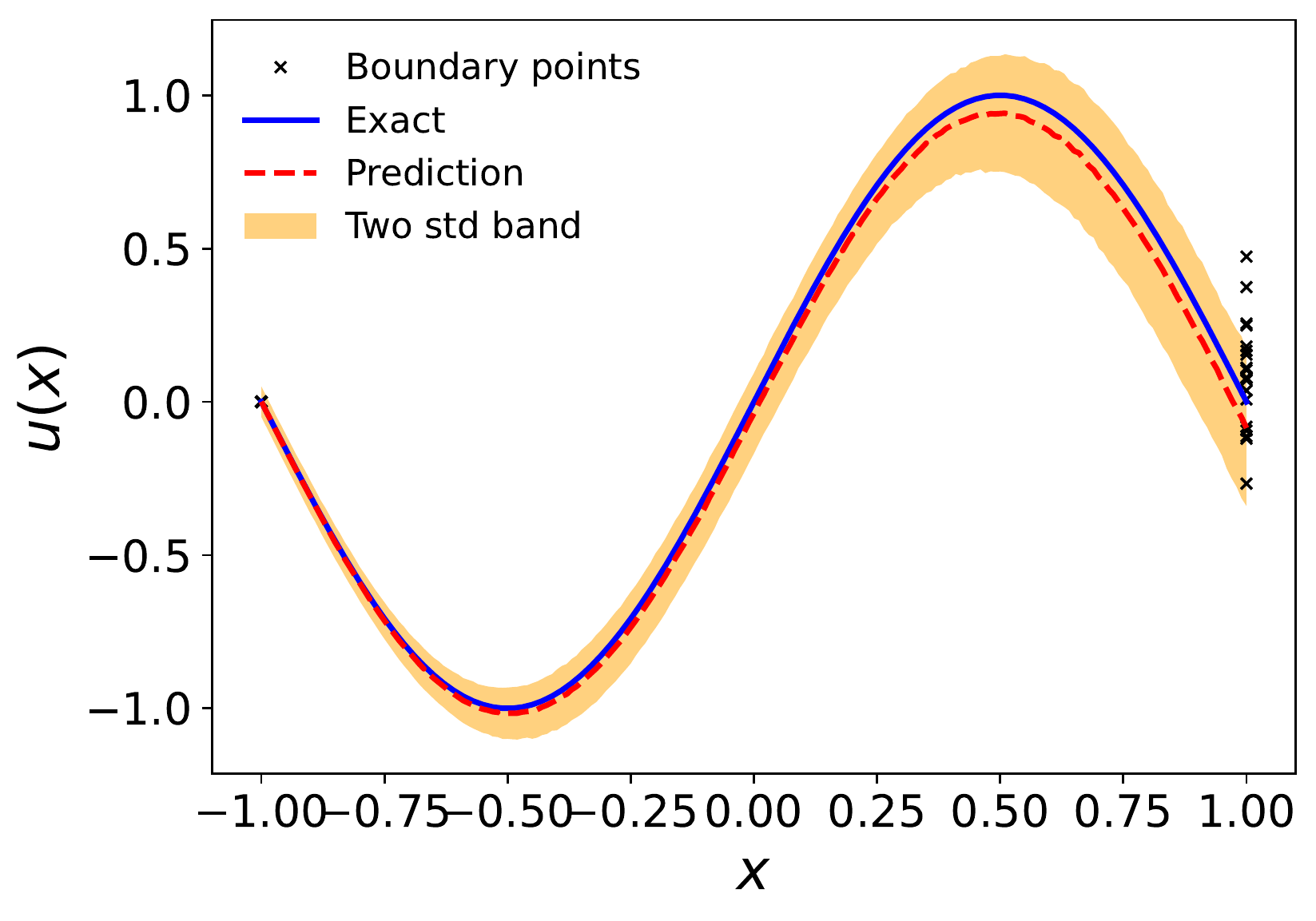}
    }\
    }
    \caption{		
The mean, the lower bound and upper bound of 
PDE solutions $p_{\Tilde{g}}(u)$ given $x$.
(a) $\text{Loss}=6.50\times 10^{-3}$, $\mathcal{E}=1.02\times 10^{-3}$; (b) $\text{Loss}=2.01\times 10^{-2}$, $\mathcal{E}=1.89\times 10^{-2}$; (c) $\text{Loss}=3.14\times 10^{-2}$, $\mathcal{E}=3.13\times 10^{-2}$; (d) $\text{Loss}=6.87 \times 10^{-2}$, $\mathcal{E}=5.98 \times 10^{-2}$. Here exact PDE solution (the blue line) refers to PDE solution without uncertainty.}	
    \label{exp1_visual_pre}
\end{figure*}

To show the effect of parameter $\lambda$ and noise levels $\sigma_1$ 
and $\sigma_2$ in the model,
we fix the model architecture and the training data and test the performance 
of the proposed WGAN-PINNs algorithm.
Here 
both the generator and the discriminator are neural networks of 3 hidden layers with 50 neurons in each layer. 
Without any priors and preferences, distributions of $\mathbf{x}$ on the boundary $\Gamma$ and in the interior domain $\Omega$ are uniform.
For two boundary conditions $u(-1)$ and $u(1)$, noise levels $\sigma_1 = \sigma_2$ are set to be from $0$ to $0.2$ and we sample 20 data points for each (i.e., $n=40$). 
Also the number $m$ of samples on such random input to generator is equal to 
$n$, i.e., $m=40$.
The number $k$ of samples on interior data is set to be 100.
A two dimensional standard normal prior for $\mathbf{z}$ is adopted, 
namely $\mathbf{z} \sim \mathcal{N}(\mathbf{0},\mathbf{I}_2)$. 

Numerical results related to the loss and the relative error are listed in Tables 
\ref{tab1} and \ref{tab2}. 
Larger noise level introduces more uncertainty in the system, leading to 
more difficulties in training.
Table \ref{tab1} reveals the role of PINNs as a regularization term that it accelerates the convergence and improves stability of WGANs.
An intuitive explanation is that PINNs term narrows the function space and enforces generators to approximate the implicit solution. 
We see from Tables \ref{tab1} and \ref{tab2} that the relative error by 
WGAN-PINNs is smaller than that by GAN-PINNs.
Visualized results of our generated solution with different noise levels when $\lambda=100$ are shown in Figure \ref{exp1_visual_pre}. Statistical comparisons between $p_{\Tilde{g}}(\mathbf{x},\mathbf{u})$ and  $p_{\Gamma}(\mathbf{x},\mathbf{u})$ at $\mathbf{x}=-1, 0, 1$ shown in Figure \ref{exp1_visual_his}
further demonstrate the effectiveness of WGAN-PINNs. In particular, the means $\hat{\mu}$ 
and the standard derivations $\hat{\sigma}$
of the generated distributions at the two boundary points 
are close to those $\mu=0$ and $\sigma_1$ (or $\sigma_2$) 
given by normal distributions. The mean and the standard derivation of the 
generated distributions at $\mathbf{x}=0$ are close to those 
($\mu_{est},\sigma_{est}$) by 
simulated partial differential equation solution data. 
Note that the exact distribution of partial differential equation solution is unknown.


\begin{figure*}[h!]
\makebox[\linewidth][c]{%
  \centering
  \subfloat
  [$u(x=-1)$] 
  {
     \label{exp1_his_1}     
    \includegraphics[width=0.3\textwidth]{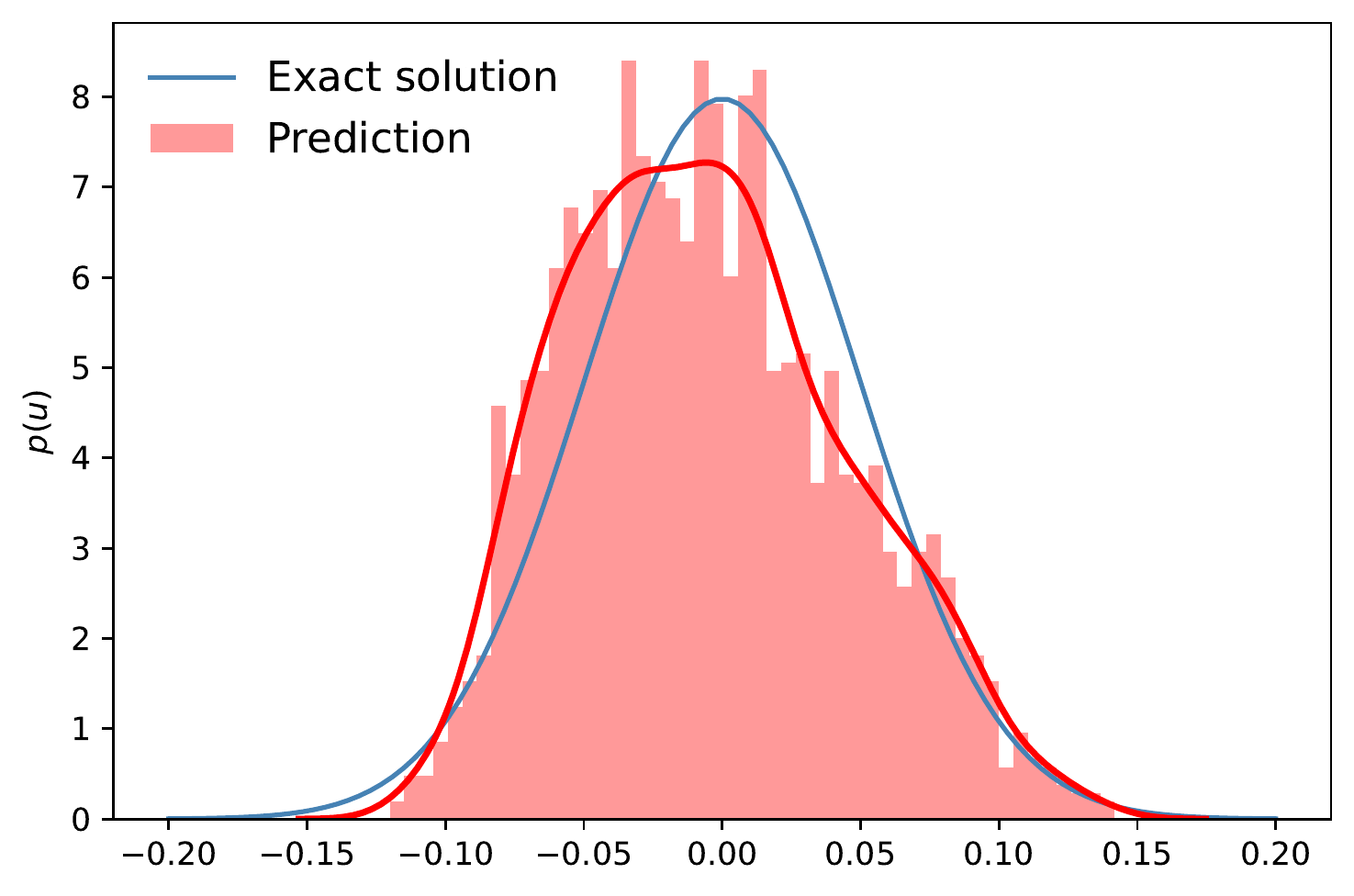}
    }\  
        \subfloat
    [$u(x=1)$] 
    {
    \label{exp1_his_5}     
    \includegraphics[width=0.3\textwidth]{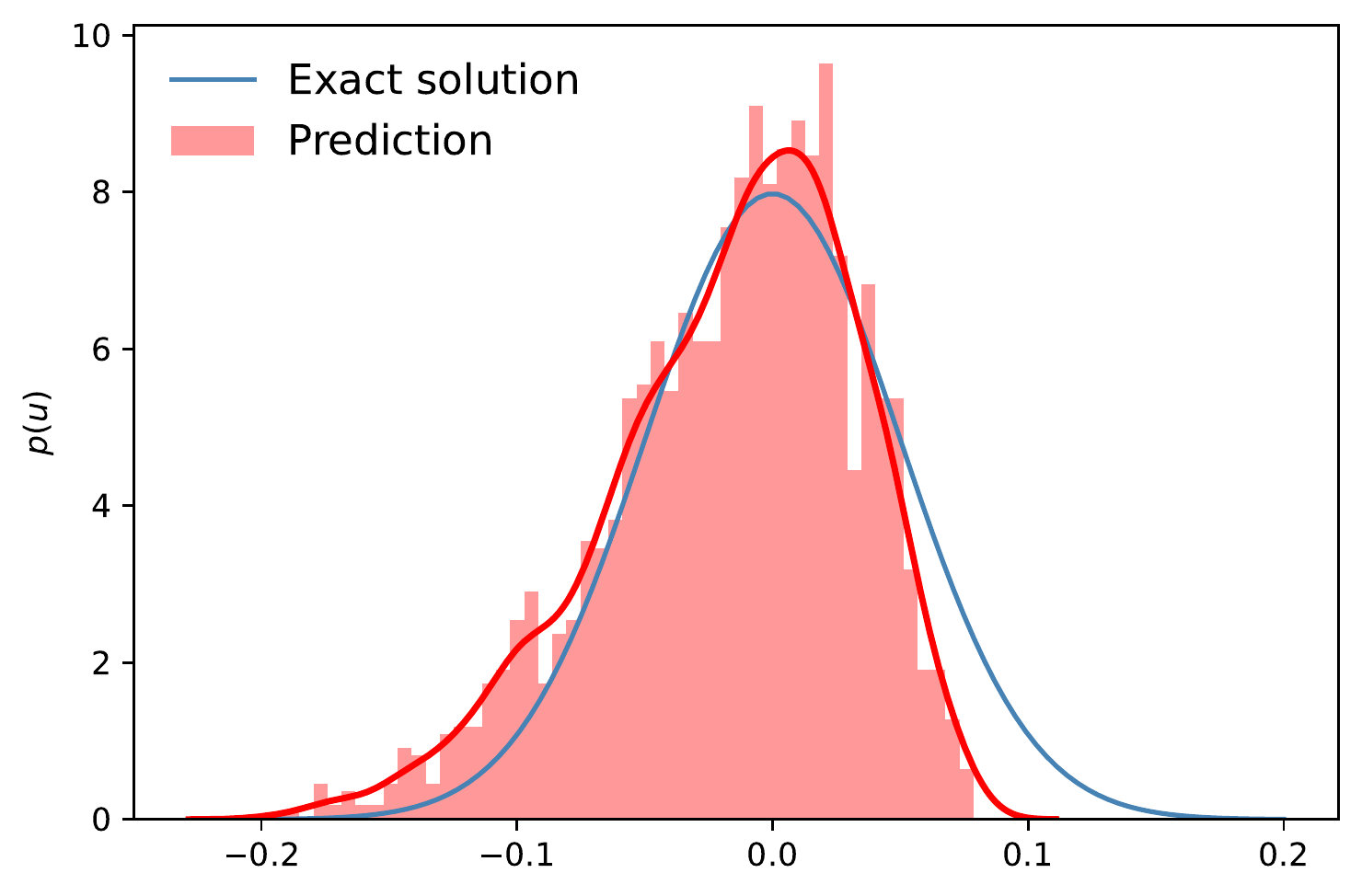}
    }
    }\\
    \makebox[\linewidth][c]{%
  \centering
   \subfloat
    [$u(x=-0.5)$] 
    {
    \label{exp1_his_2}     
    \includegraphics[width=0.3\textwidth]{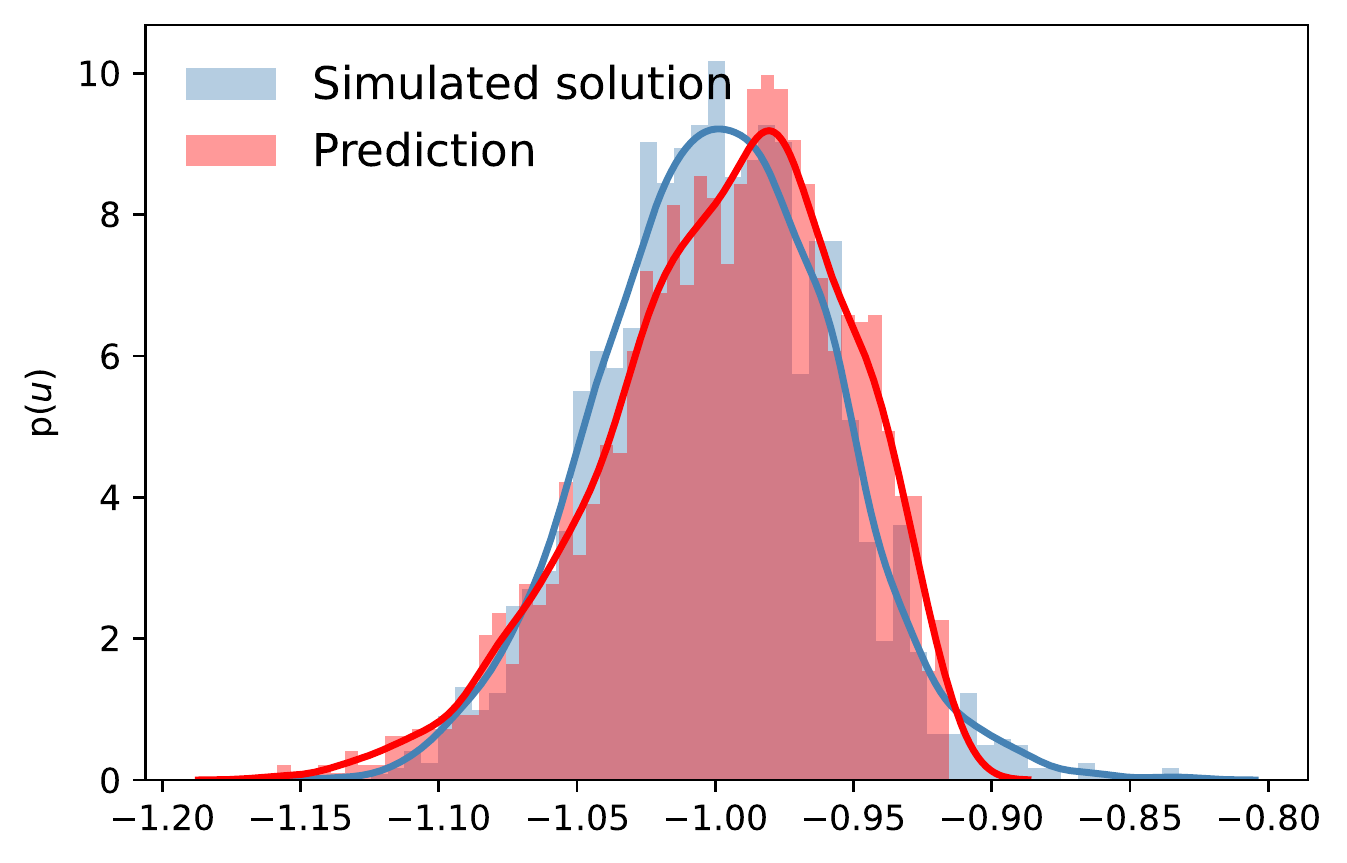}
    }\
    \subfloat
    [$u(x=0)$] 
    {
    \label{exp1_his_3}     
    \includegraphics[width=0.3\textwidth]{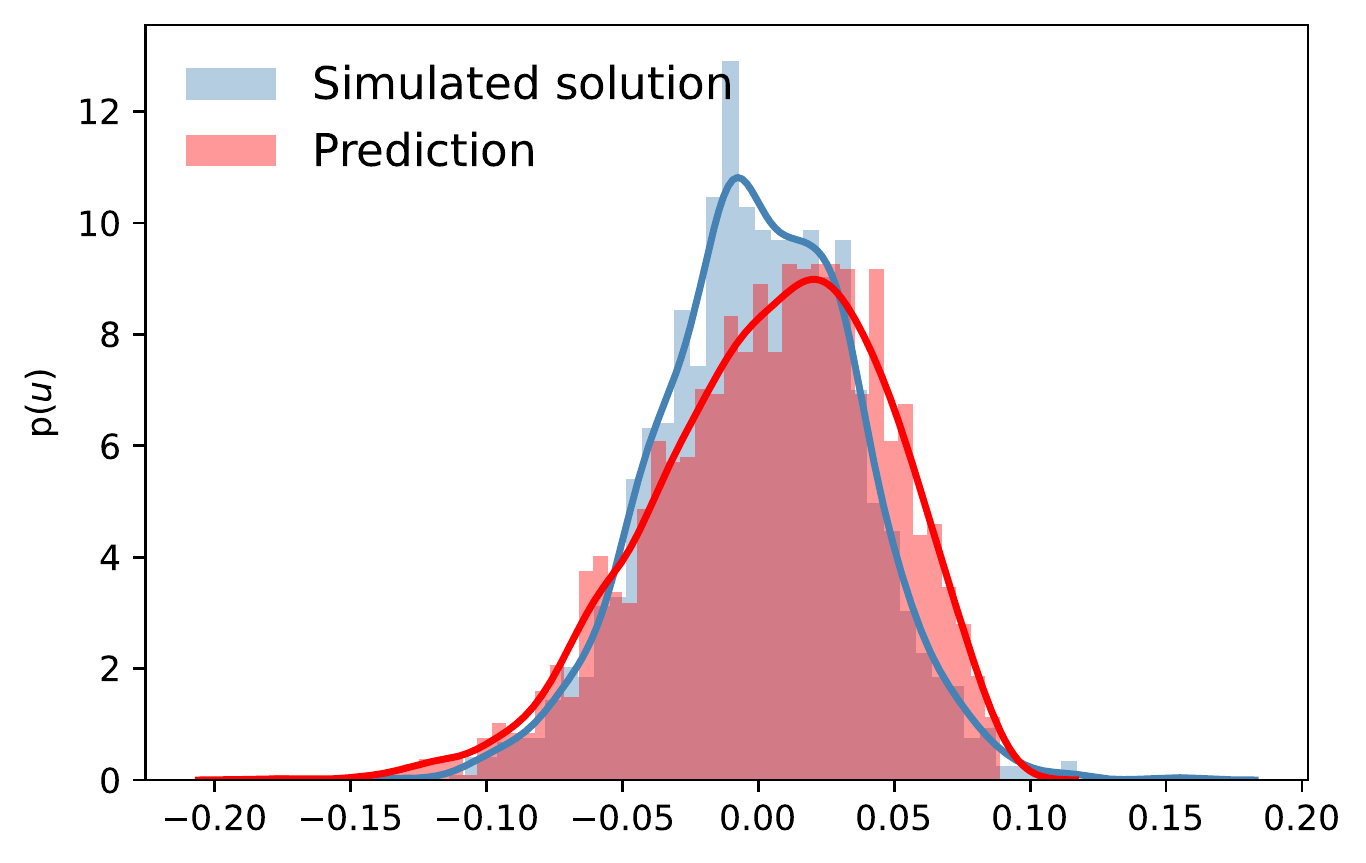}
    }\
    \subfloat
    [$u(x=0.5)$] 
    {
    \label{exp1_his_4}     
    \includegraphics[width=0.3\textwidth]{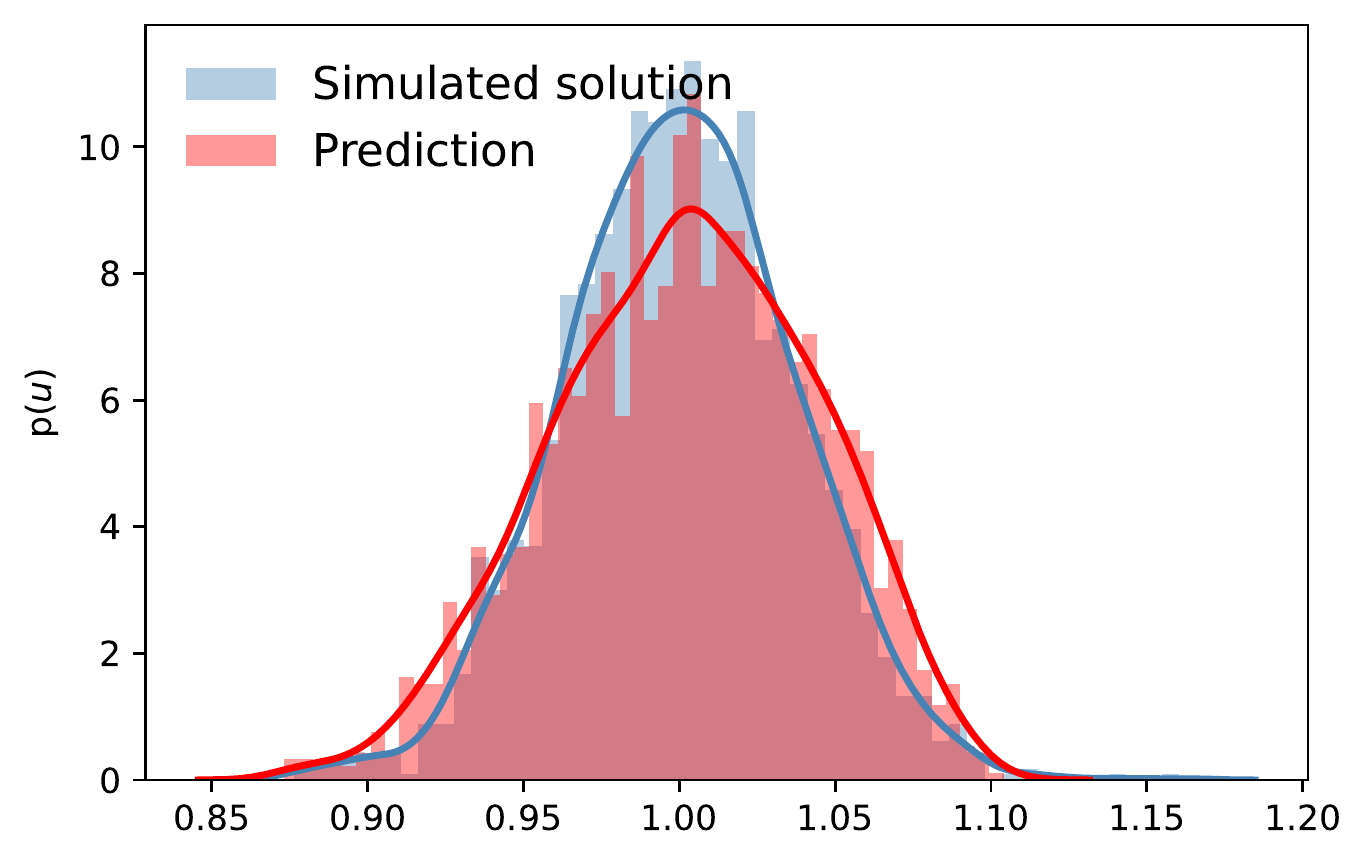}
    }\
    }
    \caption{		
	The histograms of the generated solution by WGAN-PINNs 
	and the simulated numerical solution at $x=-1$, $x=1$ (boundary) and $x=-0.5$, $x=0$, $x=0.5$  (interior) for $\sigma_1 = \sigma_2 = 0.05$. The predicted mean and standard deviation are 
	$(\hat{\mu},\hat{\sigma})$,
	the exact mean and standard deviation are $(\mu,\sigma)$, 
	and the simulated mean and standard derivation are $(\mu_{est},\sigma_{est})$.	
	(a) $(\hat{\mu},\hat{\sigma})=(-0.006,0.050)$, $(\mu,\sigma)=(0,0.05)$; (b) $(\hat{\mu},\hat{\sigma}) = (-0.016,0.049)$, $(\mu,\sigma)=(0,0.05)$; (c) $(\hat{\mu},\hat{\sigma}) = (-1.011, 0.047)$, $(\mu_{est},\sigma_{est})=(-1.000,0.042)$; (d) $(\hat{\mu},\hat{\sigma}) = (-0.014, 0.045)$, $(\mu_{est},\sigma_{est})=(0.000,0.037)$; (e) $(\hat{\mu},\hat{\sigma}) = (0.985, 0.045)$, $(\mu_{est},\sigma_{est})=(1.000,0.037)$.}	
    \label{exp1_visual_his}
\end{figure*}

\begin{figure*}[h]
\makebox[\linewidth][c]{%
  \centering
  \subfloat
  [$m=n$ boundary data] 
  {
     \label{exp1_loss_numtrain_mn}     
    \includegraphics[width=0.23\textwidth]{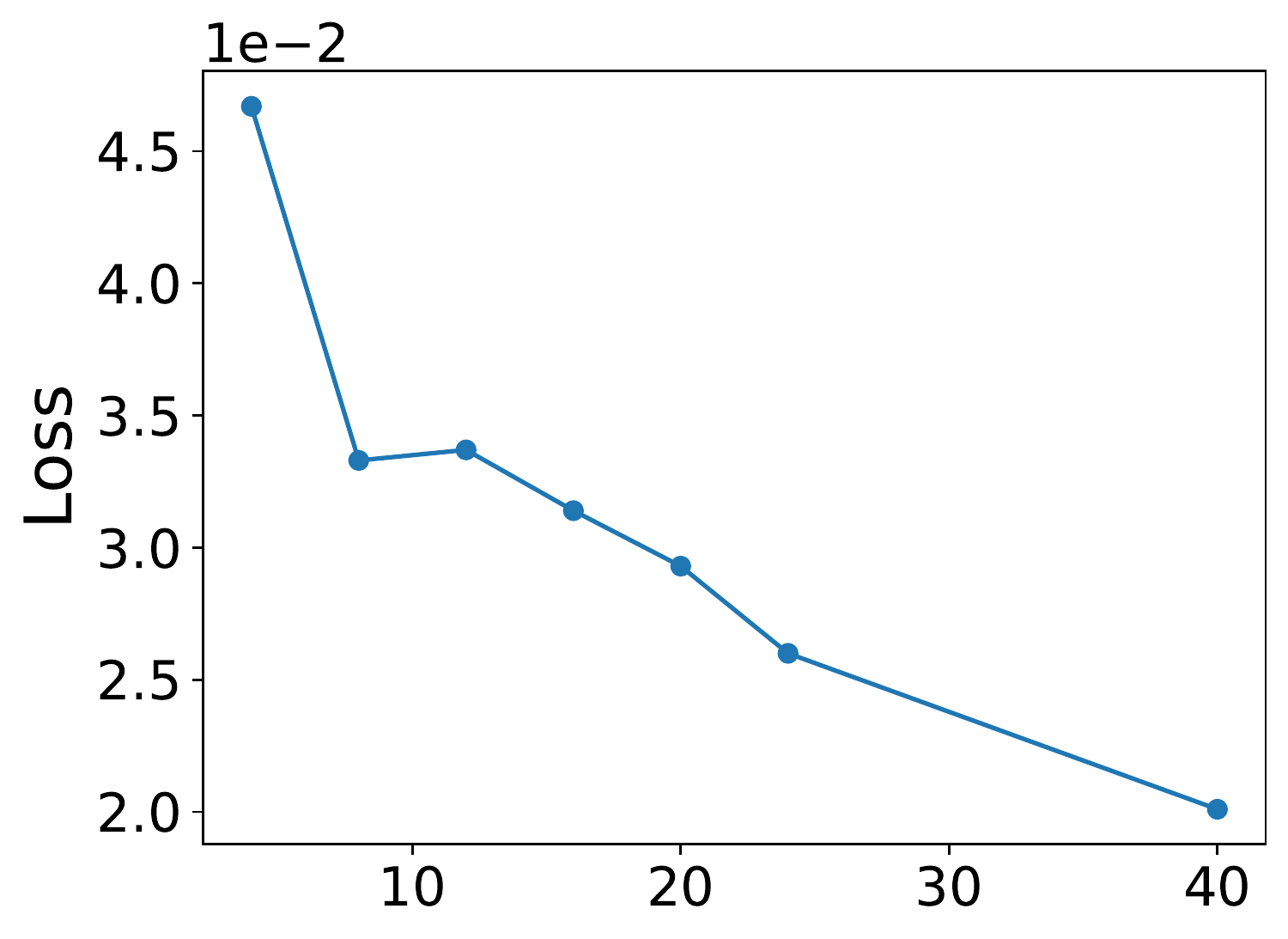}
    }\  
    \subfloat
    [$m=n$ boundary data] 
    {
    \label{exp1_re_numtrain_mn}     
    \includegraphics[width=0.23\textwidth]{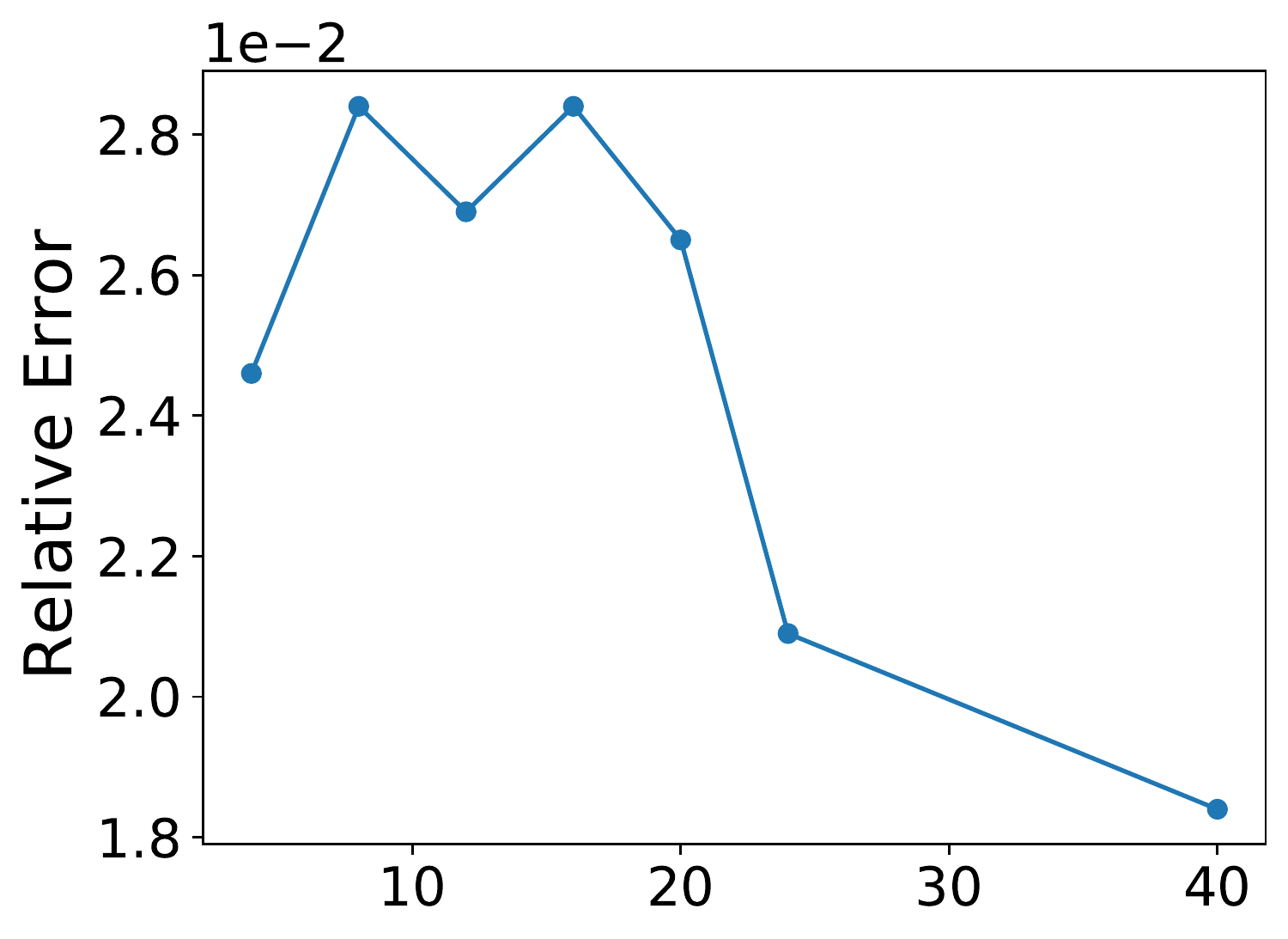}
    }\
    \subfloat
    [$k$ interior data] 
    {
    \label{exp1_loss_numtrain_k}     
    \includegraphics[width=0.230\textwidth]{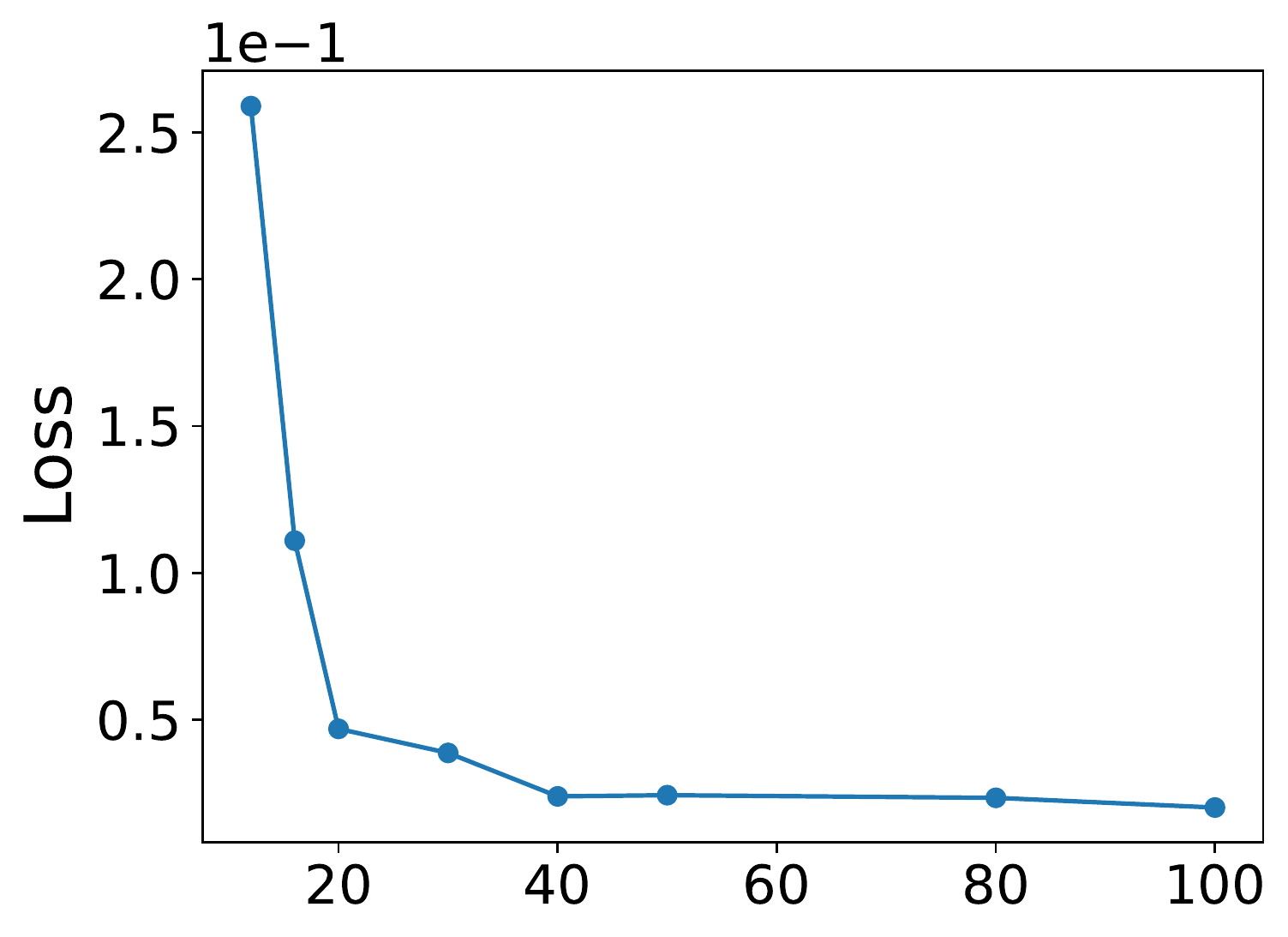}
    }\
    \subfloat
    [$k$ interior data] 
    {
    \label{exp1_re_numtrain_k}     
    \includegraphics[width=0.225\textwidth]{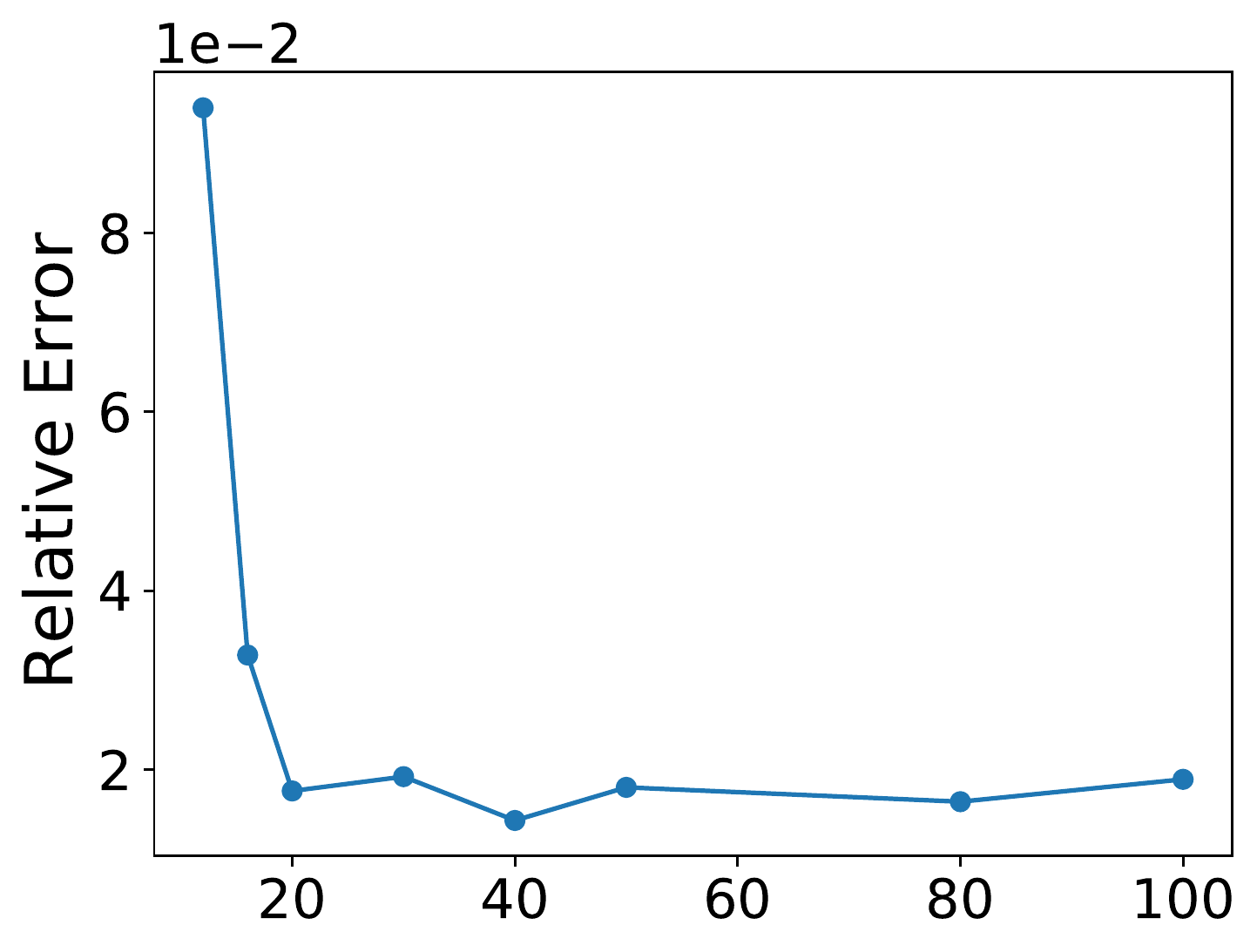}
    }\
    }\\
    \makebox[\linewidth][c]{%
  \centering
  \subfloat
  [$\log m$ ($\log n$) boundary data] 
  {
     \label{exp1_logloss_numtrain_mn}     
    \includegraphics[width=0.23\textwidth]{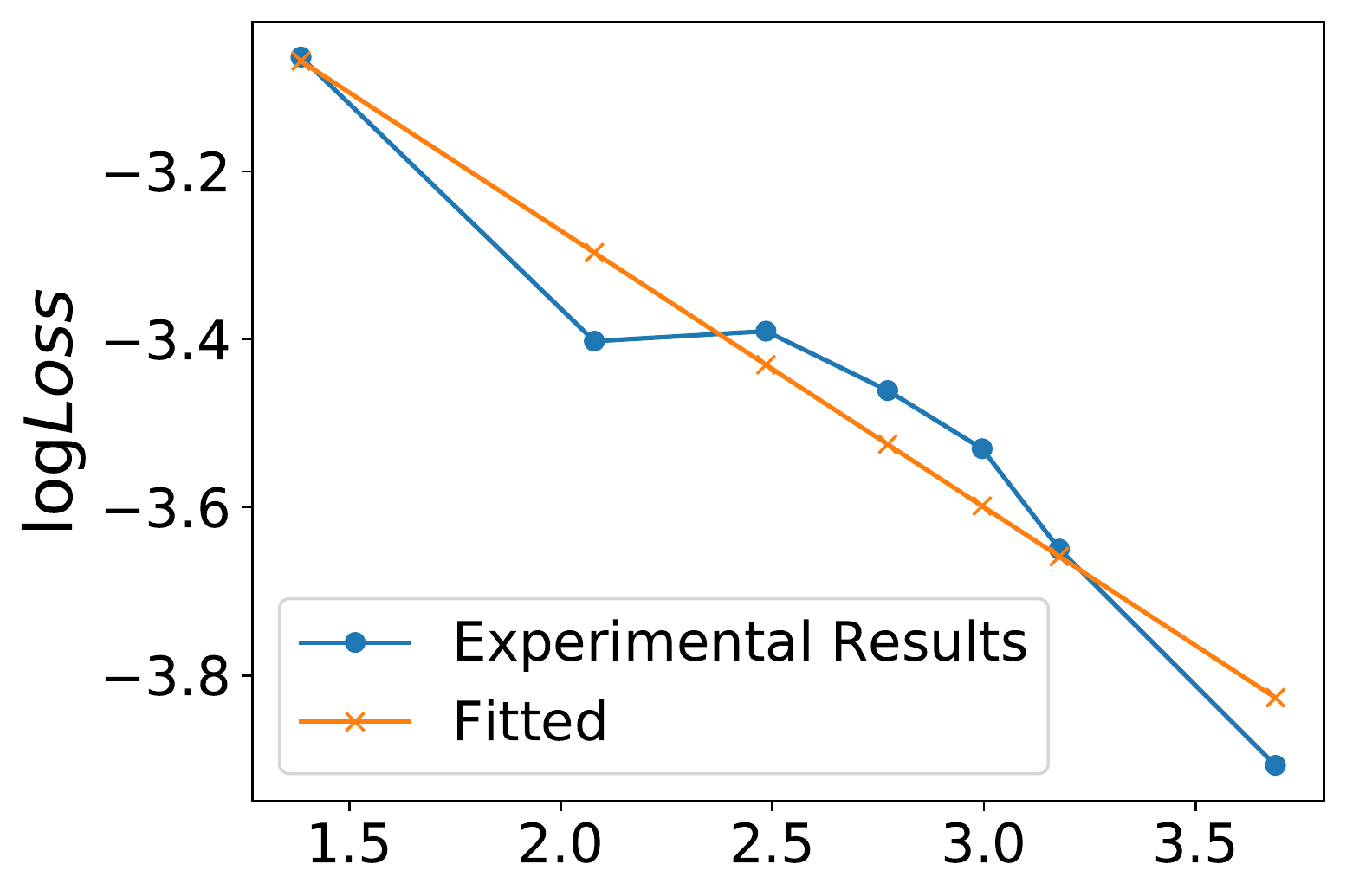}
    }\  
    \subfloat
    [$\log k$ boundary data] 
    {
    \label{exp1_logloss_numtrain_k}     
    \includegraphics[width=0.23\textwidth]{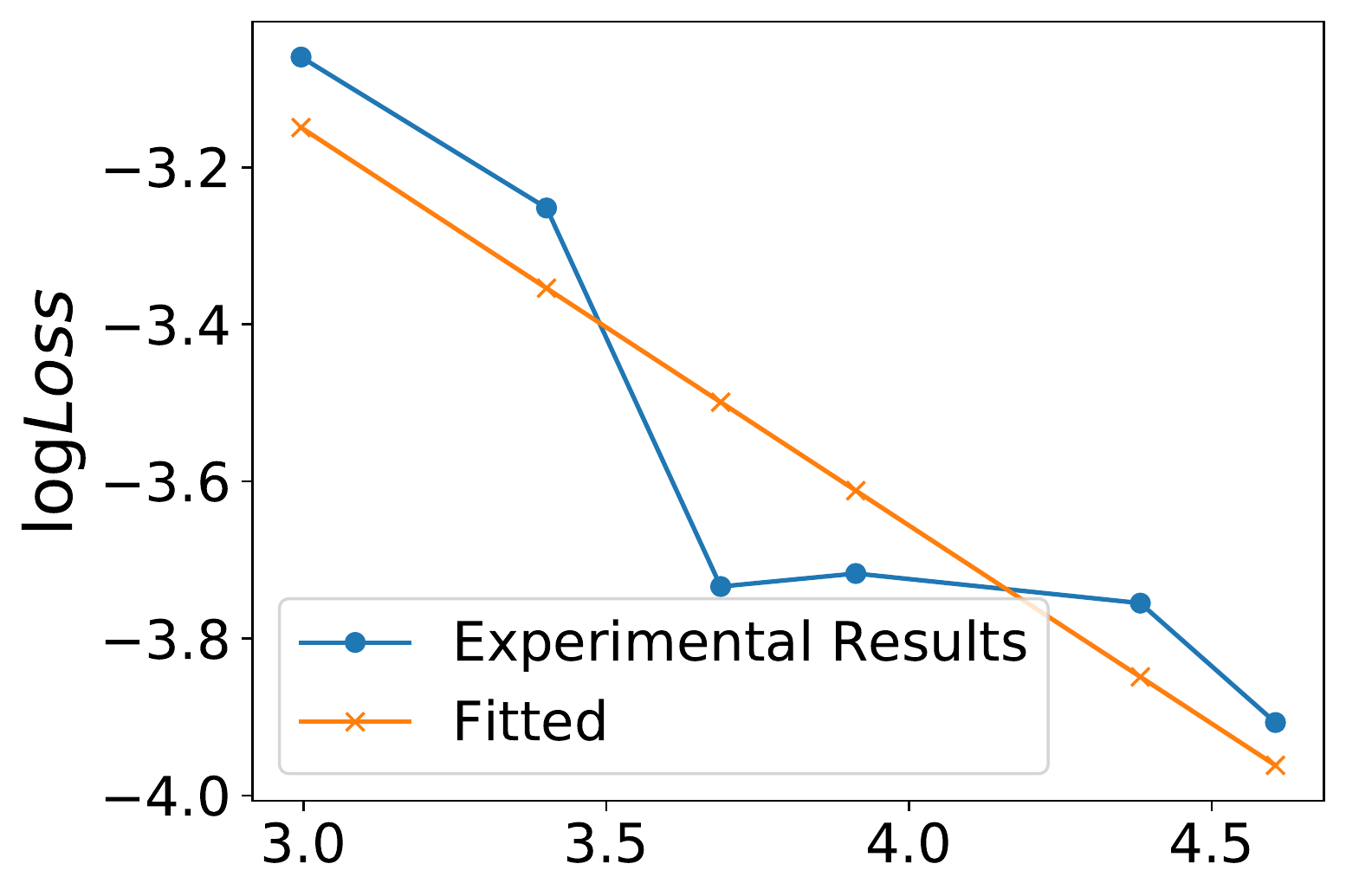}
    }\
    }
    \caption{		
		(a) The
		loss and (b) the relative error of the obtained model $\Tilde{g}$ by WGAN-PINNs with
		different numbers of training data on the boundary using $k=100$ interior data;
		(c) the loss and (d) the relative error of the obtained model $\Tilde{g}$ with
		different numbers of training data in the interior domain using $m=n=40$ 
		boundary data; (e) and (f) the loss of the obtained model $\Tilde{g}$ by WGAN-PINNs with
		different numbers of training data on the boundary and in the interior domain in log scale where $\log \text{Loss} \sim -0.33 \log m$ and $\log \text{Loss} \sim -0.50 \log k$. Generators and discriminators are of $(D_g, W_g)=(3,50)$ and $(D_f, W_f)=(3,50)$ respectively.}	
    \label{exp1_numtrain}
\end{figure*}

\begin{figure*}[h!]
\makebox[\linewidth][c]{%
  \centering
  \subfloat
  [$m=n$ boundary data] 
  {
     \label{exp1_gan_re_numtrain_mn}     
    \includegraphics[width=0.23\textwidth]{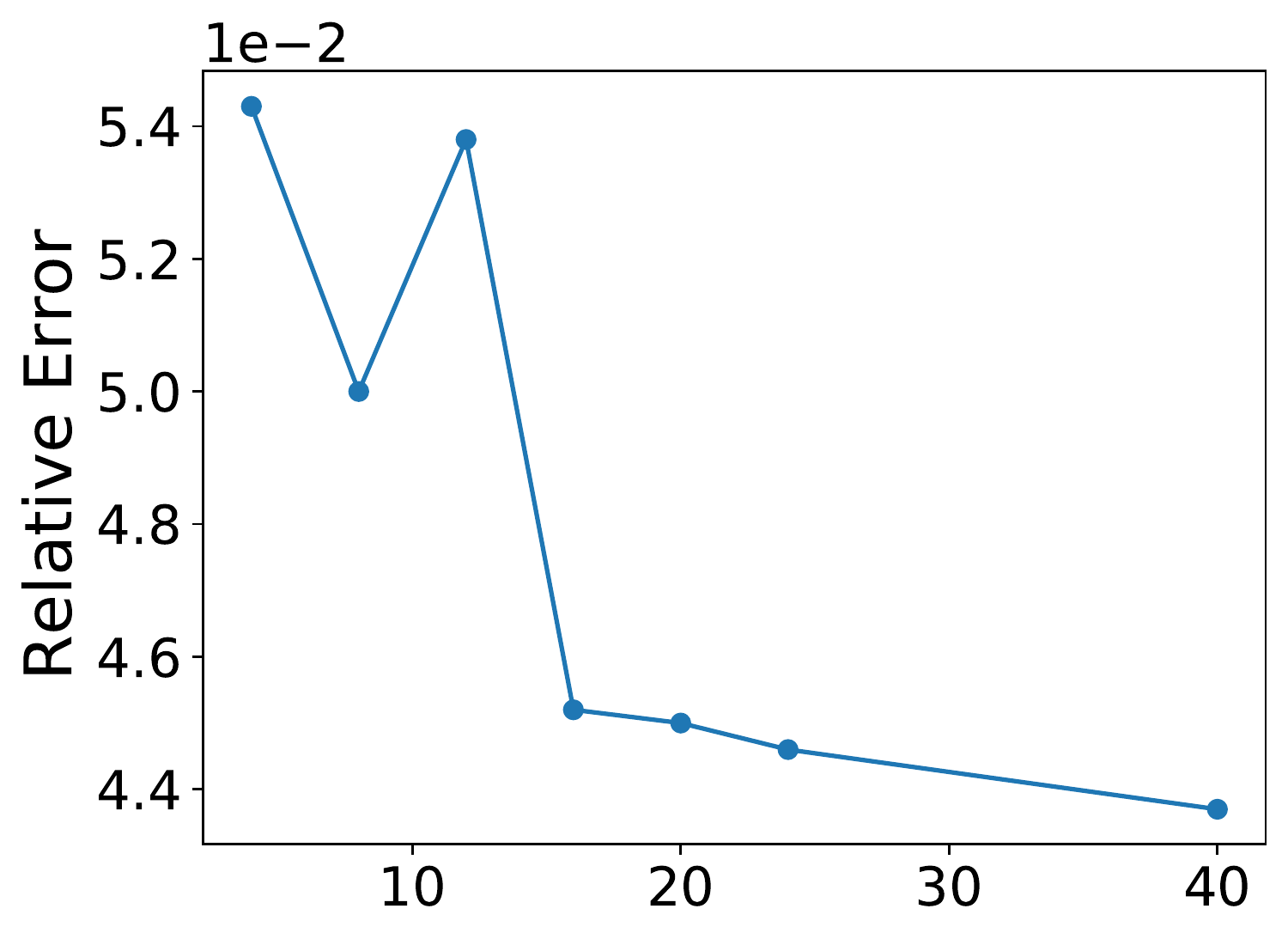}
    }\  
    \subfloat
    [$k$ interior data] 
    {
    \label{exp1_gan_re_numtrain_k}     
    \includegraphics[width=0.23\textwidth]{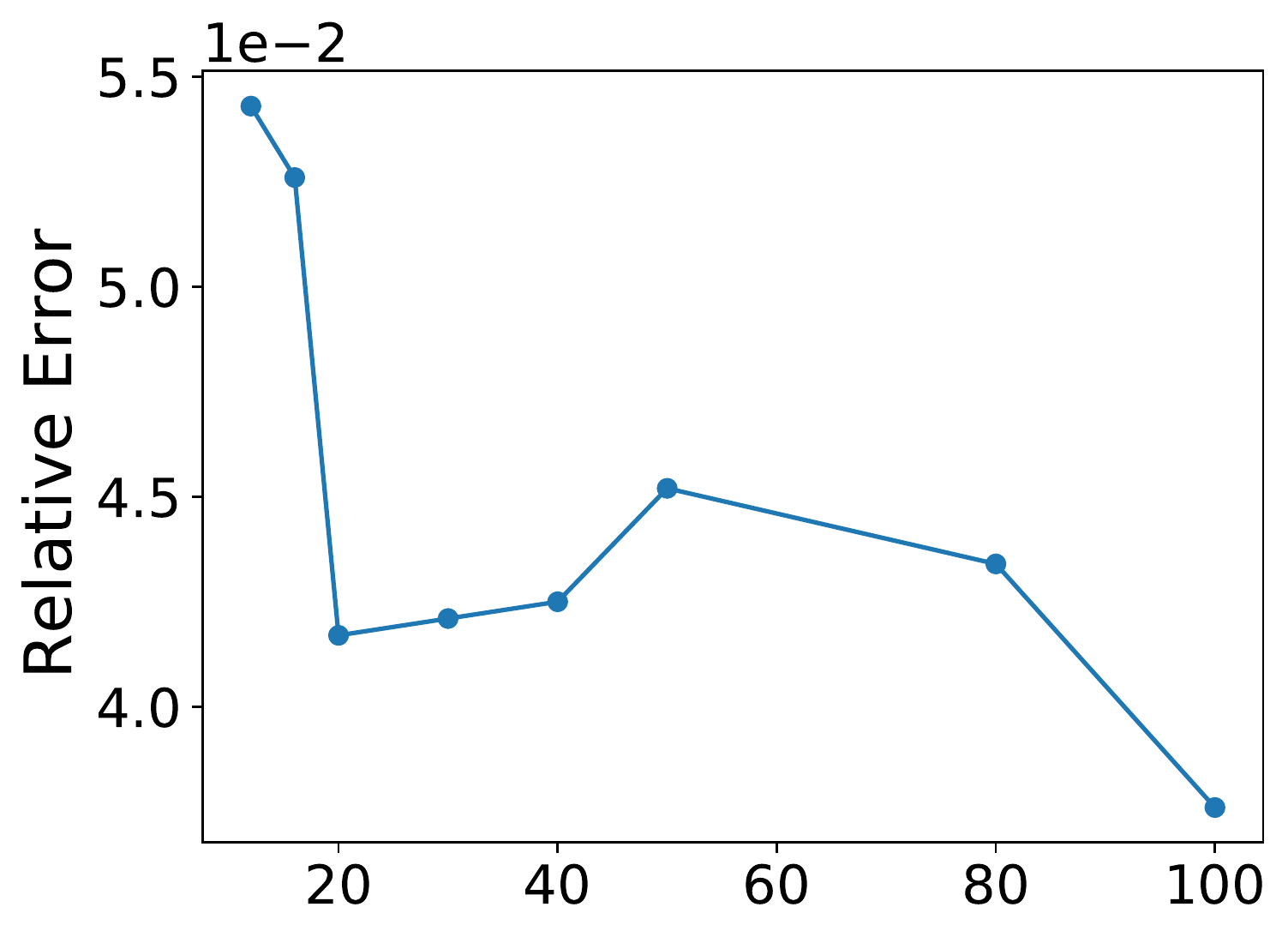}
    }\
    }
    \caption{		
		The relative error of the obtained model $\Tilde{g}$ by GAN-PINNs 
		with respect to (a) 
		different numbers of training data on the boundary using $k=100$ interior data;
		(b) different numbers of training data in the interior domain using $m=n=40$ boundary data. Generators and discriminators are of $(D_g, W_g)=(3,50)$ and $(D_f, W_f)=(3,50)$ respectively.}	
    \label{exp1_gan_numtrain}
\end{figure*}

\begin{figure*}[h]
\makebox[\linewidth][c]{%
  \centering
  \subfloat
  [discriminator width ($W_f$)] 
  {
     \label{exp1_loss_dis_width}     
    \includegraphics[width=0.23\textwidth]{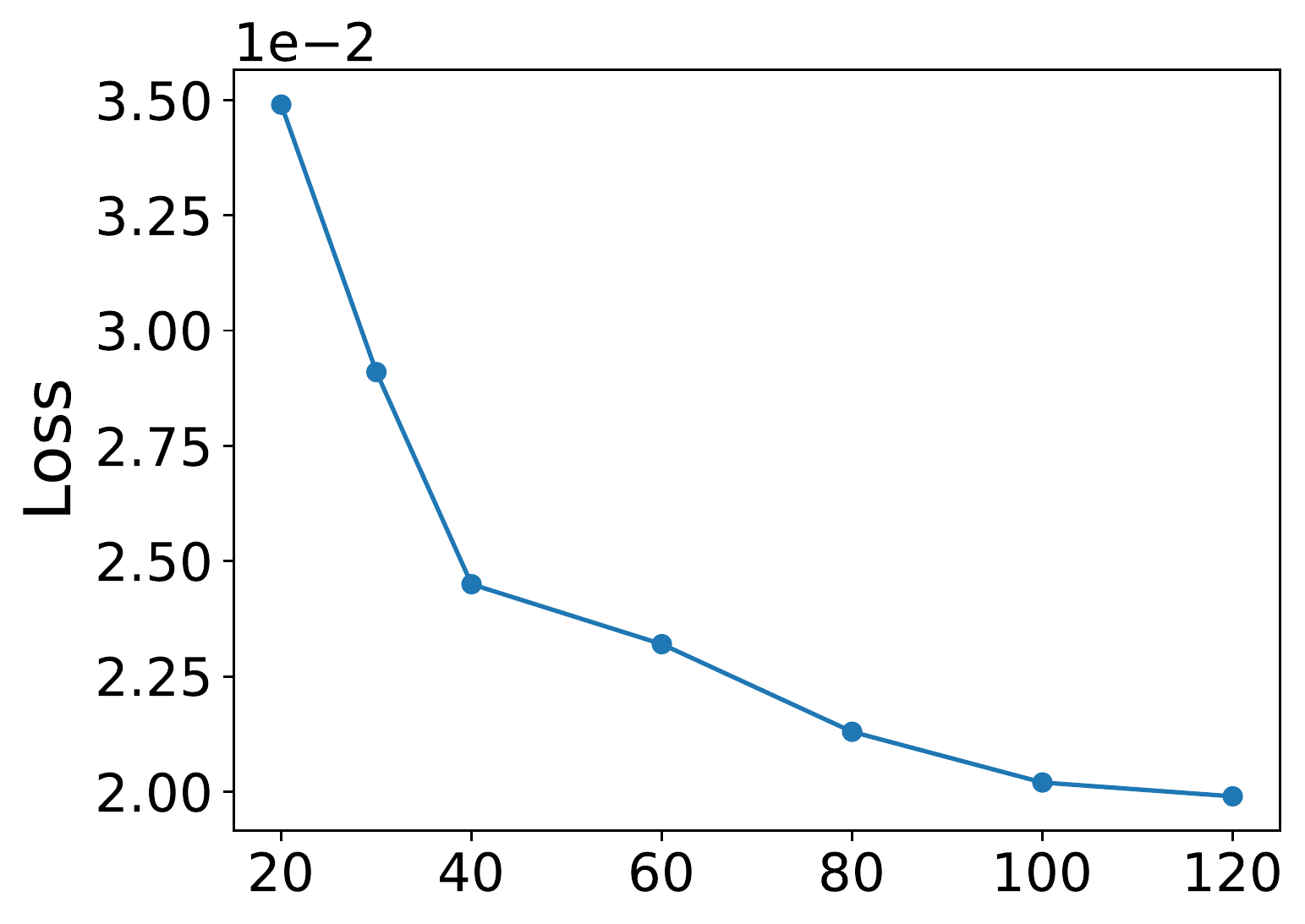}
    }\  
    \subfloat
    [discriminator width ($W_f$)] 
    {
    \label{exp1_re_dis_width}     
    \includegraphics[width=0.23\textwidth]{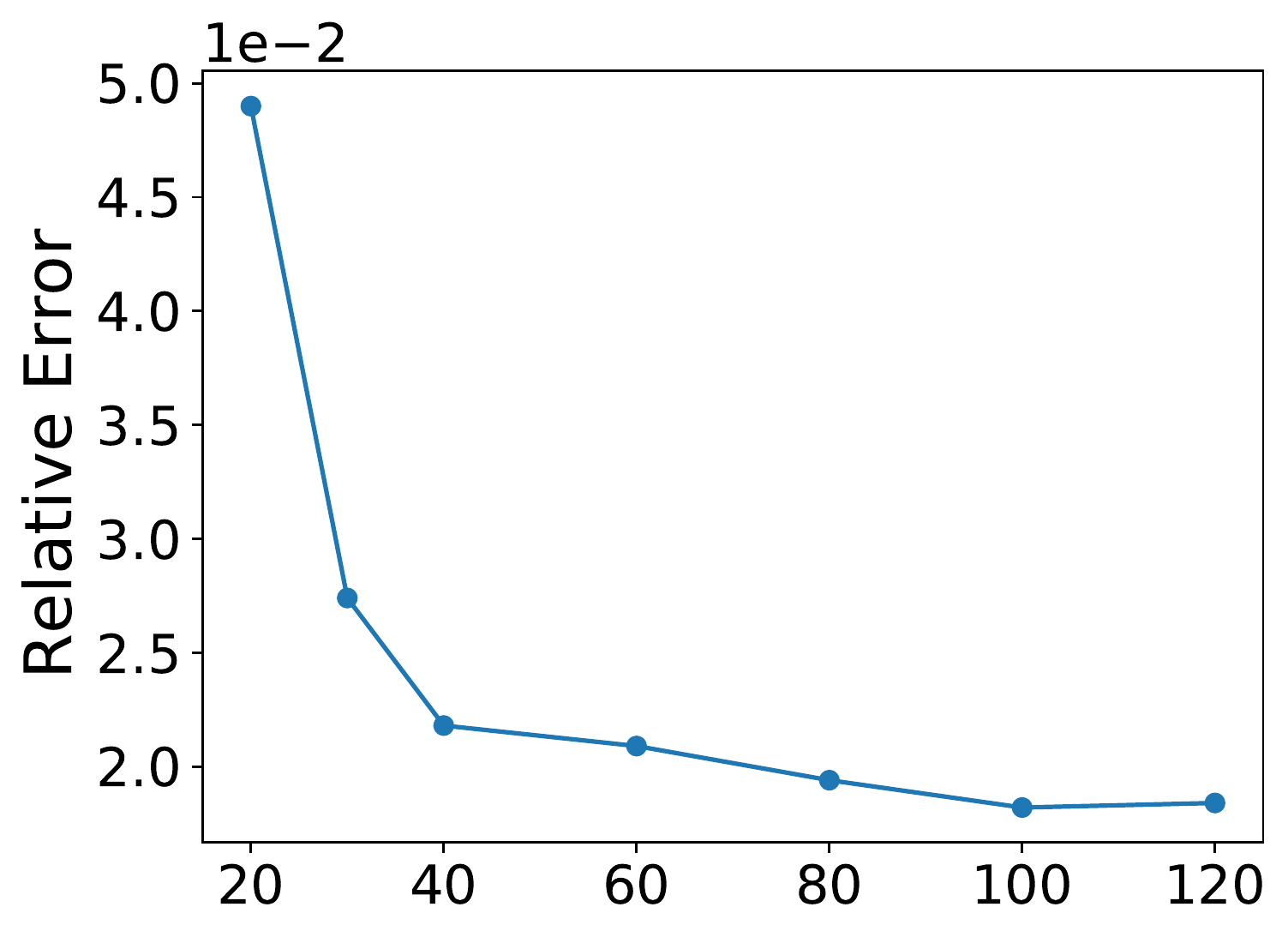}
    }\
    \subfloat
    [discriminator depth ($D_f$)] 
    {
    \label{exp1_loss_dis_depth}     
    \includegraphics[width=0.23\textwidth]{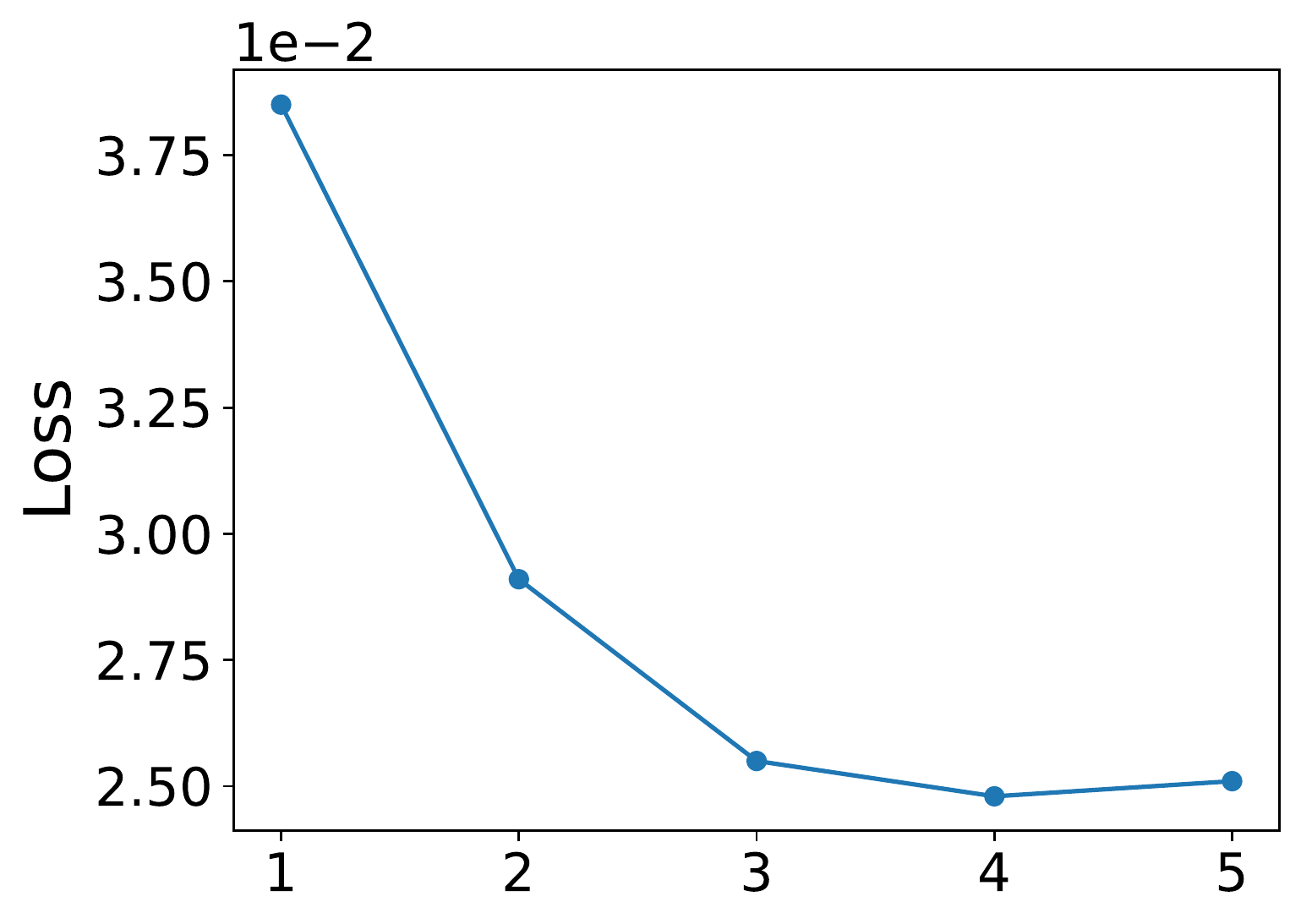}
    }\
    \subfloat
    [discriminator depth ($D_f$)] 
    {
    \label{exp1_re_dis_depth}     
    \includegraphics[width=0.23\textwidth]{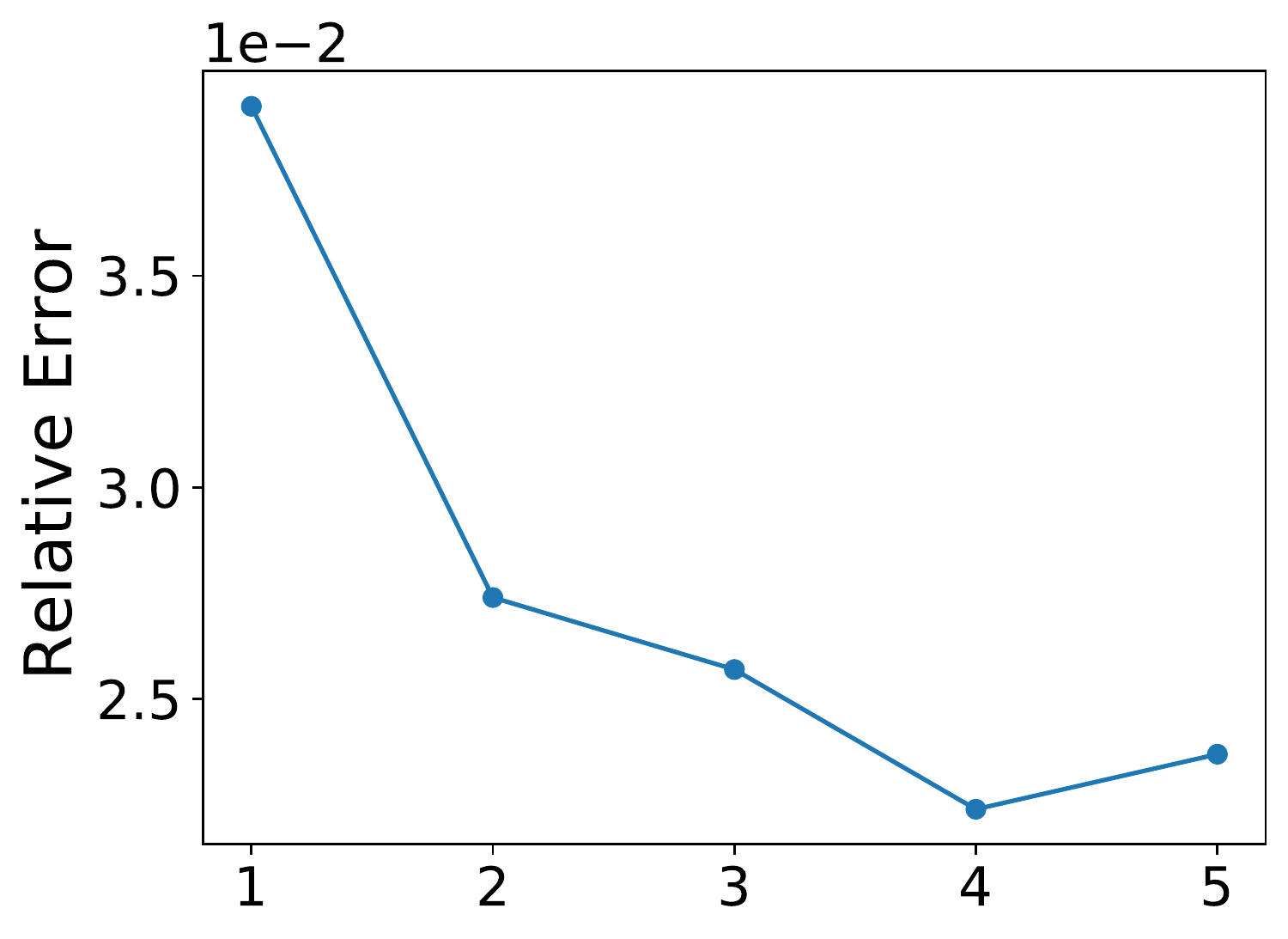}
    }\
    }\\
    \makebox[\linewidth][c]{%
  \centering
  \subfloat
  [discriminator width ($\log W_f$)] 
  {
     \label{exp1_logloss_dis_width}     
    \includegraphics[width=0.23\textwidth]{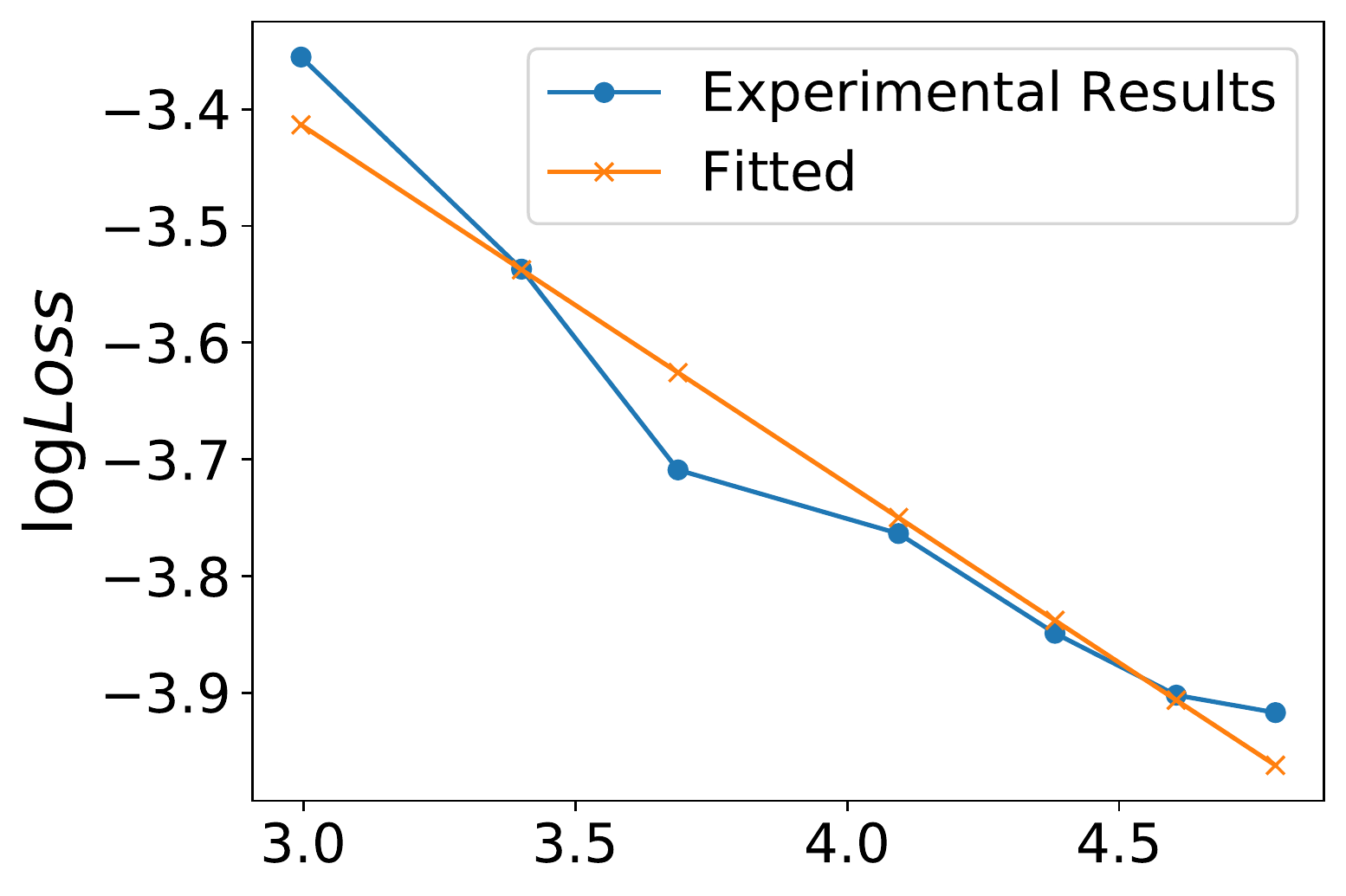}
    }\  
    \subfloat
    [discriminator depth ($D_f$)] 
    {
    \label{exp1_logloss_dis_depth}     
    \includegraphics[width=0.23\textwidth]{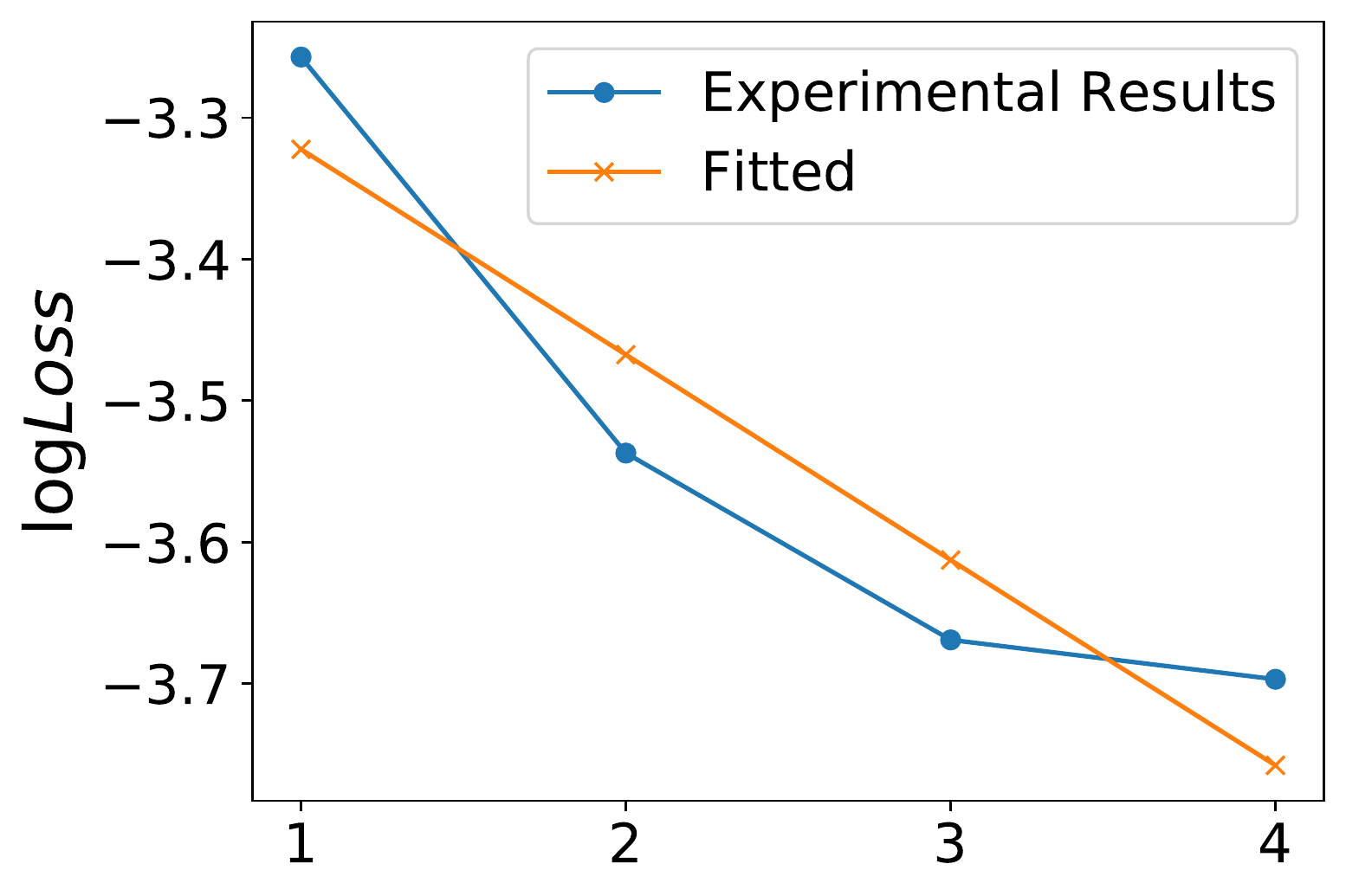}
    }\
    }
    \caption{		
		The loss and the relative error of the obtained model $\Tilde{g}$ by 
		WGAN-PINNs with respect to (a), (b): $W_f$ using  discriminators of $D_f=2$;
		(c), (d): $D_f$ using discriminators of $W_f = 30$;
		(e), (f): log scale of (a) and (b) with $\log \text{Loss} \sim -0.31 \log W_f$ and $\log \text{Loss} \sim -0.15 D_f$. For both two cases, $m=n=40$, $k=100$, and generators are of $(D_g, W_g)=(2,50)$.}	
    \label{exp1_dis}
\end{figure*}

\begin{figure*}[h!]
\makebox[\linewidth][c]{%
  \centering
  \subfloat
  [generator width ($W_g$)] 
  {
     \label{exp1_loss_gen_width}     
    \includegraphics[width=0.23\textwidth]{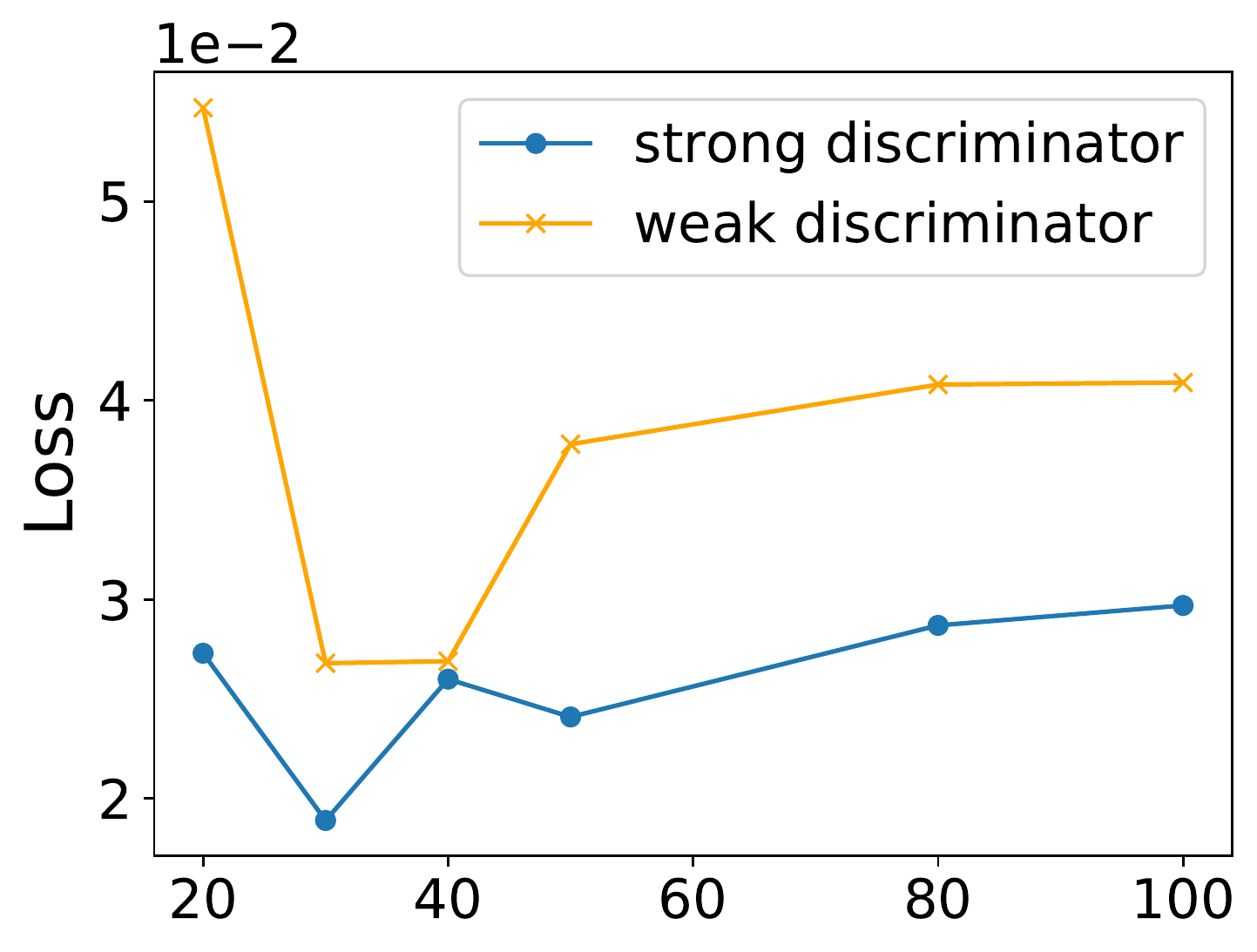}
    }\  
    \subfloat
    [generator width ($W_g$)] 
    {
    \label{exp1_re_gen_width}     
    \includegraphics[width=0.23\textwidth]{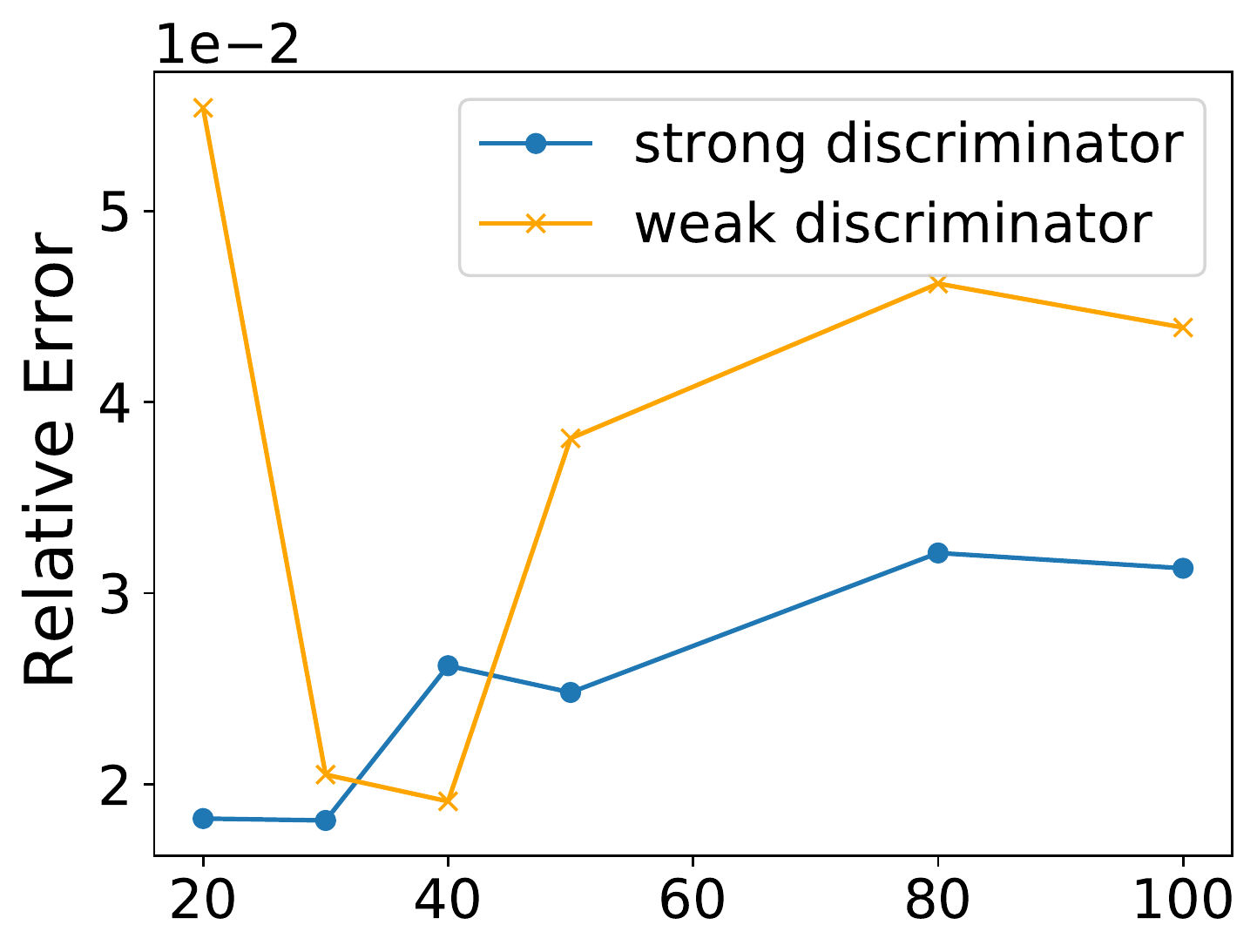}
    }\
    \subfloat
    [generator depth ($D_g$)] 
    {
    \label{exp1_loss_gen_depth}     
    \includegraphics[width=0.230\textwidth]{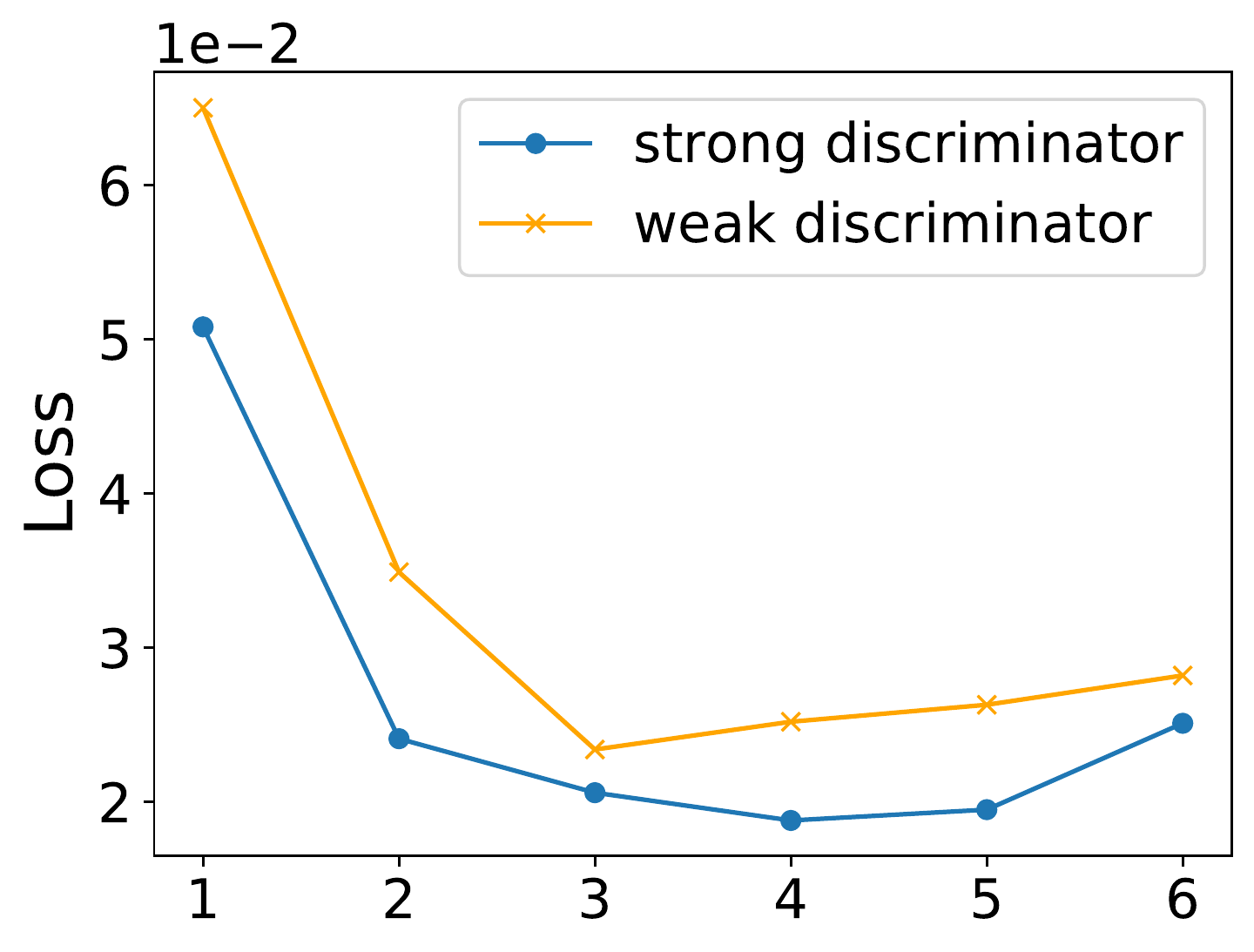}
    }\
    \subfloat
    [generator depth ($D_g$)] 
    {
    \label{exp1_re_gen_depth}     
    \includegraphics[width=0.23\textwidth]{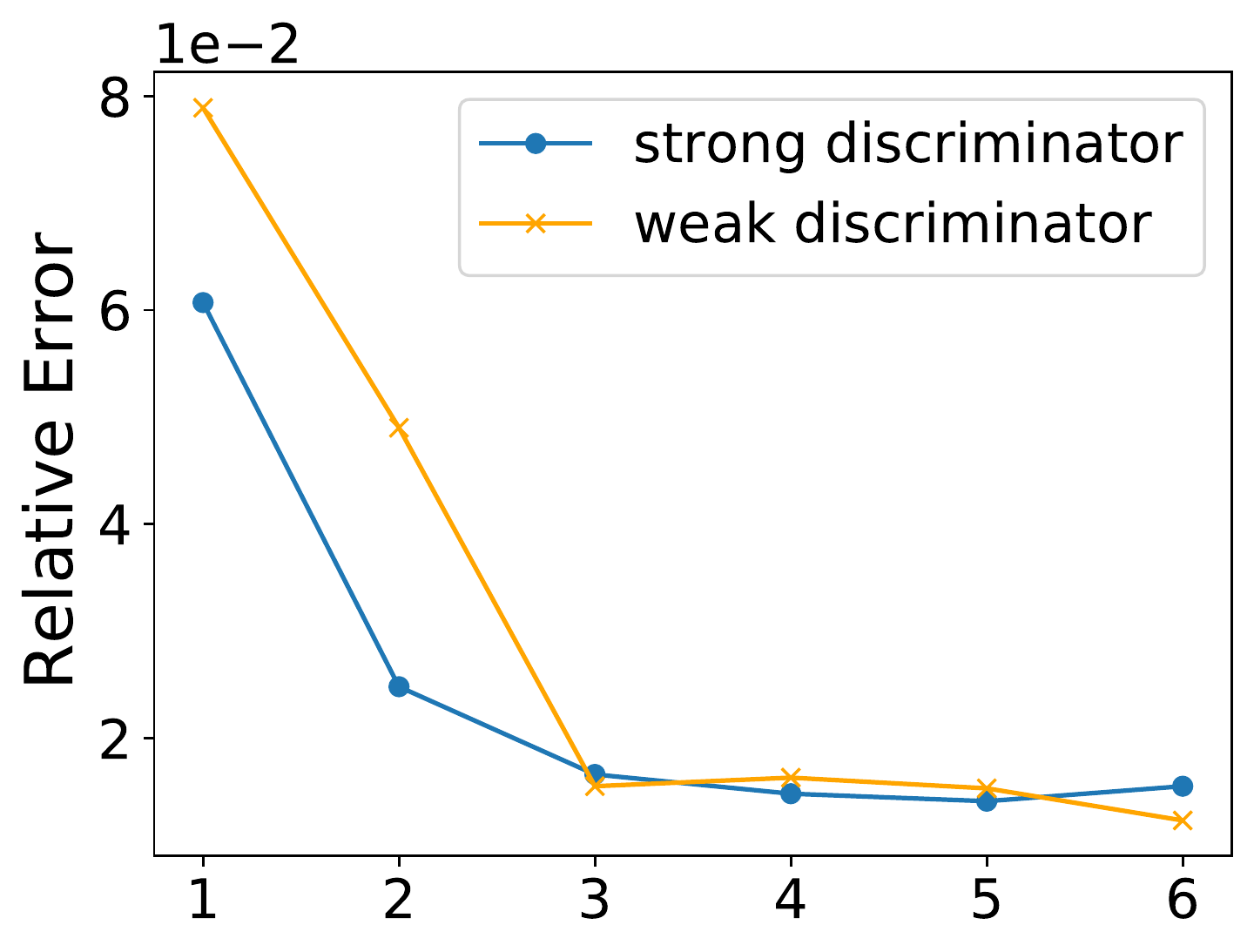}
    }\
    }
    \caption{		
		The loss and the relative error of the obtained model $\Tilde{g}$ by
		WGAN-PINNs with respect to (a), (b): $W_g$ using  generators of $D_g=2$;
		(c), (d): $D_g$ using generators of $W_g = 50$. For both two cases, $m=n=40$, $k=100$, strong discriminators $(D_f, W_f)=(2, 50)$ as well as weak discriminators $(D_f, W_f)=(2, 20)$.}	
    \label{exp1_gen}
\end{figure*}

\begin{figure*}[h!]
\makebox[\linewidth][c]{%
  \centering
  \subfloat
  [discriminator width ($W_f$)] 
  {
     \label{exp1_gan_re_dis_width}     
    \includegraphics[width=0.23\textwidth]{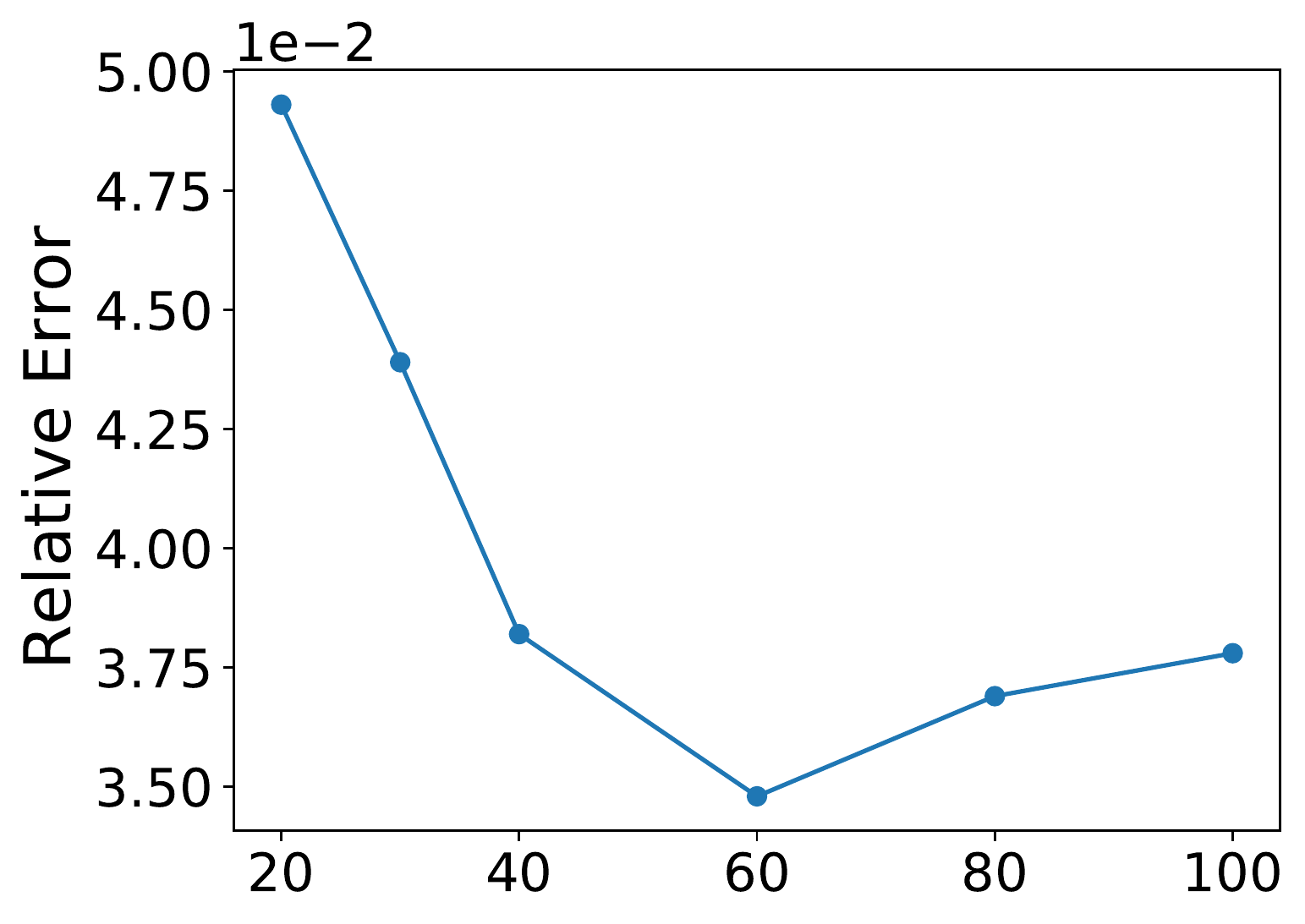}
    }\  
    \subfloat
    [discriminator depth ($D_f$)] 
    {
    \label{exp1_gan_re_dis_depth}     
    \includegraphics[width=0.23\textwidth]{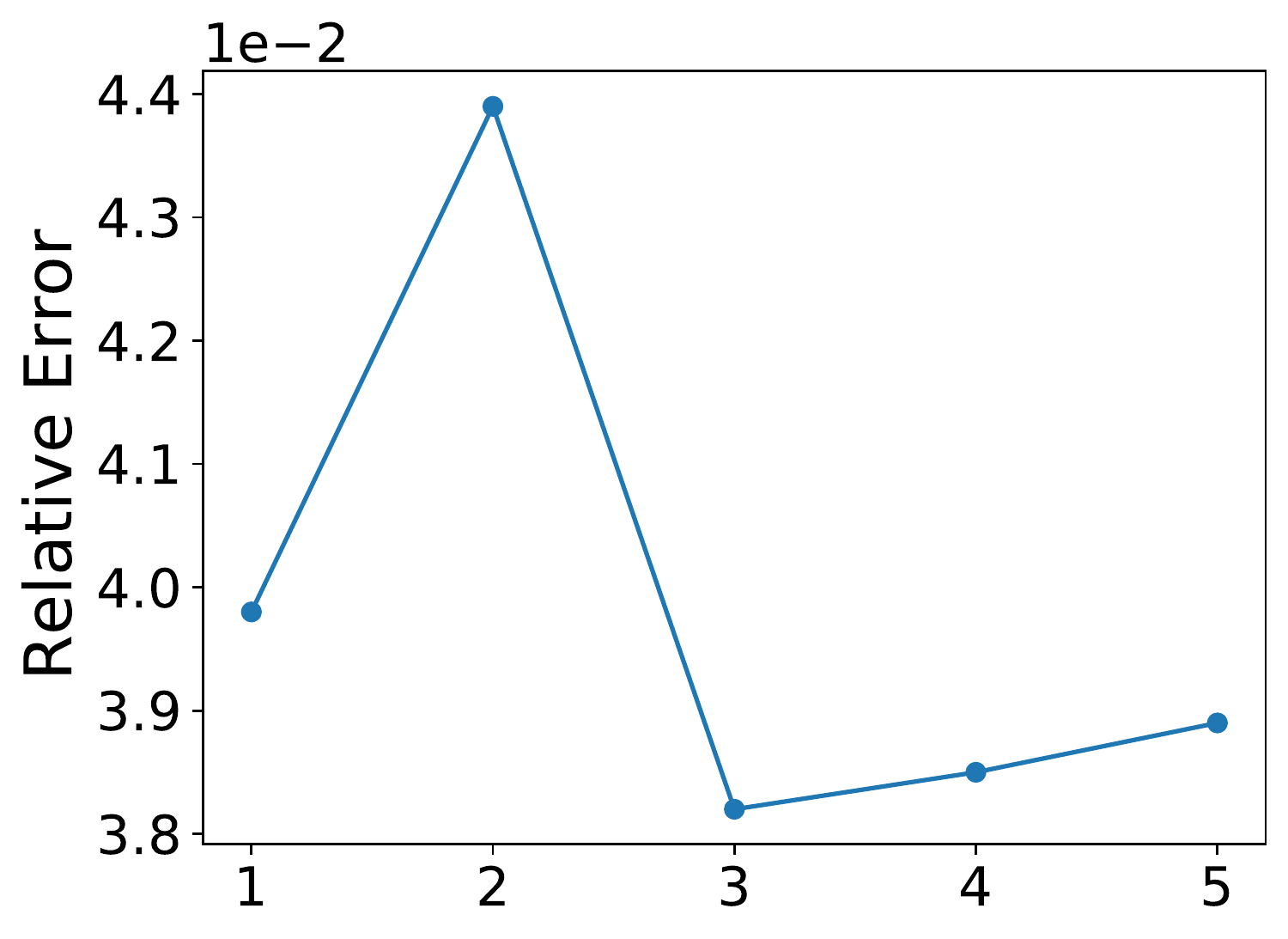}
    }\
    \subfloat
    [generator width ($W_g$)] 
    {
    \label{exp1_gan_re_gen_width}     
    \includegraphics[width=0.230\textwidth]{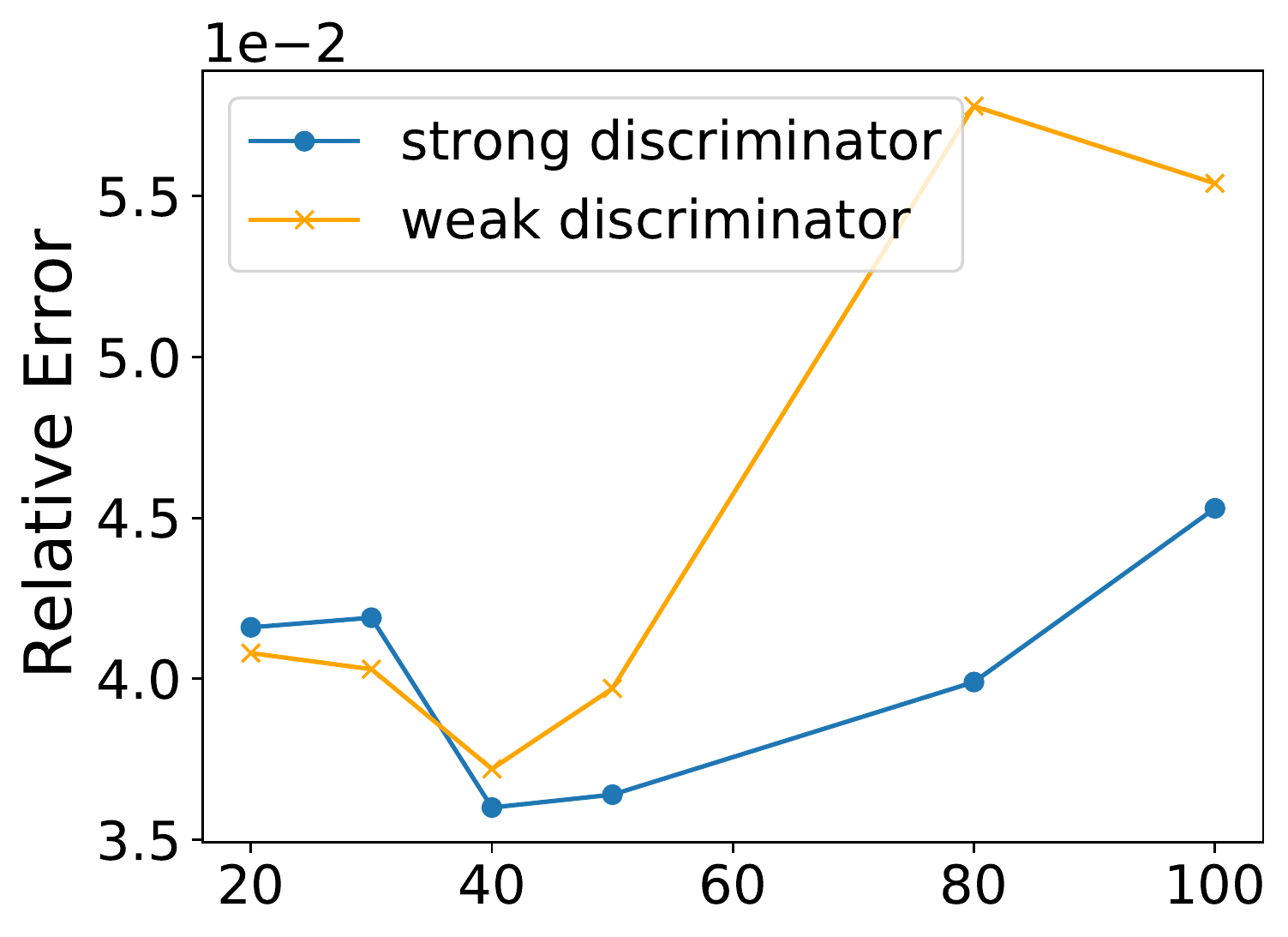}
    }\
    \subfloat
    [generator depth ($D_g$)] 
    {
    \label{exp1_gan_re_gen_depth}     
    \includegraphics[width=0.23\textwidth]{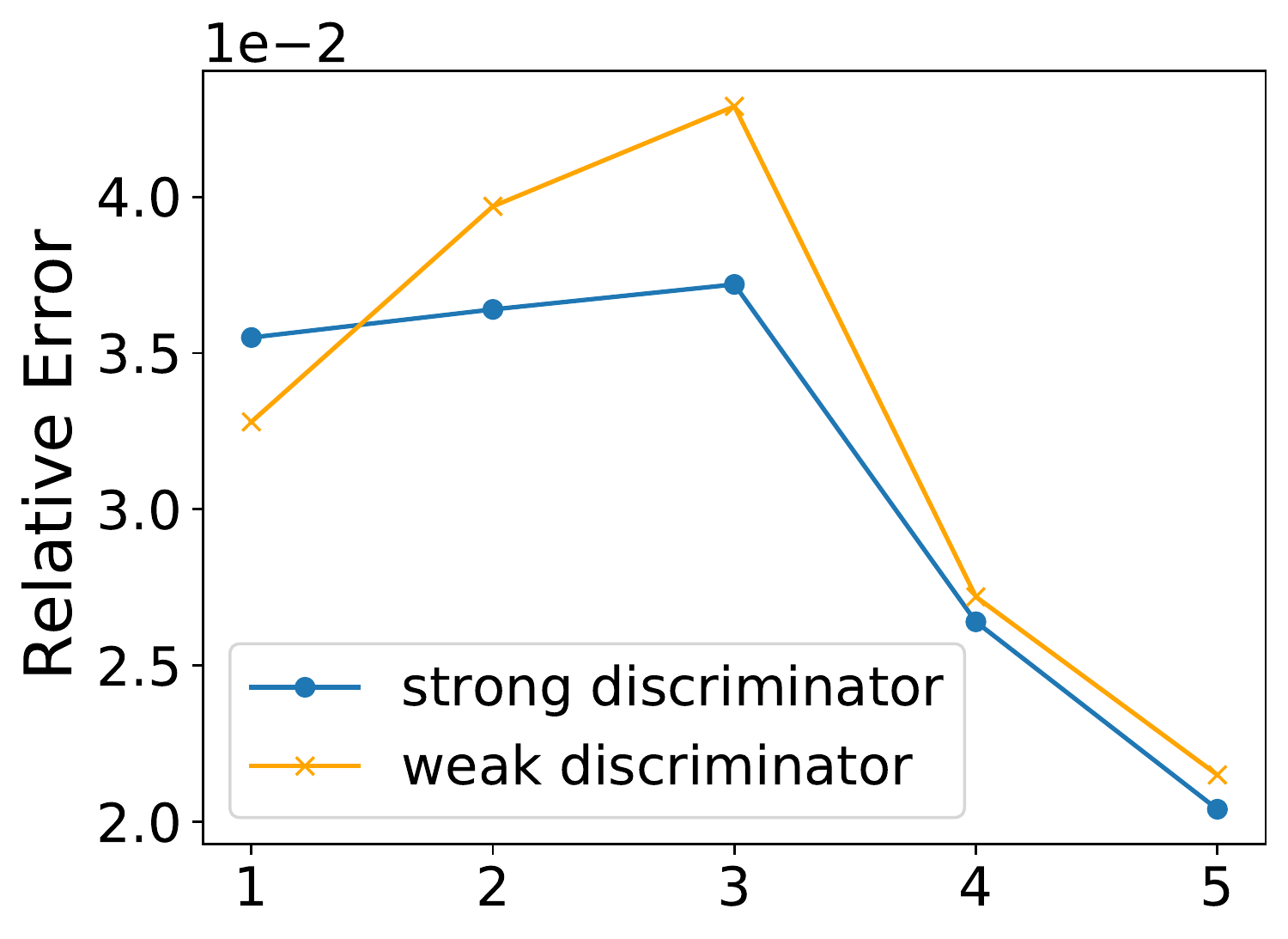}
    }\
    }
    \caption{		
		The relative error of the obtained model $\Tilde{g}$ by GAN-PINNs 
		with respect to (a): $W_f$ using  discriminators of $D_f=2$; (b): $D_f$ using  discriminators of $W_f=30$; For both two cases, $m=n=40$, $k=100$ and generators are of $(D_g,W_g)=(2,50)$;
		(c): $W_g$ using generators of $D_g = 2$; (d): $D_g$ using generators of $W_g = 50$; For both two cases, $m=n=40$, $k=100$ and strong discriminators $(D_f, W_f)=(2, 50)$ as well as weak discriminators $(D_f, W_f)=(2, 20)$.}	
    \label{exp1_gan_modelarch}
\end{figure*}

On the other hand, 
numerical results are compatible with our theoretical analysis, at least in trend. The noise levels are set to be $\sigma_1 = \sigma_2 = 0.05$. 
Figure \ref{exp1_numtrain} 
shows that more training data contributes to better convergence. Trained on $m=n=40$ boundary data and $k=100$ interior data, the generator are able to be of low loss. It suggests that our model improves with increasing numbers of training data. Moreover, We let $\mathbf{D}$ in (\ref{model_original}) to be $\text{JS}$ divergence and adopt deep ReLU feed-forward neural networks (with a $softmax$ activation function at the last layer) as discriminators, denoted as GAN-PINNs.
By comparing with WGAN-PINNs (Figure \ref{exp1_numtrain}) and 
GAN-PINNs (Figure \ref{exp1_gan_numtrain}), we see 
that the relative error by WGAN-PINNs is lower than that by GAN-PINNs.

In Figure \ref{exp1_dis},
we investigate the effects of width and depth of the discriminator with fixed $m=n=40$ data on the boundary and $k=100$ data in the interior domain. We observe that wider discriminators which have larger capacities lead to loss decreasing, consistent with the error bound for $I_1$. A mild rise of the loss in Figure \ref{exp1_loss_dis_depth} when $D_f > 4$ which may result from the optimization, does not contradict to our theory. To further validate the effectiveness of the derived error bound in theorem \ref{main_theorem}, we plot figures with relationships between Loss and $m$ ($n$), $k$, $W_f$ as well as $D_f$ in $\log$ scale, as are shown in Figure \ref{exp1_logloss_numtrain_mn}, \ref{exp1_logloss_numtrain_k}, \ref{exp1_logloss_dis_width} and \ref{exp1_logloss_dis_depth}. We found that $\log \text{Loss} \sim -0.33 (\log m \vee \log n)$, $\log \text{Loss} \sim -0.50 \log k$, $\log \text{Loss} \sim -0.31 \log W_f$ and $\log \text{Loss} \sim -0.15 D_f$, consistent with our theorem that $\text{Loss}$ decays polynomially with $m$, $n$, $k$ and $W_f$ but exponentially with $D_f$. The error bound is tight w.r.t. $m$, $n$ and $k$ but is relatively loose w.r.t. $W_f$ and $D_f$. This is within our expectation because the error bound (\ref{error_bound_1}) and (\ref{error_bound_2}) are applicable to all target distributions with finite 3-moment and the ODE (\ref{exp1_ode}) with Gaussian random boundary conditions is simple which does not require complicated discriminators.

After all, our results are somewhat ideal that $\Tilde{g}$ is the optimal solution of (\ref{emp_loss_func}), however, it is hard to achieve numerically. Deep neural networks are confronted with many challenges during training especially with adversarial networks. Another potential reason may lie in the proximal constraints of parameters of discriminators. Figure \ref{exp1_gen} displays the effects of width and depth of the generator in case of a strong and a weak discriminator. When the width of the generator increases, the loss first decreases, and then goes up again. Two reasons account for the phenomenon. Firstly, too wide generators may violate Assumption \ref{decay}, \ref{large_generator} and \ref{assump_pinns} that the generator cannot be controlled by the discriminator and PINNs regularization term. It is interesting to note that the generator trained with a strong discriminator are more stable with respect to the range of loss in Figure \ref{exp1_loss_gen_width} compared with that of with a weak discriminator. Little information is obtained from Figure \ref{exp1_loss_gen_depth}. Deeper $tanh$ neural networks do not have larger capacities theoretically based on a fact that shallow networks are not included in the set of deeper networks but may have better empirical performances sometimes. Although relative errors are not directly reflected by the loss, we observe that their movements are very similar because $\lambda=100$ is selected based on the results in Table \ref{tab1}. We are more confident with WGAN-PINNs based on the above analysis.

To further emphasize the performance of WGANs than traditional GANs, we do the same numerical experiments on our models with replacement of WGANs by GANs (the model is denoted by GAN-PINNs), although a lot of existing works have shown and explained the better performances of WGANs than GANs. Table \ref{tab2} lists relative errors of our predictions by GAN-PINNs with different noise levels and regularization parameter $\lambda$. our WGAN-PINNs performs much better than GAN-PINNs in the sense of relative errors. $\lambda=50$ is adopted in our later experiments for GAN-PINNs. More experimental details about the relationship between the numbers of training data, the model architecture as well as the relative errors are displayed in Figure \ref{exp1_gan_numtrain} and \ref{exp1_gan_modelarch}. Although KL (or JS) divergence does not hold the generalization properties w.r.t. finite number of training data, in practice, more training data leads to lower relative errors. Discriminators still paly important roles in GANs training. Deeper and wider ReLU neural networks, which have better approximation capabilities, help the model achieve lower relative error, as are shown in Figure \ref{exp1_gan_re_dis_width} and \ref{exp1_gan_re_dis_depth}. In \ref{exp1_gan_re_gen_width},  the relative errors of models with weak discriminators begin to rise up much earlier than that of with strong discriminators with increasing width of generators. Both Figure \ref{exp1_re_gen_depth} and \ref{exp1_gan_re_gen_depth} show that the depth of generators helps reduce the approximation error in this example. There are two significant differences between WGAN-PINNs and GAN-PINNs. Firstly, WGAN-PINNs achieves much lower relative errors than GAN-PINNs. Secondly, WGAN-PINNs is more stable than GAN-PINNs in training. We believe the most reason for the worse performance of GAN-PINNs is the non-generalization property of KL divergence.

\subsection{Heat Equation}
\label{1DHeatEq}

To test the performance of WGAN-PINNs, we solve a 
heat equation (including a temporal and a spatial coordinates). The PDE is as follows:
\begin{equation}
    \begin{split}
        & - u_t + \nu \cdot u_{xx} = 0, \hspace{1em} \text{(P.D.E.)}, \hspace{1em} (x,t) \in [-1,1] \times [0,1]\\
        & u(x, t=0) = sin(\pi x), \hspace{1em} \text{(I.C.)}\\
        & u(x=-1,t) = u(x=1,t) = 0, \hspace{1em} \text{(B.C.)}
    \end{split}
\end{equation}
where the thermal diffusivity is $\nu = \frac{1}{\pi^2}$. The unique solution for the above deterministic problem (without uncertainty) is $u(x,t)=\sin(\pi x) \cdot e^{-t}$. Here, we assume that the initial condition contain uncertainty, i.e., 
\begin{equation}
\label{exp2_random}
    \begin{split}
        & u(x,t=0)= \sin (\pi x) + \delta_1, \hspace{1em} \delta_1 = \frac{\epsilon_1}{\exp(3|x|-1)}, \hspace{1em} \epsilon_1 \sim \mathcal{N}(0,\sigma_1^2), \hspace{1em} \text{(I.C.)}\\
        & u(x=-1,t) = u(x=1,t) = 0, \hspace{1em} \text{(B.C.)}
    \end{split}
\end{equation}
Note that standard derivations are spatial-dependent. 
We use deep $tanh$ feed-forward neural networks (generator) $g(\mathbf{x},\mathbf{z})$ as a surrogate to the solution $\mathbf{u}(\mathbf{x},\mathbf{z})$. Our predefined architecture for both generators and discriminators are neural networks with 3 hidden layers and 50 neurons at each layer. The prior for the latent variable is chosen to be a 
two-dimensional Gaussian distribution, i.e., $\mathbf{z} \sim \mathcal{N}(\mathbf{0},\mathbf{I}_2)$. For more detailed hyper-parameters, please refer to Table \ref{tab_hyperprm}. 

In the test, we mainly consider the cases of noise-free and noise level $\sigma_1=0.05$. For both two cases, we use $m=n=1600$ boundary data (800 for initial condition plus 400 for each boundary condition) and $k=500$ interior data to train our probabilistic model. All data points are randomly and uniformly sampled. 
For the noise-free case, the relative error $\mathcal{E}=1 \times 10^{-3}$ while $\mathcal{E}=8 \times 10^{-3}$ for $\sigma_1=0.05$. We display the mean solution 
values as our prediction and the lower and upper bounds of solution values (shown as yellow bars) which are associated with uncertainties captured by our model. Visualized results are summarized in Figure \ref{exp2_visual_pre} and \ref{exp2_visual_his} for comparisons of both mean and standard derivations and 
histograms between the simulated solution and generated data. Our model is capable to capture the uncertainty of the initial/boundary data, and propagate it to the interior domain following physical laws.

\begin{figure*}[h!]
\makebox[\linewidth][c]{%
  \centering
  \subfloat
  [$\sigma_1 = 0$, \hspace{0.3em} $t=0$] 
  {
     \label{exp2_pre1}     
    \includegraphics[width=0.3\textwidth]{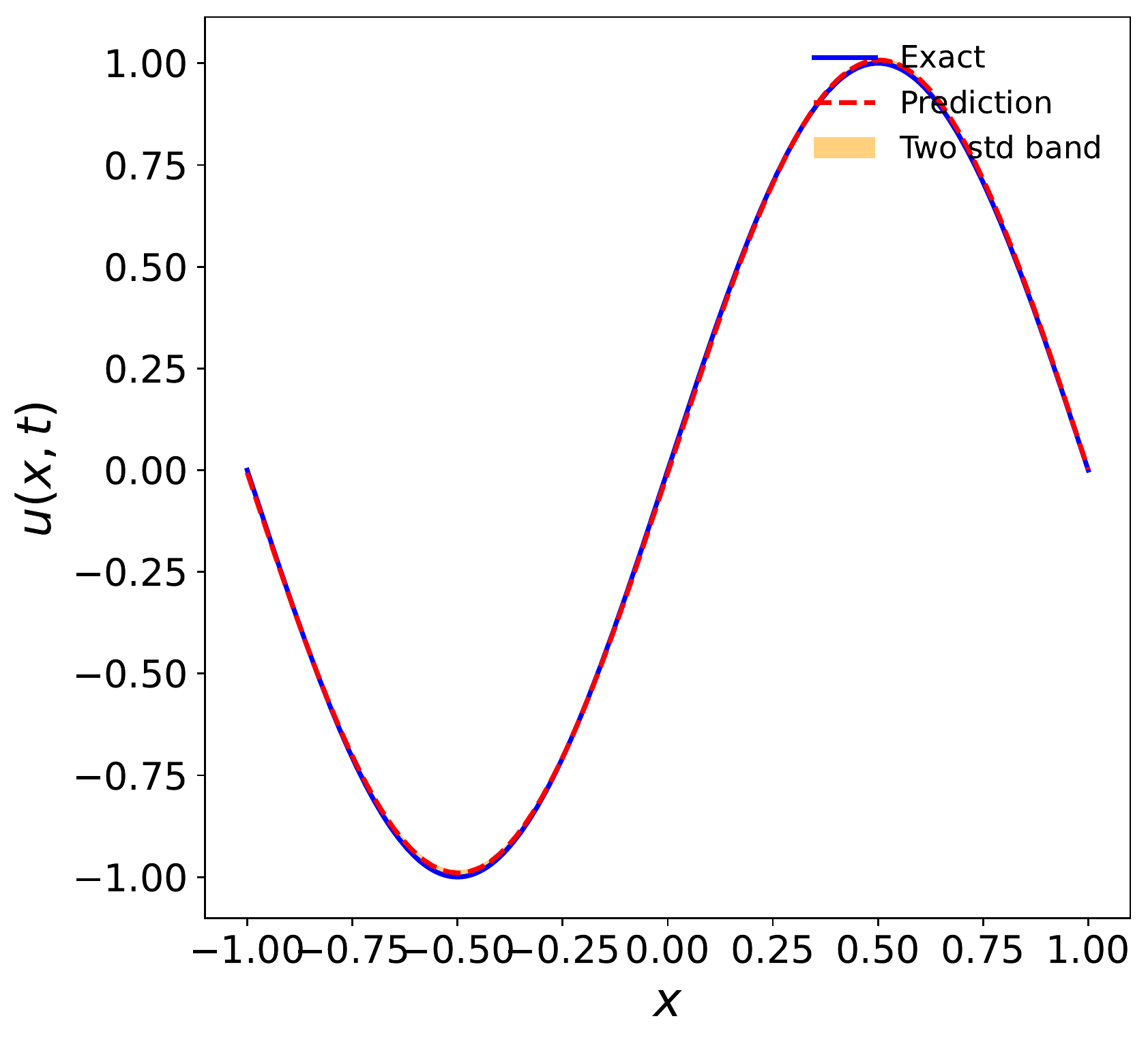}
    }\  
    \subfloat
    [$\sigma_1 = 0$, \hspace{0.3em} $t=0.5$] 
    {
    \label{exp2_pre2}     
    \includegraphics[width=0.3\textwidth]{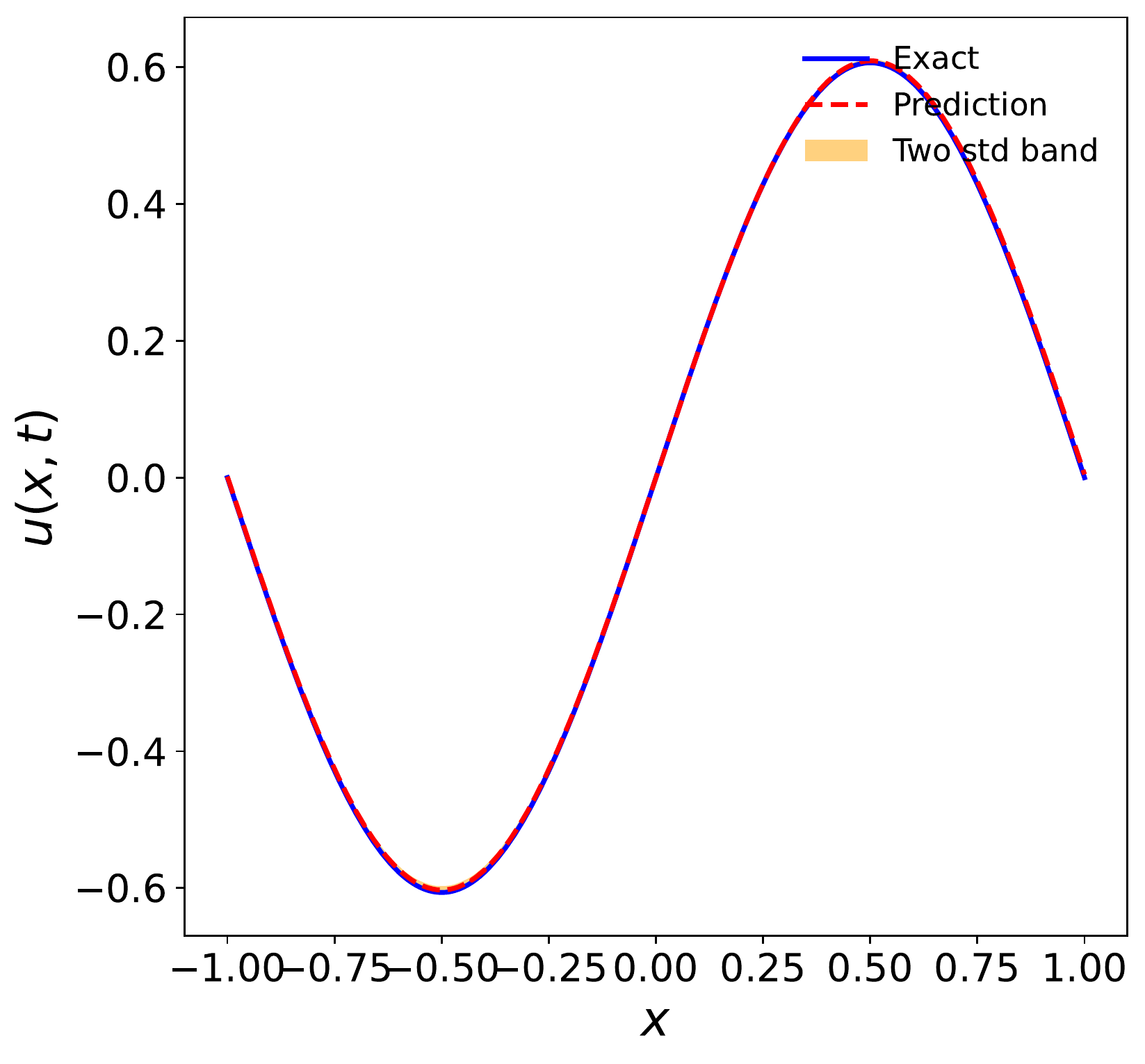}
    }\
        \subfloat
    [$\sigma_1 = 0$, \hspace{0.3em} $t=1$] 
    {
    \label{exp2_pre3}     
    \includegraphics[width=0.3\textwidth]{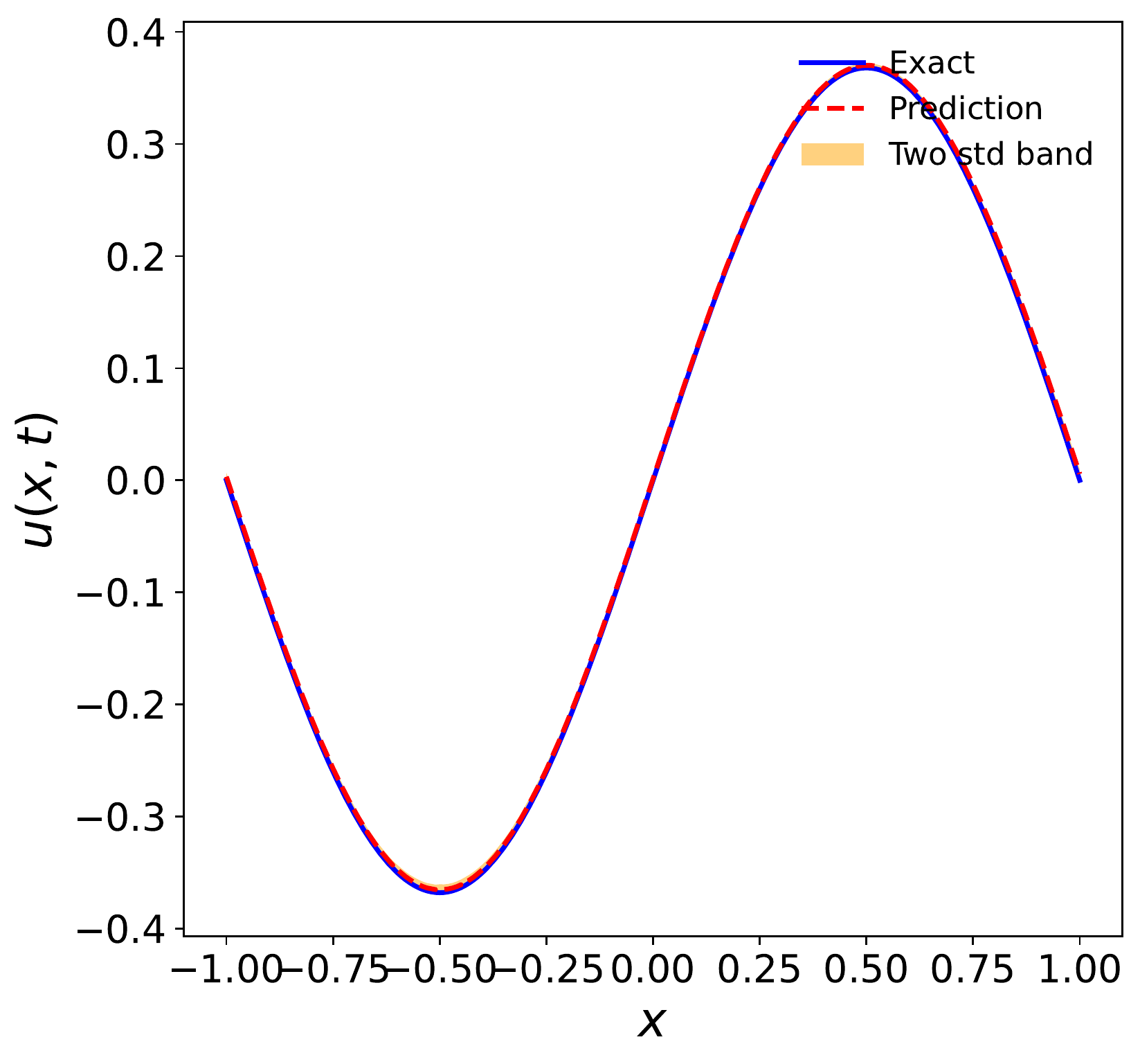}
    }\
    }\\
    
    \makebox[\linewidth][c]{%
  \centering
  \subfloat
  [$\sigma_1 = 0.05$, \hspace{0.3em} $t=0$] 
  {
     \label{exp2_pre4}     
    \includegraphics[width=0.3\textwidth]{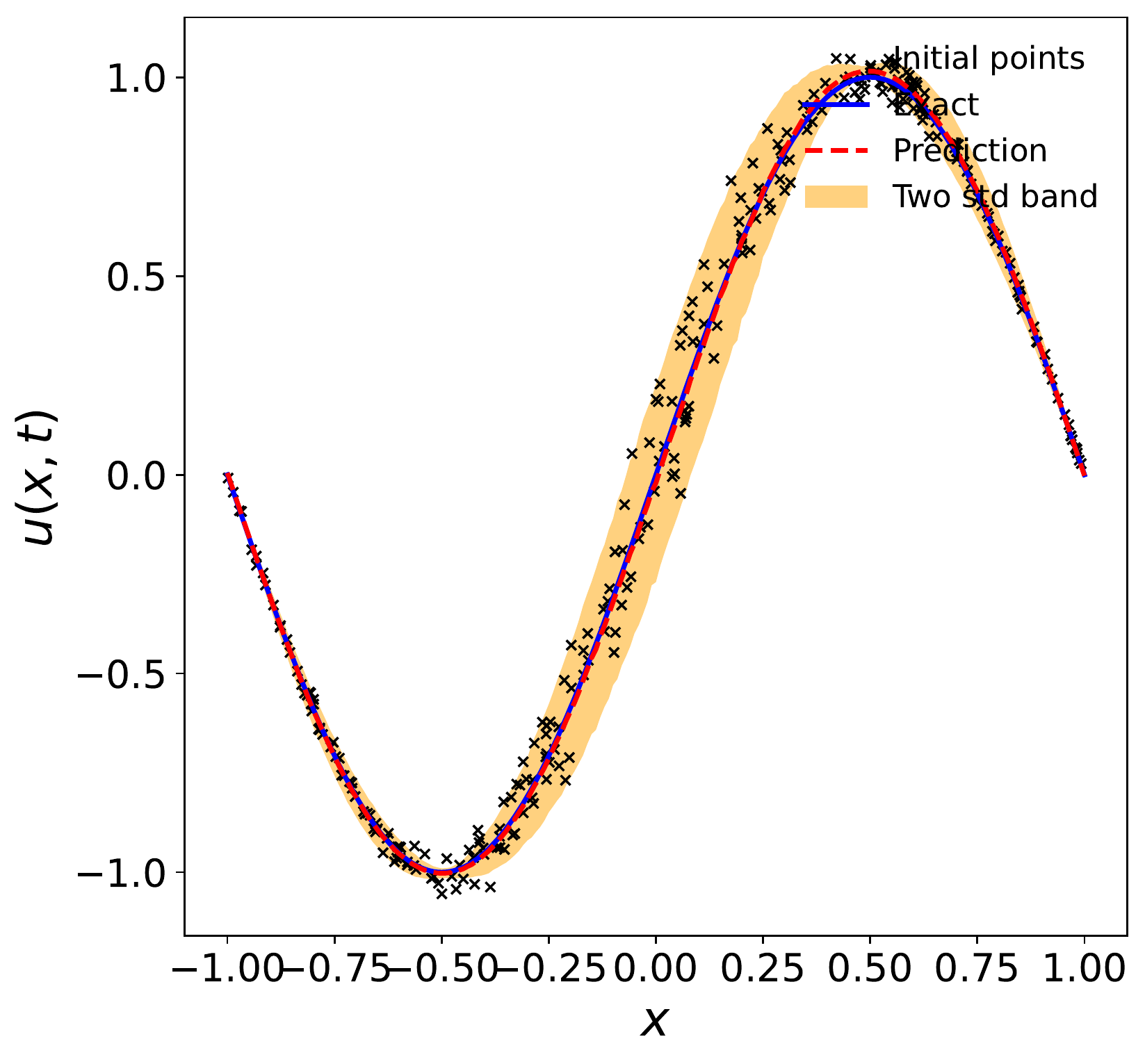}
    }\  
    \subfloat
    [$\sigma_1 = 0.05$, \hspace{0.3em} $t=0.5$] 
    {
    \label{exp2_pre5}     
    \includegraphics[width=0.3\textwidth]{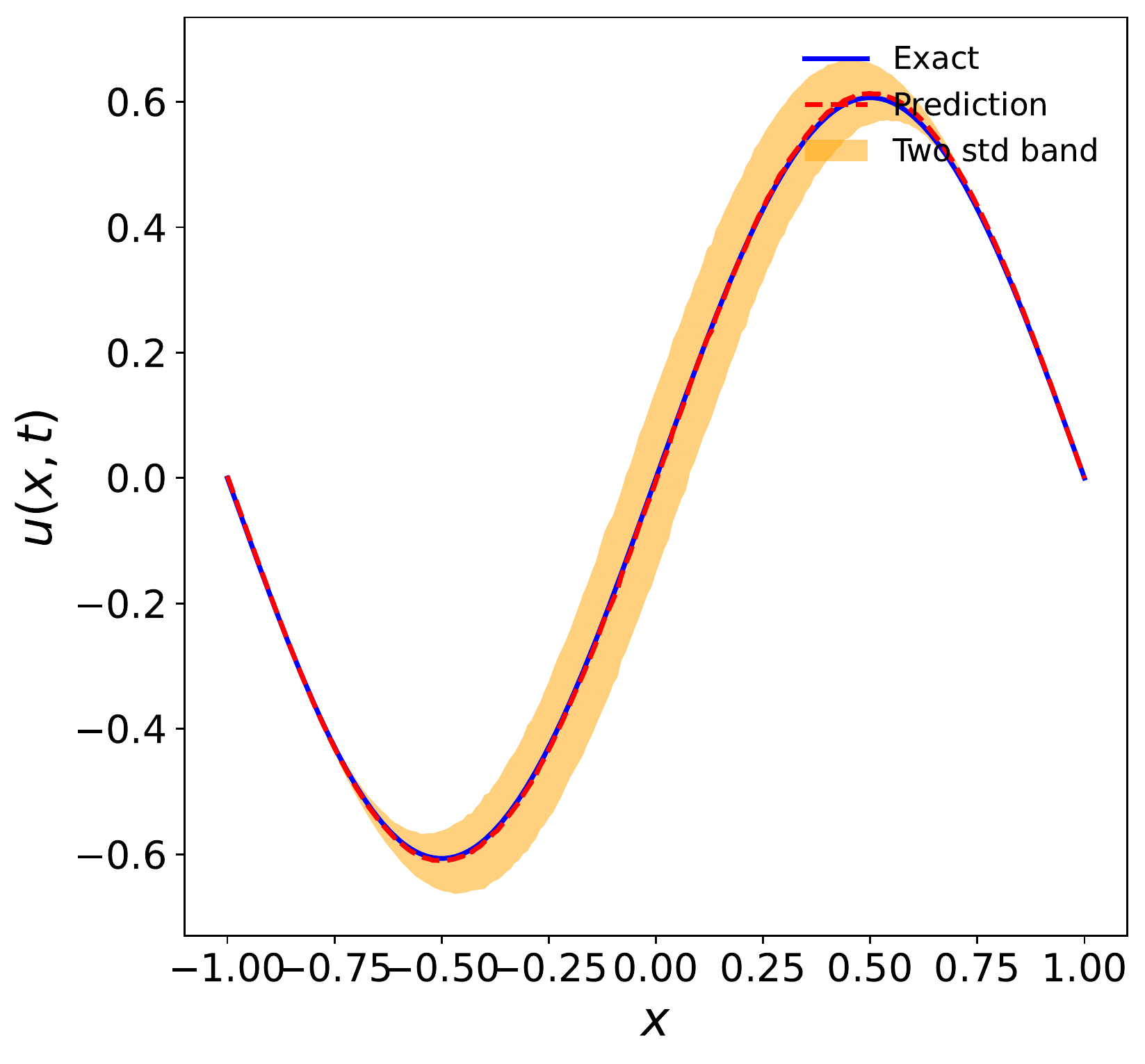}
    }\
        \subfloat
    [$\sigma_1 = 0.05$, \hspace{0.3em} $t=1$] 
    {
    \label{exp2_pre6}     
    \includegraphics[width=0.3\textwidth]{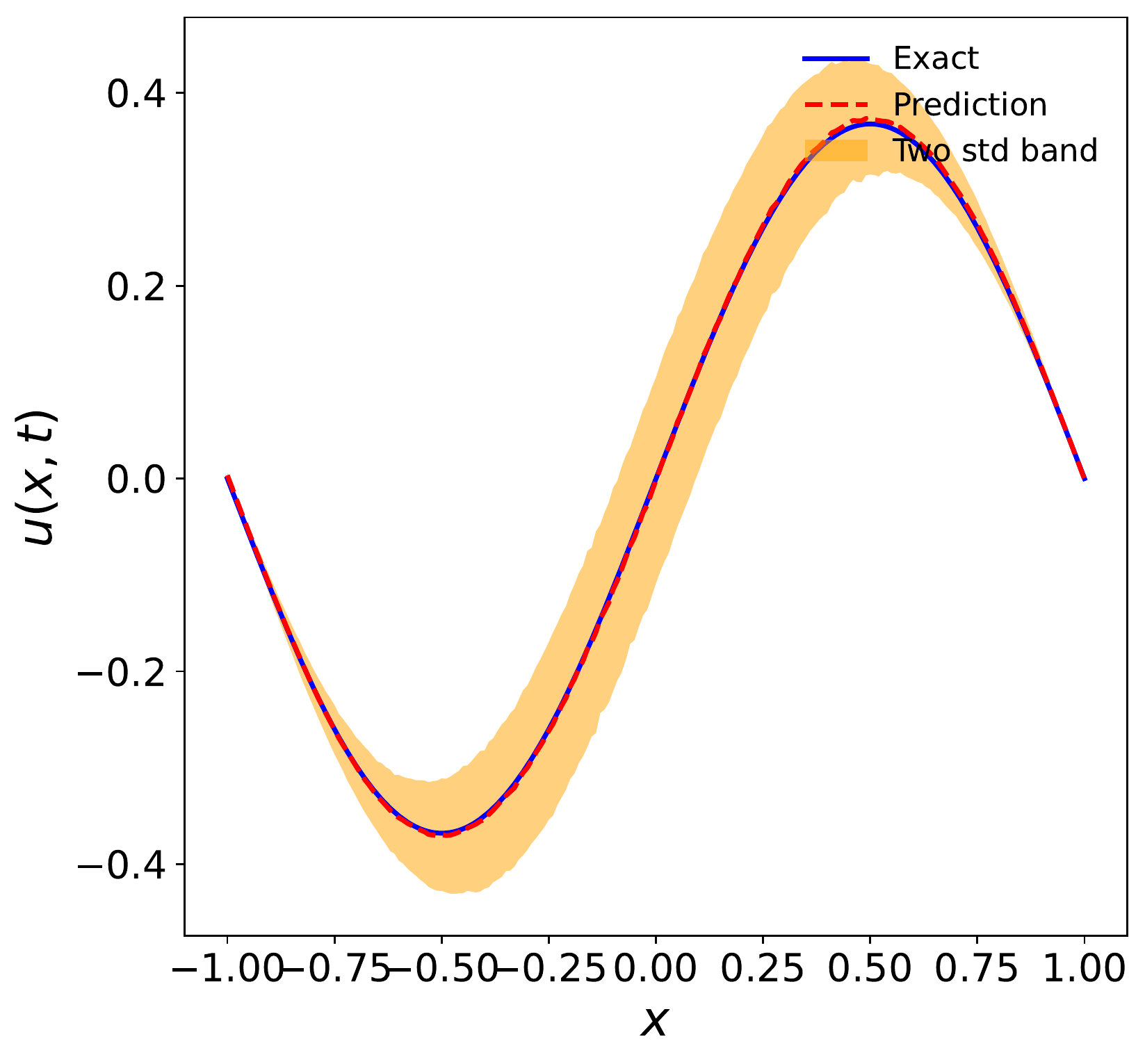}
    }\
    }
    \caption{		
	 Heat equation: the mean, the lower and the upper bounds of solution values
	 by $p_{\Tilde{g}}(u)$ given $(x,t)$.}	
    \label{exp2_visual_pre}
\end{figure*}

\begin{figure*}[h!]
\makebox[\linewidth][c]{%
  \centering
  \subfloat
  [$u(x=-0.25,t=0)$] 
  {
     \label{exp2_his1}     
    \includegraphics[width=0.3\textwidth]{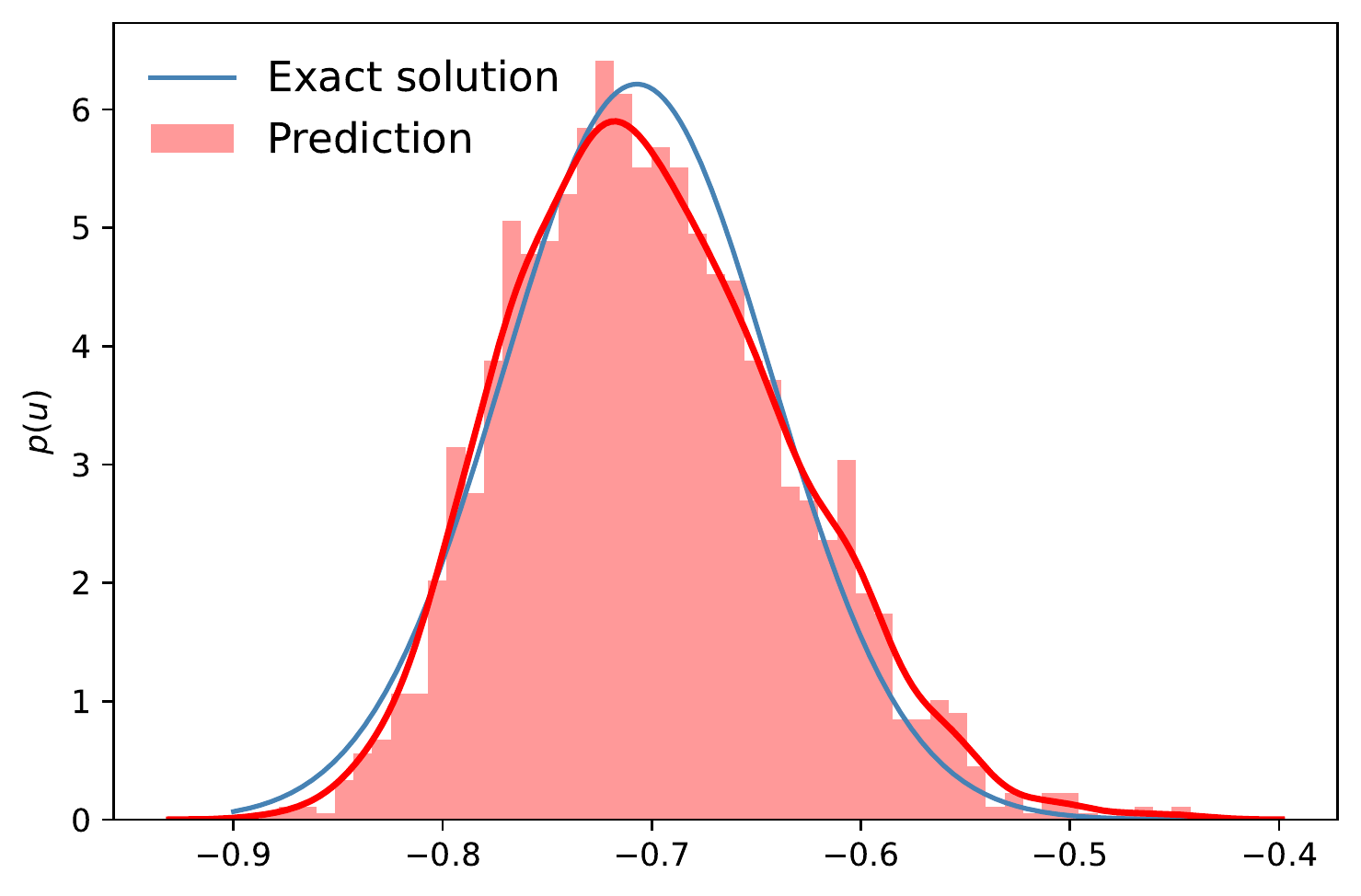}
    }\  
    \subfloat
    [$u(x=0,t=0)$] 
    {
    \label{exp2_his2}     
    \includegraphics[width=0.3\textwidth]{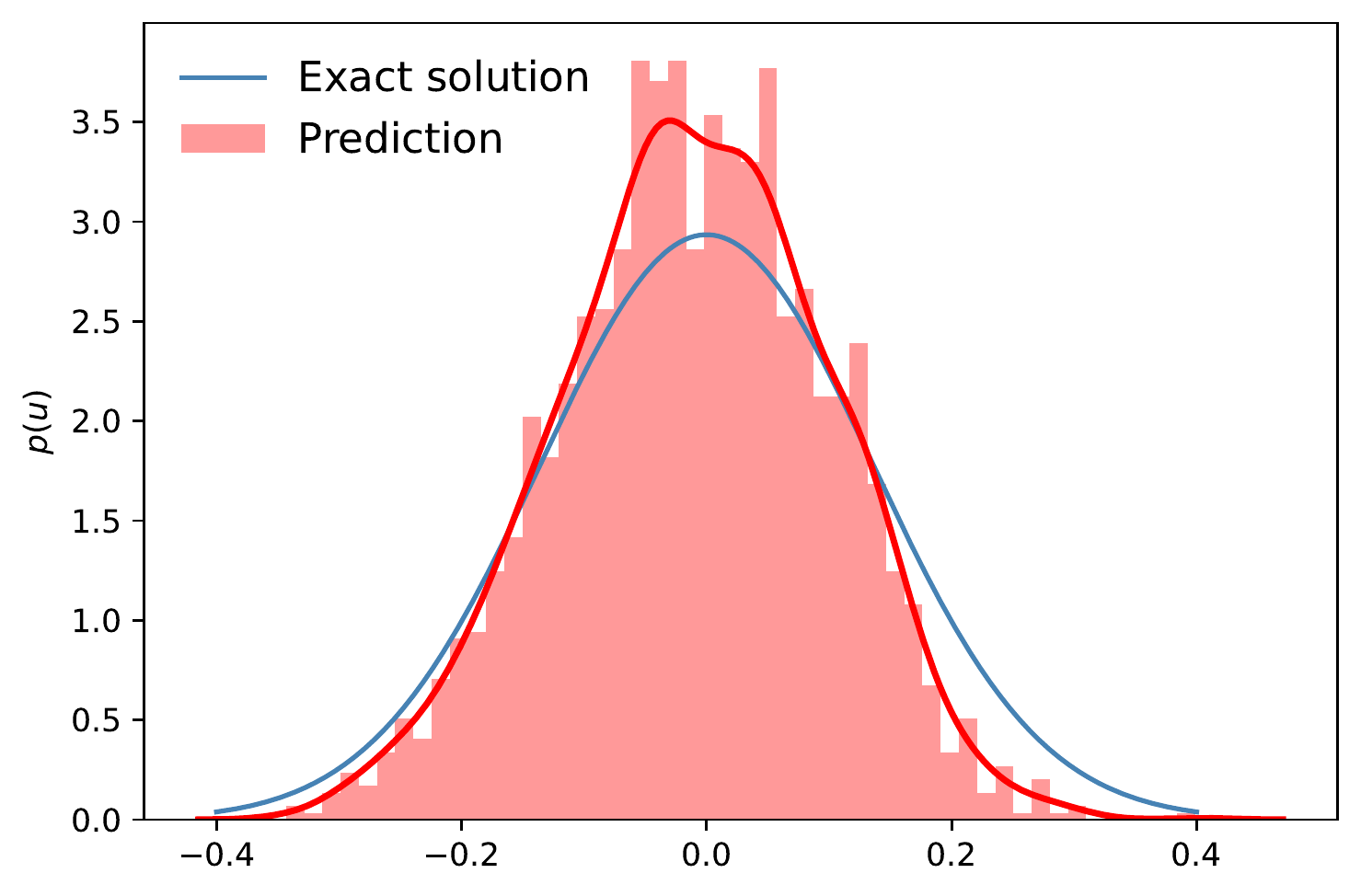}
    }\
    \subfloat
    [$u(x=0.25,t=0)$] 
    {
    \label{exp2_his3}     
    \includegraphics[width=0.3\textwidth]{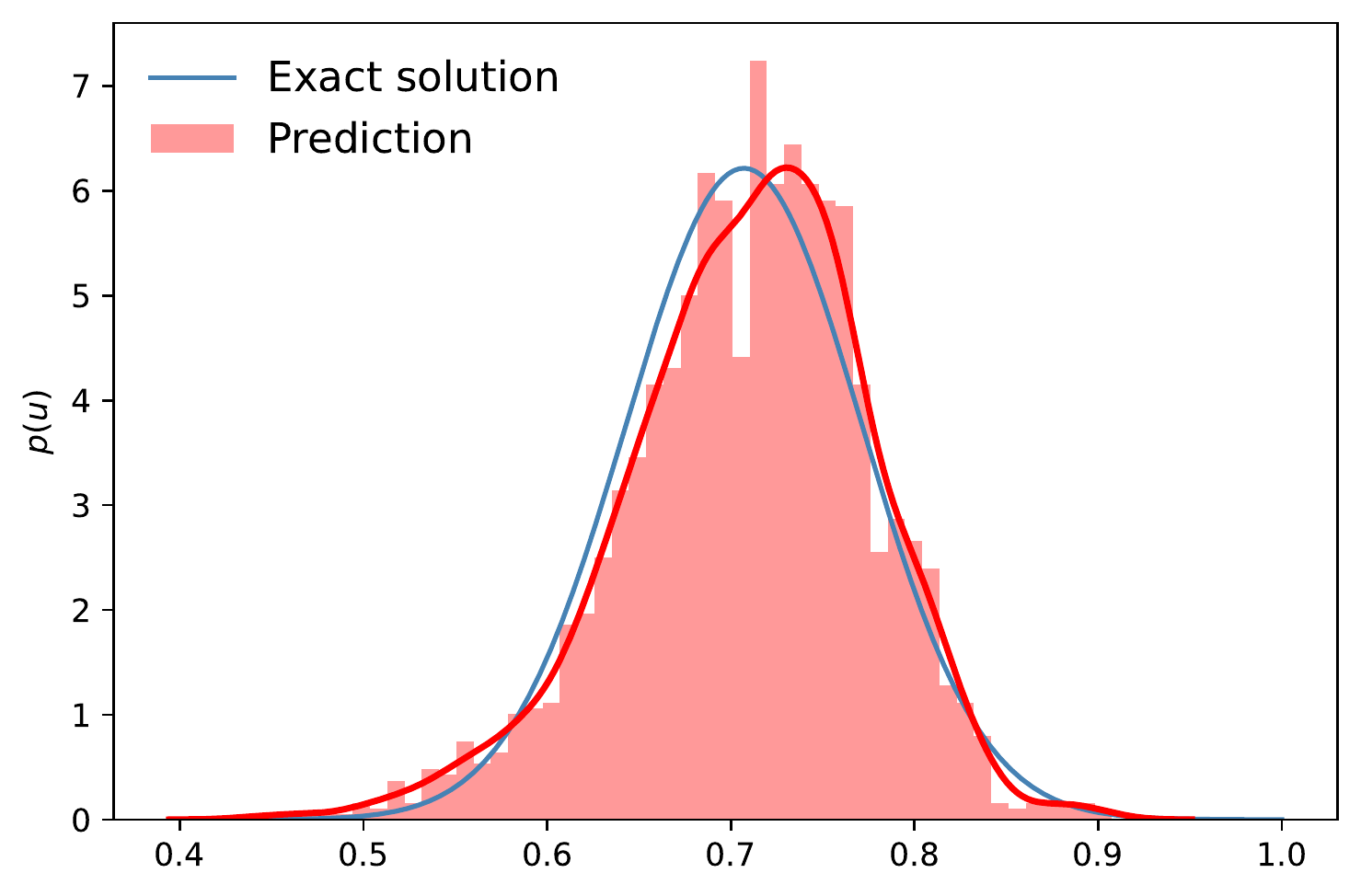}
    }\
    }
    \caption{		
	 Heat equation: the histograms of the generated data and the simulated solution at $(x=-0.25,t=0)$, $(x=0,t=0)$ and $(x=0.25,t=0)$ for $\sigma_1 = 0.05$. The 
	 predicted mean and standard deviation are  
	 $(\hat{\mu},\hat{\sigma})$, and the exact mean and standard deviation are $(\mu,\sigma)$. (a): $(\hat{\mu},\hat{\sigma})=(-0.701,0.066)$, $(\mu,\sigma)=(-0.707,0.064)$; (b): $(\hat{\mu},\hat{\sigma}) = (-0.010, 0.109)$, $(\mu,\sigma)=(0,0.136)$ ; (c): $(\hat{\mu},\hat{\sigma}) = (0.707, 0.065)$, $(\mu,\sigma)=(0.707,0.064)$.}	
    \label{exp2_visual_his}
\end{figure*}

\subsection{Burgers Equation}

\begin{figure*}[h!]
\makebox[\linewidth][c]{%
  \centering
  \subfloat
  [$\sigma_1 = \sigma_2 = \sigma_3 = 0$, \hspace{0.3em} $t=0$] 
  {
     \label{exp3_pre1}     
    \includegraphics[width=0.5\textwidth]{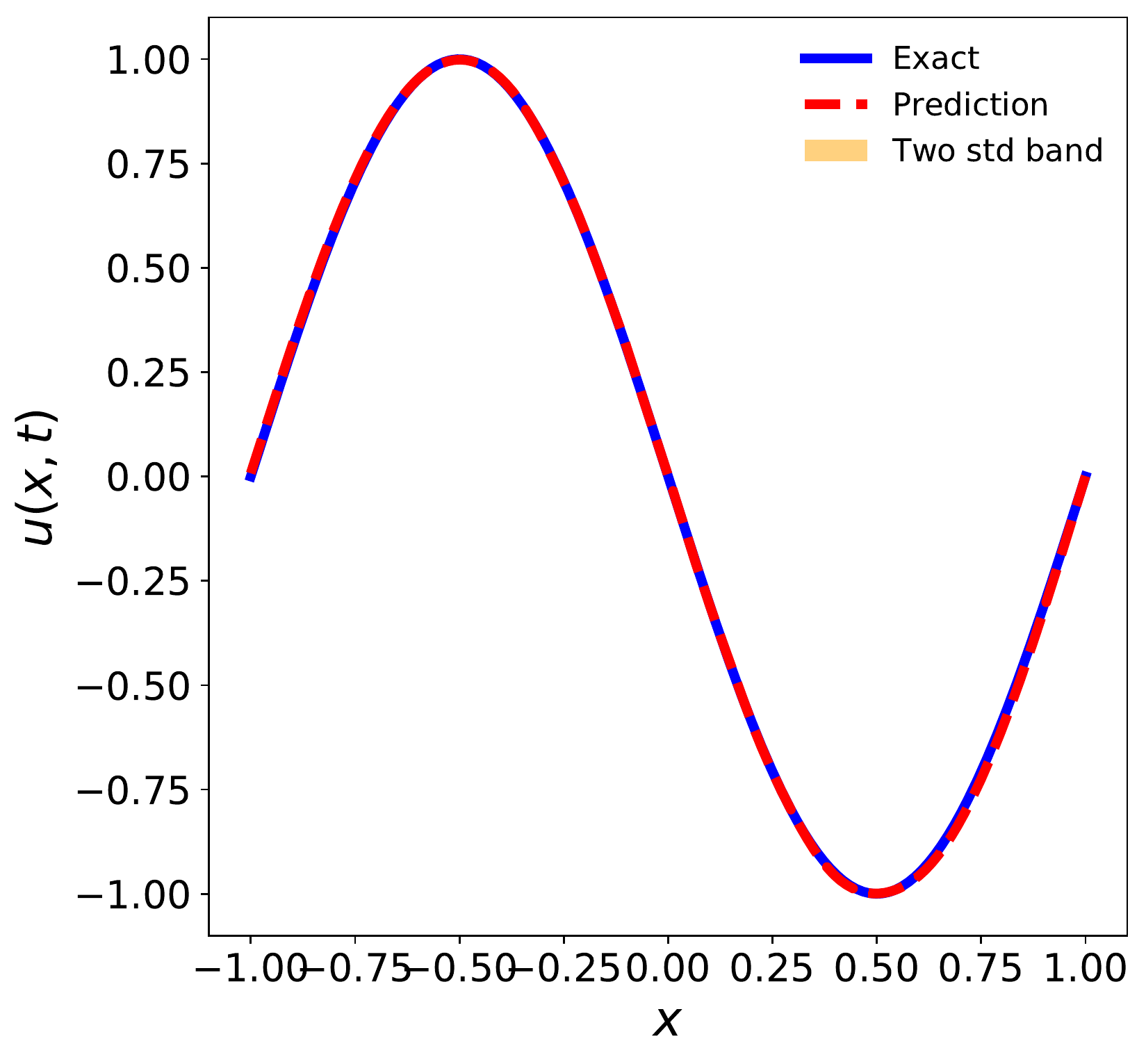}
    }\  
    \subfloat
    [$\sigma_1 = \sigma_2 = \sigma_3 = 0$, \hspace{0.3em} $t=0.25$] 
    {
    \label{exp3_pre2}     
    \includegraphics[width=0.5\textwidth]{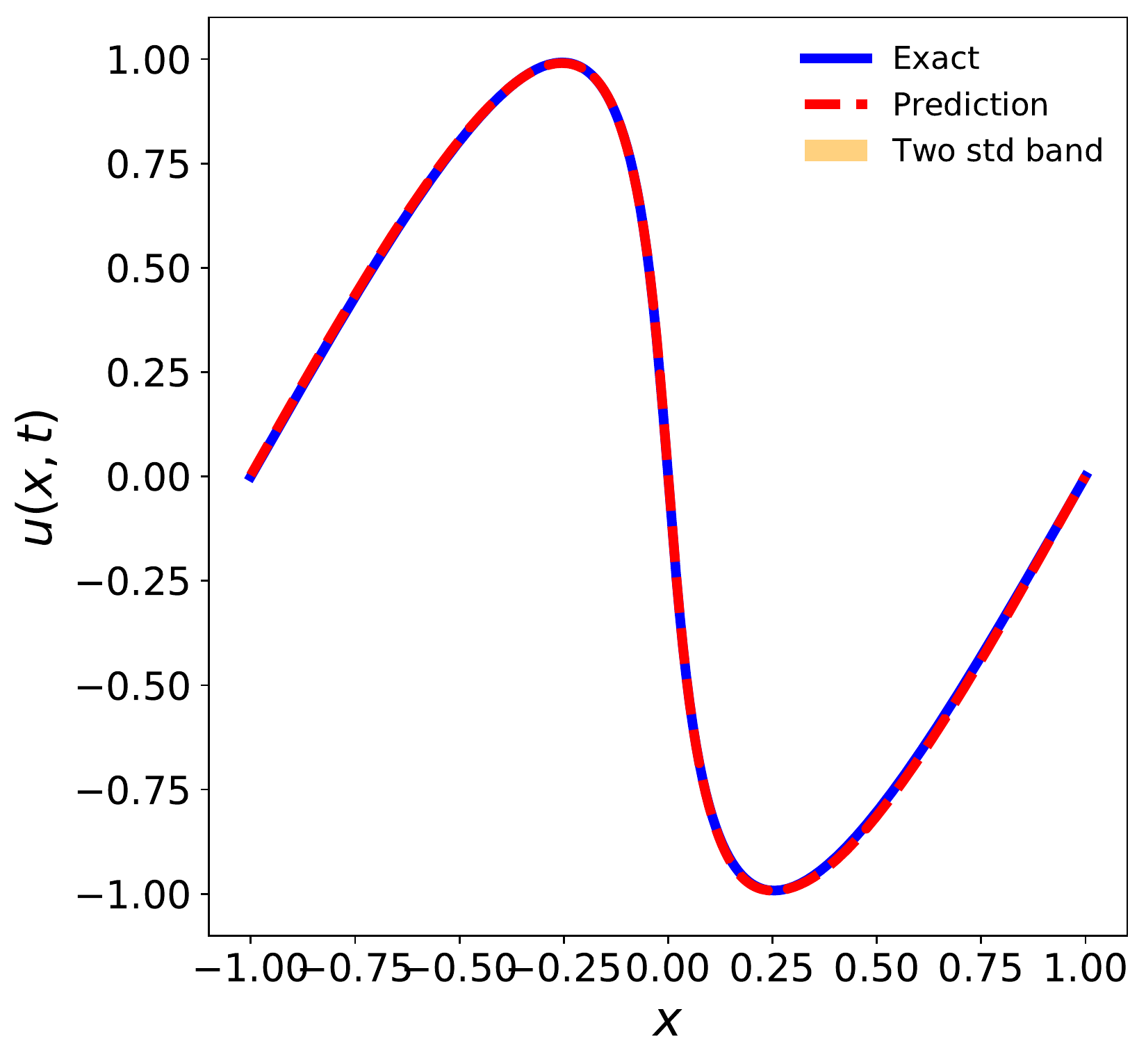}
    }\
    }\\
    
    \makebox[\linewidth][c]{%
  \centering
  \subfloat
  [$\sigma_1 = \sigma_2 = \sigma_3 = 0$, \hspace{0.3em} $t=0.5$] 
  {
     \label{exp3_pre3}     
    \includegraphics[width=0.5\textwidth]{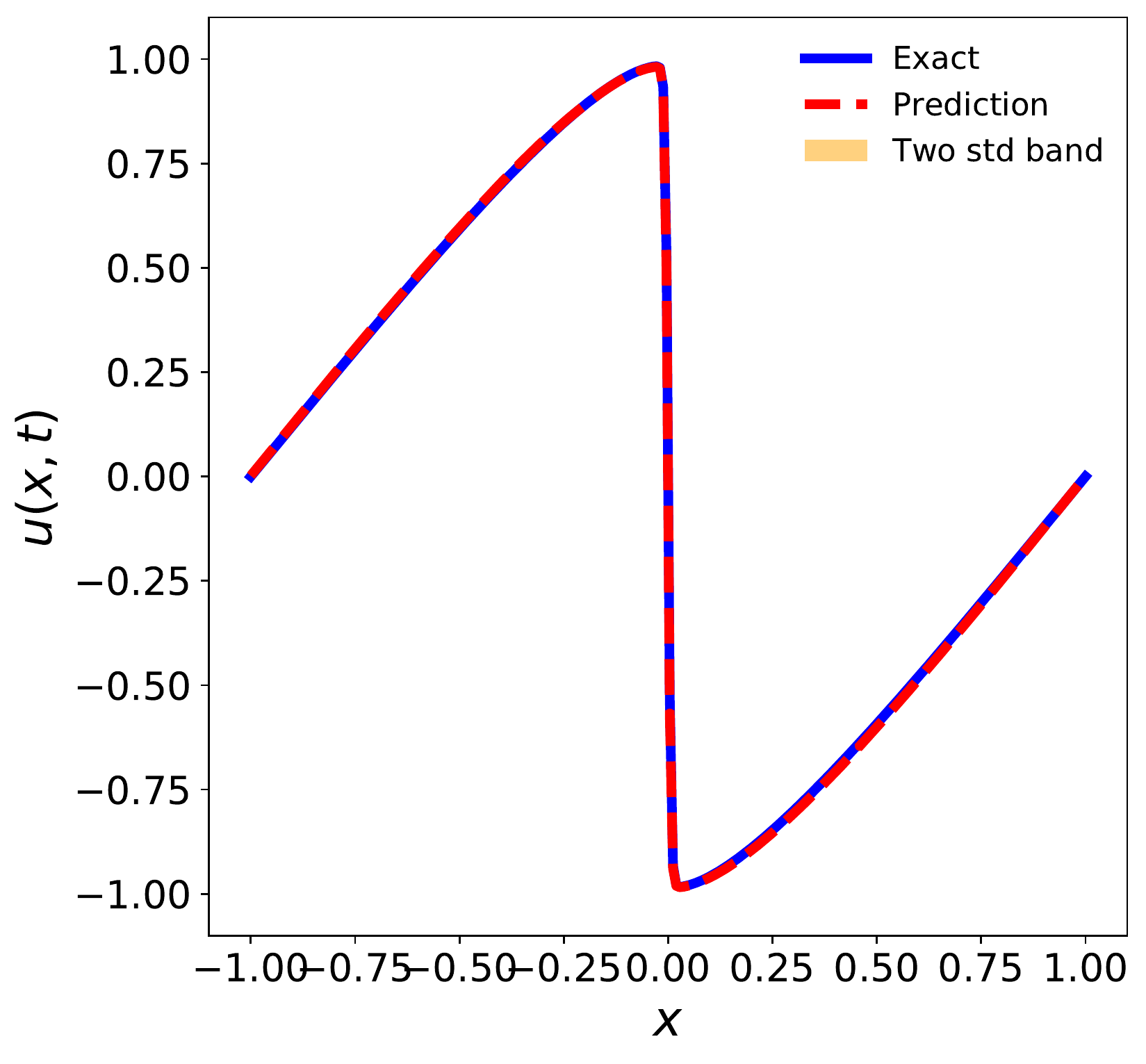}
    }\  
    \subfloat
    [$\sigma_1 = \sigma_2 = \sigma_3 = 0$, \hspace{0.3em} $t=0.75$] 
    {
    \label{exp3_pre4}     
    \includegraphics[width=0.5\textwidth]{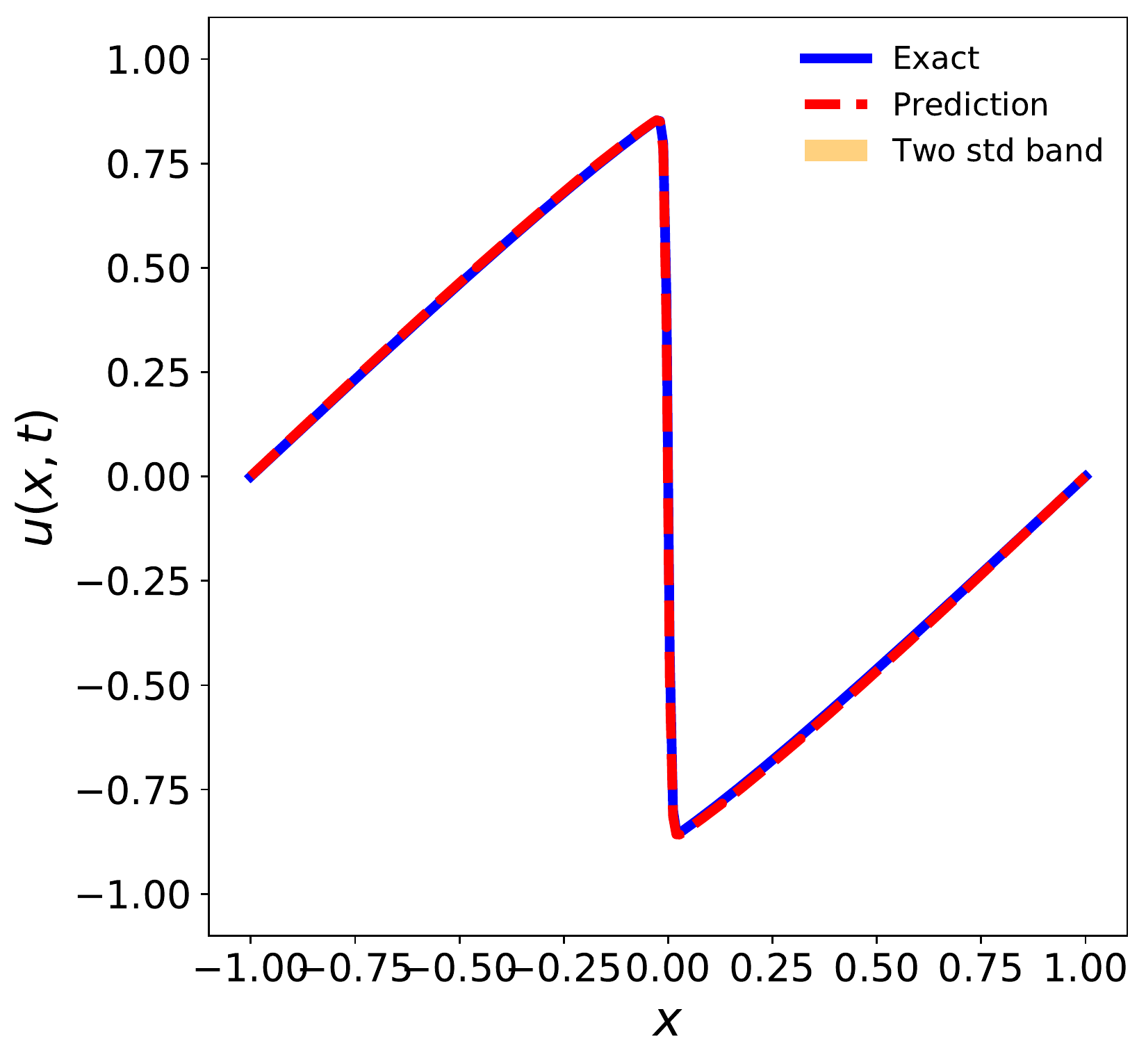}
    }\
    }
    \caption{		
     Burgers equation (noise-free): the mean, the lower and upper bounds 
	  of $p_{\Tilde{g}}(u)$ given $(x,t)$.}	
    \label{exp3_visual_pre1}
\end{figure*}

\begin{figure*}[h!]
\makebox[\linewidth][c]{%
  \centering
  \subfloat
  [$\sigma_1 = \sigma_2 = \sigma_3 = 0.1$, \hspace{0.3em} $t=0$] 
  {
     \label{exp3_pre5}     
    \includegraphics[width=0.5\textwidth]{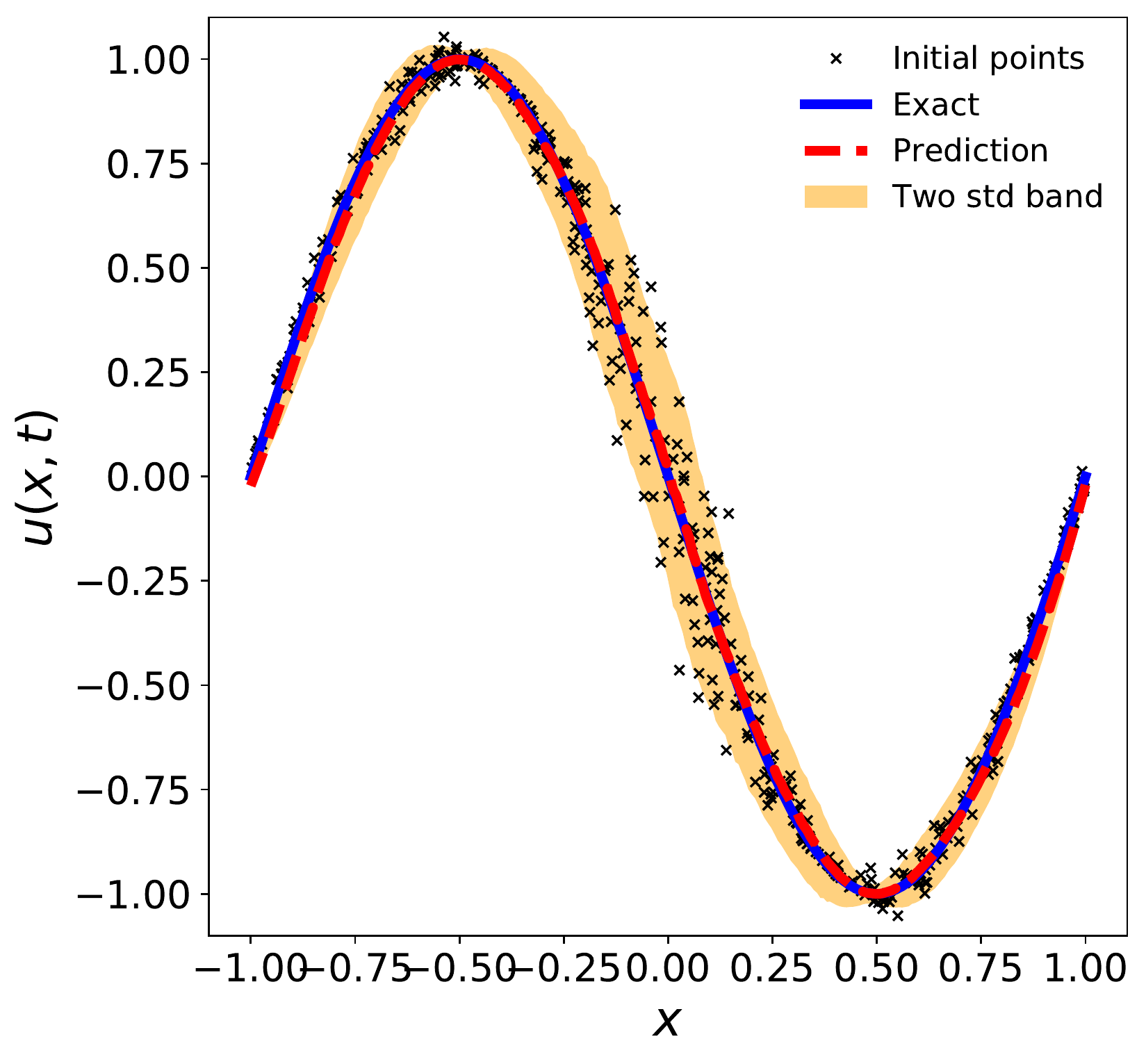}
    }\  
    \subfloat
    [$\sigma_1 = \sigma_2 = \sigma_3 = 0.1$, \hspace{0.3em} $t=0.25$] 
    {
    \label{exp3_pre6}     
    \includegraphics[width=0.5\textwidth]{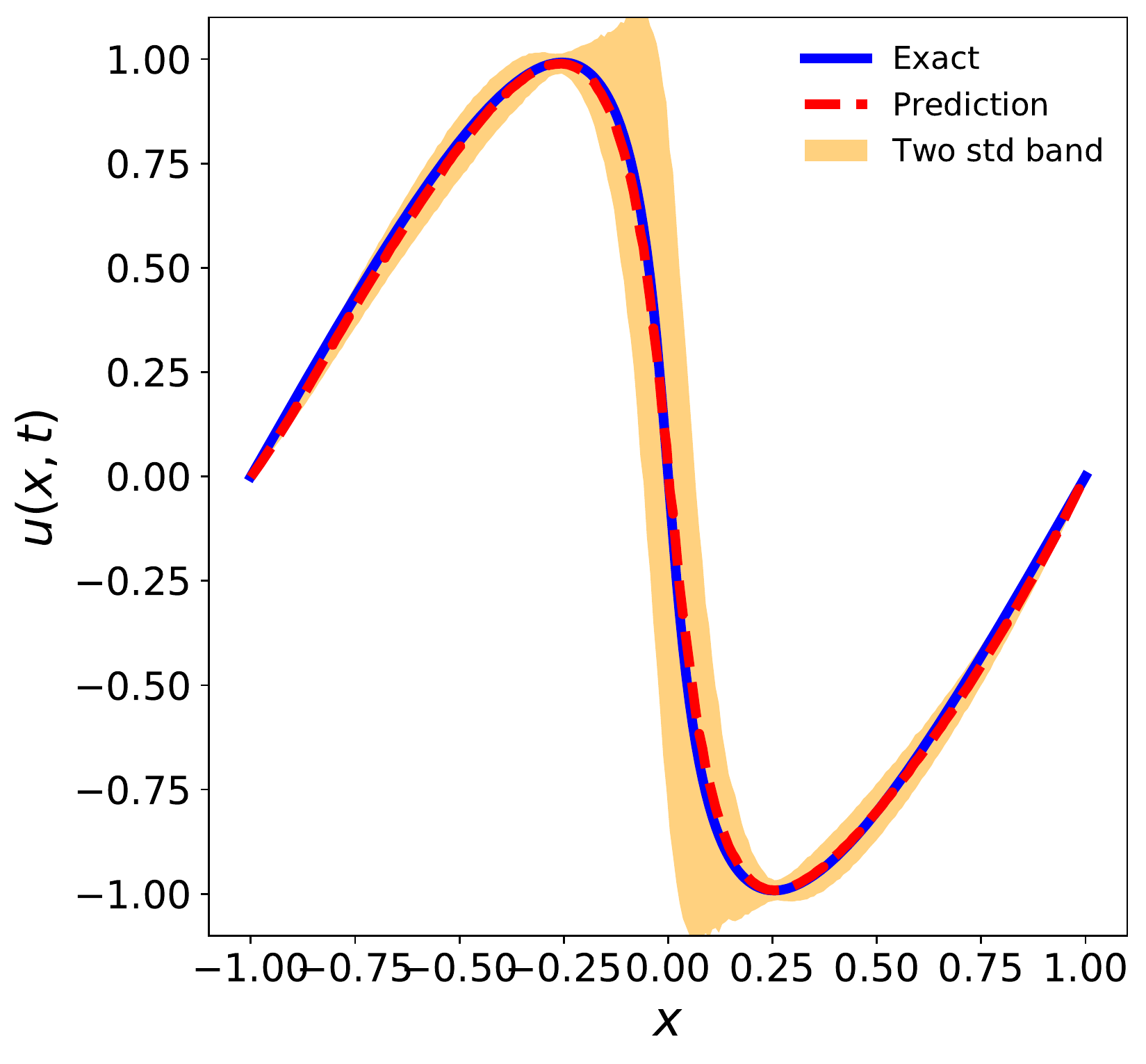}
    }\
    }\\
    
    \makebox[\linewidth][c]{%
  \centering
  \subfloat
  [$\sigma_1 = \sigma_2 = \sigma_3 = 0.1$, \hspace{0.3em} $t=0.5$] 
  {
     \label{exp3_pre7}     
    \includegraphics[width=0.5\textwidth]{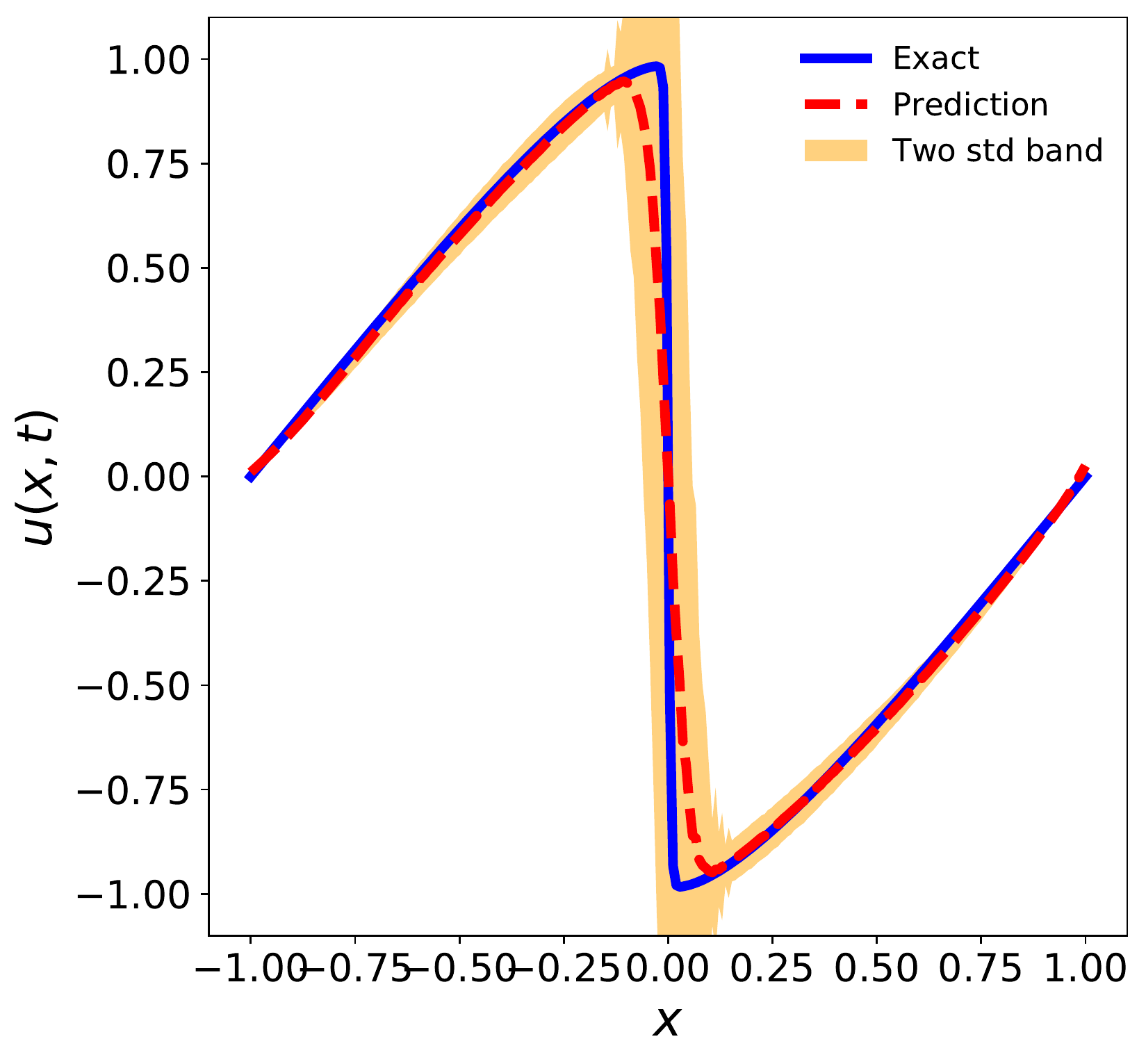}
    }\  
    \subfloat
    [$\sigma_1 = \sigma_2 = \sigma_3 = 0.1$, \hspace{0.3em} $t=0.75$] 
    {
    \label{exp3_pre8}     
    \includegraphics[width=0.5\textwidth]{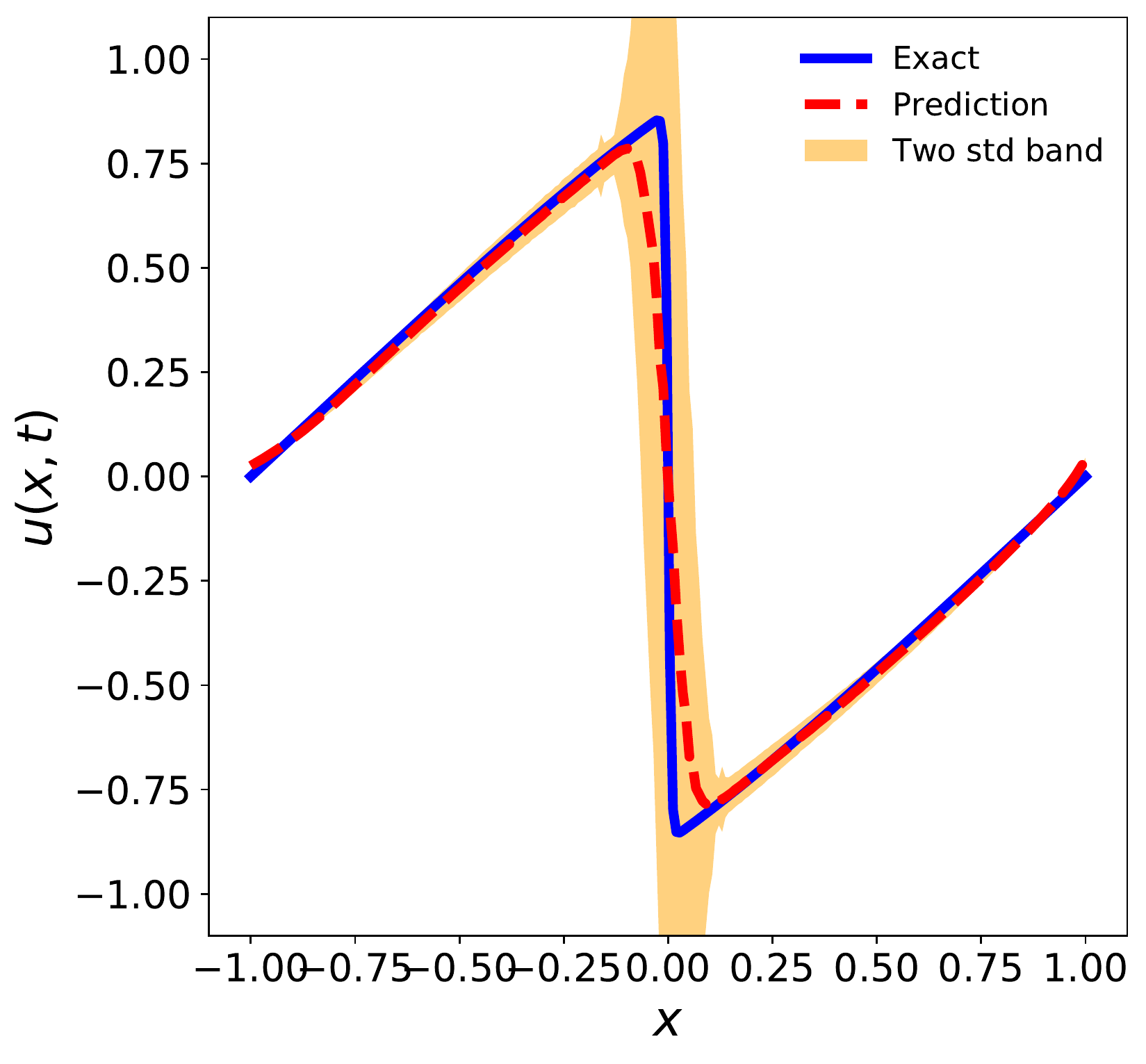}
    }\
    }
    \caption{		
	 Burgers equation ($\sigma=0.1$): the mean, the lower and upper bounds 
	  of $p_{\Tilde{g}}(u)$ given $(x,t)$.}	
    \label{exp3_visual_pre2}
\end{figure*}

\begin{figure*}[h!]
\makebox[\linewidth][c]{%
  \centering
    \subfloat
    [$u(x=-0.25,t=0)$] 
    {
     \label{exp3_his1}     
    \includegraphics[width=0.3\textwidth]{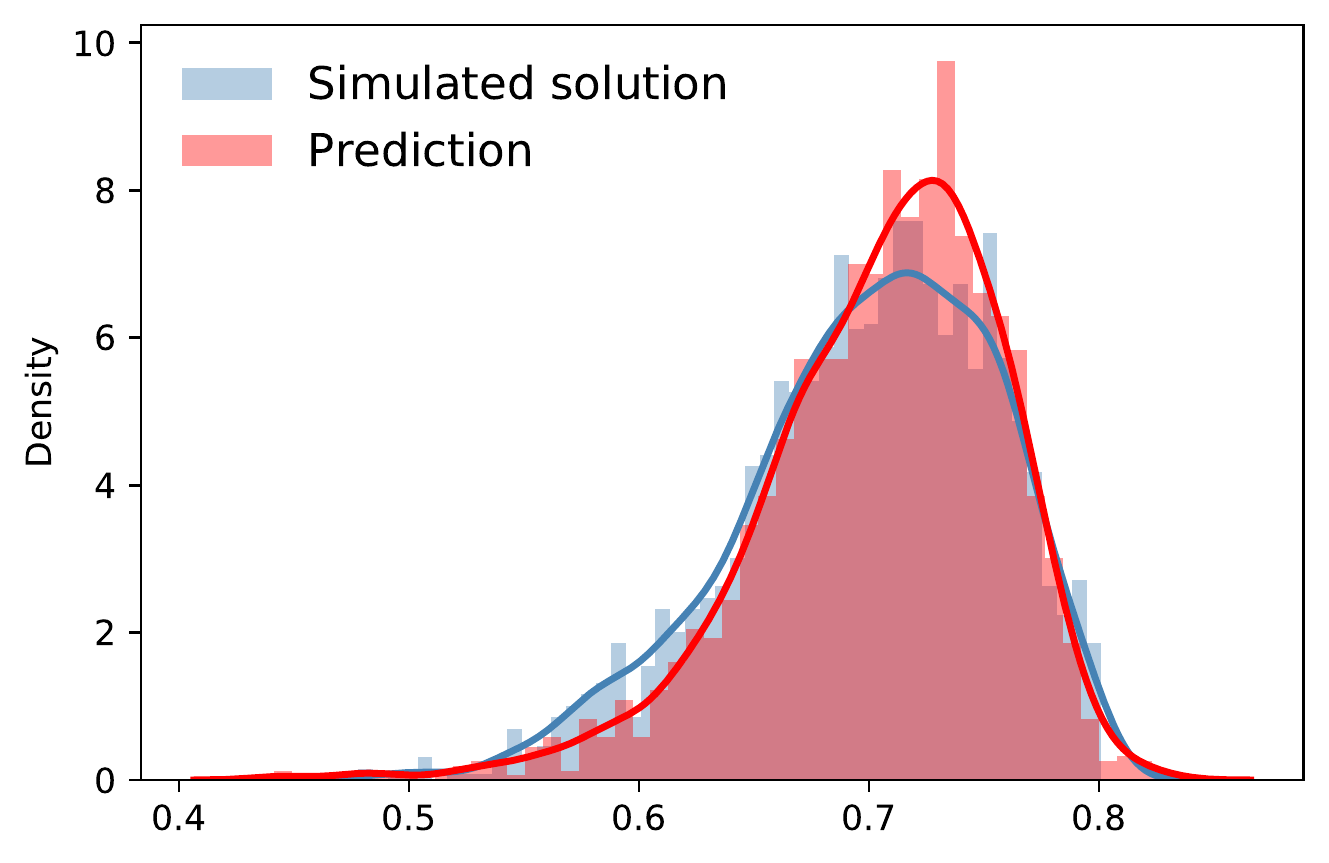}
    }\  
    \subfloat
    [$u(x=0,t=0)$] 
    {
    \label{exp3_his2}     
    \includegraphics[width=0.3\textwidth]{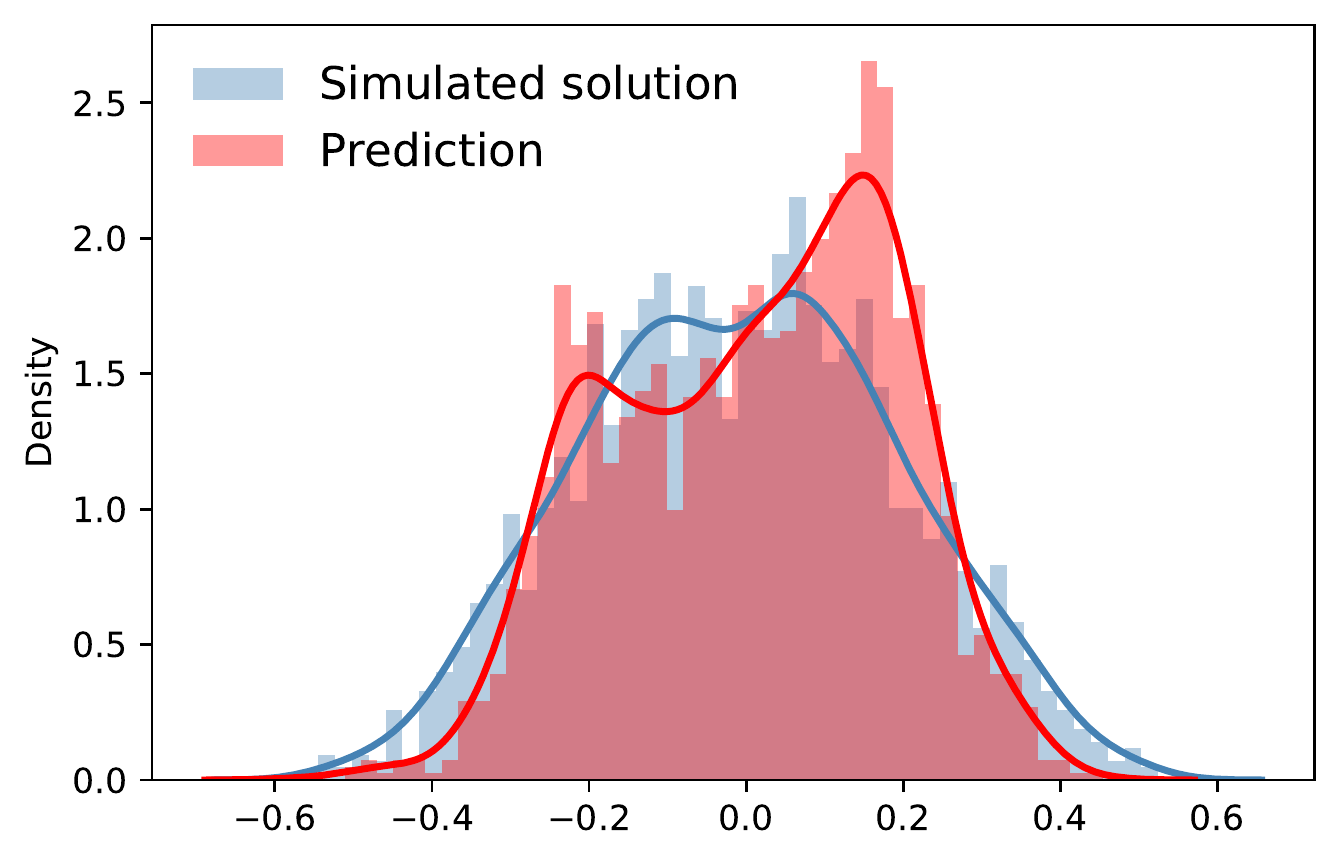}
    }\
    \subfloat
    [$u(x=0.25,t=0)$] 
    {
    \label{exp3_his3}     
    \includegraphics[width=0.3\textwidth]{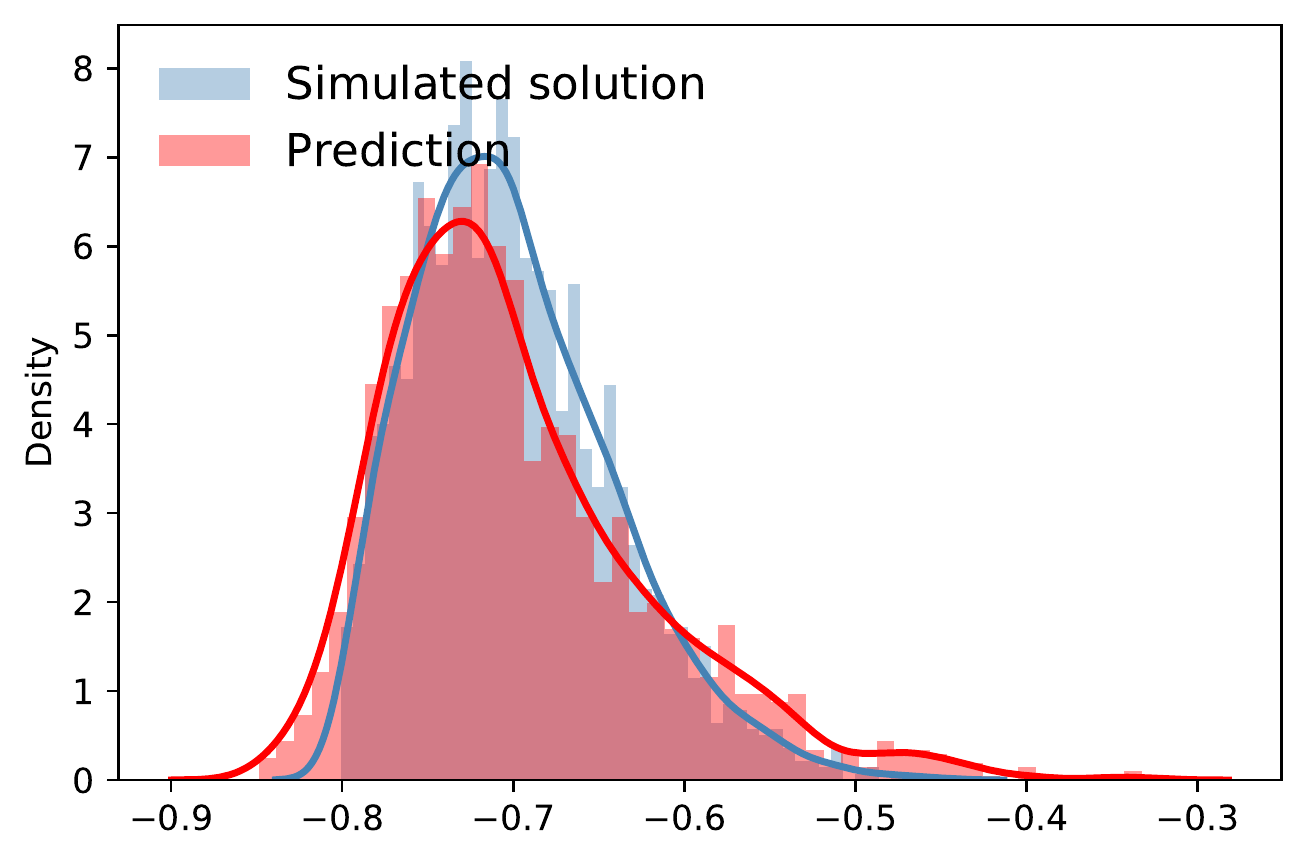}
    }\
    }
    \caption{		
	 Burgers equation: 
	 the histograms of the generated data and the simulated solution at $(x=-0.25,t=0)$, $(x=0,t=0)$ and  $(x=0.25,t=0)$ for $\sigma_1 = 0.1$. The 
	 predicted mean and standard deviation are $(\hat{\mu},\hat{\sigma})$ and the 
	 simulated mean and standard deviation are $(\mu_{est},\sigma_{est})$. (a) $(\hat{\mu},\hat{\sigma})=(0.705,0.053)$, $(\mu_{est},\sigma_{est})=(0.701,0.056)$; (b) $(\hat{\mu},\hat{\sigma}) = (0.011, 0.178)$, $(\mu_{est},\sigma_{est})=(-0.001,0.202)$; (c) $(\hat{\mu},\hat{\sigma}) = (-0.698, 0.077)$, $(\mu_{est},\sigma_{est})=(-0.698,0.058)$.}	
    \label{exp3_visual_his}
\end{figure*}

In this example, we test the effectiveness and robustness of WGAN-PINNs on solving a time-dependent nonlinear Burgers equation (one spatial dimension) which contains
a steep solution at $x=0$. The Burgers equation is given as follows:
\begin{equation}
\begin{split}
    & u_t + u u_{x} - \nu u_{xx} = 0, \hspace{1em} \text{(P.D.E.)}, \hspace{1em} (x,t) \in [-1,1] \times [0,1]\\
    & u(x,t=0) = - \sin(\pi x), \hspace{1em} \text{(I.C.)}\\
    & u(x=-1,t) = u(x=1,t)=0, \hspace{1em} \text{(B.C.)}
\end{split}
\end{equation}
where the viscosity parameter is chosen as $\nu = \frac{0.01}{\pi}$. 
The exact solution (without uncertainty) is solved through the Cole-Hopf transformation \cite{hopf1950partial}.
Here 
we adopt a more complicated non-Gaussian and spatially dependent random initial conditions, i.e.,
\begin{equation}
\label{exp3_random}
    \begin{split}
        & u(x,t=0)= \sin (\pi (x+\delta_1)) + \delta_1, \hspace{1em} \delta_1 = \frac{\epsilon_1}{\exp(3|x|)}, \hspace{1em} \epsilon_1 \sim \mathcal{N}(0,\sigma_1^2), \hspace{1em} \text{(I.C.)}\\
    \end{split}
\end{equation}

In the experiment, we merely display its predictive abilities. The generator and the discriminator are neural networks of $(D_g,W_g)=(4,50)$ and $(D_f,W_f)=(3,50)$ respectively. We have 500 samples for each boundary/initial conditions (therefore, $m=n=1000$ for 500 boundary data and 500 initial data) and 10,000 samples in interior domain for PINNs ($k=10,000$). We sample uniformly for $(x,t)$ on the boundary and in the domain. We find that the training will fail if we use small interior samples. 
The number of interior training data $k$ here is much larger than that required 
in the last two experiments. The steepness property of the exact solution at $x=0$ results in the phenomenon, consistent with Assumption \ref{assump_pinns} in section \ref{conv}. We adopt a two dimensional latent variable $\mathbf{z} \sim \mathcal{N}(\mathbf{0},\mathbf{I}_2)$ in the generative model. 

Figure \ref{exp3_visual_pre1} and \ref{exp3_visual_pre2} show visualized results for our predictions at $t=0$, $0.25$, $0.5$ and $0.75$ with noise levels $\sigma_1=0$ (noise-free) and $\sigma_1=0.1$. The prediction lines match the solution
without certainty lines well. Yellow bars represent the lower and upper bounds 
of generated solutions, cover
most of the initial random interior data as is shown in Figure \ref{exp3_pre5}. 
Statistical comparison of the generated data and the simulated solution are displayed in Figure \ref{exp3_visual_his}. The results further demonstrate that our model is able to capture non-Gaussian and data-dependent noise (uncertainty), which are more common in real applications. The remarkable prediction accuracy indicates the effectiveness and robustness of WGAN-PINNs.

\subsection{Nonlinear Allen-Cahn equation}

We further investigate the performance of WGAN-PINNs on the following 
two-dimensional nonlinear Allen-Cahn equation:
\begin{equation}
    \begin{split}
        & c \cdot ((u_x)^2 + (u_y)^2) + u(u^2-1) = f, \hspace{1em} \text{(P.D.E.),} \hspace{1em} (x,y) \in [-1,1] \times [-1,1]\\
        & u(x,y) = 0, \hspace{1em}  \text{(B.C.).}
    \end{split}
\end{equation}
where $c=0.01$ represents the mobility. The Allen-Cahn equation is widely applied for multi-phase flows. Here, the exact solution is set to be $u(x,y) = \sin(\pi  x) \cdot \sin (\pi  y)$ and the corresponding $f(x,y)$ can be obtained. Similarly, instead of deterministic boundary data, we consider random (noised) boundary conditions, i.e., 
\begin{equation}
    \begin{split}
         u(x,y=-1) \sim \mathcal{N}(0,\sigma_1^2), \hspace{0.5em}  \text{(B.C.1)} & \hspace{1em}
         u(x,y=1) \sim \mathcal{N}(0,\sigma_2^2), \hspace{0.5em}  \text{(B.C.2)}\\
         u(x=-1,y) \sim \mathcal{N}(0,\sigma_3^2), \hspace{0.5em} \text{(B.C.3)} &  \hspace{1em}  
         u(x=1,y) \sim \mathcal{N}(0,\sigma_4^2), \hspace{0.5em} \text{(B.C.4)}
    \end{split}
\end{equation}
and we let noise levels $\sigma_1 = \cdots = \sigma_4$. 

In this example, we adopt relatively shallow and narrow generators and 
discriminators of $(D_g, W_g)=(2,30)$ and $(D_f, W_f)=(2,30)$ respectively for lower computational cost. The latent variable is assumed to be $\mathbf{z} \sim \mathcal{N}(\mathbf{0},\mathbf{I}_2)$. Given $k=500$ uniformly distributed interior data and $m=n=800$ (equally 200 data for each boundary) uniformly distributed boundary data, WGAN-PINNs are able to solve the PDE with reasonable accuracy 
even with noisy boundary data. We test both the deterministic (noise free) and noisy boundaries cases. 
WGANs term enforces the generator $g$ to satisfy four boundary conditions while PINNs regularization term narrows down the function space. For deterministic boundary conditions, WGAN-PINNs can achieve $\mathcal{E}=5 \cdot 10^{-3}$ 
while $\mathcal{E}=2.2 \cdot 10^{-2}$ for noisy conditions. Figure \ref{exp4_visual} displays visualized results. Bright and dark regions represent the peaks and valleys of the solution. Figure \ref{exp4_visual_his} shows histograms for both the generated data and the simulated solution. Our predictions coincide well with the solution.

The above four experiments may also somewhat imply the concern from the theory that the uncertainty propagation can be entirely wrong and may not converge to a correct one in the interior. Histograms in Figure \ref{exp1_visual_his} show that the predicted uncertainty in the interior differs from the simulated solution in the interior domain, compared with excellent data matching on the boundary. Similar observations can also be found in Figure \ref{exp2_visual_his}, \ref{exp3_visual_his} and \ref{exp4_visual_his}. It seems that the uncertainty propagation is not done so well based on the experimental results of histograms, compared with boundary data matching, although their mean and standard deviation are close. This is consistent with our 
limitations discussed in Section \ref{discussion_limitations} for the quality of uncertainty propagation of the proposed model from the boundary to the interior domain. Moreover, the practical difficulties of uncertain propagation may come from the non-Gaussian uncertainty distributions in the interior. Therefore, it is an open problem and left as a future work to investigate the satisfactory uncertainty quantification in more complicated non-Gaussian and multi-modal distributions. 

\begin{figure*}[h!]
\makebox[\linewidth][c]{%
  \centering
  \subfloat
  [exact solution] 
  {
     \label{exp4_pre1}     
    \includegraphics[width=0.3\textwidth]{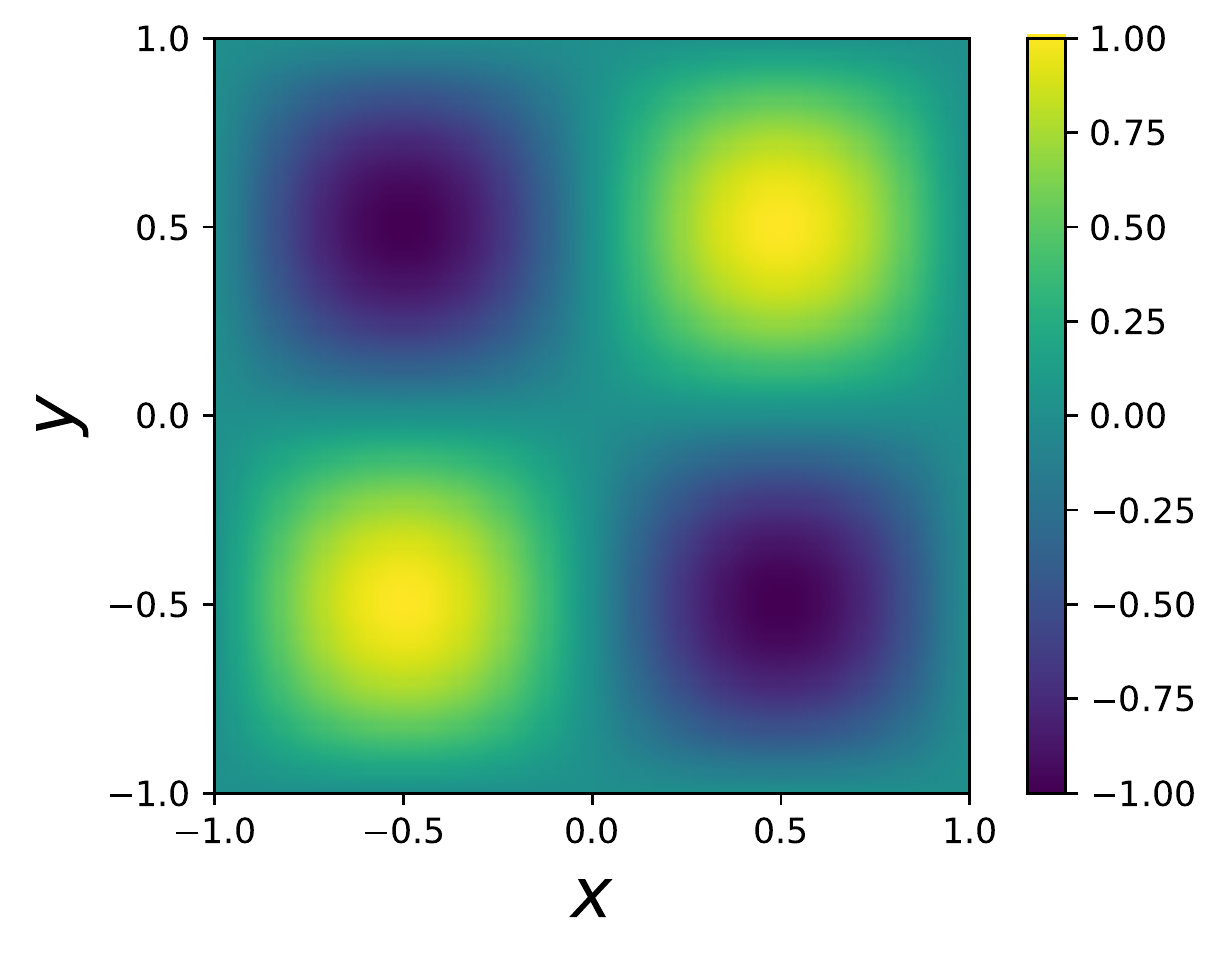}
    }\  
    \subfloat
    [prediction, $\sigma_1 =\cdots = \sigma_4 = 0$] 
    {
    \label{exp4_pre2}     
    \includegraphics[width=0.3\textwidth]{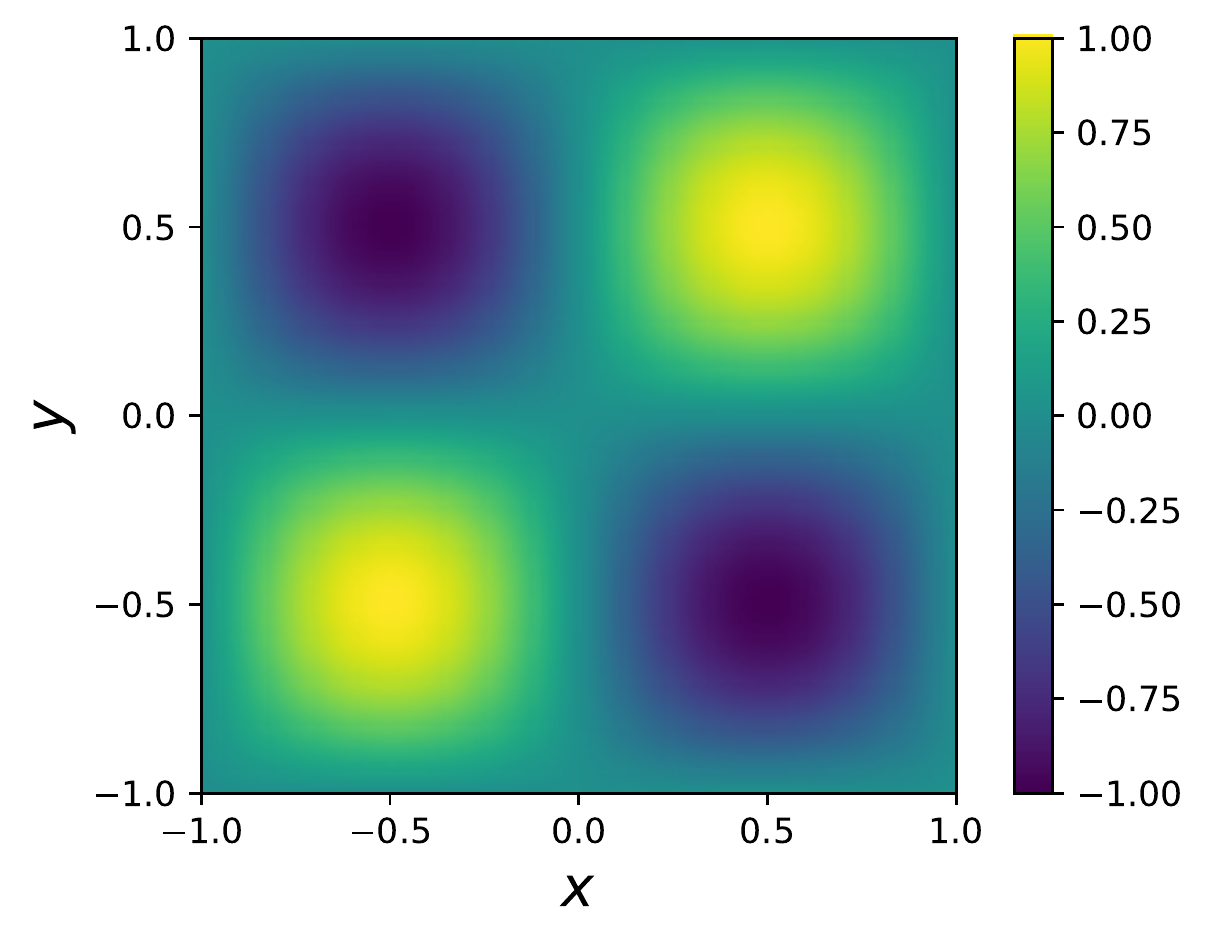}
    }\
     \subfloat
    [prediction, $\sigma_1 =\cdots = \sigma_4 = 0.05$] 
    {
    \label{exp4_pre3}     
    \includegraphics[width=0.3\textwidth]{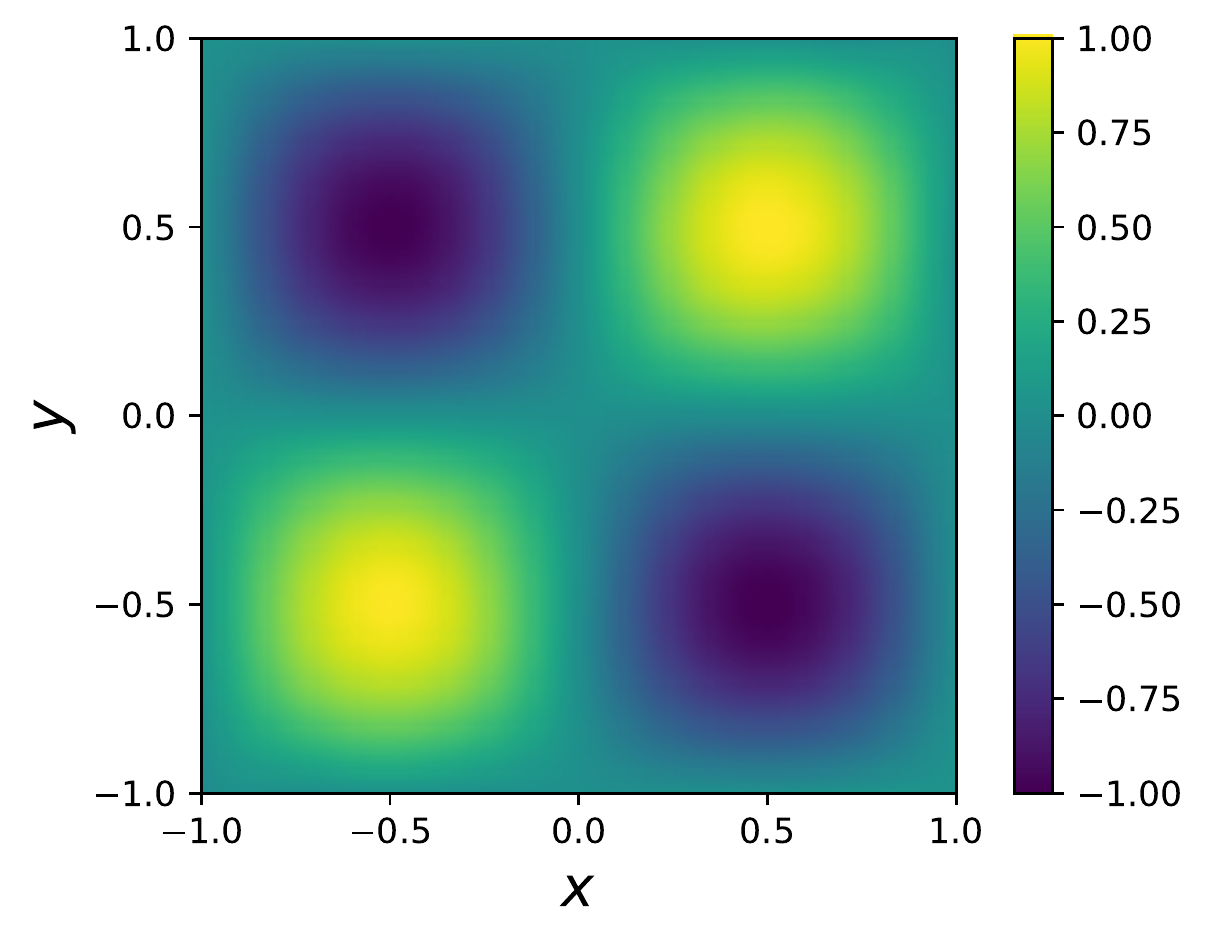}
    }\
    }\\
    
    \makebox[\linewidth][c]{%
  \centering
  \subfloat
  [err., $\sigma_1 =\cdots = \sigma_4 = 0$] 
  {
    \label{exp4_pre4}     
    \includegraphics[width=0.25\textwidth]{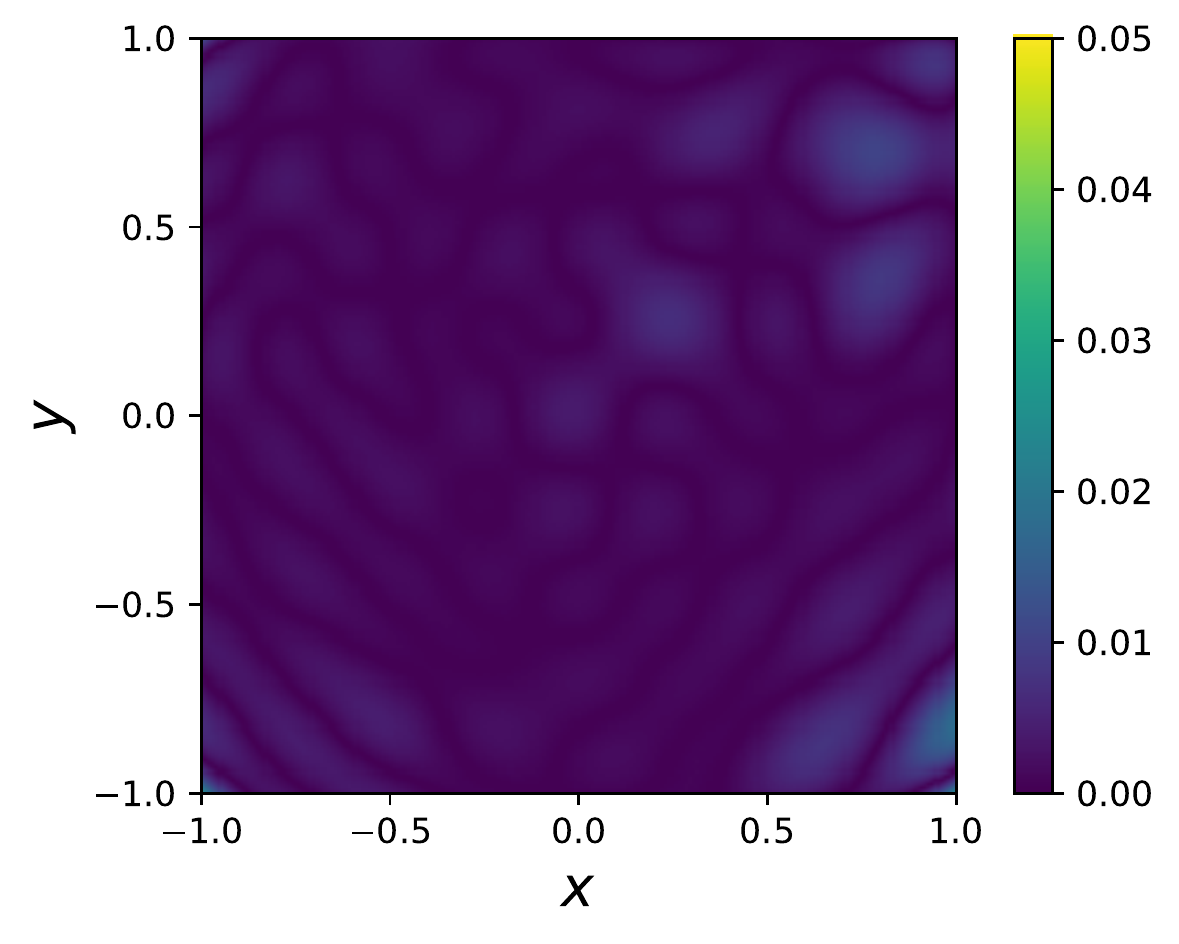}
    }\  
    \subfloat
    [err., $\sigma_1 =\cdots = \sigma_4 = 0.05$] 
    {
    \label{exp4_pre5}     
    \includegraphics[width=0.25\textwidth]{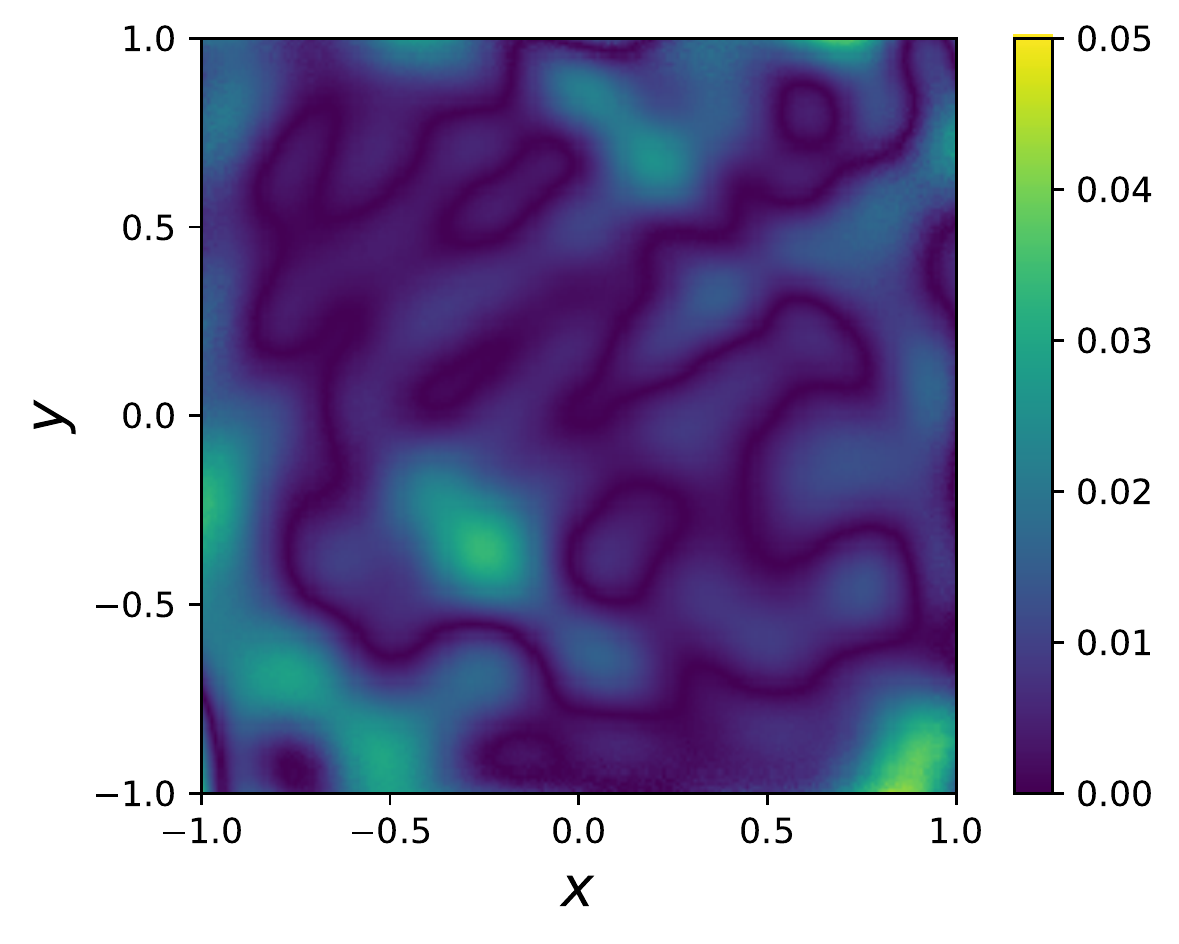}
    }\
    \subfloat
    [std, $\sigma_1 =\cdots = \sigma_4 = 0$] 
    {
    \label{exp4_pre6}     
    \includegraphics[width=0.25\textwidth]{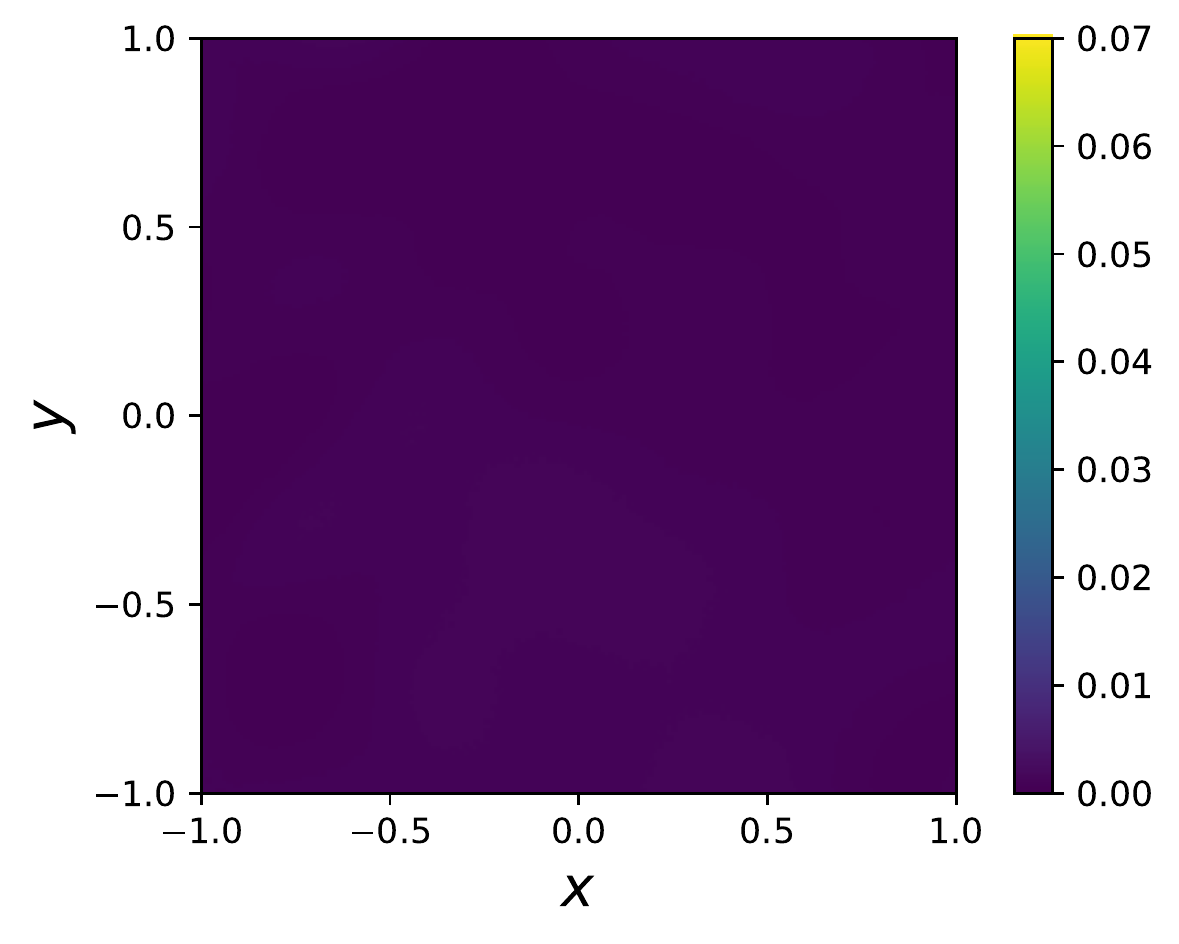}
    }\
    \subfloat
    [std, $\sigma_1 =\cdots = \sigma_4 = 0.05$] 
    {
    \label{exp4_pre7}     
    \includegraphics[width=0.25\textwidth]{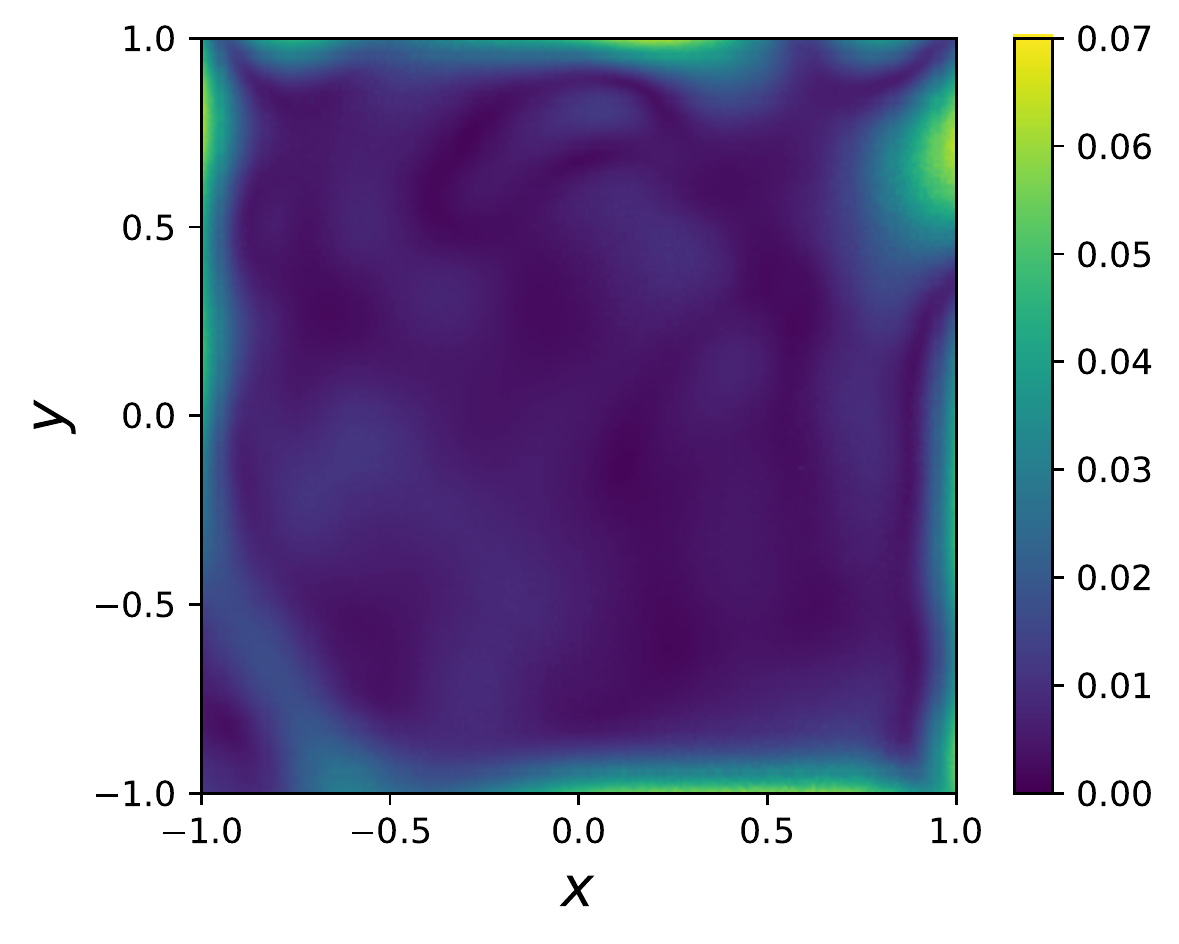}
    }\
    }
    \caption{		
	 Two-dimensional Allen-Cahn equation: errors and standard derivations of the obtained generator $\Tilde{g}$ at $(x,y)$.}	
    \label{exp4_visual}
\end{figure*}

\begin{figure*}[h!]
\makebox[\linewidth][c]{%
  \centering
  \subfloat
  [$u(x=0,y=-1)$] 
  {
     \label{exp4_his1}     
    \includegraphics[width=0.25\textwidth]{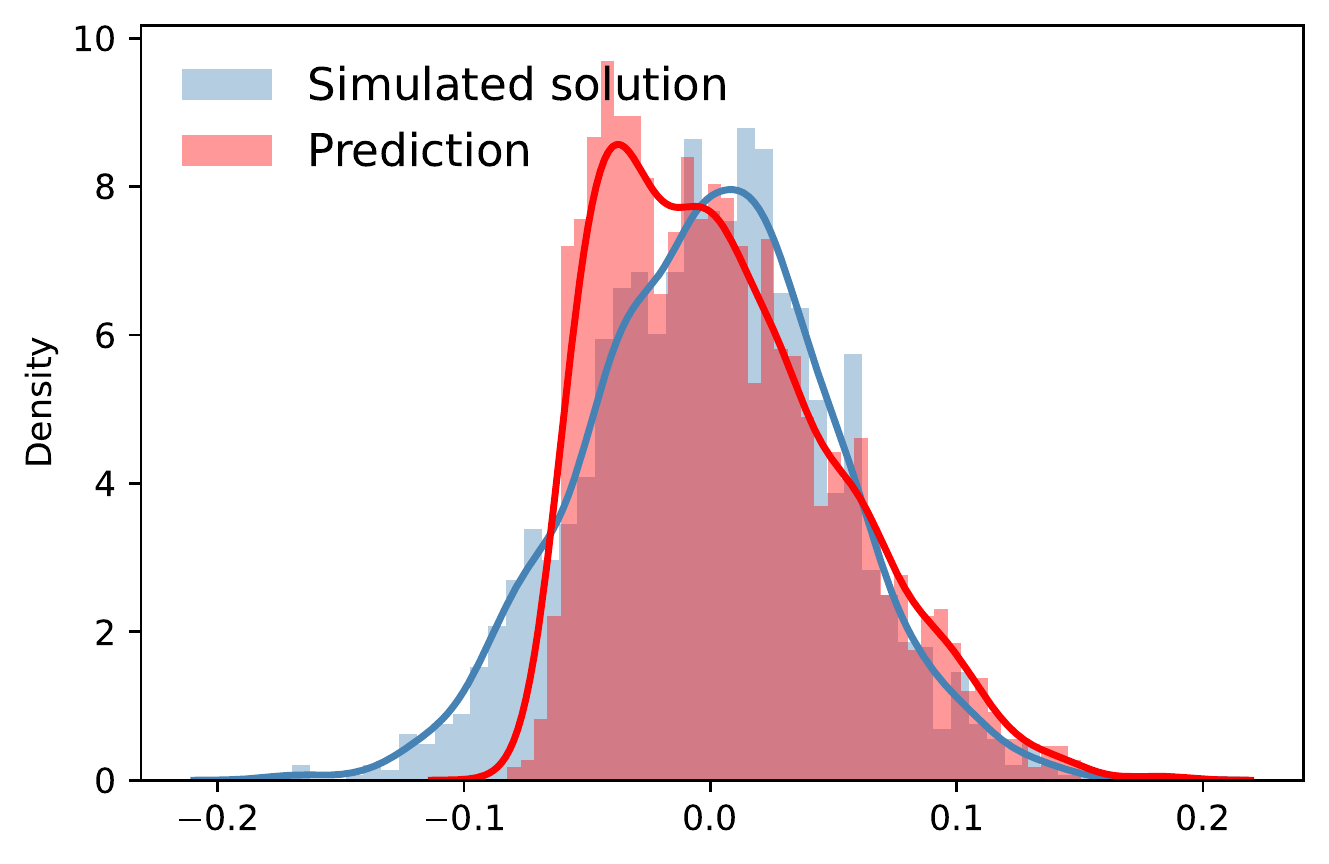}
    }\  
    \subfloat
    [$u(x=-1,y=0)$] 
    {
    \label{exp4_his2}     
    \includegraphics[width=0.25\textwidth]{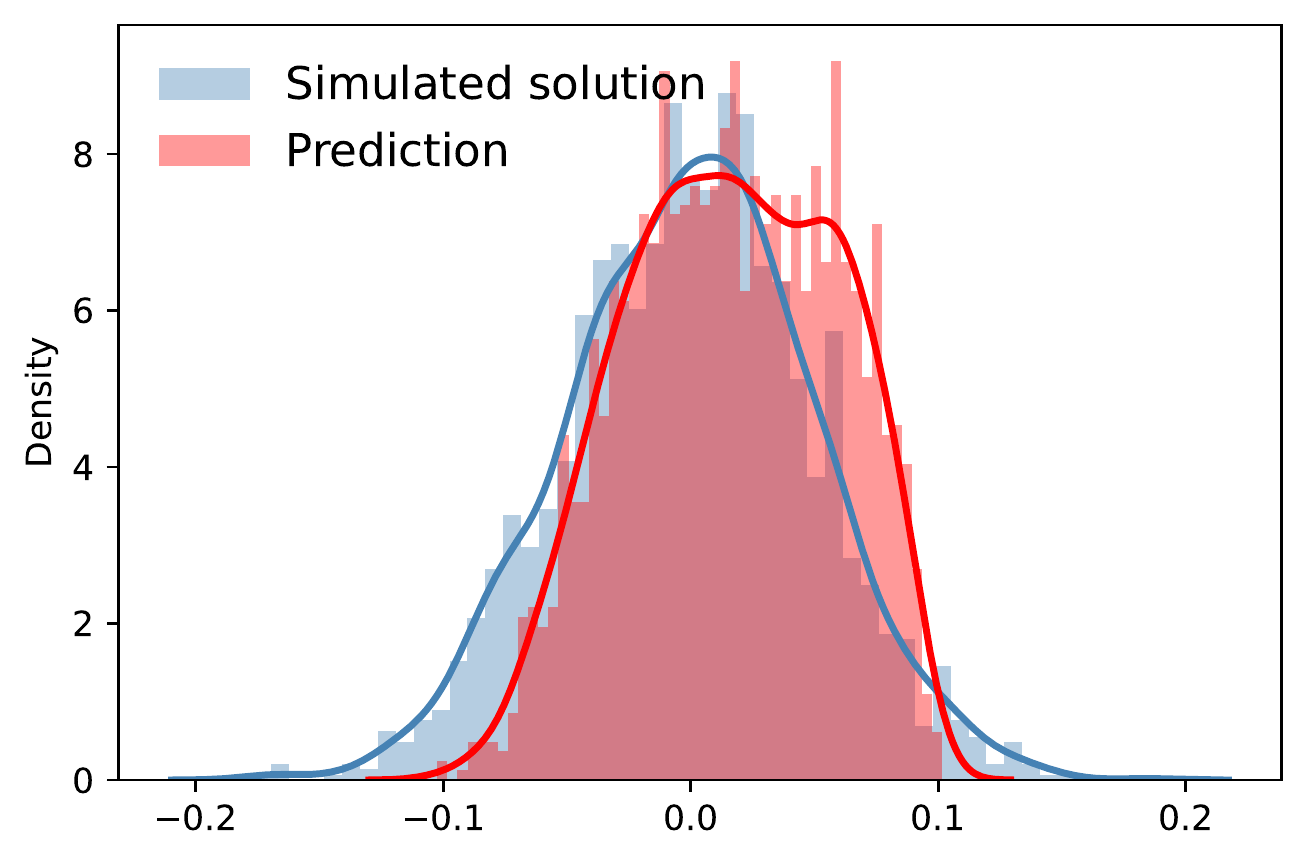}
    }\
    \subfloat
    [$u(x=1,y=0)$] 
    {
    \label{exp4_his3}     
    \includegraphics[width=0.25\textwidth]{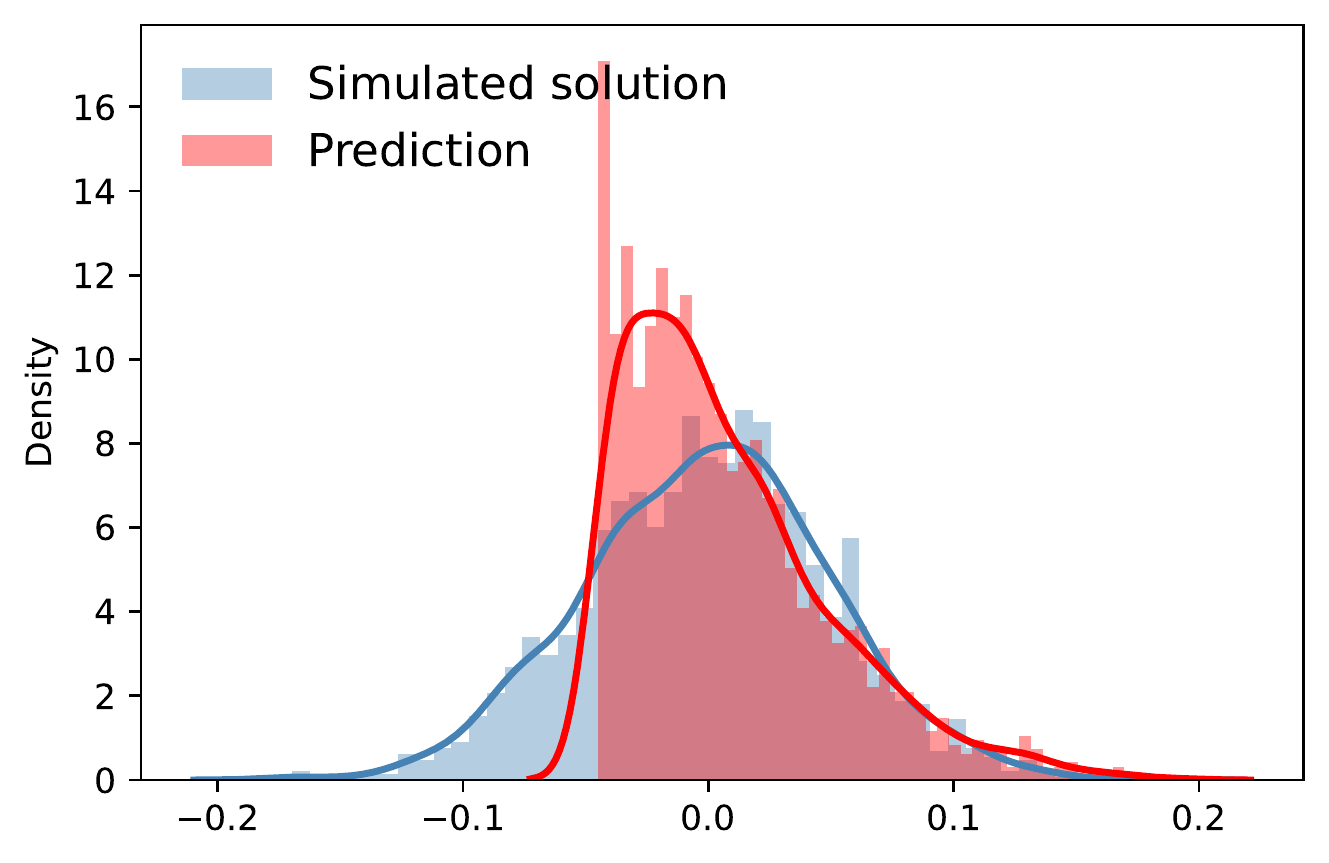}
    }\
    \subfloat
    [$u(x=0,y=1)$] 
    {
    \label{exp4_his4}     
    \includegraphics[width=0.25\textwidth]{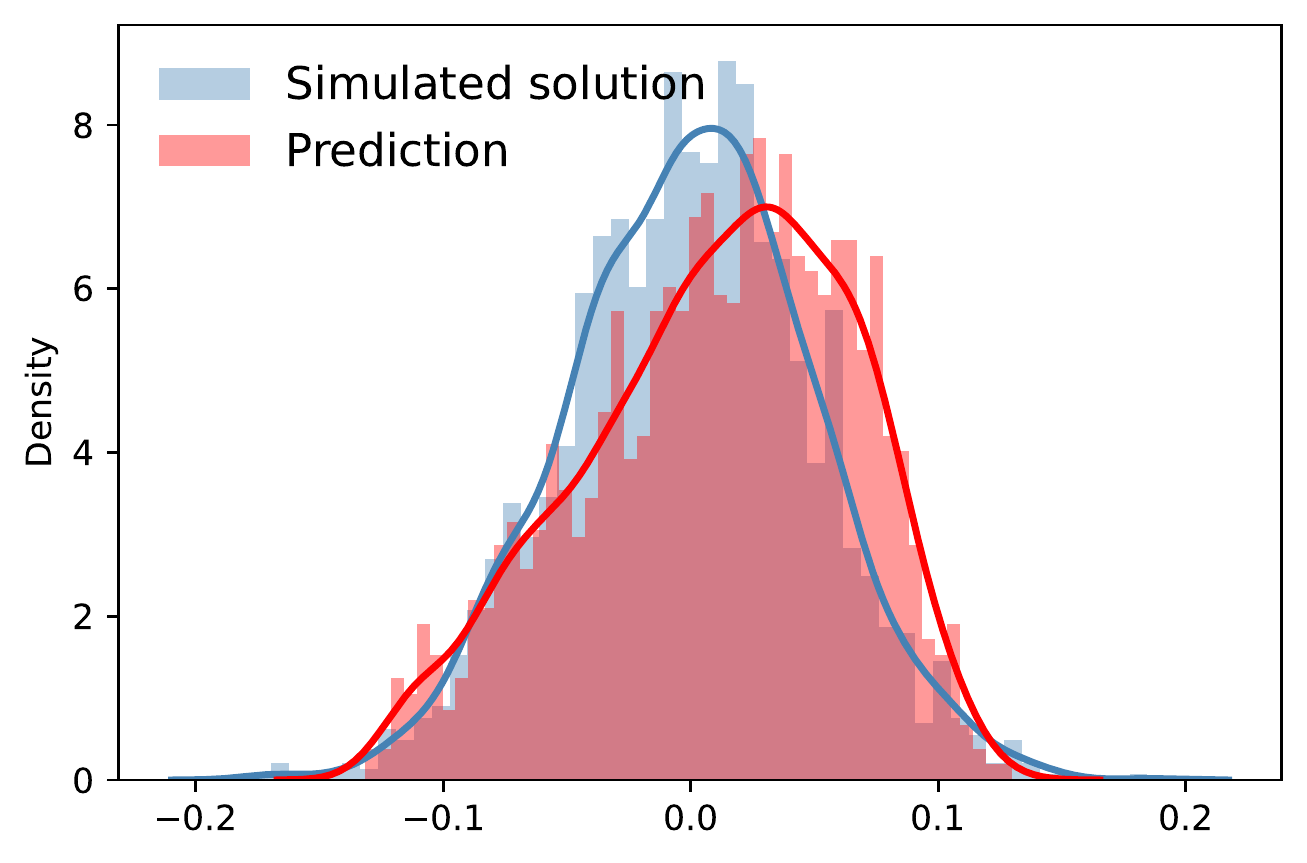}
    }\
    }
    \caption{		
	 Two-dimensional Allen-Cahn equation:
	 the histograms of the generated data and the simulated solution at 
	 $(x=0,t=0)$, $(x=0,y=-1)$, $(x=-1,y=0)$ and  $(x=0,y=1)$ for $\sigma_1 = \cdots \sigma_4 = 0.05$. The predicted mean and standard deviation are 
	 $(\hat{\mu},\hat{\sigma})$ and the simulated mean and standard deviation 
	 are $(\mu_{est},\sigma_{est})$. (a) $(\hat{\mu},\hat{\sigma})=(0.005,0.047)$, $(\mu_{est},\sigma_{est})=(0,0.05)$; (b) $(\hat{\mu},\hat{\sigma}) = (0.016, 0.042)$, $(\mu_{est},\sigma_{est})=(0,0.05)$; (c) $(\hat{\mu},\hat{\sigma}) = (0.007, 0.042)$,
	 $(\mu_{est},\sigma_{est})=(0,0.05)$; (d) $(\hat{\mu},\hat{\sigma}) = 
	 (0.009, 0.054)$, $(\mu_{est},\sigma_{est})=(0,0.05)$.}	
    \label{exp4_visual_his}
\end{figure*}

\section{Conclusion}

We present a class of generative physics-informed neural networks (WGAN-PINNs) that can generate observed uncertainty and propagate with physical laws (partial differential equations). An explicit error bound for the generalization and the convergence of WGAN-PINNs is developed. We show that with sufficient training data and strong (wide and deep) discriminators the (exact) loss of our obtained model converges to the approximation error (the minimal error among all generators in the predefined class). Numerical results verify our theories and demonstrate the extraordinary performance of WGAN-PINNs. 

However, there are still some less rigorous points (limitations) of the proposed model. Theoretical works are urgently needed to enhance the performance of neural networks in solving PDEs problems. Firstly, about the role of PINNs regularization term. We find in the numerical part that PINNs term accelerates the convergence of WGANs. An implicit explanation for the phenomenon lies in that PINNs narrows the function space. Besides the convergence of the model, the most important thing is the role of PINNs term for the quality of uncertainty propagation in the interior domain. As is shown in the experimental part, the uncertainty propagation is less than satisfactory compared with boundary distribution matching, which also reflects the concerns for the uncertainty propagation in the interior. We left it as a future research work to investigate whether the proposed probabilistic model guarantee the quality of uncertainty propagation or how to adjust it for satisfactory uncertainty quantification in more complicated non-Gaussian multi-modal problems. Secondly, The reasonableness of the proposed model with boundary/initial samples and governing equations. As is mentioned in section \ref{proposed_wgan_pinns_model}, the core idea of the proposed model comes from traditional PDEs solvers that solve PDEs with given boundary/initial data and governing equations, where we do not require interior solution data. We can simply extend the model to that with interior solution samples, which contributes to the accuracy of the uncertainty quantification in the interior. We left it as a future work to study the experimental performances of the new and simply extended model. Thirdly, from an optimization perspective. Our theories have an prerequisite that 
the obtained model $\Tilde{g}$ is the optimal solution of the empirical loss. 
In general, generative adversarial networks are especially difficult to train, where we are required to solve a min-max problem. PINNs, including differential operations to the input data, raises more challenges to adversarial networks training. Stagnating at a local optima, gradients vanishing especially for deep models, high order and high dimensional 
differential equations, as well as training instability bring about difficulties in training deep neural networks for solving PDEs problems. All requires further research in the next
step.
Explicit forms of PDEs are required, however, sometimes it is hard to mathematically formulate the differential equations governing systems. More generalized models that can learn equations (functional operators) totally from data sometimes also have wide applications \cite{lu2019deeponet,li2020neural,guo2016convolutional,bhatnagar2019prediction}. Final and the most important one is the generalization with respect to the prediction accuracy in the interior. Our theories derive the generalization of WGAN-PINNs with respect to the loss, but their approximations to solutions in the interior are actually our goals. How to appropriately measure the accuracy or quality of our probabilistic models (with respect to PDEs solutions) is still an open question. For more challenges and difficulties of PINNs related machine learning problems, please refer to \cite{raissi2019physics}.

\section*{Acknowledgments}

This work is supported by Hong Kong Research Grant Council 
GRF 12300519, 17201020,
17300021, C1013-21GF, C7004-21GF and 
N$\underline{\mbox{\hspace{2mm}}}$HKU769\/21. We thank all reviewers for their helpful comments and suggestions. 


\appendix

\section{Group size of groupsort neural networks}

Here, we extend the Theorem 2 and 3 in \cite{tanielian2021approximating} to the following Theorem to show that group size contributes to approximation capabilities of groupsort neural networks theoretically. 
\begin{theorem}
[Approximation of GroupSort neural networks to 1-Lipschitz functions]
For any 1-Lipschitz function $f$ defined on $[-M_u,M_u]^{d}$, there exists $f_{\alpha} \in \mathcal{F}_{\text{GS}}$ with $group_{size}=k$, $D_f=n+1$ and $W_f=k^n$, such that $\|f_{\alpha}-f\|_{L^{\infty}([-M_u,M_u]^{d})} \leq 2M_u \cdot \epsilon$, where $\epsilon = 2\sqrt{d} \cdot  W_f^{-1/d^2}$. Furthermore, let $group_{size}=\sqrt{W_f}$ and $D_f = 3$, then $\|f_{\alpha}-f\|_{L^{\infty}([-M_u,M_u]^{d})} \leq 2M_u \cdot \epsilon$.
\end{theorem}
Note the the error $\epsilon =\mathcal{O}(W_f^{-1/d^2})$ with group size $k=\sqrt{W_f}$ and $\epsilon =\mathcal{O}(2^{-D_f/d^2} \vee W_f^{-1/d^2})$ with $k=2$ theoretically imply that large group size cannot reduce the error to be lower than $\mathcal{O}(W_f^{-1/d^2})$ for deep groupsort neural networks. Although theoretical results recommend to use larger group size in groupsort neural networks, empirical results in \cite{anil2019sorting} show slight improvement of groupsort neural networks with larger group sizes and they exactly outperform ReLU neural networks by a large margin.  Groupsort neural networks work well with group size of 2 and group size is not that essential empirically. Moreover, it is more convenient to fix the group size to be 2 since width $W_f$ must be a multiple of the group size.

\section{Constants in the derived error bound}
There are several constants independent of the model architecture (numbers of training data $m$ and $n$, generator class $\mathcal{G}$ and discriminator class $\mathcal{F}_{\text{GS}}$) in the derived error bound (\ref{bound_emp_loss_1})-(\ref{error_bound_2}) and (\ref{error_bound_3}). In this part, we will discuss them one by one to see whether they are dominant in the error bound.

For the constant $C_1$, recall the proof of Lemma \ref{bound_I1} that
\begin{equation*}
    \begin{split}
        I_1 & \leq \min_{\Tilde{f} \in \mathcal{F}_{GS}} \{ \| \Tilde{f}-f \|_{L^{\infty}([-M_u,M_u]^{d+r})} \cdot \left (\mathbb{P}(\|\mathbf{u}(\mathbf{x})\|_2 \leq M_u) +1 \right )\\
        & \hspace{4em} +  \mathbb{E}_{(\mathbf{x},\mathbf{u})\sim p_{\Gamma}(\mathbf{x},\mathbf{u})} [\Tilde{f}(\mathbf{x},\mathbf{u}) - f(\mathbf{x},\mathbf{u})] \cdot \mathbb{I}(\|\mathbf{u}(\mathbf{x})\|_2 \geq M_u) \}\\
        & \leq 2 \cdot \min_{\Tilde{f} \in \mathcal{F}_{GS}} \{ \| \Tilde{f}-f \|_{L^{\infty}([-M_u,M_u]^{d+r})} \cdot \left (\mathbb{P}(\|\mathbf{u}(\mathbf{x})\|_2 \leq M_u) +1 \right )\\
        & \hspace{4em} + \mathbb{E}_{(\mathbf{x},\mathbf{u})\sim p_{\Gamma}(\mathbf{x},\mathbf{u})} [(3M_u + \|\mathbf{u}(\mathbf{x})\|_2) \cdot \mathbb{I}(\|\mathbf{u}(\mathbf{x})\|_2 \geq M_u)] \}\\
        & \leq C_{11} \cdot M_u \cdot ( 2^{-D_f/(d+r)^2} \vee W_f^{-1/(d+r)^2})  + C_{12} \cdot M_u^{-s+1},
    \end{split}
\end{equation*}
where $C_{11} \approx 8 \sqrt{d+r}$ according to Proposition \ref{approx_dis} and $C_{12} = 4$ by Assumption \ref{decay}, due to the equality that $\int_{0}^{+\infty} \mathbb{P}(|X| \geq x) \hspace{0.2em} dx = \mathbb{E}[|X|]$. Therefore, let $s=2$ and $M_u = (C_{12}/C_{11})^{1/2} \cdot ( 2^{-D_f/2(d+r)^2} \vee W_f^{-1/2(d+r)^2})$ which satisfies Assumption \ref{large_capacity}, we have
$$
I_1 \leq C_1 \cdot ( 2^{-D_f/2(d+r)^2} \vee W_f^{-1/2(d+r)^2}),
$$
where $C_1=\sqrt{C_{11}\cdot C_{12}} \approx 4\sqrt{2} \cdot (d+r)^{1/4}$.

For the constant $C_2$ in Lemma \ref{bound_I2}, we have $C_2 = 2\sqrt{d+r} \cdot C^{\prime}$, where we assume that $M_x \leq ((W_g+1)\cdot M)$ without loss of generality and $C^{\prime}$ is a constant defined in Proposition \ref{convg_W1}. After checking Proposition 3.1 in \cite{NEURIPS2020_2000f632}, Theorem 2.3 and Theorem 3.1 in \cite{lei2020convergence} with $(p,q)=(1,3)$, we have $C^{\prime} \approx \frac{64}{3} \sqrt[3]{M_3} $  in Proposition \ref{convg_W1}, when the dimension $d \ll log n$ and $M_3$ is the 3-moment for $\nu$. Therefore, we have $C_2 \approx \frac{128}{3} \sqrt{d+r}$. Similarly, $C_3 \approx \frac{128}{3} \sqrt[3]{M_3}$, where $M_3$ is the 3-moment for the solution data $(\mathbf{x},\mathbf{u}(\mathbf{x}))$ in Assumption \ref{cong_sol_dis}. As is discussed in \ref{appendix_proof}, $C_5 \approx 24 \sqrt{d+r}$ when $\log m \gg 1$ while $C_6 \approx 24$ when $\log n \gg 1$.

Both constants $C_4$ and $\Tilde{C}$ depend on the given PDE itself. In the remark right after Lemma \ref{bound_I4}, we provide a simple example to illustrate values of constants $C_4$ and $\Tilde{C}$. For the given PDE:
\begin{equation*}
        u_{xx} - u^2 u_x = b(x), \hspace{1em} x \in [-1,1],
\end{equation*}
with the pre-defined generator class (\ref{generator_class}), we have $\Tilde{C}=2$ and $C_4 \approx 3$. Theoretically, considering the uniform upper bound, $\Tilde{C}$ is close to the order of the PDE and $C_4$ is proportional to the number of terms in the given PDEs. For more details, please refer to the remark of Lemma \ref{bound_I4}.

\section{Proof of Theorem \ref{improved_generalization_bound}}
\label{appendix_proof}

The proof for Theorem \ref{improved_generalization_bound} mainly consists of deriving new error bounds for $I_2$ and $I_3$ due to the finite capacity of $\mathcal{F}_{\text{GS}}$. To fully understand the details, some definitions (e.g., covering balls, covering numbers, Rademacher complexity, pseudo-dimension, etc.) and theorems (e.g., Dudley's theorem, etc.) in deep learning theory are required. Please refer to \cite{ma2021lecturenotes} for details of these basic tools. 

We use $\mathcal{R}_{\mathcal{S}}(\mathcal{F})$ and $\mathcal{R}_{n}(\mathcal{F})$ to denote the empirical Rademacher complexity of the function space $\mathcal{F}$ on the independent and identically distributed data $\mathcal{S}=\{X_1, \cdots, X_n\}$ and the corresponding Rademacher complexity with $\mathcal{R}_{n}(\mathcal{F})=\mathbb{E}_{\mathcal{S}} \mathcal{R}_{\mathcal{S}}(\mathcal{F})$ respectively. Here, the empirical Rademacher complexity is defined as 
$$
\mathcal{R}_{\mathcal{S}}(\mathcal{F}):=\mathbb{E}_{\epsilon} \sup _{f \in \mathcal{F}} \frac{1}{n} \sum_{i=1}^{n} \epsilon_{i} f\left(X_{i}\right),
$$
where $\epsilon=\left(\epsilon_{1}, \ldots, \epsilon_{n}\right)$ are i.i.d. Rademacher random variables. The next three lemmas show that the generalization error is bounded by the Rademacher complexity of the function class and then is further controlled by the pseudo-dimension.

\begin{lemma}
[Translating the generalization error to the Rademacher complexity, Theorem 4.13 in \cite{ma2021lecturenotes}] Let $\mathcal{S}=\left \{X_1, \cdots, X_n\right \}$ be a set of $n$ i.i.d. samples, then we have
$$
\mathbb{E}_{\mathcal{S}} \left [ \sup_{f \in \mathcal{F}} \left [ \frac{1}{n} \sum_{i=1}^{n} f(X_i) - \mathbb{E}_{X} \left [f(X) \right ]  \right ] \right ]\leq 2 \mathcal{R}_{n}(\mathcal{F}).
$$
\end{lemma}

\begin{lemma}
[Translating the Rademacher complexity to the covering number, Dudley’s Theorem, Theorem 4.26 in \cite{ma2021lecturenotes}] If $\mathcal{F}$ is a function class from $\mathcal{X} \to \mathbb{R}$, where $\mathcal{S}$ is sampled from $\mathcal{X}$. For $B>0$, we assume that $\max_{i} |f(X_i)| \leq B$ for all $f \in \mathcal{F}$. Then,
\begin{equation}
\label{bound_rademacher_complexity}
    R_{S}(\mathcal{F}) \leq 12 \int_{0}^{B} \sqrt{\frac{\log N\left(\epsilon, \mathcal{F}, \|\cdot\|_{2,n} \right)}{n}} d \epsilon,
\end{equation}
where $\|f\|_n = \sqrt{\frac{1}{n}\sum_{i=1}^{2,n}\left(f(X_i)\right)^2}$ and $N \left(\epsilon, \mathcal{F}, \|\cdot\|_{2,n} \right)$ is the covering number for $\mathcal{F}$ with radius $\epsilon$ and metric $\|\cdot\|_{2,n}$.
\end{lemma}

\begin{lemma}
[Translating the covering number to the pseudo-dimension \cite{anthony1999neural}] For $B>0$, assume that  for all $f \in \mathcal{F}$, we have $\max_{i} |f(X_i)| \leq B$. Define the metric $\|\cdot\|_{\infty,n}$ as $\|f\|_{\infty,n}=\max_{i\in \{1, \cdots, n\}}|f(X_i)|$. Then,
\begin{equation}
\label{bound_cover_number_1}
    N\left(\epsilon,\mathcal{F}, \|\cdot\|_{\infty,n}\right) \leq\left(\frac{2 e B \cdot n}{\epsilon \cdot \operatorname{Pdim}(\mathcal{F})}\right)^{\operatorname{Pdim}(\mathcal{F})}
\end{equation}
for any $\epsilon>0$, where $\operatorname{Pdim}(\mathcal{F})$ denotes the pseudo-dimension of $\mathcal{F}$ and $N \left(\epsilon, \mathcal{F}, \|\cdot\|_{\infty,n} \right)$ is the covering number for $\mathcal{F}$ with radius $\epsilon$ and metric $\|\cdot\|_{\infty,n}$.
\end{lemma}

Obviously $N\left(\epsilon,\mathcal{F}, \|\cdot\|_{2,n}\right) \leq N\left(\epsilon,\mathcal{F}, \|\cdot\|_{\infty,n}\right)$, since $\|f\|_{2,n} \leq \|f\|_{\infty,n}$ for all $f \in \mathcal{F}$. Therefore,
\begin{equation}
\label{bound_cover_number_2}
    \log N\left(\epsilon,\mathcal{F}, \|\cdot\|_{2,n}\right) \leq \operatorname{Pdim}(\mathcal{F}) \cdot \log \left( \frac{2eB \cdot n}{\epsilon \cdot \operatorname{Pdim}(\mathcal{F})} \right) \leq \operatorname{Pdim}(\mathcal{F}) \cdot \log \left(\frac{2e B \cdot n}{\epsilon} \right),
\end{equation}
where the last inequality comes from $\operatorname{Pdim}(\mathcal{F}) \geq 1$. Combining with (\ref{bound_rademacher_complexity}) and (\ref{bound_cover_number_2}), we have
\begin{equation*}
    \begin{split}
        R_{S}(\mathcal{F}) & \leq 12 \int_{0}^{B} \sqrt{\frac{\operatorname{Pdim}(\mathcal{F})}{n} \cdot \log \left(\frac{2e B \cdot n}{\epsilon} \right) } d \epsilon \leq 12B \int_{0}^{1}\sqrt{\frac{\operatorname{Pdim}(\mathcal{F})}{n} \cdot \log \left(\frac{2e \cdot n}{\epsilon} \right) } d \epsilon\\
        & \leq 12B \int_{0}^{1}\sqrt{\frac{\operatorname{Pdim}(\mathcal{F})}{n}} \cdot \left(\sqrt{\log(2e) + \log n} + \sqrt{\log\left(\frac{1}{\epsilon}\right)}\right) d \epsilon\\
        & \leq 12B \int_{0}^{1}\sqrt{\frac{\operatorname{Pdim}(\mathcal{F})}{n}} \cdot \left(\sqrt{\log(2e) + \log n}  + 1 \right) d \epsilon\\
        & \leq 12B C_{51} \cdot \sqrt{\frac{\operatorname{Pdim}(\mathcal{F}) \cdot \log n}{n}},
    \end{split}
\end{equation*}
where the constant $C_{51}$ satisfies $\sqrt{\log(2e) + \log n}  + 1 \leq C_{51} \cdot \sqrt{\log n} $ and $C_{51} \approx 1$ if $\log n \gg 1$. Equipped with those results, we have
\begin{equation*}
    \begin{split}
        \mathbb{E}I_2 & = \mathbb{E} \max_{{\Tilde{f} \in \mathcal{F}_{GS}}} \{ \mathbb{E}_{\mathbf{x},\mathbf{z} \sim p_{\Gamma}(\mathbf{x}),p(\mathbf{z})}\Tilde{f}(\mathbf{x},\Tilde{g}(\mathbf{x},\mathbf{z})) - \hat{\mathbb{E}}^{m}_{\mathbf{x},\mathbf{z} \sim p_{\Gamma}(\mathbf{x}),p(\mathbf{z})}\Tilde{f}(\mathbf{x},\Tilde{g}(\mathbf{x},\mathbf{z})) \} \leq 2\mathbb{E} \mathcal{R}_{\mathcal{S}_m}(\mathcal{F})\\
        & \leq 24C_{51 }\sqrt{\frac{\operatorname{Pdim}(\mathcal{F}) \cdot \log m}{m}} \cdot \left (\sqrt{d+r} \cdot \left((W_g +1) \cdot M \right)\right),
    \end{split}
\end{equation*}
where $\mathcal{S}_m$ is the set of i.i.d. samples from the generated distribution and the inequality holds with the assumption that $M_x \leq (W_g + 1) \cdot M$, because $|\Tilde{f}(\mathbf{x}, g(\mathbf{x},\mathbf{z}))|=\|\Tilde{f}(\mathbf{x}, g(\mathbf{x},\mathbf{z})) - \Tilde{f}(\mathbf{0},\mathbf{0})\| \leq \|(\mathbf{x}, g(\mathbf{x},\mathbf{z}))\|_2$ for all $g$, $\mathbf{x}$ and $\mathbf{z}$. By Markov-Inequality, with probability of at least $1-m^{-1/4}$,
\begin{eqnarray}
    I_2 & \leq & 24C_{51 } \cdot \left (\sqrt{d+r} \cdot \left((W_g +1) \cdot M \right )\right) \cdot \sqrt{\operatorname{Pdim}(\mathcal{F}) \cdot \log m} \cdot m^{-1/4}\\
    & = & C_5 \cdot  \left ((W_g +1) \cdot M \right ) \cdot \sqrt{\operatorname{Pdim}(\mathcal{F}) \cdot \log m} \cdot m^{-1/4},
\end{eqnarray}
where $C_5 = 24 C_{51} \cdot \sqrt{d+r} \approx 24 \sqrt{d+r}$ if $\log m \gg 1$. Similarly, we derive the error bound for $I_3$ that
\begin{equation*}
    \begin{split}
        \mathbb{E}I_3 & = \mathbb{E} \max_{{\Tilde{f} \in \mathcal{F}_{GS}}} \{ \hat{\mathbb{E}}^{n}_{(\mathbf{x},\mathbf{u})\sim p_{\Gamma}(\mathbf{x},\mathbf{u})} \Tilde{f}(\mathbf{x},\mathbf{u}) - \mathbb{E}_{(\mathbf{x},\mathbf{u})\sim p_{\Gamma}(\mathbf{x},\mathbf{u})} \Tilde{f}(\mathbf{x},\mathbf{u}) \} \leq 2 \mathbb{E} \mathcal{R}_{\mathcal{S}_n}(\mathcal{F})\\
        & \leq 24C_{51 }\sqrt{\frac{\operatorname{Pdim}(\mathcal{F}) \cdot \log n}{n}} \cdot F\left(p_{\Gamma}(\mathbf{x},\mathbf{u}),n\right),
    \end{split}
\end{equation*}
where $\mathcal{S}_n$ is the set of i.i.d. samples from the solution $(\mathbf{x},\mathbf{u})$ and $F\left(p_{\Gamma}(\mathbf{x},\mathbf{u}),n\right)$ is defined in Definition \ref{def_expec_max}.   By Markov-Inequality, with probability of at least $1-n^{-1/4}$,
$$
I_3 \leq 24C_{51 }  \cdot F\left(p_{\Gamma}(\mathbf{x},\mathbf{u}),n\right) \cdot \sqrt{\operatorname{Pdim}(\mathcal{F}) \cdot \log n} \cdot n^{-1/4} = C_6  \cdot F\left(p_{\Gamma}(\mathbf{x},\mathbf{u}),n\right) \cdot \sqrt{\operatorname{Pdim}(\mathcal{F}) \cdot \log n} \cdot n^{-1/4},
$$
where $C_6 =24C_{51} \approx 24$ if $\log n \gg 1$.

\section{Additional experiments}
In this part, we further test the performance of the proposed model when the noise is large enough that it dominates the solution data. We utilize the same experimental setting as that of the pedagogical example in section \ref{pedagogical_example}. Here, we enlarge the noise level to be $\sigma_1=\sigma_2=0.3$ and $\sigma_1=\sigma_2=0.5$ where different numbers of training data $m=n$ are adopted. The relative errors and visualized results are displayed in Figure \ref{exp1_visual_pre_appendix}.

When the noise level is up to $0.3$ and $0.5$, the noise dominates the solution data as is shown in Figure \ref{exp1_visual_pre_appendix}. The prediction stays away from the exact deterministic solution if the numbers of training data $m=n$ are small because WGAN cannot catch the boundary uncertainty well, while the relative error decreases if we collect more boundary data, compatible with our theoretical analysis. The results further raise the confidence that the proposed model can detect the uncertainty and approximate the exact solution, even though the noise is dominant, if provided sufficient training data.

\begin{figure*}[h!]
\makebox[\linewidth][c]{
  \centering
  \subfloat
  [$\sigma_1 = \sigma_2 = 0.3$, $m=n=20$] 
  {
     \label{exp1_pre_03_20}     
    \includegraphics[width=0.5\textwidth]{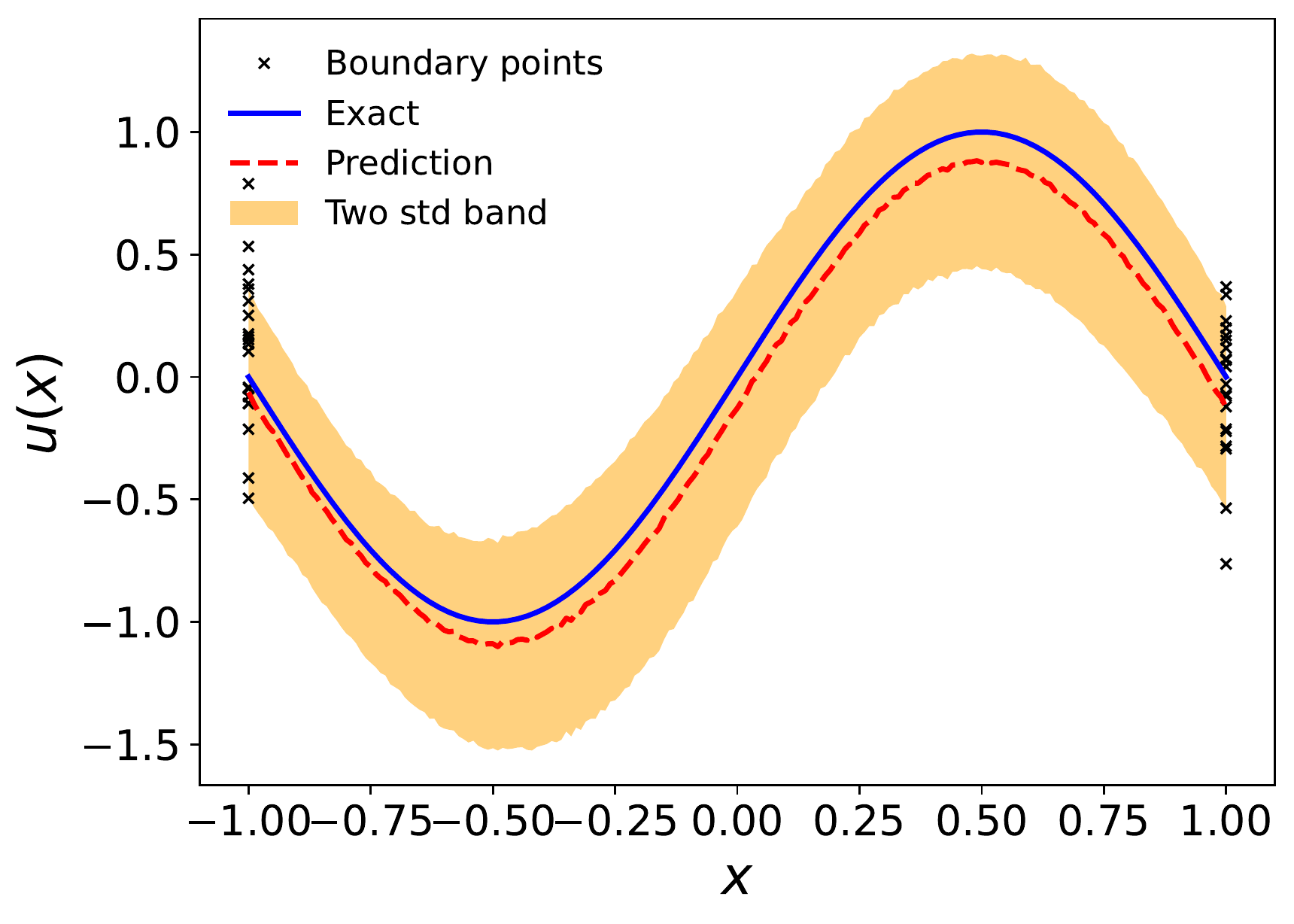}
    }\  
    \subfloat
    [$\sigma_1 = \sigma_2 = 0.3$, $m=n=50$] 
    {
    \label{exp1_pre_03_50}     
    \includegraphics[width=0.5\textwidth]{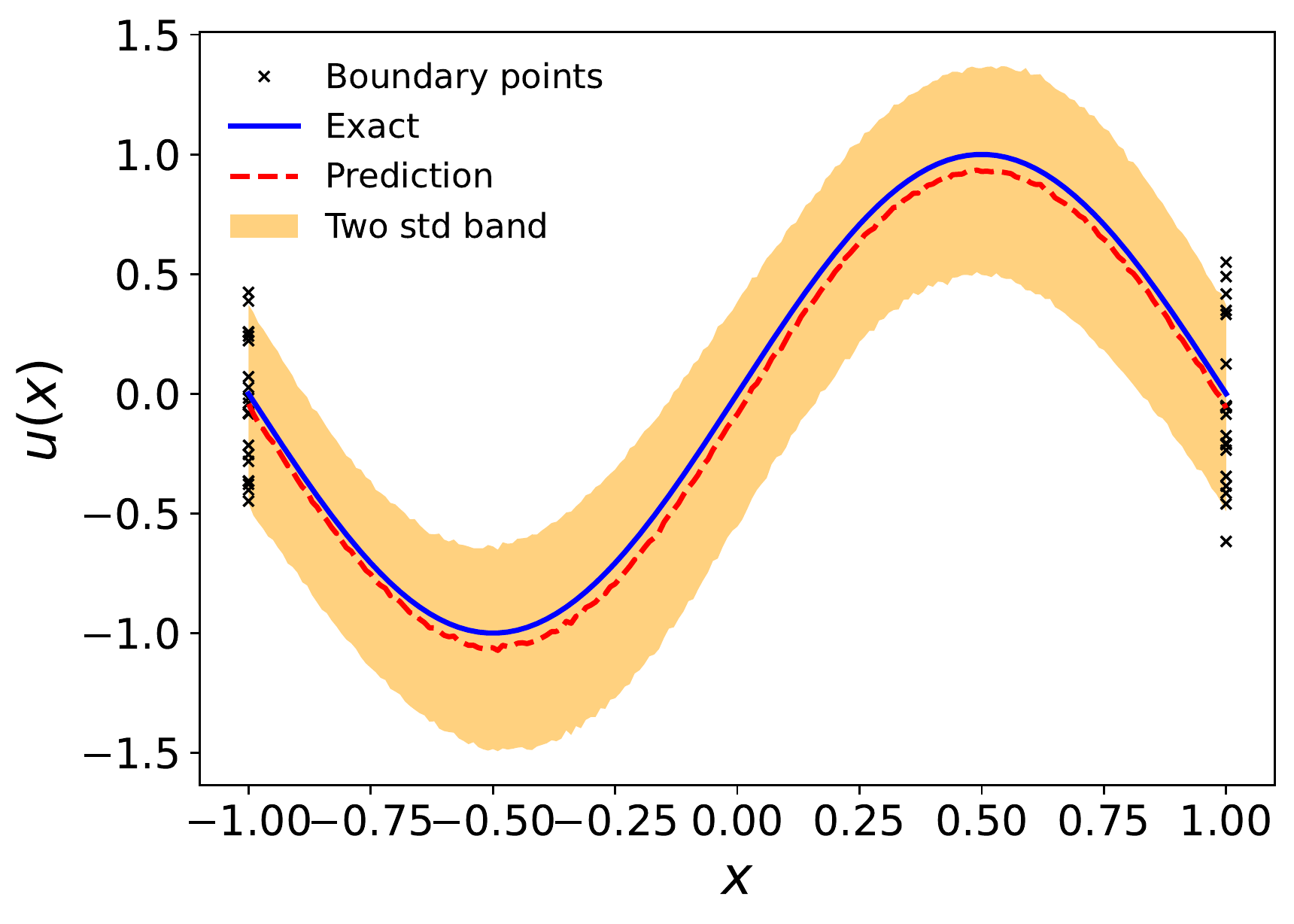}
    }\ 
    %
		}
    \makebox[\linewidth][c]{%
  \centering
  \subfloat
  [$\sigma_1 = \sigma_2 = 0.3$, $m=n=80$] 
  {
     \label{exp1_pre_03_80}     
    \includegraphics[width=0.5\textwidth]{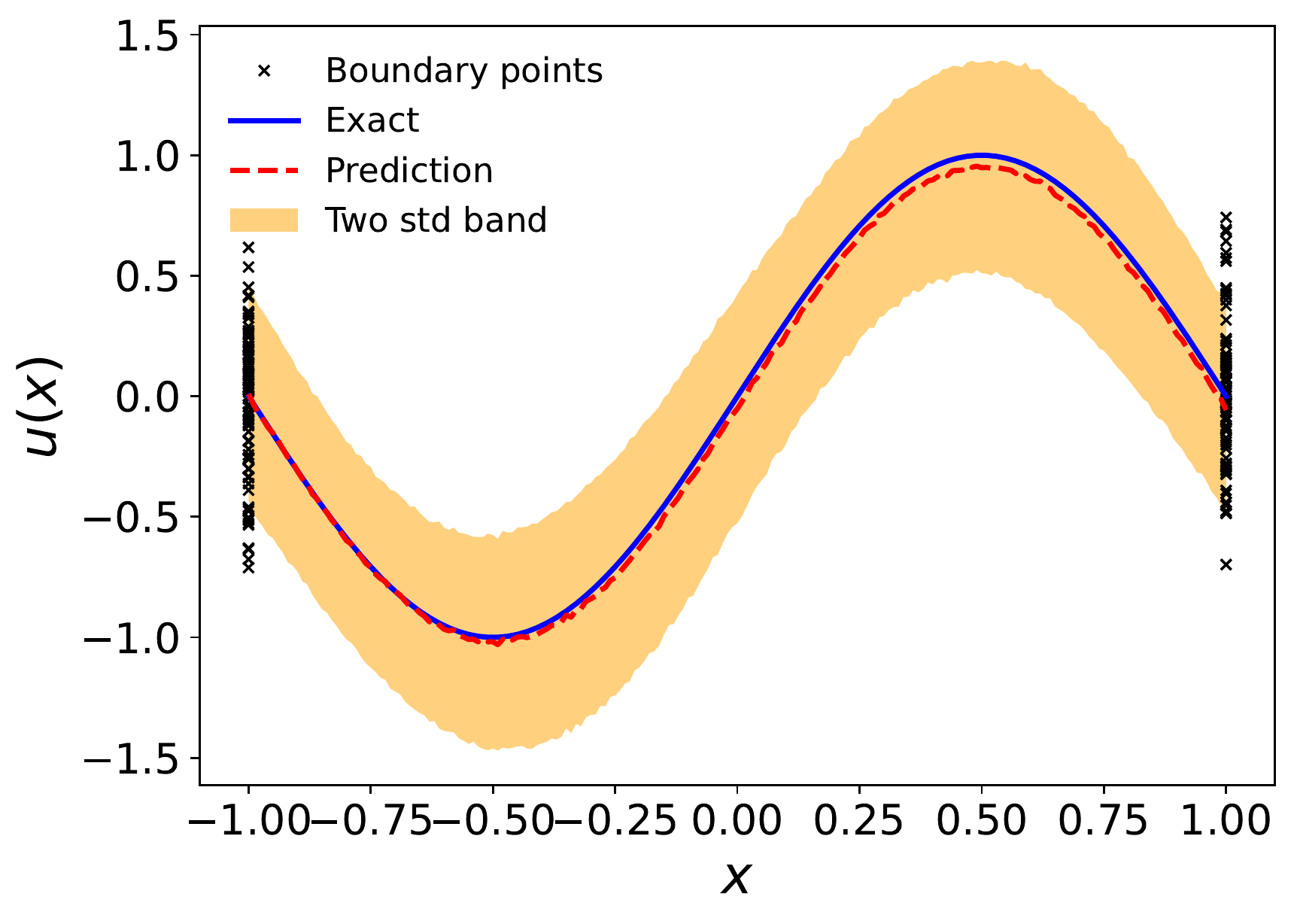}
    }\  
    \subfloat
    [$\sigma_1 = \sigma_2 = 0.3$, $m=n=100$] 
    {
    \label{exp1_pre_03_100}     
    \includegraphics[width=0.5\textwidth]{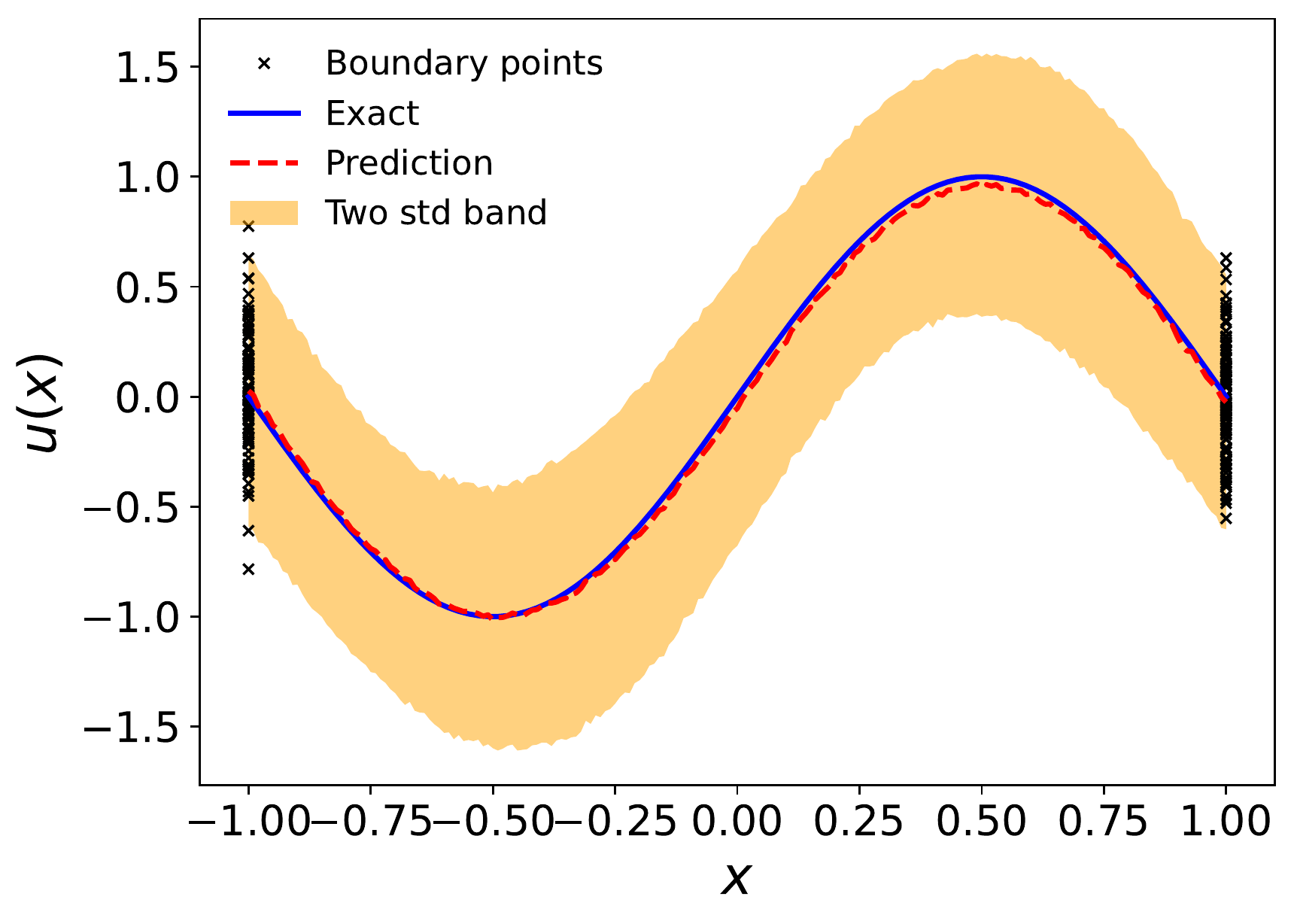}
    }\
    }
       \makebox[\linewidth][c]{%
  \centering
  \subfloat
  [$\sigma_1 = \sigma_2 = 0.5$, $m=n=100$] 
  {
     \label{exp1_pre_05_100}     
    \includegraphics[width=0.5\textwidth]{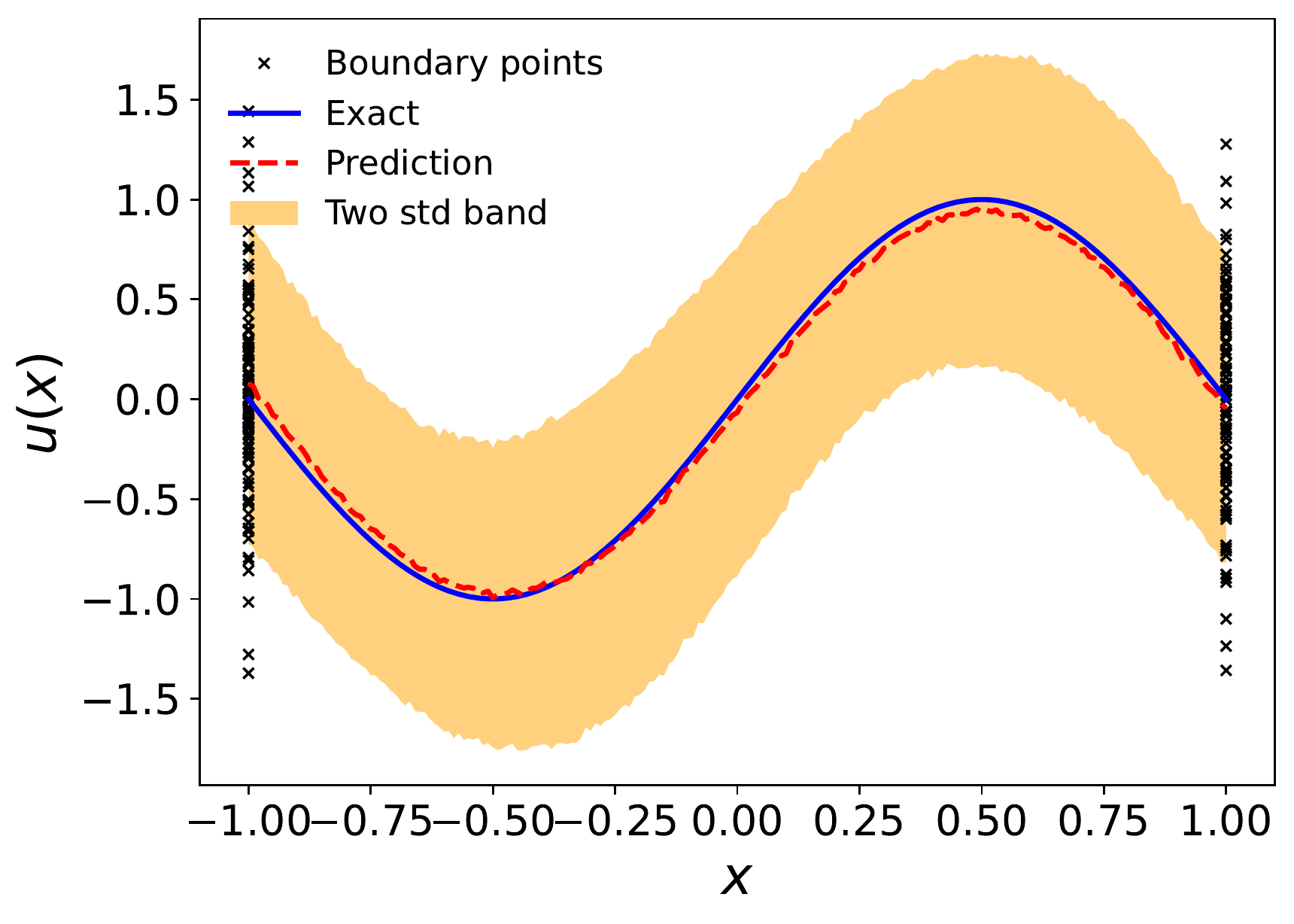}
    }\  
    \subfloat
    [$\sigma_1 = \sigma_2 = 0.5$, $m=n=160$] 
    {
    \label{exp1_pre_05_160}     
    \includegraphics[width=0.5\textwidth]{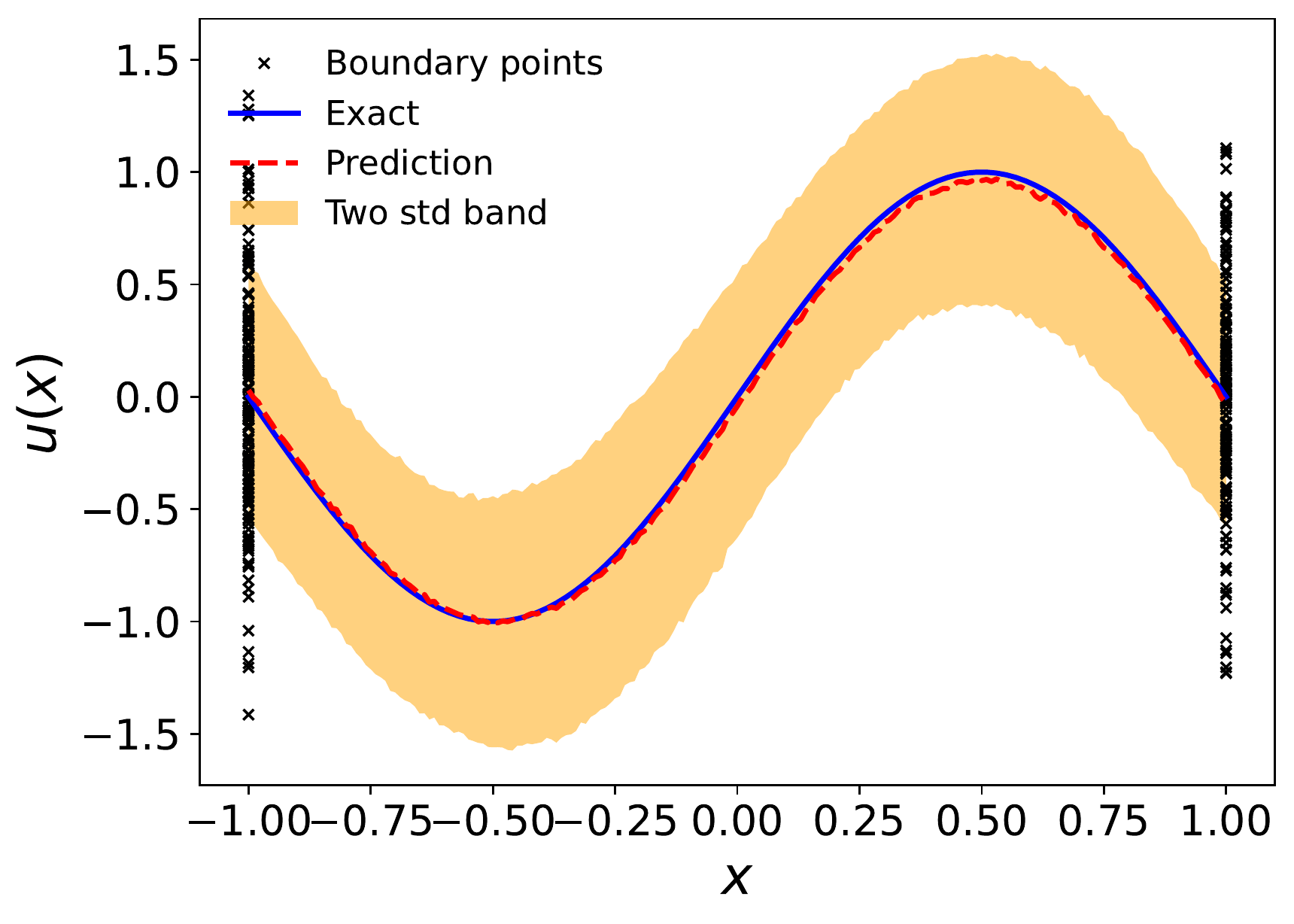}
    }\
    }
    \caption{		
The mean, the lower bound and upper bound of 
PDE solutions $p_{\Tilde{g}}(u)$ given $x$.
(a) $\mathcal{E}=1.55\times 10^{-1}$; (b) $\mathcal{E}=8.72\times 10^{-2}$; (c) $\mathcal{E}=5.44\times 10^{-2}$; (d) $\mathcal{E}=3.99 \times 10^{-2}$; (e) $\mathcal{E}=6.78 \times 10^{-2}$; (f) $\mathcal{E}=4.43 \times 10^{-2}$. Here exact PDE solution (the blue line) refers to PDE solution without uncertainty.}	
    \label{exp1_visual_pre_appendix}
\end{figure*}

\clearpage
\bibliographystyle{plain}
\bibliography{main}

\begin{thebibliography}{10}

\bibitem{anil2019sorting}
Cem Anil, James Lucas, and Roger Grosse.
\newblock Sorting out lipschitz function approximation.
\newblock In {\em International Conference on Machine Learning}, pages
  291--301. PMLR, 2019.

\bibitem{anthony1999neural}
Martin Anthony, Peter~L Bartlett, Peter~L Bartlett, et~al.
\newblock {\em Neural network learning: Theoretical foundations}, volume~9.
\newblock cambridge university press Cambridge, 1999.

\bibitem{arjovsky2017towards}
Martin Arjovsky and L{\'e}on Bottou.
\newblock Towards principled methods for training generative adversarial
  networks.
\newblock {\em arXiv preprint arXiv:1701.04862}, 2017.

\bibitem{arjovsky2017wasserstein}
Martin Arjovsky, Soumith Chintala, and L{\'e}on Bottou.
\newblock {W}asserstein generative adversarial networks.
\newblock In {\em Proceedings of the 34th International Conference on Machine
  Learning}, volume~70 of {\em Proceedings of Machine Learning Research}, pages
  214--223. PMLR, 06--11 Aug 2017.

\bibitem{arora2017generalization}
Sanjeev Arora, Rong Ge, Yingyu Liang, Tengyu Ma, and Yi~Zhang.
\newblock Generalization and equilibrium in generative adversarial nets (gans).
\newblock In {\em International Conference on Machine Learning}, pages
  224--232. PMLR, 2017.

\bibitem{arora2018gans}
Sanjeev Arora, Andrej Risteski, and Yi~Zhang.
\newblock Do gans learn the distribution? some theory and empirics.
\newblock In {\em International Conference on Learning Representations}, 2018.

\bibitem{bahdanau2014neural}
Dzmitry Bahdanau, Kyunghyun Cho, and Yoshua Bengio.
\newblock Neural machine translation by jointly learning to align and
  translate.
\newblock {\em arXiv preprint arXiv:1409.0473}, 2014.

\bibitem{bai2018approximability}
Yu~Bai, Tengyu Ma, and Andrej Risteski.
\newblock Approximability of discriminators implies diversity in {GAN}s.
\newblock In {\em International Conference on Learning Representations}, 2019.

\bibitem{barth2011multi}
Andrea Barth, Christoph Schwab, and Nathaniel Zollinger.
\newblock Multi-level monte carlo finite element method for elliptic pdes with
  stochastic coefficients.
\newblock {\em Numerische Mathematik}, 119(1):123--161, 2011.

\bibitem{baydin2018automatic}
Atilim~Gunes Baydin, Barak~A Pearlmutter, Alexey~Andreyevich Radul, and
  Jeffrey~Mark Siskind.
\newblock Automatic differentiation in machine learning: a survey.
\newblock {\em Journal of machine learning research}, 18, 2018.

\bibitem{bhatnagar2019prediction}
Saakaar Bhatnagar, Yaser Afshar, Shaowu Pan, Karthik Duraisamy, and Shailendra
  Kaushik.
\newblock Prediction of aerodynamic flow fields using convolutional neural
  networks.
\newblock {\em Computational Mechanics}, 64(2):525--545, 2019.

\bibitem{bilionis2016probabilistic}
Ilias Bilionis.
\newblock Probabilistic solvers for partial differential equations.
\newblock {\em arXiv preprint arXiv:1607.03526}, 2016.

\bibitem{bjorck1971iterative}
{\AA}ke Bj{\"o}rck and Clazett Bowie.
\newblock An iterative algorithm for computing the best estimate of an
  orthogonal matrix.
\newblock {\em SIAM Journal on Numerical Analysis}, 8(2):358--364, 1971.

\bibitem{bolcskei2019optimal}
Helmut Bolcskei, Philipp Grohs, Gitta Kutyniok, and Philipp Petersen.
\newblock Optimal approximation with sparsely connected deep neural networks.
\newblock {\em SIAM Journal on Mathematics of Data Science}, 1(1):8--45, 2019.

\bibitem{bowman2015generating}
Samuel~R. Bowman, Luke Vilnis, Oriol Vinyals, Andrew Dai, Rafal Jozefowicz, and
  Samy Bengio.
\newblock Generating sentences from a continuous space.
\newblock In {\em Proceedings of The 20th {SIGNLL} Conference on Computational
  Natural Language Learning}, pages 10--21, Berlin, Germany, August 2016.
  Association for Computational Linguistics.

\bibitem{brock2016neural}
Andrew Brock, Theodore Lim, James~M Ritchie, and Nick Weston.
\newblock Neural photo editing with introspective adversarial networks.
\newblock {\em arXiv preprint arXiv:1609.07093}, 2016.

\bibitem{chen2021learning}
Xiaoli Chen, Jinqiao Duan, and George~Em Karniadakis.
\newblock Learning and meta-learning of stochastic
  advection--diffusion--reaction systems from sparse measurements.
\newblock {\em European Journal of Applied Mathematics}, 32(3):397--420, 2021.

\bibitem{darbon2020overcoming}
Jerome Darbon, Gabriel~P Langlois, and Tingwei Meng.
\newblock Overcoming the curse of dimensionality for some hamilton--jacobi
  partial differential equations via neural network architectures.
\newblock {\em Research in the Mathematical Sciences}, 7(3):1--50, 2020.

\bibitem{de2021approximation}
Tim De~Ryck, Samuel Lanthaler, and Siddhartha Mishra.
\newblock On the approximation of functions by tanh neural networks.
\newblock {\em arXiv preprint arXiv:2104.08938}, 2021.

\bibitem{flamary2021pot}
R{\'e}mi Flamary, Nicolas Courty, Alexandre Gramfort, Mokhtar~Z. Alaya,
  Aur{\'e}lie Boisbunon, Stanislas Chambon, Laetitia Chapel, Adrien Corenflos,
  Kilian Fatras, Nemo Fournier, L{\'e}o Gautheron, Nathalie~T.H. Gayraud,
  Hicham Janati, Alain Rakotomamonjy, Ievgen Redko, Antoine Rolet, Antony
  Schutz, Vivien Seguy, Danica~J. Sutherland, Romain Tavenard, Alexander Tong,
  and Titouan Vayer.
\newblock Pot: Python optimal transport.
\newblock {\em Journal of Machine Learning Research}, 22(78):1--8, 2021.

\bibitem{goodfellow2014generative}
Ian Goodfellow, Jean Pouget-Abadie, Mehdi Mirza, Bing Xu, David Warde-Farley,
  Sherjil Ozair, Aaron Courville, and Yoshua Bengio.
\newblock Generative adversarial nets.
\newblock In {\em Advances in Neural Information Processing Systems},
  volume~27. Curran Associates, Inc., 2014.

\bibitem{graepel2003solving}
Thore Graepel.
\newblock Solving noisy linear operator equations by gaussian processes:
  Application to ordinary and partial differential equations.
\newblock In {\em ICML}, volume~3, pages 234--241, 2003.

\bibitem{gulrajani2017improved}
Ishaan Gulrajani, Faruk Ahmed, Martin Arjovsky, Vincent Dumoulin, and Aaron~C
  Courville.
\newblock Improved training of wasserstein gans.
\newblock In {\em Advances in Neural Information Processing Systems},
  volume~30. Curran Associates, Inc., 2017.

\bibitem{guo2016convolutional}
Xiaoxiao Guo, Wei Li, and Francesco Iorio.
\newblock Convolutional neural networks for steady flow approximation.
\newblock In {\em Proceedings of the 22nd ACM SIGKDD international conference
  on knowledge discovery and data mining}, pages 481--490, 2016.

\bibitem{han2020algorithms}
Jiequn Han, Arnulf Jentzen, et~al.
\newblock Algorithms for solving high dimensional pdes: From nonlinear monte
  carlo to machine learning.
\newblock {\em arXiv preprint arXiv:2008.13333}, 2020.

\bibitem{hopf1950partial}
Eberhard Hopf.
\newblock The partial differential equation ut+ uux= $\mu$xx.
\newblock {\em Communications on Pure and Applied mathematics}, 3(3):201--230,
  1950.

\bibitem{hutzenthaler2020overcoming}
Martin Hutzenthaler, Arnulf Jentzen, Thomas Kruse, Tuan Anh~Nguyen, and
  Philippe von Wurstemberger.
\newblock Overcoming the curse of dimensionality in the numerical approximation
  of semilinear parabolic partial differential equations.
\newblock {\em Proceedings of the Royal Society A}, 476(2244):20190630, 2020.

\bibitem{jagtap2020locally}
Ameya~D Jagtap, Kenji Kawaguchi, and George Em~Karniadakis.
\newblock Locally adaptive activation functions with slope recovery for deep
  and physics-informed neural networks.
\newblock {\em Proceedings of the Royal Society A}, 476(2239):20200334, 2020.

\bibitem{jagtap2020adaptive}
Ameya~D Jagtap, Kenji Kawaguchi, and George~Em Karniadakis.
\newblock Adaptive activation functions accelerate convergence in deep and
  physics-informed neural networks.
\newblock {\em Journal of Computational Physics}, 404:109136, 2020.

\bibitem{jentzen2018proof}
Arnulf Jentzen, Diyora Salimova, and Timo Welti.
\newblock A proof that deep artificial neural networks overcome the curse of
  dimensionality in the numerical approximation of kolmogorov partial
  differential equations with constant diffusion and nonlinear drift
  coefficients.
\newblock {\em arXiv preprint arXiv:1809.07321}, 2018.

\bibitem{kingma2014adam}
Diederik~P Kingma and Jimmy Ba.
\newblock Adam: A method for stochastic optimization.
\newblock {\em arXiv preprint arXiv:1412.6980}, 2014.

\bibitem{kingma2013auto}
Diederik~P Kingma and Max Welling.
\newblock Auto-encoding variational bayes.
\newblock {\em arXiv preprint arXiv:1312.6114}, 2013.

\bibitem{krizhevsky2012imagenet}
Alex Krizhevsky, Ilya Sutskever, and Geoffrey~E Hinton.
\newblock Imagenet classification with deep convolutional neural networks.
\newblock {\em Advances in neural information processing systems},
  25:1097--1105, 2012.

\bibitem{lagaris1998artificial}
Isaac~E Lagaris, Aristidis Likas, and Dimitrios~I Fotiadis.
\newblock Artificial neural networks for solving ordinary and partial
  differential equations.
\newblock {\em IEEE transactions on neural networks}, 9(5):987--1000, 1998.

\bibitem{lei2020convergence}
Jing Lei.
\newblock Convergence and concentration of empirical measures under wasserstein
  distance in unbounded functional spaces.
\newblock {\em Bernoulli}, 26(1):767--798, 2020.

\bibitem{li2020neural}
Zongyi Li, Nikola Kovachki, Kamyar Azizzadenesheli, Burigede Liu, Kaushik
  Bhattacharya, Andrew Stuart, and Anima Anandkumar.
\newblock Neural operator: Graph kernel network for partial differential
  equations.
\newblock {\em arXiv preprint arXiv:2003.03485}, 2020.

\bibitem{liang2018well}
Tengyuan Liang.
\newblock How well generative adversarial networks learn distributions.
\newblock {\em arXiv preprint arXiv:1811.03179}, 2018.

\bibitem{liang2014deep}
Znaonui Liang, Gang Zhang, Jimmy~Xiangji Huang, and Qmming~Vivian Hu.
\newblock Deep learning for healthcare decision making with emrs.
\newblock In {\em 2014 IEEE International Conference on Bioinformatics and
  Biomedicine (BIBM)}, pages 556--559. IEEE, 2014.

\bibitem{liu1989limited}
Dong~C Liu and Jorge Nocedal.
\newblock On the limited memory bfgs method for large scale optimization.
\newblock {\em Mathematical programming}, 45(1):503--528, 1989.

\bibitem{lu2019deeponet}
Lu~Lu, Pengzhan Jin, and George~Em Karniadakis.
\newblock Deeponet: Learning nonlinear operators for identifying differential
  equations based on the universal approximation theorem of operators.
\newblock {\em arXiv preprint arXiv:1910.03193}, 2019.

\bibitem{NEURIPS2020_2000f632}
Yulong Lu and Jianfeng Lu.
\newblock A universal approximation theorem of deep neural networks for
  expressing probability distributions.
\newblock In {\em Advances in Neural Information Processing Systems},
  volume~33, pages 3094--3105. Curran Associates, Inc., 2020.

\bibitem{luo2020two}
Tao Luo and Haizhao Yang.
\newblock Two-layer neural networks for partial differential equations:
  Optimization and generalization theory.
\newblock {\em arXiv preprint arXiv:2006.15733}, 2020.

\bibitem{ma2021lecturenotes}
Tengyu Ma.
\newblock Lecture notes for machine learning theory.
\newblock 2021.

\bibitem{miotto2018deep}
Riccardo Miotto, Fei Wang, Shuang Wang, Xiaoqian Jiang, and Joel~T Dudley.
\newblock Deep learning for healthcare: review, opportunities and challenges.
\newblock {\em Briefings in bioinformatics}, 19(6):1236--1246, 2018.

\bibitem{mishra2020estimates}
Siddhartha Mishra and Roberto Molinaro.
\newblock Estimates on the generalization error of physics informed neural
  networks (pinns) for approximating pdes.
\newblock {\em arXiv preprint arXiv:2006.16144}, 2020.

\bibitem{odena2017conditional}
Augustus Odena, Christopher Olah, and Jonathon Shlens.
\newblock Conditional image synthesis with auxiliary classifier gans.
\newblock In {\em International conference on machine learning}, pages
  2642--2651. PMLR, 2017.

\bibitem{pang2019fpinns}
Guofei Pang, Lu~Lu, and George~Em Karniadakis.
\newblock fpinns: Fractional physics-informed neural networks.
\newblock {\em SIAM Journal on Scientific Computing}, 41(4):A2603--A2626, 2019.

\bibitem{pinkus1999approximation}
Allan Pinkus.
\newblock Approximation theory of the mlp model.
\newblock {\em Acta Numerica 1999: Volume 8}, 8:143--195, 1999.

\bibitem{psichogios1992hybrid}
Dimitris~C Psichogios and Lyle~H Ungar.
\newblock A hybrid neural network-first principles approach to process
  modeling.
\newblock {\em AIChE Journal}, 38(10):1499--1511, 1992.

\bibitem{raissi2019physics}
Maziar Raissi, Paris Perdikaris, and George~E Karniadakis.
\newblock Physics-informed neural networks: A deep learning framework for
  solving forward and inverse problems involving nonlinear partial differential
  equations.
\newblock {\em Journal of Computational Physics}, 378:686--707, 2019.

\bibitem{raissi2018numerical}
Maziar Raissi, Paris Perdikaris, and George~Em Karniadakis.
\newblock Numerical gaussian processes for time-dependent and nonlinear partial
  differential equations.
\newblock {\em SIAM Journal on Scientific Computing}, 40(1):A172--A198, 2018.

\bibitem{shin2020convergence}
Yeonjong Shin, Jerome Darbon, and George~Em Karniadakis.
\newblock On the convergence of physics informed neural networks for linear
  second-order elliptic and parabolic type pdes.
\newblock {\em arXiv preprint arXiv:2004.01806}, 2020.

\bibitem{siegel2020approximation}
Jonathan~W Siegel and Jinchao Xu.
\newblock Approximation rates for neural networks with general activation
  functions.
\newblock {\em Neural Networks}, 128:313--321, 2020.

\bibitem{siegel2021high}
Jonathan~W Siegel and Jinchao Xu.
\newblock High-order approximation rates for shallow neural networks with
  cosine and reluk activation functions.
\newblock {\em Applied and Computational Harmonic Analysis}, 2021.

\bibitem{sirignano2018dgm}
Justin Sirignano and Konstantinos Spiliopoulos.
\newblock Dgm: A deep learning algorithm for solving partial differential
  equations.
\newblock {\em Journal of computational physics}, 375:1339--1364, 2018.

\bibitem{stuart2010inverse}
Andrew~M Stuart.
\newblock Inverse problems: a bayesian perspective.
\newblock {\em Acta numerica}, 19:451--559, 2010.

\bibitem{tanielian2021approximating}
Ugo Tanielian and Gerard Biau.
\newblock Approximating lipschitz continuous functions with groupsort neural
  networks.
\newblock In {\em International Conference on Artificial Intelligence and
  Statistics}, pages 442--450. PMLR, 2021.

\bibitem{vinyals2015grammar}
Oriol Vinyals, {\L}ukasz Kaiser, Terry Koo, Slav Petrov, Ilya Sutskever, and
  Geoffrey Hinton.
\newblock Grammar as a foreign language.
\newblock {\em Advances in neural information processing systems},
  28:2773--2781, 2015.

\bibitem{wang2014similarity}
Bo~Wang, Aziz~M Mezlini, Feyyaz Demir, Marc Fiume, Zhuowen Tu, Michael Brudno,
  Benjamin Haibe-Kains, and Anna Goldenberg.
\newblock Similarity network fusion for aggregating data types on a genomic
  scale.
\newblock {\em Nature methods}, 11(3):333, 2014.

\bibitem{wang2020and}
Sifan Wang, Xinling Yu, and Paris Perdikaris.
\newblock When and why pinns fail to train: A neural tangent kernel
  perspective.
\newblock {\em arXiv preprint arXiv:2007.14527}, 2020.

\bibitem{wojtowytsch2020representation}
Stephan Wojtowytsch et~al.
\newblock Representation formulas and pointwise properties for barron
  functions.
\newblock {\em arXiv preprint arXiv:2006.05982}, 2020.

\bibitem{yang2021b}
Liu Yang, Xuhui Meng, and George~Em Karniadakis.
\newblock B-pinns: Bayesian physics-informed neural networks for forward and
  inverse pde problems with noisy data.
\newblock {\em Journal of Computational Physics}, 425:109913, 2021.

\bibitem{yang2019adversarial}
Yibo Yang and Paris Perdikaris.
\newblock Adversarial uncertainty quantification in physics-informed neural
  networks.
\newblock {\em Journal of Computational Physics}, 394:136--152, 2019.

\bibitem{yang2021capacity}
Yunfei Yang, Zhen Li, and Yang Wang.
\newblock On the capacity of deep generative networks for approximating
  distributions.
\newblock {\em arXiv preprint arXiv:2101.12353}, 2021.

\bibitem{yarotsky2017error}
Dmitry Yarotsky.
\newblock Error bounds for approximations with deep relu networks.
\newblock {\em Neural Networks}, 94:103--114, 2017.

\bibitem{zhang2019quantifying}
Dongkun Zhang, Lu~Lu, Ling Guo, and George~Em Karniadakis.
\newblock Quantifying total uncertainty in physics-informed neural networks for
  solving forward and inverse stochastic problems.
\newblock {\em Journal of Computational Physics}, 397:108850, 2019.

\bibitem{zhu2018bayesian}
Yinhao Zhu and Nicholas Zabaras.
\newblock Bayesian deep convolutional encoder--decoder networks for surrogate
  modeling and uncertainty quantification.
\newblock {\em Journal of Computational Physics}, 366:415--447, 2018.

\end{thebibliography}
\end{document}